%% file: Arxiv3.tex
\pdfoutput=1
\documentclass[11pt, letterpaper]{article}

\usepackage{algpseudocode} 
\usepackage{algorithmicx}
\usepackage{algorithm}
\input{stuff.tex}

\title{Accelerated Stochastic Matrix Inversion:  General Theory and  Speeding up BFGS Rules for Faster Second-Order Optimization}

\author{Robert M. Gower\footnote{T\'{e}l\'{e}com ParisTech, Paris, France} \and Filip Hanzely\footnote{King Abdullah University of Science and Technology, Thuwal, Saudi Arabia} \and Peter Richt\'{a}rik\footnote{King Abdullah University of Science and Technology, Thuwal, Saudi Arabia --- University of Edinburgh, Edinburgh, United Kingdom --- Moscow Institute of Physics and Technology, Moscow, Russia} \and Sebastian Stich\footnote{\'{E}cole polytechnique f\'{e}d\'{e}rale de Lausanne (EPFL), Lausanne, Switzerland}}

\date{June 13, 2018}

\begin{document}

\maketitle

\begin{abstract} We present the first accelerated randomized algorithm for solving linear systems in Euclidean spaces. One essential problem of this type is the matrix inversion problem. In particular, our algorithm can be specialized to invert positive definite matrices in such a way that all iterates (approximate solutions) generated by the algorithm are positive definite matrices themselves. This opens the way for many applications in the field of optimization and machine learning.  As an application of our general theory, we develop the {\em first  accelerated (deterministic and stochastic) quasi-Newton updates}. Our updates lead to provably more aggressive approximations of the inverse Hessian, and lead to speed-ups over classical non-accelerated rules in numerical experiments. Experiments with empirical risk minimization show that our rules can accelerate training of machine learning models.
\end{abstract}


\section{Introduction}
\label{s:into}


%
 Consider the optimization problem
\begin{equation}\label{eq:opt_main}\min_{w\in \R^n} f(w),\end{equation}
and assume $f$ is sufficiently smooth. A new wave of second order stochastic methods are being developed with the aim of solving large scale optimization problems. In particular, many of these new methods are based on stochastic BFGS updates~\cite{Schraudolph2007,wang2017stochastic,Mokhtari2014, moritz2016linearly,
Byrd2015}. Here we develop a new stochastic accelerated BFGS update that can form the basis of new stochastic quasi-Newton methods. 

Another approach to scaling up second order methods is to use randomized~\emph{sketching} to reduce the dimension, and hence the complexity of the Hessian and the updates involving the Hessian~ \cite{PilanciW17,Xu2016}, or \emph{subsampled} 
  Hessian matrices  when the objective function is a sum of many loss functions~\cite{BerahasBN17,agarwal2017second}. 

The starting point for developing second order methods is arguably Newton's method, which performs the iterative process
\begin{align}
 w_{k+1} = w_k - (\nabla^2 f(w_k))^{-1} \nabla f(w_k),
\end{align}
where $\nabla^2 f(w_k)$ and $\nabla f(w_k)$ are the Hessian and gradient of $f$, respectively. However, it is inefficient for solving large scale problems as it requires the computation of the Hessian and then solving a linear system in each iteration. Several methods have been developed to address this issue, based on the idea of approximating the exact update.

\emph{Quasi-Newton} methods, in particular the BFGS~\cite{broyden1967quasi,fletcher1970new,goldfarb1970family,shanno1970conditioning}, have been the leading optimization algorithm in various fields since the late 60's until the rise of big data, which brought a need for simpler first order algorithms. It is well known that Nesterov's acceleration \cite{nesterov1983method} is a reliable way to speed up first order methods. However until now, acceleration techniques have been applied exclusively to speeding up gradient updates. In this paper we present an accelerated BFGS algorithm, opening up new applications for acceleration. The acceleration in fact comes from an accelerated algorithm for inverting the Hessian matrix.

To be more specific, recall that quasi-Newton rules aim to maintain an estimate of the inverse Hessian $X_k$, adjusting it every iteration so that the inverse Hessian acts appropriately in a particular direction, while enforcing symmetry:
\begin{equation} \label{eq:bfgs_const}
X_k(\nabla f(w_{k})- \nabla f(w_{k-1})) =w_{k}-w_{k-1}, \qquad X_k  =X^\top_k.
\end{equation}



A notable research direction is the development of stochastic quasi-Newton methods~\cite{Gower:2017}, where the estimated inverse is equal to the true inverse over a subspace:
\begin{equation}
X_k\nabla^2f(w_k) S_k=S_k, \qquad X_k=X^\top_k,
\label{eq:sbfgs}
\end{equation}
where $S_k \in \R^{n \times \tau}$ is a randomly generated matrix.

In fact, \eqref{eq:sbfgs} can be seen as the so called sketch-and-project iteration for inverting $\nabla^2f(w_k)$. In this paper we first develop the accelerated algorithm for inverting positive definite matrices. As a direct application, our algorithm can be used as a primitive in quasi-Newton methods
 which results in a novel accelerated (stochastic) quasi-Newton method of the type \eqref{eq:sbfgs}. In addition, our acceleration technique can also be incorporated in the classical (non stochastic) BFGS method. This results in the accelerated BFGS method. Whereas the matrix inversion contribution is accompanied by strong theoretical justifications, this does not apply to the latter. Rather, we verify the effectiveness of this new  accelerated BFGS method through numerical experiments.

%

\subsection{Sketch-and-project for linear systems}
Our accelerated algorithm can be applied to more general tasks than only inverting matrices. In its most general form, it can be seen as an accelerated version of a  \emph{sketch-and-project} method in Euclidean spaces which we present now.
Consider a linear system $Ax=b$ such that $b\in \Range{A}$. One step of the 
sketch-and-project algorithm reads as:
\begin{equation} \label{eq:sap}
x_{k+1}=\argmin_{x} \; \norm{x_k-x}_B^2 \quad \text{subject to} \quad S_k^\top Ax=S_k^\top b,
\end{equation}
where $\norm{x}^2_B=\dotprod{Bx,x}$ for some $B\succ 0$ and $S_k$ is a random sketching matrix sampled i.i.d at each iteration from a fixed distribution.

Randomized Kaczmarz~\cite{K-1937,strohmer2009randomized} was the first algorithm of this type.
In~\cite{Gower2015}, this sketch-and-project algorithm was analyzed in its full generality.
Note that the dual problem of~\eqref{eq:sap} takes the form of a quadratic minimization problem \cite{Gower2015c}, and randomized methods such as coordinate descent~\cite{Nesterov12,Wright:2015}, random pursuit~\cite{Stich14,Stich2016} or stochastic dual ascent~\cite{Gower2015c} can thus also be captured as special instances of this method. Richt\'{a}rik and Tak\'{a}\v{c} \cite{RT2017_stoch_reformulations} adopt a new point of view through a theory of stochastic reformulations of linear systems. In addition, they consider the addition of a relaxation parameter, as well as  mini-batch and accelerated variants. Acceleration was only achieved for the expected iterates, and not in the L2 sense as we do here. We refer to Richt\'{a}rik and Tak\'{a}\v{c}  \cite{RT2017_stoch_reformulations} for interpretation of sketch-and-project as stochastic gradient descent, stochastic Newton, stochastic proximal point method, and stochastic fixed point method.

Gower \cite{Gower:2017} observed that the procedure~\eqref{eq:sap} can also be applied to find the inverse of a matrix. Assume the optimization variable itself is a matrix, $x=X$, $b= I$, the identity matrix, then sketch-and-project  converges (under mild assumptions) to a solution of $AX=I$. Even the symmetry constraint $X =X^\top$ can be incorporated into the sketch-and-project framework since it is a linear constraint.

There has been recent development in speeding up the sketch-and-project method using the idea of Nesterov's acceleration \cite{nesterov1983method}. In~\cite{Liu:2016} an accelerated Kaczmarz algorithm was presented for special sketches of rank one. Arbitrary sketches of rank one where considered in~\cite{Stich14}, block sketches in~\cite{Nesterov:2017} and recently, Tu and coathors \cite{TuVWGJR17} developed acceleration for special sketching matrices, assuming the matrix $A$ is square. This assumption, along with any assumptions on $A$, was later dropped in~\cite{MartinRichtarikAccell}. Another notable way to accelerate the sketch-and-project algorithm is by using momentum or stochastic momentum~\cite{loizou2017momentum}.

We build on recent work of Richt\'{a}rik and Tak\'{a}\v{c}  \cite{MartinRichtarikAccell} and further extend their analysis by studying accelerated sketch-and-project in general Euclidean spaces. This allows us to deduce the result for matrix inversion as a special case. However, there is one additional caveat that has to be considered for the intended application in quasi-Newton methods: ideally, all iterates of the algorithm should be symmetric positive definite matrices. This is not the case in general, but we address this problem by constructing special sketch operators that preserve symmetry and positive definiteness.

\section{Contributions} \label{sec:contrib}
We now present our main contributions.\\
\textbf{Accelerated Sketch and Project in Euclidean Spaces.} We generalize the analysis of an accelerated version of the sketch-and-project algorithm~\cite{MartinRichtarikAccell} to linear operator systems in Euclidean spaces. We provide a self-contained convergence analysis, recovering the original results in a more general setting. 

\textbf{Faster Algorithms for Matrix Inversion.} We develop an accelerated algorithm for inverting positive definite matrices. This algorithm can be seen as a special case of the accelerated sketch-and-project in Euclidean space, thus its convergence follows from the main theorem. However, we also provide a different formulation of the proof that is specialized to this setting. Similarly to~\cite{TuVWGJR17}, the performance of the algorithm depends on two parameters $\mu$ and $\nu$ that capture spectral properties of the input matrix and the sketches that are used. 
Whilst for the non-accelerated sketch-and-project algorithm for matrix inversion~\cite{Gower:2017} the knowledge of these parameters is not necessary, they need to be given as input to the accelerated scheme. When employed with the correct choice of parameters, the accelerated algorithm is always faster than the non-accelerated one. We also provide a theoretical rate for sub-optimal parameters $\mu, \nu$, and we perform numerical experiments to argue the choice of $\mu, \nu$ in practice.

\textbf{Randomized Accelerated Quasi-Newton.} 
The proposed iterative algorithm for matrix inversion is designed in such a way that each iterate is a symmetric matrix. This means, we can use the generated approximate solutions as estimators for the inverse Hessian in quasi-Newton methods, which is a direct extension of stochastic quasi-Newton methods. To the best of our knowledge, this yields the first accelerated (stochastic) quasi-Newton method.

\textbf{Accelerated Quasi-Newton.}
In the standard BFGS method the updates to the Hessian estimate are not chosen randomly, but deterministically. Based on the intuition gained from the accelerated random method, we propose an accelerated scheme for BFGS. The main idea is that we replace the random sketching of the Hessian with a deterministic update. The theoretical convergence rates do not transfer to this scheme, but we demonstrate by numerical experiments that it is possible to choose a parameter combination which yields a slightly faster convergence. We believe that the novel idea of accelerating BFGS update is extremely valuable, as until now, acceleration techniques were only considered to improve gradient updates.


\subsection{Outline} 
Our accelerated sketch-and-project algorithm for solving linear systems in Euclidean spaces is developed and analyzed in Section~\ref{sec:ami_paper}, and is used later in Section~\ref{sec:asqn_pap} to analyze an accelerated sketch-and-project algorithm for matrix inversion. The accelerated sketch-and-project algorithm for matrix inversion is then used to accelerate the BFGS update, which in term leads to the development of an accelerated BFGS optimization method. Lastly in Section~\ref{sec:num_pap}, we perform numerical experiments to gain  different insights into the newly developed methods.  Proofs of all results and additional insights can be found in the appendix.

\section{Accelerated Stochastic Algorithm for Matrix Inversion \label{sec:ami_paper}}
In this section we propose an accelerated randomized algorithm to solve linear systems in Euclidean spaces. 
This is a very general problem class and it comprises for instance also the matrix inversion problem.
Thus, we will use the result of this section later to analyze our newly proposed matrix inversion algorithm, which we then use to estimate the inverse of the Hessian within a quasi-Newton method.\footnote{Quasi-Newton methods do not compute an exact matrix inverse, rather, they only compute an incremental update. Thus, it suffices to apply \emph{one step} of our proposed scheme per iteration. This will be detailed in Section~\ref{sec:asqn_pap}.}

Let $\cX$ and $\cY$ be finite dimensional Euclidean spaces and 
let $\A:\cX \mapsto \cY$ be  a linear operator. 
Let $L(\cX,\cY)$ denote the space of linear operators that map from $\cX$ to $\cY.$ Consider the linear system  
\begin{equation}\label{eq:system} \A x = b,\end{equation}
where $x\in \cX$ and $b\in \Range{\A}.$
Consequently there exists a solution to the equation~\eqref{eq:system}. In particular,  we aim to find the solution closest to a given initial point $x_0 \in \cX$:
\begin{equation} \label{eq:primal}
x^* \eqdef \arg\min_{x \in \cX} \tfrac{1}{2}\norm{x-x_0}^2 \quad \mbox{subject to} \quad \A x = b.
\end{equation}
Using the pseudoinverse and Lemma~\ref{lem:pseudo} item~\emph{\ref{it:pseudoleastnorm}}, the solution to~\eqref{eq:primal} is given by
\begin{equation}\label{eq:xsol}
x^* = x_0- \A^{\dagger}(\A x_0 -b) \in x_0+ \Range{\A^*},
\end{equation}
where $A^{\dagger}$ and $A^*$ denote the pseudoinverse and the adjoint of $A,$ respectively.

\subsection{The algorithm}

Let $\cZ$ be a  Euclidean space and consider a random linear operator $\cS_k \in L(\cY,\cZ)$ chosen from some distribution $\cD$ over $L(\cY,\cZ)$ at iteration $k$.
Our method is given in Algorithm~\ref{alg:SketchJac}, where $Z_k \in L(\cX)$ is a random linear operator given by the following compositions
\begin{equation}\label{eq:Z}
Z_k =Z(\cS_k) \eqdef \A^*\cS_k^*(\cS_k\A\A^*\cS_k^*)^{\dagger}\cS_k \A.
\end{equation}
The updates of variables $g_k$ and $x_{k+1}$ on lines~8 and~9, respectively, correspond to what is known as the \emph{sketch-and-project} update:
\begin{equation}  \label{eq:sk}
 x_{k+1} =   \arg\min_{x \in \cX} \tfrac{1}{2}\norm{x-y_k}^2 \quad \text{subject to} \quad \cS_k \A x = \cS_k b, 
\end{equation}
which can also be written as the following operation 
\begin{equation}\label{eq:IZprojres}
x_{k+1} - x_*  = (I- Z_k)(y_k -x_*). 
\end{equation}
This follows from the fact that $b \in \Range{\A}$, together with item~\ref{it:pseudoTTdagT} of Lemma~\ref{lem:pseudo}. 
Furthermore, note that the adjoint $\A^*$ and the pseudoinverse in Algorithm~\ref{alg:SketchJac} are taken with respect to the norm  in~\eqref{eq:primal}. 


\begin{algorithm}[!h]
\begin{algorithmic}[1]
\State \textbf{Parameters:} $\mu, \nu >0$, ${\cal D}$ = distribution over random linear operators.
\State Choose  $x_0\in \cX$ and set $v_0 = x_0$, $\beta  =1 - \sqrt{\frac{\mu}{\nu}},$ $\gamma = \sqrt{\frac{1}{\mu \nu}},$ $\alpha = \frac{1}{1+\gamma\nu}.$  
\For{$k =  0, 1, \dots$}
\State $y_k = \alpha v_k + (1-\alpha) x_k$ 
	\State Sample an independent copy $S_k\sim {\cal D}$
	\State 
	$g_k = \A^*\cS_k^*(\cS_k\A\A^*\cS_k^*)^{\dagger}\cS_k(\A y_k -b)=Z_k(y_k -x_*)$  \label{ln:stochgrad}	
\State $x_{k+1} =y_k -g_k$\label{ln:xupdate}
	\State $v_{k+1} = \beta v_k +(1-\beta)y_k - \gamma g_k$ \label{ln:vupdate}
\EndFor
\end{algorithmic}
\caption{Accelerated  Sketch-and-Project  for solving \eqref{eq:sk}  \cite{MartinRichtarikAccell}}
\label{alg:SketchJac}
\end{algorithm}

Algorithm~\ref{alg:SketchJac} was first proposed and analyzed by Richt\'{a}rik and Tak\'{a}\v{c}  \cite{MartinRichtarikAccell} in the special case when $\cX = \R^n$ and $\cY = \R^m$. Our contribution here is in extending the algorithm and analysis to the more abstract setting of Euclidean spaces.  In addition, we provide some further extensions of this method in Sections~\ref{sec:omega} and~\ref{sec:alpha}, allowing for a non-unit stepsize  and variable $\alpha$, respectively.

\subsection{Key assumptions and quantities}
Denote $Z=Z(\cS)$ for $\cS\sim \cD$. Assume that the \emph{exactness property} holds
\begin{equation}
\Null{\A} = \Null{\E{Z}} ; 
 \label{eq:exactness}
\end{equation}
this is also equivalent to $\Range{\A^*} = \Range{\E{Z}}$. The exactness assumption is of key importance in the sketch-and-project framework, and indeed it is not very strong. For example, it holds for the matrix inversion problem with every sketching strategy we consider. 
We further assume that $\A \neq 0$ and $\E{Z}$ is finite. First we collect a few observation on the $Z$ operator

\begin{lemma}\label{lem:Z}
The $Z$ operator~\eqref{eq:Z} is a self-adjoint positive projection. Consequently $\E{Z}$ is a self-adjoint positive operator.
\end{lemma}

The two parameters that govern the acceleration are
\begin{equation}
 \label{eq:mu+nu}
\mu  \eqdef   \inf_{x \in \Range{\A^*}} \tfrac{\dotprod{\E{Z}x,x}}{\dotprod{x,x}}, \qquad \qquad 
\nu  \eqdef   \sup_{x \in \Range{\A^*}} \tfrac{\dotprod{\E{Z\E{Z}^\dagger Z}x,x}}{\dotprod{\E{Z}x,x}}.
\end{equation}

The supremum in the definition of $\nu$ is well defined due to the exactness assumption together with $\A \neq 0.$

\begin{lemma} \label{lem:Z_bounds}
We have 
\begin{equation} \label{eq:nubnds}
1\quad \leq \quad\nu\quad \leq \quad  \tfrac{1}{\mu} \quad = \quad  \norm{\E{Z}^{\dagger}}.
\end{equation}
Moreover, if $\Range{\A^*}=\cX$, we have
\begin{equation} \label{eq:nu_lower}
\tfrac{\Rank{\A^*} }{\E{\Rank{Z}}}\leq \nu.
\end{equation}
\end{lemma}

\subsection{Convergence and change of the norm}
 For a positive self-adjoint $G \in L(\cX)$  and $x \in \cX$ let $\norm{x}_G \eqdef \sqrt{\dotprod{x,x}_G} \eqdef \sqrt{\dotprod{Gx,x}}$. We now informally state the convergence rate of Algorithm \ref{alg:SketchJac}. Theorem~\ref{theo:conv} generalizes  the main theorem from \cite{MartinRichtarikAccell} to linear systems in Euclidean spaces.

\begin{theorem}\label{theo:conv} Let $x_k, v_k$ be the random iterates of Algorithm \ref{alg:SketchJac}. Then
\[
 \E{\norm{v_{k} -x_*}_{\E{Z}^\dagger}^2 +\tfrac{1}{\mu}\norm{x_{k} -x_*}^2 }  \leq \left(1 - \sqrt{\tfrac{\mu}{\nu}} \right)^k \E{\norm{v_0 -x_*}_{\E{Z}^\dagger}^2 + \tfrac{1}{\mu}\norm{x_0-x_*}^2}.
\]
\end{theorem}

This theorem shows the accelerated Sketch-and-Project algorithm converges linearly with a rate of $ \bigl( 1-\sqrt{\frac{\mu}{\nu}} \bigr),$ which translates to a total of 
$O(\sqrt{\nu/\mu}\log\left( 1/\epsilon\right) )$ iterations to bring the given error in Theorem~\ref{theo:conv} below $\epsilon >0.$ This is in contrast with the non-accelerated 
Sketch-and-Project algorithm which requires $O((1/\mu)\log\left( 1/\epsilon\right) )$ iterations, as shown in~\cite{Gower2015} for solving linear systems. From~\eqref{eq:nubnds}, we have the bounds
$1/\sqrt{\mu}  \leq  \sqrt{\nu/\mu}  \leq 1/\mu.$
On one extreme, this  inequality shows that the iteration complexity of the accelerated algorithm is at least as good as its non-accelerated counterpart. On the other extreme, the accelerated algorithm might require as little as the square root of the number of iterations of its non-accelerated counterpart. Since the cost of a single iteration of the accelerated algorithm is of the same order as the non-accelerated algorithm, this theorem shows that acceleration can offer a significant speed-up, which is verified numerically in Section~\ref{sec:num_pap}. 
It is also possible to get the convergence rate of accelerated sketch-and-project where projections are taken with respect to a different weighted norm. For technical details, see Section~\ref{sec:change_norm} of the Appendix. 

\subsection{Coordinate sketches with convenient probabilities \label{sec:convenient_munu}}
Let us consider a simple example in the setting for Algorithm~\ref{alg:SketchJac} where we can understand parameters $\mu, \nu$. 
 In particular, consider a linear system $Ax=b$ in $\R^n$ where $A$ is symmetric positive definite. 
 
\begin{corollary} \label{cor:sss}
 Choose $B=A$ and $S=e_i$ with probability proportional to $A_{i,i}$. Then
\begin{equation}
\mu=\tfrac{\lambda_{\min}(A)}{\Tr{A}} =: \mu^P \quad \mbox{and} \quad \nu = \tfrac{\Tr{A}}{\min_i A_{i,i}} =: \nu^P
\label{eq:munu_conv_paper}
\end{equation}
and therefore the convergence rate given in Theorem~\ref{theo:conv} for the accelerated algorithm is 
\begin{equation}\label{eq:89g9d8g098f}
\biggl( 1- \sqrt{\tfrac{\mu}{\nu}}\biggr)^k\quad  =\quad   \left( 1-\tfrac{\sqrt{\lambda_{\min}(A) \min_i A_{i,i}}}{\Tr{A}} \right) ^k.
\end{equation}
\end{corollary}

Rate \eqref{eq:89g9d8g098f} of our accelerated method is to be contrasted with the rate of the non-accelerated method: 
$
(1- \mu )^k  =  ( 1- \lambda_{\min}(A)/ \Tr{A}) ) ^k.
$
Clearly, we gain from acceleration if the smallest diagonal element of $A$ is significantly larger than the smallest eigenvalue.

 In fact, parameters $\mu^P, \nu^P$ above are the correct choice for the matrix inversion algorithm, when symmetry is not enforced, as we shall see later.  Unfortunately, we are not able to estimate the parameters while enforcing symmetry for different sketching strategies. We dedicate a section in numerical experiments to test, if the parameter selection \eqref{eq:munu_conv_paper} performs well under enforced symmetry and different sketching strategies, and also how one might safely choose $\mu,\nu$ in practice.


\section{Accelerated Stochastic BFGS Update \label{sec:asqn_pap}}
The update of the inverse Hessian used in quasi-Newton methods (e.g., in BFGS) can be seen as a sketch-and-project update applied to the linear system $AX=I$, while $X= X^\top$ is enforced, and where $A$ denotes and approximation of the Hessian.
In this section, we present an accelerated version of these updates. We provide two different proofs: one based on Theorem~\ref{theo:conv} and one based on vectorization.   By mimicking the updates of the accelerated stochastic BFGS method for inverting matrices, we determine a heuristic for accelerating the classic deterministic BFGS update. We then incorporate this acceleration into the classic BFGS optimization method and show that the resulting  algorithm can offer a speed-up of the standard BFGS algorithm.
 

\subsection{Accelerated matrix inversion}
Consider the symmetric positive definite matrix $A \in \R^{n \times n}$ and the following projection problem
\begin{equation}  \label{eq:primalqN}
 A^{-1} =  \arg\min_{X} \; \norm{X}_{F(A)}^2 \quad \text{subject to} \quad AX= I, \quad X = X^\top,
\end{equation}
where $\norm{X}_{F(A)}\eqdef \Tr{AX^\top AX} = \norm{A^{1/2}XA^{1/2}}_F^2.$
This projection problem can be cast as an instantiation of the general projection problem~\eqref{eq:primal}. Indeed, we need only note that the constraint in~\eqref{eq:primalqN} is linear and equivalent to
$\mathcal{A}(X) \eqdef \left(\begin{smallmatrix} AX \\ X -X^\top \end{smallmatrix} \right)= 
\left(\begin{smallmatrix} I \\ 0 \end{smallmatrix}\right). $
The matrix inversion problem can be efficiently solved using sketch-and-project with a symmetric sketch~\cite{Gower:2017}. The symmetric sketch is given by
$ \mathcal{S}_k\mathcal{A}(X)  = \left(\begin{smallmatrix} S_k^\top AX \\ X -X^\top \end{smallmatrix}\right), $
where $S_k \in \R^{n \times \tau}$ is a random matrix drawn from a distribution $\mathcal{D}$ and $\tau \in \N.$ The resulting sketch-and-project method is as follows
\begin{equation}  \label{eq:primalqNX}
 X_{k+1}=   \arg\min_{X} \; \norm{X - X_k}_{F(A)}^2 \quad   \text{subject to} \quad S_k^\top A X = S_k^\top, \quad X = X^\top,
\end{equation}
 the closed form solution of which is 
\begin{equation}
X_{k+1}  =S_k(S_k^\top A S_k)^{-1}S_k^\top + \left(I-S_k(S_k^\top AS_k)^{-1}S_k^\top A\right) X_{k} \left(I -AS_k(S_k^\top AS_k)^{-1}S_k^\top  \right).\label{eq:qunac} 
\end{equation}

By observing that~\eqref{eq:primalqNX} is the sketch-and-project algorithm applied to a linear operator equation, we have constructed an accelerated version in Algorithm~\ref{alg:qn}. We can also apply Theorem~\ref{theo:conv} to prove that 
 Algorithm~\ref{alg:qn} is indeed accelerated.
\begin{theorem}\label{theo:qn}
Let $L^k\eqdef \norm{V_k -A^{-1}}_{M}^2 + \tfrac{1}{\mu}\norm{X_k-A^{-1}}^2_{F(A)}$. The iterates of Algorithm~\ref{alg:qn} satisfy 
\begin{equation}\label{eq:qnaccconv}
 \E{L_{k+1}} \leq \left(1 - \sqrt{\tfrac{\mu}{\nu}} \right) \E{L_k},
\end{equation} 
where 
$\norm{X}_{M}^2 = \Tr{A^{1/2}X^\top A^{1/2} \E{Z}^\dagger A^{1/2} X A^{1/2}}.$
Furthermore,
\begin{equation}
\mu  \eqdef   \inf_{X \in \R^{n\times n}} \tfrac{\dotprod{\E{Z}X,X}}{\dotprod{X,X}}= \lambda_{\min}(\E{\bigZ}), \qquad \nu \eqdef   \sup_{X \in \R^{n\times n}} \tfrac{\dotprod{\E{Z\E{Z}^\dagger Z}X,X}}{\dotprod{\E{Z}X,X}},\label{eq:nuqn} 
\end{equation}
where 
\begin{equation}
\bigZ \eqdef  I\otimes I- (I-P)\otimes(I-P) , \qquad P \eqdef A^{1/2}S(S^\top AS)^{-1}S^\top A^{1/2}, \label{eq:bigz}
\end{equation}
and
$Z: X \in \R^{n\times n} \rightarrow \R^{n\times n}$ is given by
$Z(X) = X - \left(I-P\right) X\left(I -P \right) = XP + PX(I-P).$
Moreover, $ 2\lambda_{\min}(\E{P})\geq \lambda_{\min}(\E{\bigZ})\geq \lambda_{\min}(\E{P}).$
\end{theorem}

Notice that preserving symmetry yields $\mu =\lambda_{\min}(\E{\bigZ})$ , which can be up to twice as large as  $\lambda_{\min}(\E{P})$, which is the value of the $\mu$ parameter of the method without preserving symmetry. 
This improved rate is new, and was not present in the algorithm's debut publication~\cite{Gower:2017}.  In terms of parameter estimation, once symmetry is not preserved, we fall back onto the setting from Section~\ref{sec:convenient_munu}. Unfortunately, we were not able to quantify the effect of enforcing symmetry on the parameter $\nu$.

\begin{algorithm}[!h]
\begin{algorithmic}[1]
\State \textbf{Parameters:} $\mu, \nu >0$, ${\cal D}$ = distribution over random linear operators.
\State Choose  $X_0\in \cX$ and set $V_0 = X_0$, $\beta  =1 - \sqrt{\frac{\mu}{\nu}},$ $\gamma = \sqrt{\frac{1}{\mu \nu}},$ $\alpha = \frac{1}{1+\gamma\nu}$  
\For {$k =  0, 1, \dots$}
\State $Y_k = \alpha V_k + (1-\alpha) X_k$ 
	\State Sample an independent copy $S\sim {\cal D}$
	\State $X_{k+1} = Y_k + (Y_kA-I)S(S^\top AS)^{-1}S^\top  
- S(S^\top AS)^{-1}S^\top AY_k$ \\
		\qquad \qquad $+ S(S^\top AS)^{-1}S^\top AY_kAS(S^\top AS)^{-1}S^\top $  \label{ln:stochgradqn}	
	\State $V_{k+1} = \beta V_k +(1-\beta)Y_k - \gamma (Y_k-X_{k+1})$ \label{ln:vupdateqn}
\EndFor
\end{algorithmic}
\caption{Accelerated  BFGS matrix inversion (solving \eqref{eq:primalqN})}
\label{alg:qn}
\end{algorithm}

\subsection{Vectorizing -- a different insight}
Define ${\bf Vec}: \R^{n\times n}\rightarrow \R^{n^2}$ to be a vectorization operator of column-wise stacking and denote $x\eqdef \Vect{X}$. It can be shown that the sketch-and-project operation for matrix inversion \eqref{eq:primalqNX} is equivalent to 
\begin{equation*}  \label{eq:primalqNX}
 x_{k+1}=   \arg\min_{x}\; \norm{x- x_k}_{A\otimes A}^2  \quad \text{subject to} \quad  (I\otimes S_k^\top) (I\otimes A) x = (I\otimes S_k^\top) \Vect{I}, \; Cx=0,
\end{equation*}
where $C$ is defined so that  $Cx=0$ if and only if $X~=~X^\top$. The above is a sketch-and-project update for a linear system in $\R^{n^2}$, which allows to obtain an alternative proof of Theorem \ref{theo:qn}, without using our results from Euclidean spaces. The details are provided in Section~\ref{sec:alternate} of the Appendix.

\subsection{Accelerated BFGS as an optimization algorithm}
\label{sec:accBFGSmethod}
 As a tweak in the stochastic BFGS allows for a faster estimation of Hessian inverse and therefore more accurate steps of the method, one might wonder if a equivalent tweak might speed up the standard, deterministic BFGS algorithm for solving~\ref{eq:opt_main}. The mentioned tweaked version of standard BFGS is proposed as Algorithm~\ref{alg:bfgs_opt}. We do not state a convergence theorem for this algorithm---due to the deterministic updates the analysis is currently elusive---nor propose to use it as a default solver, but we rather introduce it as a novel idea for accelerating optimization algorithms. We leave theoretical analysis for the future work. For now, we perform several numerical experiments, in order to understand the potential and limitations of this new method.

\begin{algorithm}[!h]
\begin{algorithmic}[1]
\State \textbf{Parameters:} $\mu, \nu >0$, 
stepsize $\eta$.
\State Choose $X_0\in \cX$, $w_0$ and set $V_0 = X_0$, $\beta  =1 - \sqrt{\frac{\mu}{\nu}},$ $\gamma = \sqrt{\frac{1}{\mu \nu}},$ $\alpha = \frac{1}{1+\gamma\nu}.$  
\For {$k =  0, 1, \dots$}
\State $w_{k+1}=w_{k}-\eta X_{k}\nabla f(w_{k})$
\State $s_k = w_{k+1}-w_k$, \quad $\zeta_k = \nabla f(w_{k+1})- \nabla f(w_k)$
\State $Y_k = \alpha V_k + (1-\alpha) X_k$ 
		 
\State $X_{k+1} =   \frac{\delta_k \delta_k^\top}{\delta_k^\top \zeta_k}+ \left(I-\frac{\delta_k \zeta_k^\top}{\delta_k^\top \zeta_k} \right) Y_{k} \left(I -\frac{\zeta_k\delta_k^\top}{\delta_k^\top \zeta_k}   \right) $ \label{ln:updateX}
	\State $V_{k+1} = \beta V_k+(1-\beta)Y_k-\gamma (Y_k - X_{k+1})$ 
\EndFor
\end{algorithmic}
\caption{BFGS method with accelerated BFGS update for solving \eqref{eq:opt_main}}
\label{alg:bfgs_opt}
\end{algorithm}

To better understand Algorithm~\ref{alg:bfgs_opt}, recall that the BFGS updates an estimate of the inverse Hessian via
\begin{equation}
 X_{k+1}=\argmin_{X} \; \|X-X_k \|^2_{F(A)} \quad \text{subject to} \quad X\delta_k=\zeta_k ,\, X=X^\top, 
\end{equation}
where $\delta_k = w_{k+1}-w_k $ and $ \zeta_k = \nabla f(w_{k+1})- \nabla f(w_k)$. The above has the following  closed form solution
$
X_{k+1}=\tfrac{\delta_k \delta_k^\top}{\delta_k^\top \zeta_k}+ \left(I-\tfrac{\delta_k \zeta_k^\top}{\delta_k^\top \zeta_k} \right) X_{k} \left(I -\tfrac{\zeta_k\delta_k^\top}{\delta_k^\top \zeta_k}   \right).
$
This update appears on line~\ref{ln:updateX} of Algorithm~\ref{alg:bfgs_opt} with the difference being that it is applied to a matrix $Y_k$.

\section{Numerical Experiments \label{sec:num_pap}}

We perform extensive numerical experiments to bring additional insight to both the performance of and to parameter selection for Algorithms~\ref{alg:qn} and~\ref{alg:bfgs_opt}. More numerical experiments can be found in Section~\ref{sec:exp_appendix} of the appendix. We first test our  accelerated matrix inversion algorithm, and subsequently perform experiments related to Section~\ref{sec:accBFGSmethod}. 

\subsection{Accelerated Matrix Inversion} \label{sec:ex-BFGS-opt}

We consider the problem of inverting a matrix symmetric positive matrix $A$. 
We focus on a few particular choices of matrices $A$ (specified when describing each experiment), that differ in their eigenvalue spectra.
Three different sketching strategies are studied: Coordinate sketches with convenient probabilities ($S=e_i$ with probability proportional to $A_{i,i}$), coordinate sketches with uniform probabilities ($S=e_i$ with probability $\tfrac1n$) and Gaussian sketches ($S\sim \cN(0,I)$).
As matrices to be inverted, we use both artificially generated matrices with the access to the spectrum and also Hessians of ridge regression problems from LIBSVM. 

We have shown earlier that $\mu, \nu$ can be estimated as per \eqref{eq:munu_conv_paper} for coordinate sketches with convenient probabilities without enforcing symmetry. We use the mentioned parameters for the other sketching strategies while enforcing the symmetry. Since in practice one might not have an access to the exact parameters $\mu, \nu$ for given sketching strategy, we test sensitivity of the algorithm to parameter choice . We also test test for $\nu$  chosen by \eqref{eq:munu_conv_paper}, $\mu = \tfrac{1}{100\nu}$ and $\mu = \tfrac{1}{10000\nu}$.

\begin{figure}[!h]
    \centering
\begin{minipage}{0.25\textwidth}
  \centering
\includegraphics[width =  \textwidth ]{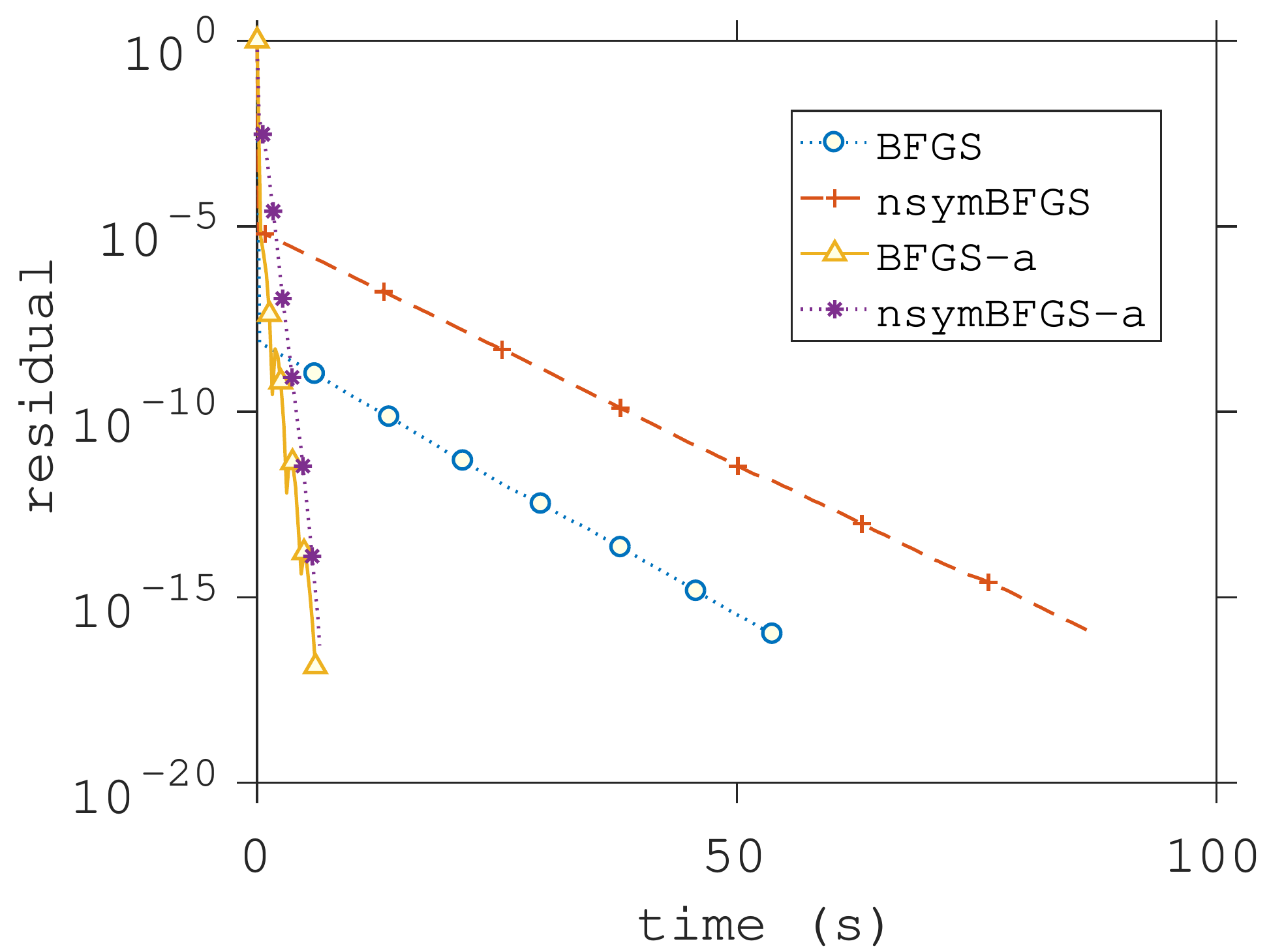}
\end{minipage}%
\begin{minipage}{0.25\textwidth}
  \centering
\includegraphics[width =  \textwidth ]{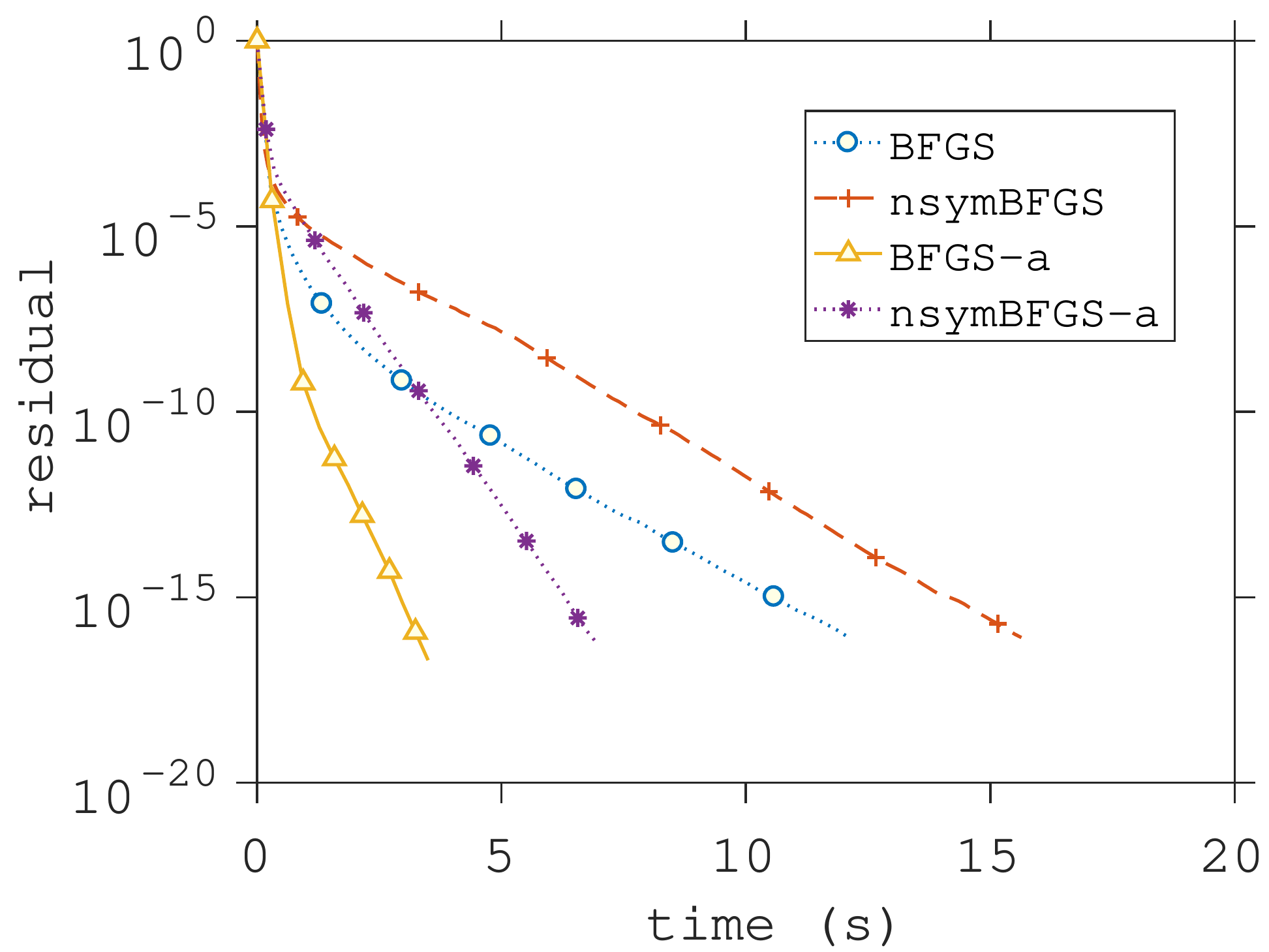}
\end{minipage}%
\begin{minipage}{0.25\textwidth}
  \centering
\includegraphics[width =  \textwidth ]{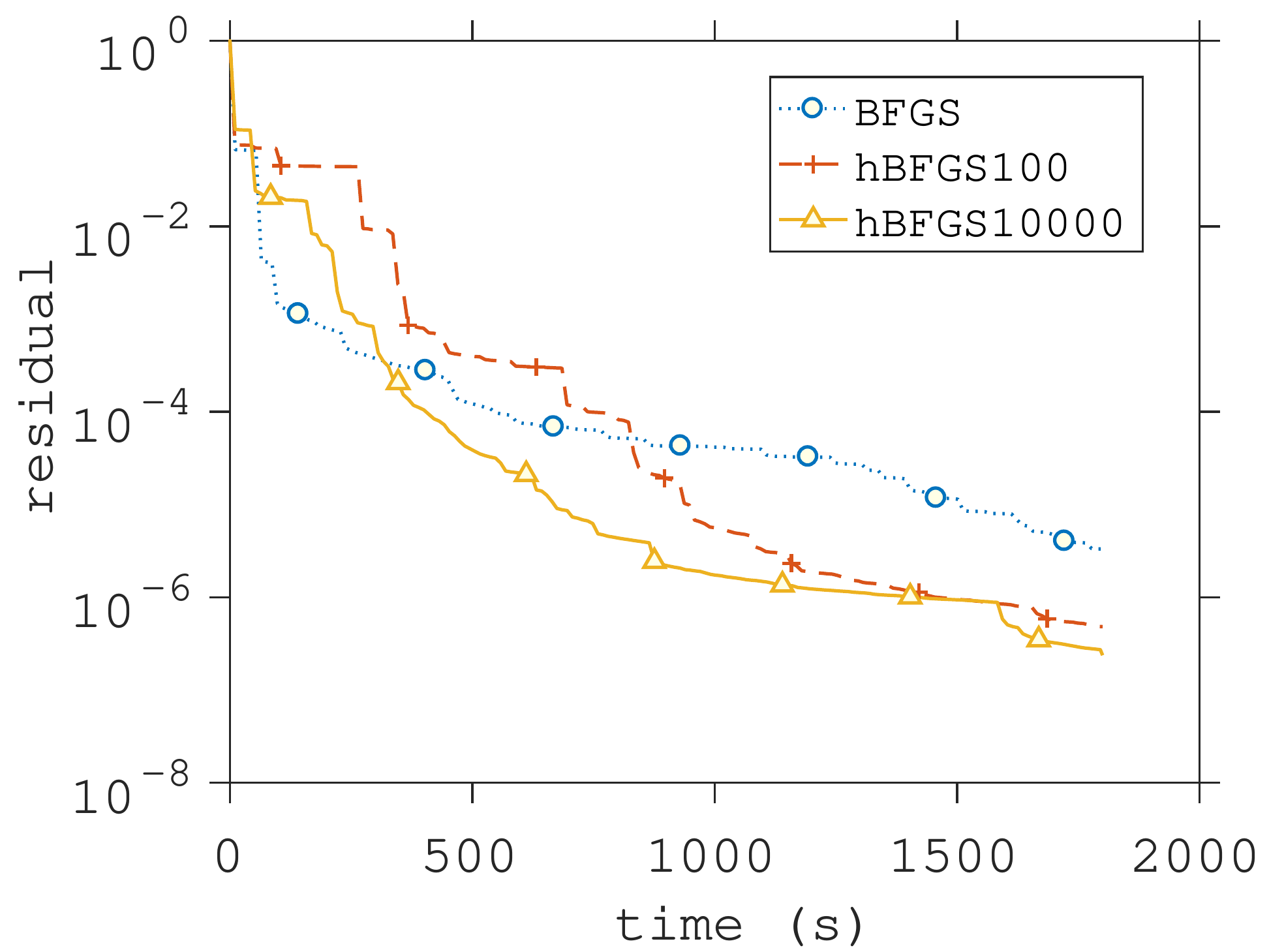}
\end{minipage}%
\begin{minipage}{0.25\textwidth}
  \centering
\includegraphics[width =  \textwidth ]{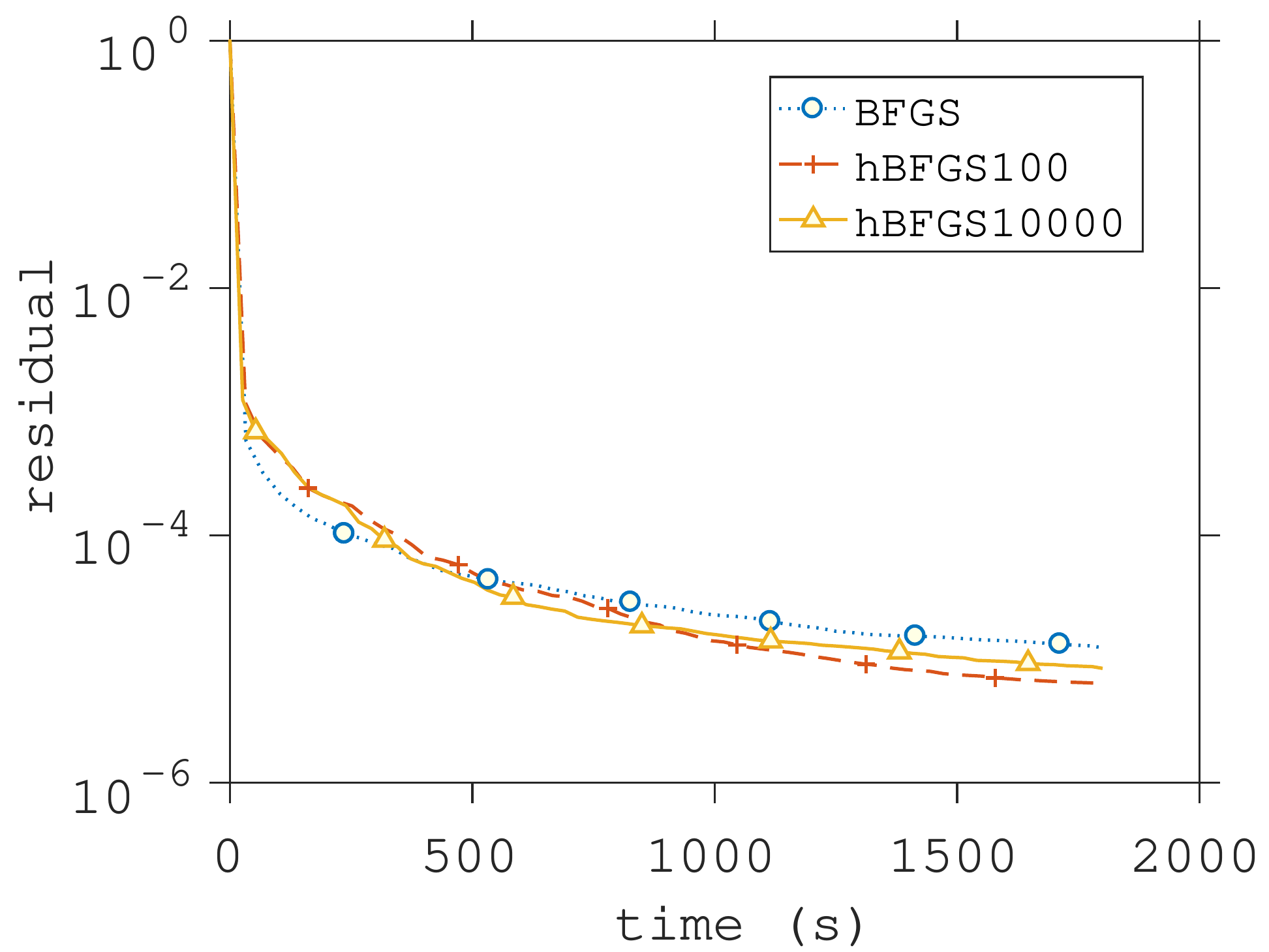}
\end{minipage}%
       \caption{\tiny From left to right: (i) Eigenvalues of $A \in \R^{100\times 100}$ are $1,10^3, 10^3, \dots,10^3$ and coordinate sketches with convenient probabilities are used. (ii) Eigenvalues of $A \in \R^{100\times 100}$  are $1,2,\dots,n$ and Gaussian sketches are used. Label ``nsym'' indicates non-enforcing symmetry and  ``-a'' indicates acceleration. (iii)   Epsilon dataset ($n=2000$), coordinate sketches with uniform probabilities. (iv) SVHN dataset ($n=3072$), coordinate sketches with convenient probabilities. Label ``h'' indicates that $\lambda_{\min}$ was not precomputed, but $\mu$ was chosen as described in the text.}
\label{fig:artificial_paper}
\end{figure}

For more plots, see Section~\ref{sec:exp_appendix} in the appendix as here we provide only a tiny fraction of all plots. The experiments suggest that once the parameters $\mu,\nu$ are estimated exactly, we get a speedup comparing to the nonaccelerated method; and the amount of speedup depends on the structure of $A$ and the sketching strategy. We observe from Figure~\ref{fig:artificial_paper} that we gain a great speedup for ill conditioned problems once the eigenvalues are concentrated around the largest eigenvalue. We also observe from Figure~\ref{fig:artificial_paper} that enforcing symmetry combines well with $\mu, \nu$ computed for the algorithm which do not enforce symmetry. On top of that, choice of $\mu, \nu$ per \eqref{eq:munu_conv_paper} seems to be robust to different sketching strategies, and in worst case performs as fast as nonaccelerated algorithm.

%
%

\subsection{BFGS Optimization Method}
\label{sec:accBFGSexp}
 We test Algorithm~\ref{alg:bfgs_opt} on several logistic regression problems using data from  LIBSVM~\cite{Chang2011}.  In all our tests we centered and normalized the data, included a bias term (a linear intercept), and choose the regularization parameter as $\lambda =1/m$, where $m$ is the number of data points. To keep things as simple as possible, we also used a fixed stepsize which 
was determined using grid search.  Since our theory regarding the choice for the parameters $\mu$ and $\nu$ does not apply in this setting, we simply probed the space of parameters manually and reported the best found result, 
 see Figure~\ref{fig:australian}. In the legend  we use \texttt{BFGS}-a-$\mu$-$\nu$ to denote the accelerated BFGS method (Alg~\ref{alg:bfgs_opt}) with parameters $\mu$ and $\nu$.


\begin{figure}[!h]
    \centering
\begin{minipage}{0.24 \textwidth}
\includegraphics[width =  \textwidth ]{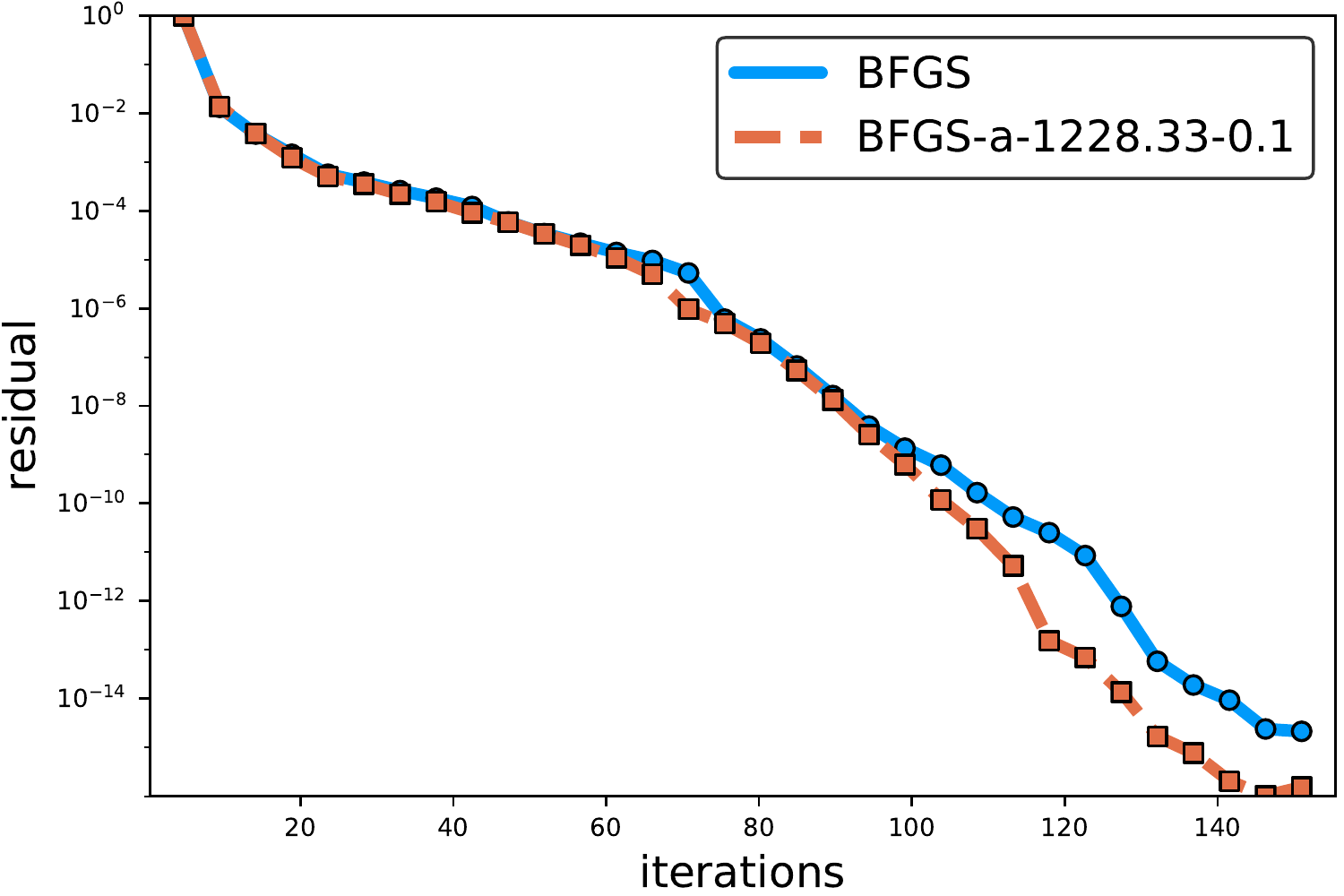}
\end{minipage}%
\begin{minipage}{0.24 \textwidth}
\includegraphics[width =  \textwidth ]{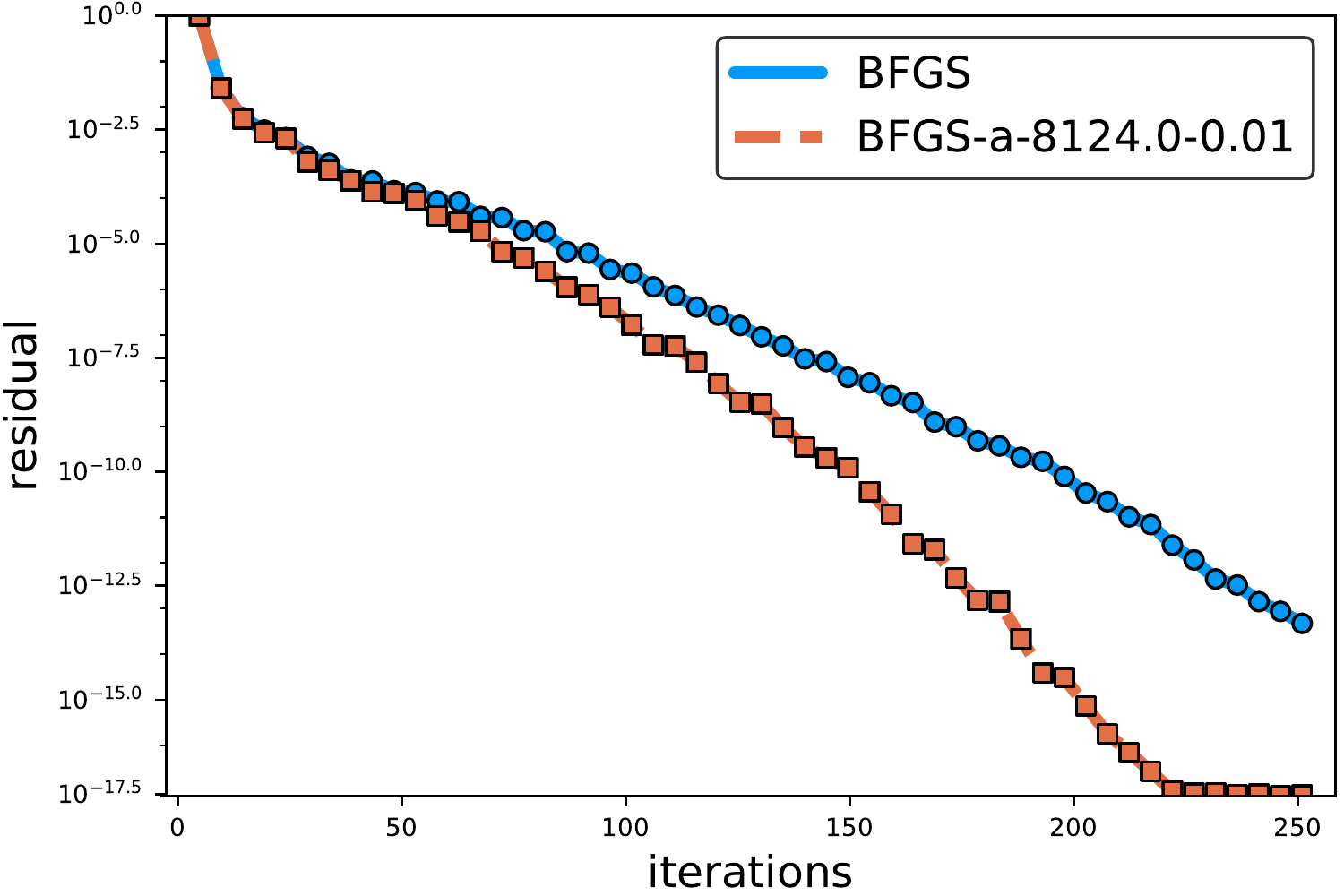}
\end{minipage}
\begin{minipage}{0.24	 \textwidth}
\includegraphics[width =  \textwidth ]{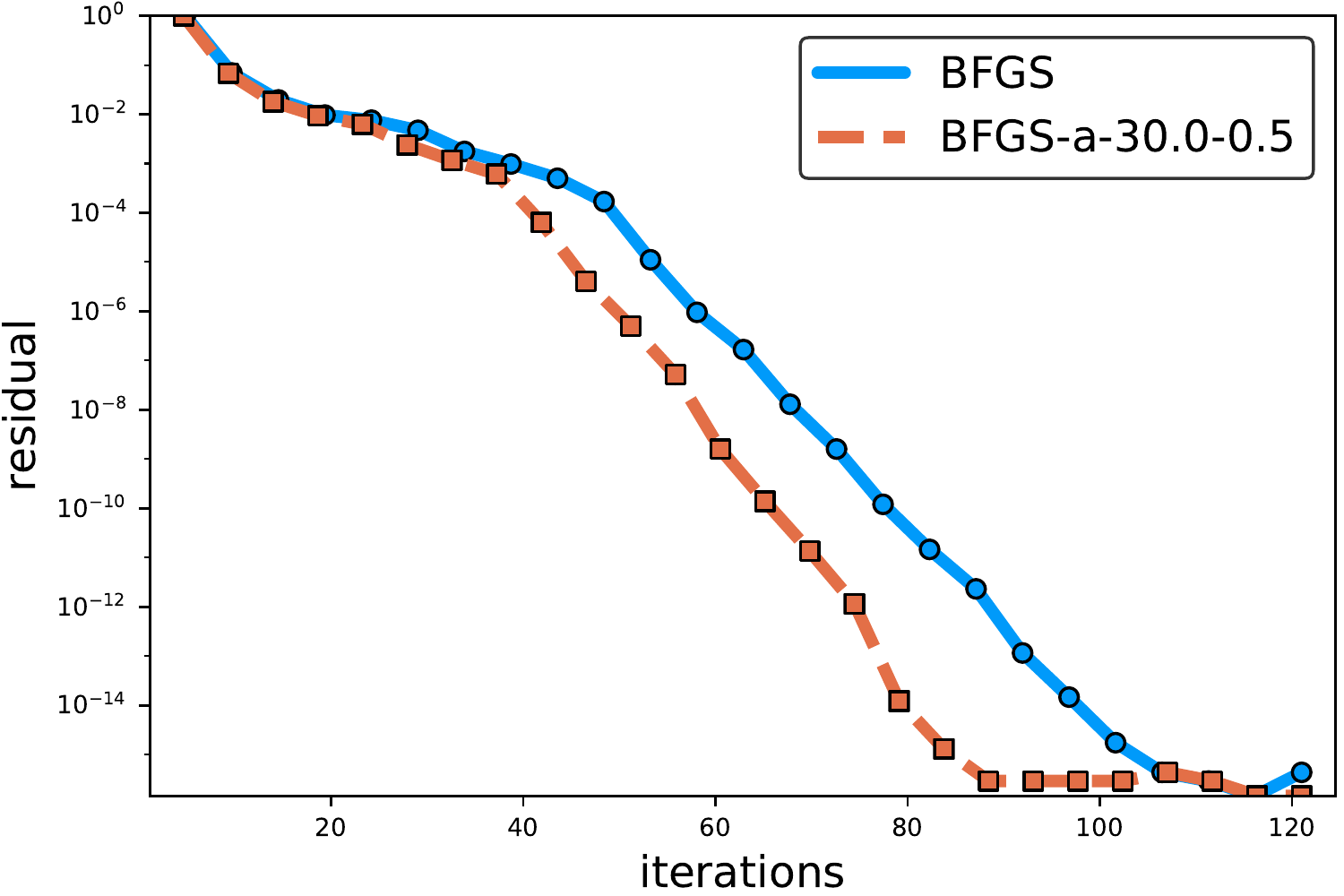}	
\end{minipage}%
\begin{minipage}{0.24 \textwidth}
\includegraphics[width =  \textwidth ]{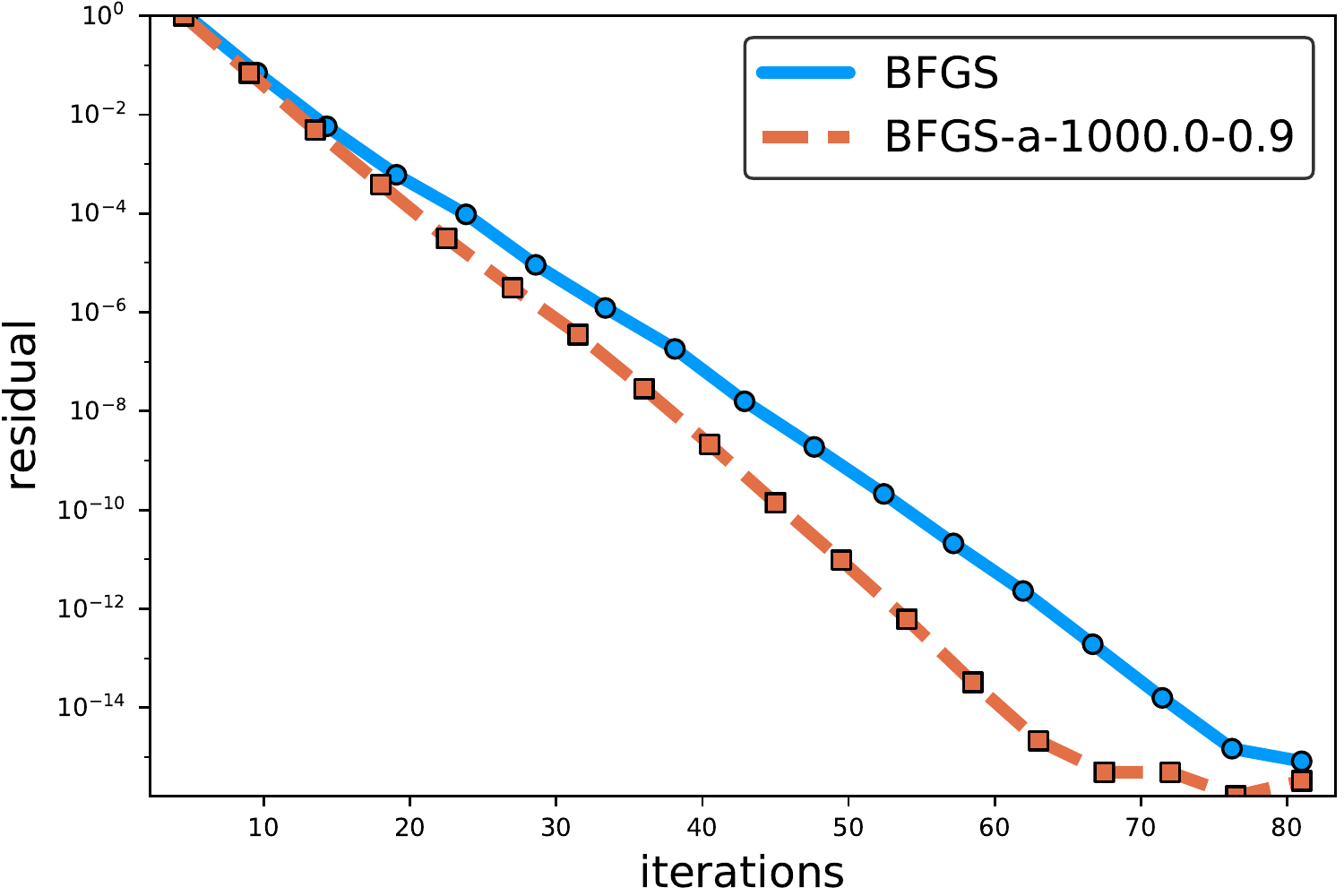}
\end{minipage}%
\caption{Algorithm~\ref{alg:bfgs_opt} (BFGS with accelerated matrix inversion quasi-Newton update) vs standard BFGS. From left to right:  \texttt{phishing}, \texttt{mushrooms}, \texttt{australian} and \texttt{splice} dataset.}
\label{fig:australian}
\end{figure}
On all four datasets, our method outperforms the classic BFGS method, indicating that replacing classic BFGS update rules for learning the inverse Hessian by our new accelerated rules can be beneficial in practice.
In A.4 in the appendix we also show the time plots for solving the problems in Figure~\ref{fig:australian}, and show that the accelerated BFGS method also converges faster in time.
 

\section{Further Experiments with Accelerated quasi-Newton Updates \label{sec:exp_appendix}}
In this section, we test the the empirical rate of convergence of Algorithm~\ref{alg:qn}, the accelerated BFGS update for inverting positive definite matrices. Only vector sketches are considered, as the standard quasi-Newton methods also update the inverse Hessian only according to the action in one direction. We compare the speed of the accelerated method with precomputed estimates of the parameters $\mu, \nu$ to the nonaccelerated method. The precomputed estimates of $\mu^P, \nu^P$ are set as per \eqref{eq:munu_conv_paper}:
\begin{equation*}
\mu^P=\frac{\lambda_{\min}(A)}{\Tr{A}}, \qquad \nu^P=\frac{\Tr{A}}{\min_i(A_{i,i})}, 
\end{equation*}
which is the optimal choice for coordinate sketches with convenient probabilities without enforcing symmetry. In practice we might not have an access to $\lambda_{\min}(A)$, thus we cannot compute $\mu^P$ exactly. Therefore we also test sensitivity of the algorithm to the choice of parameters, and we run some experiments where we only guess parameter $\mu^P$. 

Lastly, the tests are performed on both artificial examples and LIBSVM~\cite{Chang2011} data. We shall also explain the legend of plots: ``a'' indicates acceleration, ``nsym'' indicates the algorithm without enforcing symmetry and ``h'' indicates the setting when $\nu^P$ is not known, and a naive heuristic choice is casted. 

\subsection{Simple and well understood artificial example \label{sec:bl}}

Let us consider inverting the matrix $A = \alpha I + \beta {\bf 1 1^\top}$ for $\alpha>0$ and $\beta\geq -\frac{\alpha}{n}$ so as in this case we have control over both $\mu$ and $\nu$. This artificial example was considered in \cite{TuVWGJR17} for solving linear systems. In particular, we show that for coordinate sketches with convenient probabilities (which is indeed the same as uniform probabilities in this example), we have
\begin{eqnarray*}
\mu^P &\eqdef& \lambda_{\min}(\E{P}) =  \frac{\min \left(\alpha, \alpha+n\beta \right)}{n(\alpha+\beta)},\\
\nu^P &\eqdef& \lambda_{\max}\left(\E{ \E{P}^{-\frac12 } P \E{P}^{-1}P \E{P}^{-\frac12}}\right)  = n. \\
\end{eqnarray*}

Due to the fact that we do not have a theoretical justification of $\mu,\nu$ for $n>2$ when enforcing symmetry, we set $\mu=\mu^P$ and $\nu=\nu^P$ for Gaussian sketches as well.

\begin{figure}[H]
    \centering
\begin{minipage}{0.40\textwidth}
  \centering
\includegraphics[width =  \textwidth]{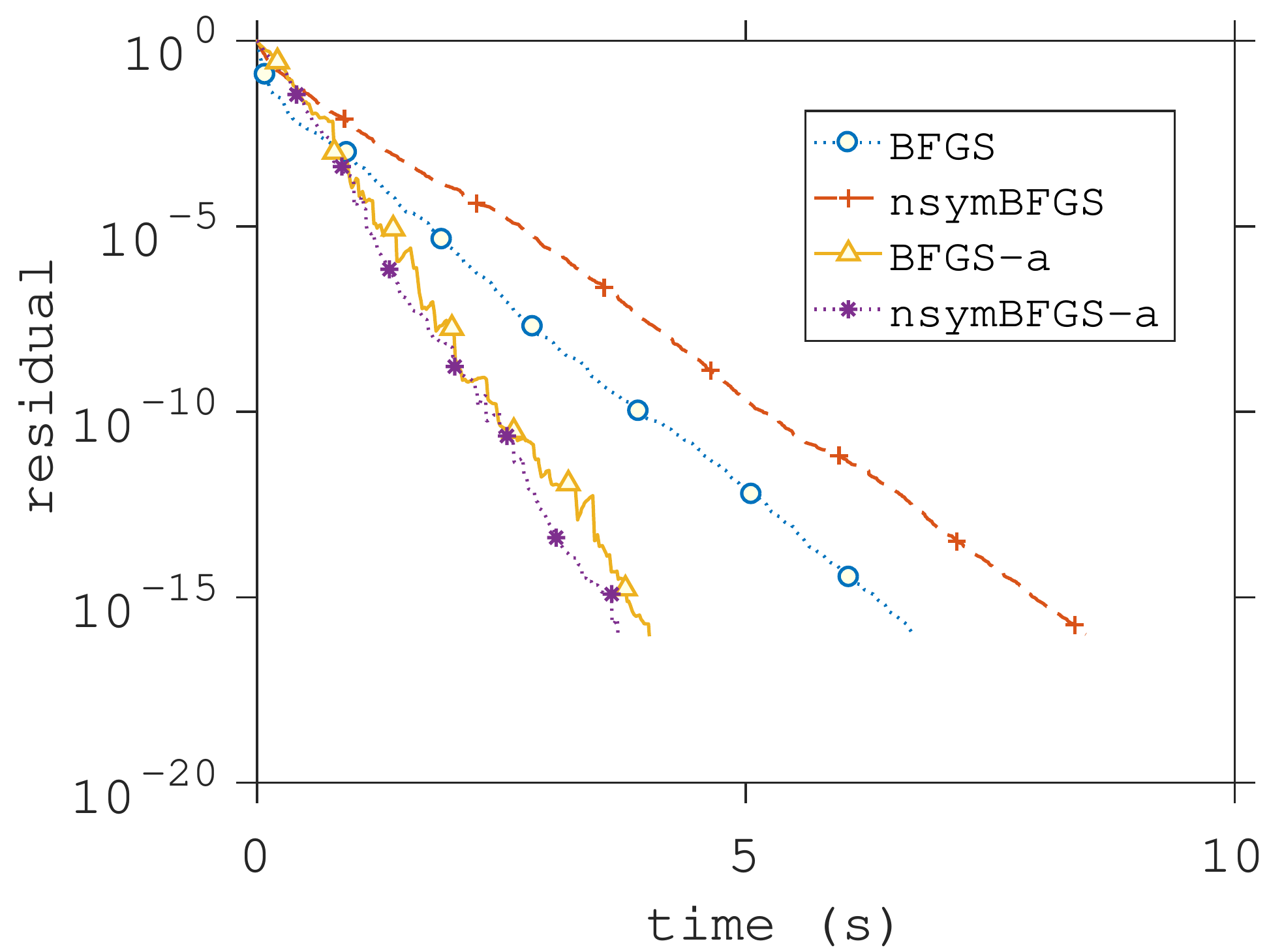}
\end{minipage}%
\begin{minipage}{0.40\textwidth}
  \centering
\includegraphics[width =  \textwidth ]{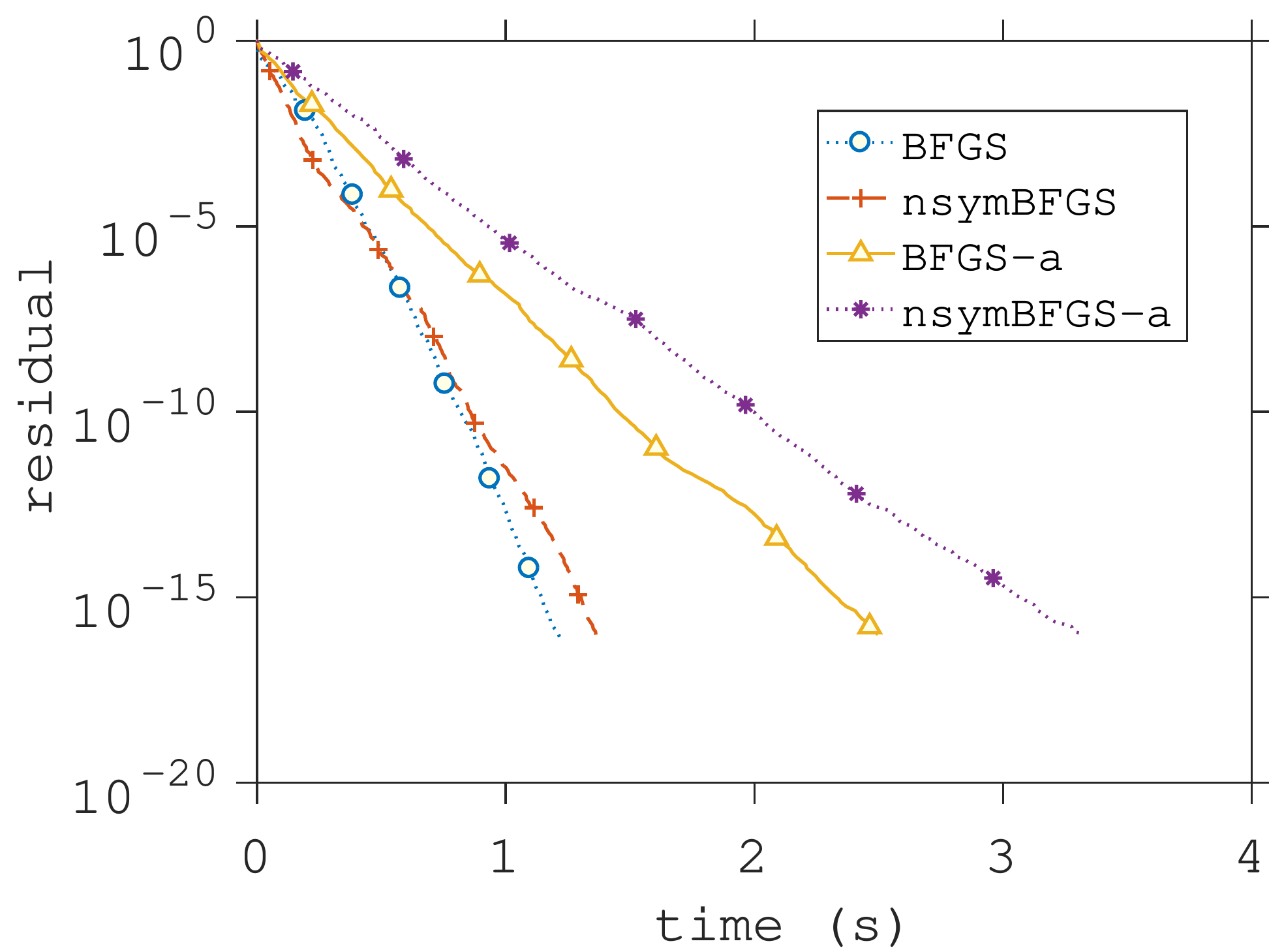}
\end{minipage}%
    \caption{Parameter choice: $\alpha=1+10^{-1}, \beta=-n^{-1}, n=100$. From left to right we have: Coordinate sketch with uniform (convenient) probabilities and Gaussian sketch respectively. 
}\label{fig:bl_ex}
\end{figure}

\begin{figure}[H]
    \centering
\begin{minipage}{0.40\textwidth}
  \centering
\includegraphics[width =  \textwidth]{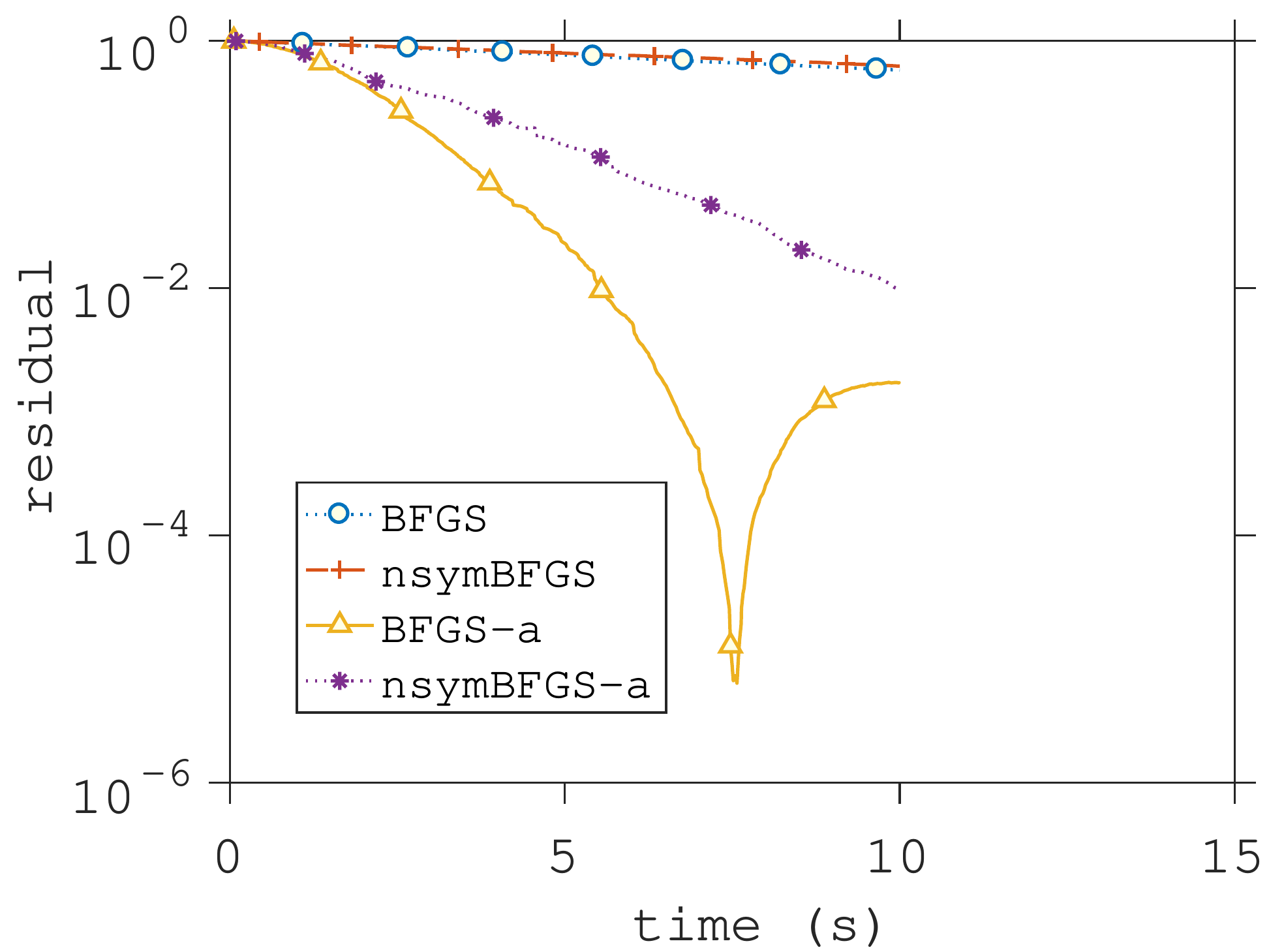}
\end{minipage}%
\begin{minipage}{0.40\textwidth}
  \centering
\includegraphics[width =  \textwidth ]{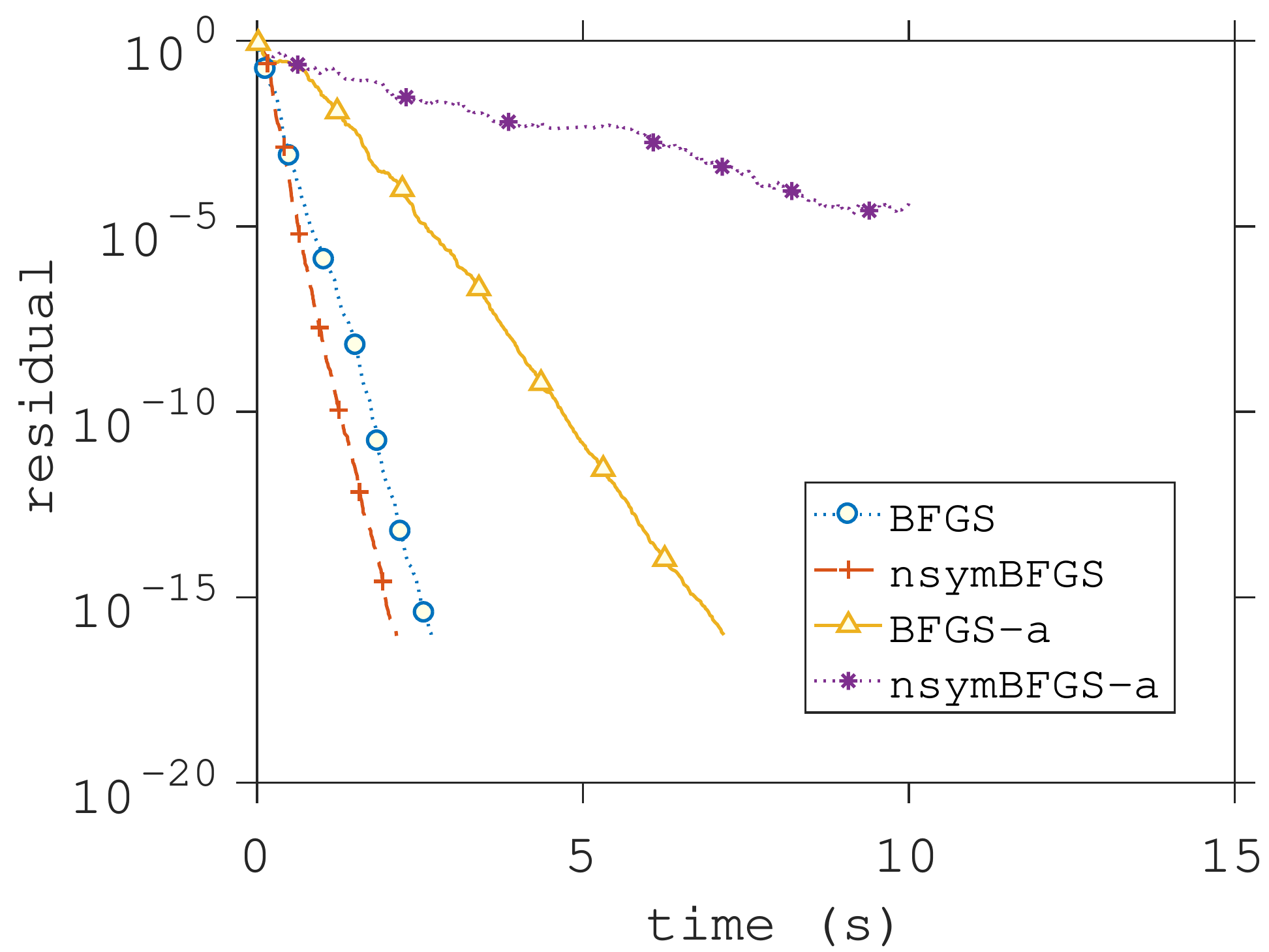}
\end{minipage}%
    \caption{Parameter choice: $\alpha=1+10^{-3}, \beta=-n^{-1}, n=100$. From left to right we have: Coordinate sketch with uniform (convenient) probabilities and Gaussian sketch respectively. 
}\label{fig:bl_ex}
\end{figure}

\begin{figure}[H]
    \centering
\begin{minipage}{0.40\textwidth}
  \centering
\includegraphics[width =  \textwidth]{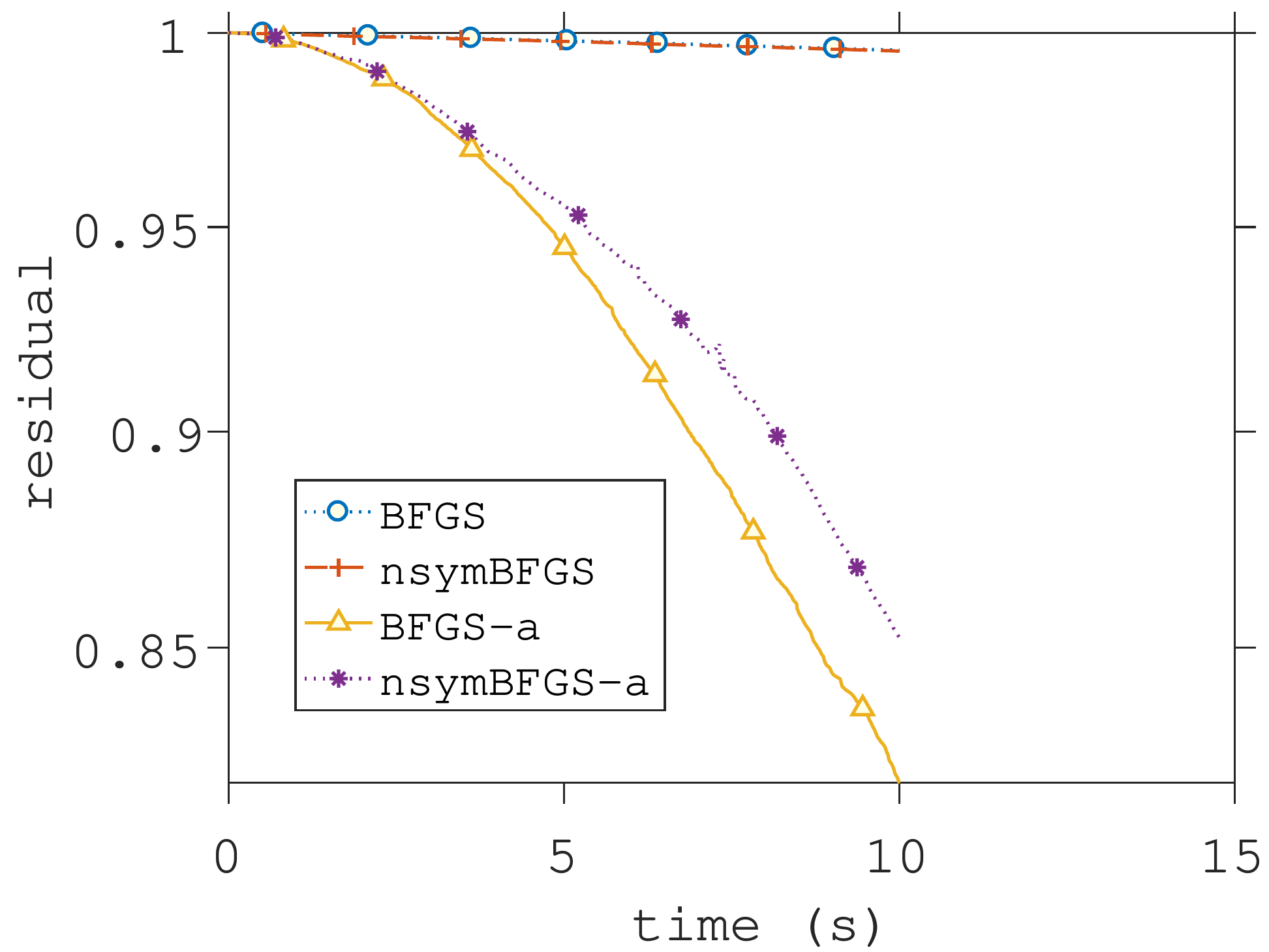}
\end{minipage}%
\begin{minipage}{0.40\textwidth}
  \centering
\includegraphics[width =  \textwidth ]{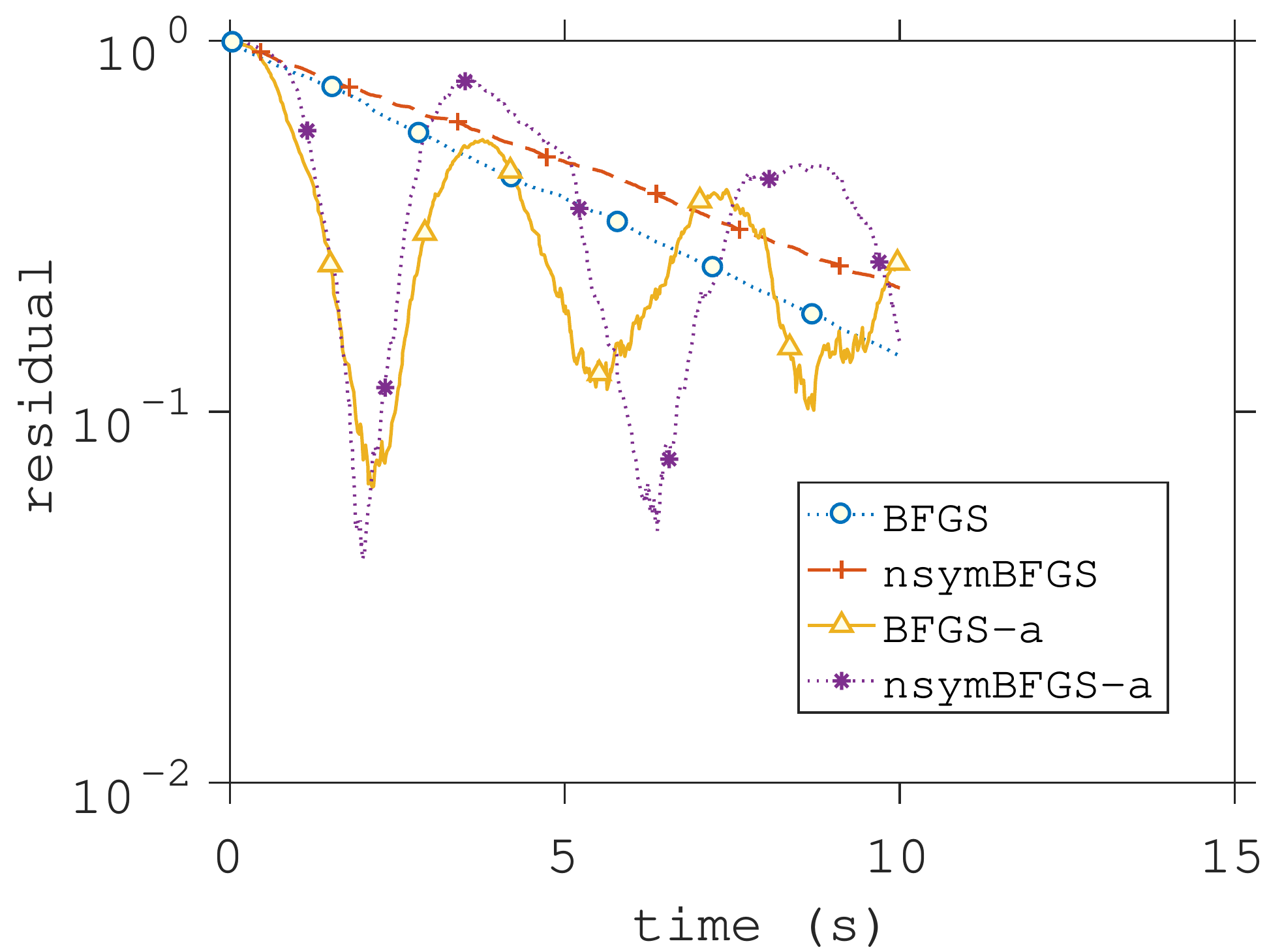}
\end{minipage}%
    \caption{Parameter choice: $\alpha=1+10^{-5}, \beta=-n^{-1}, n=100$. From left to right we have: Coordinate sketch with uniform (convenient) probabilities and Gaussian sketch, respectively. 
}\label{fig:bl_ex}
\end{figure}

As expected from the theory, as the matrix to be inverted becomes more ill conditioned, the accelerated method performs significantly better compared to the nonaccelerated method for coordinate sketches. In fact, an arbitrary speedup can be obtained by setting $\beta=-n^{-1}$ and $\alpha \rightarrow 1$ for the coordinate sketches setup. On the other hand, Gaussian sketches report the slowing of the algorithm, most likely caused by the fact that the theoretical parameters $\mu, \nu$ for Gaussian sketches with enforced symmetry are different to $\mu^P, \nu^P$, which are estimated for coordinate sketches without enforced symmetry. In the case of coordinate sketches with symmetry enforced, we suspect a great speedup even though the parameters $\mu, \nu$ were set to $\mu^P, \nu^P$.

\subsection{Random artificial example}

We randomly generate an orthonormal matrix $U$, choose diagonal matrix $D$, and set $A=UDU^\top$. Clearly, diagonal elements of $D$ are eigenvalues of $A$. We set them in the following way:

\begin{itemize}
\item Uniform grid. The eigenvalues are set to $1,2,\dots,n$.
\item One small, the rest larger. The smallest eigenvalue is $1$, remaining eigenvalues are all $10$ in the first example, all $100$ in the second example and all $1000$ in the third example in this category.
\item One large, the rest small. The largest eigenvalue is $10^4$, the remaining eigenvalues are all $1$.
\end{itemize}

Firstly, consider coordinate sketches with convenient probabilities. Notice that we can easily estimate $\nu^P, \mu^P$ due to the results from Section~\ref{sec:convenient_munu} since we have control of $\lambda_{\min}(A)$ and therefore also of $\mu$. Therefore, we set $\mu=\mu^P=\min D_{i,i}$ and $\nu=\nu^P$ for Algorithm~\ref{alg:qn}. Then, we consider coordinate sketches with uniform probabilities and Gaussian sketches. In both cases, we set the parameters $\mu,\nu$ as for coordinate sketches with convenient probabilities.

\begin{figure}[H]
    \centering
\begin{minipage}{0.30\textwidth}
  \centering
\includegraphics[width =  \textwidth ]{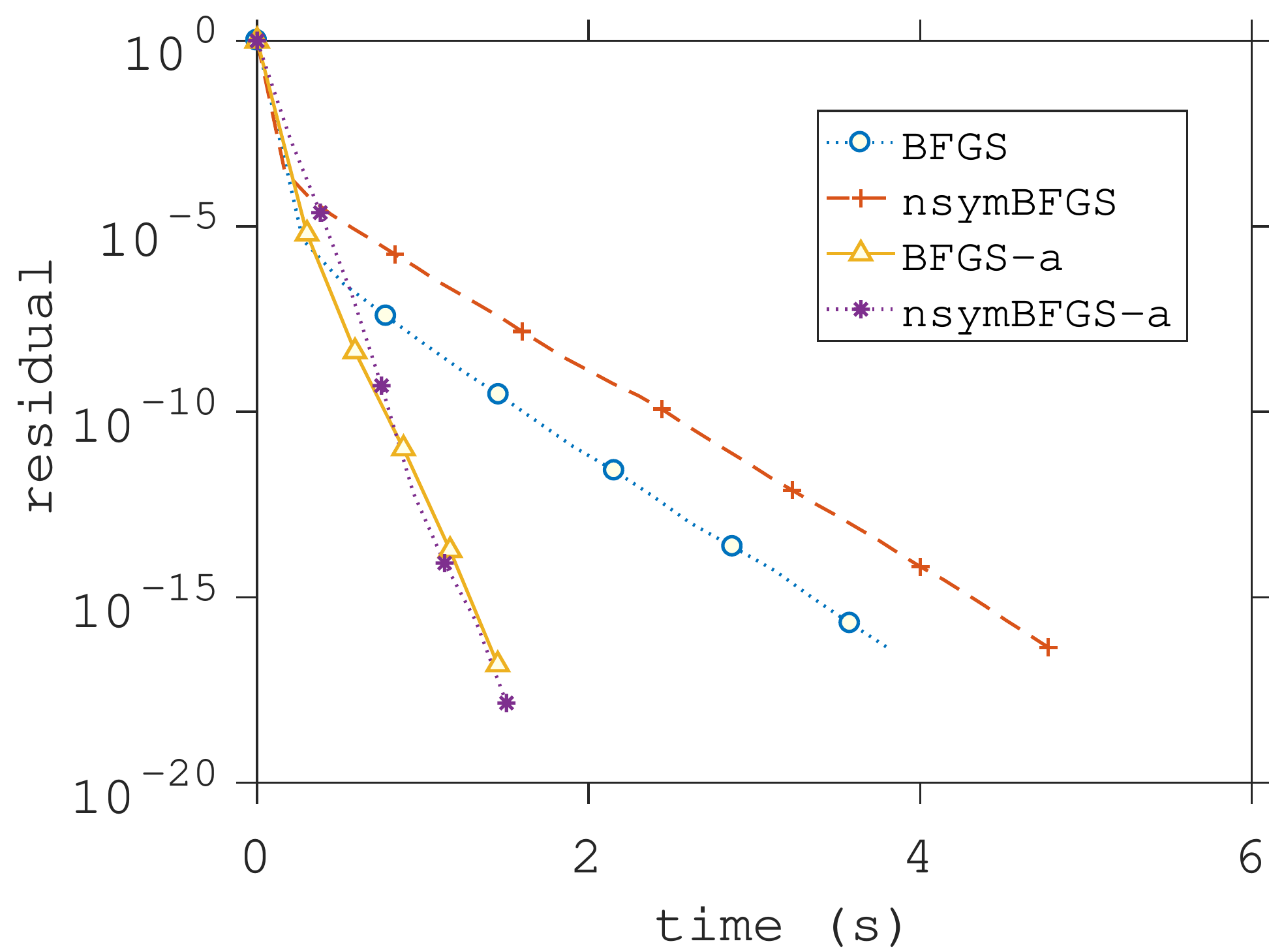}
\end{minipage}%
\begin{minipage}{0.30\textwidth}
  \centering
\includegraphics[width =  \textwidth ]{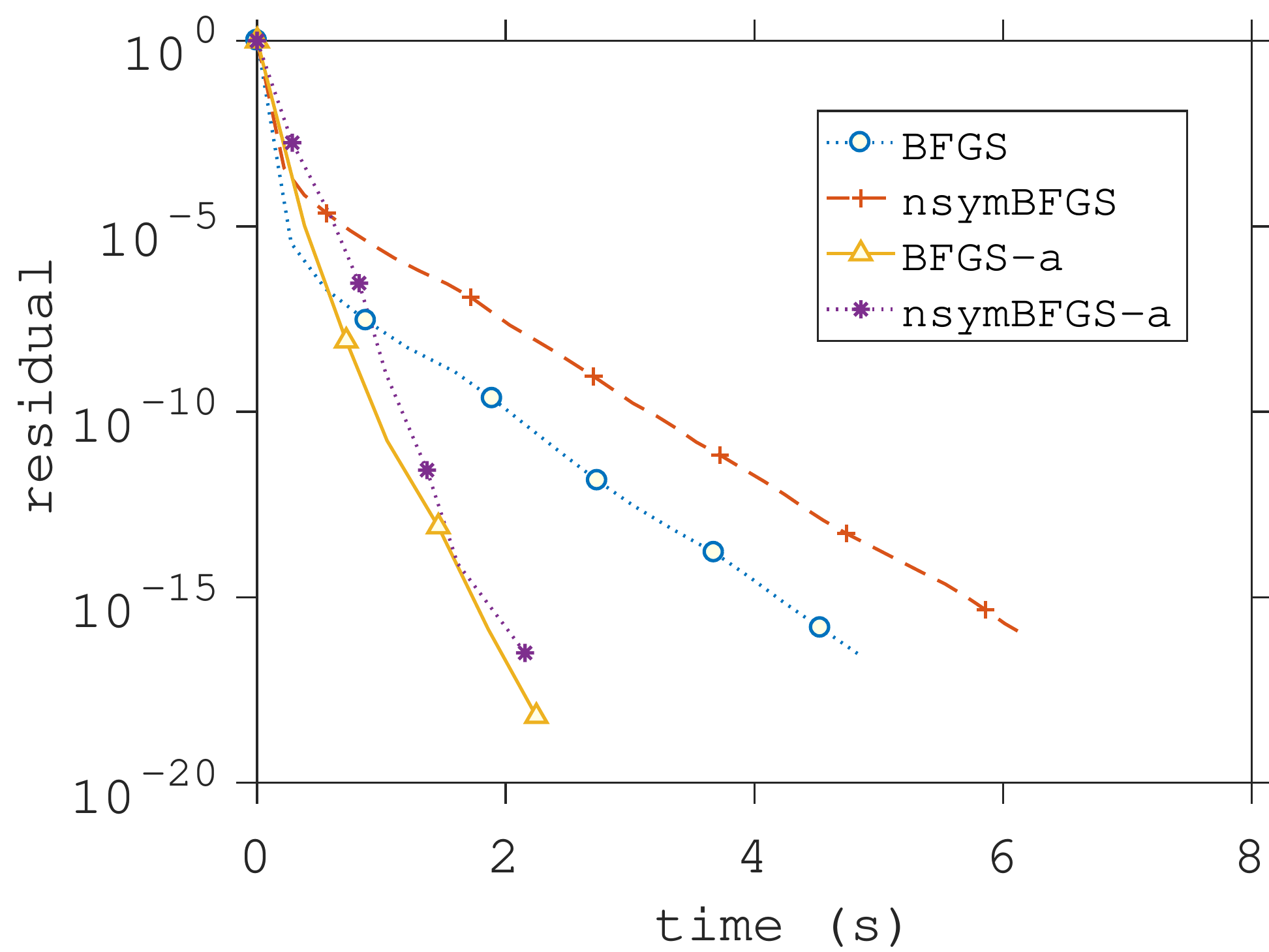}
\end{minipage}%
\begin{minipage}{0.30\textwidth}
  \centering
\includegraphics[width =  \textwidth ]{gaussrandomjj=1-time}
\end{minipage}%
    \caption{   Eigenvalues set to $1,2,3,\dots n$. From left to right we have: Coordinate sketch with convenient probabilities, coordinate sketch with uniform probabilities and Gaussian sketch respectively. 
}
\label{fig:rand_conv_lin}
\end{figure}

\begin{figure}[H]
    \centering
\begin{minipage}{0.30\textwidth}
  \centering
\includegraphics[width =  \textwidth ]{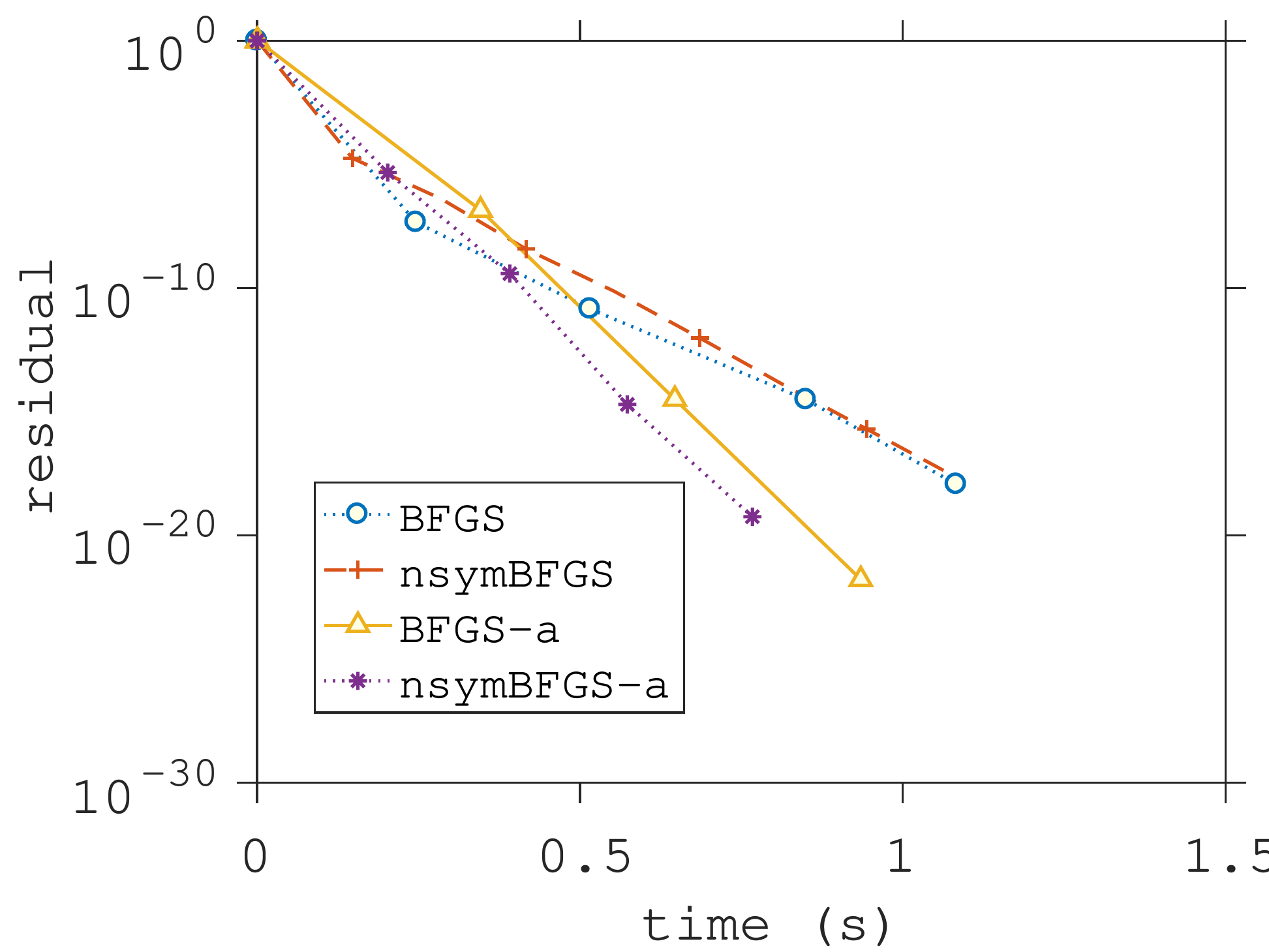}
\end{minipage}%
\begin{minipage}{0.30\textwidth}
  \centering
\includegraphics[width =  \textwidth ]{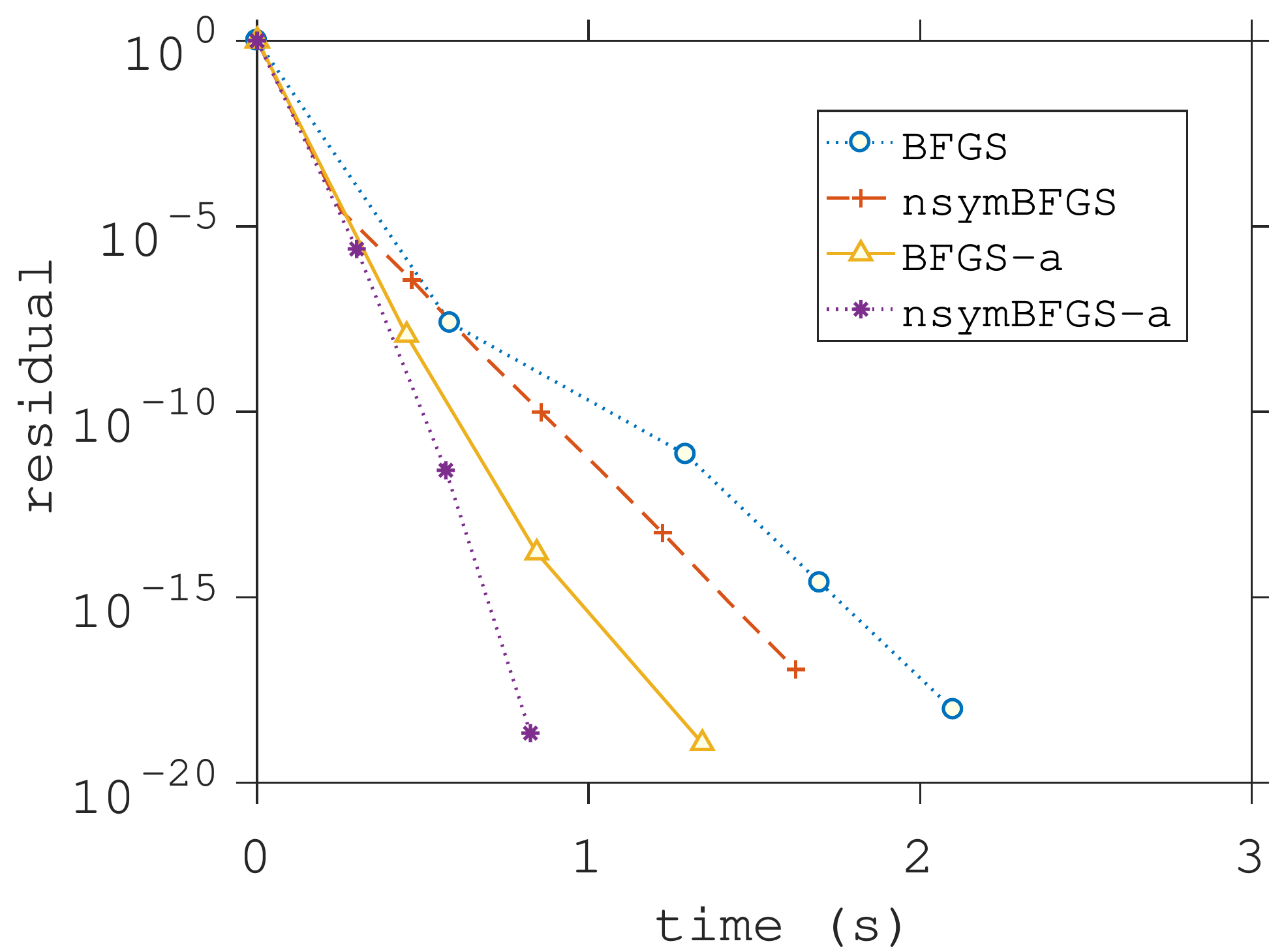}
\end{minipage}%
\begin{minipage}{0.30\textwidth}
  \centering
\includegraphics[width =  \textwidth ]{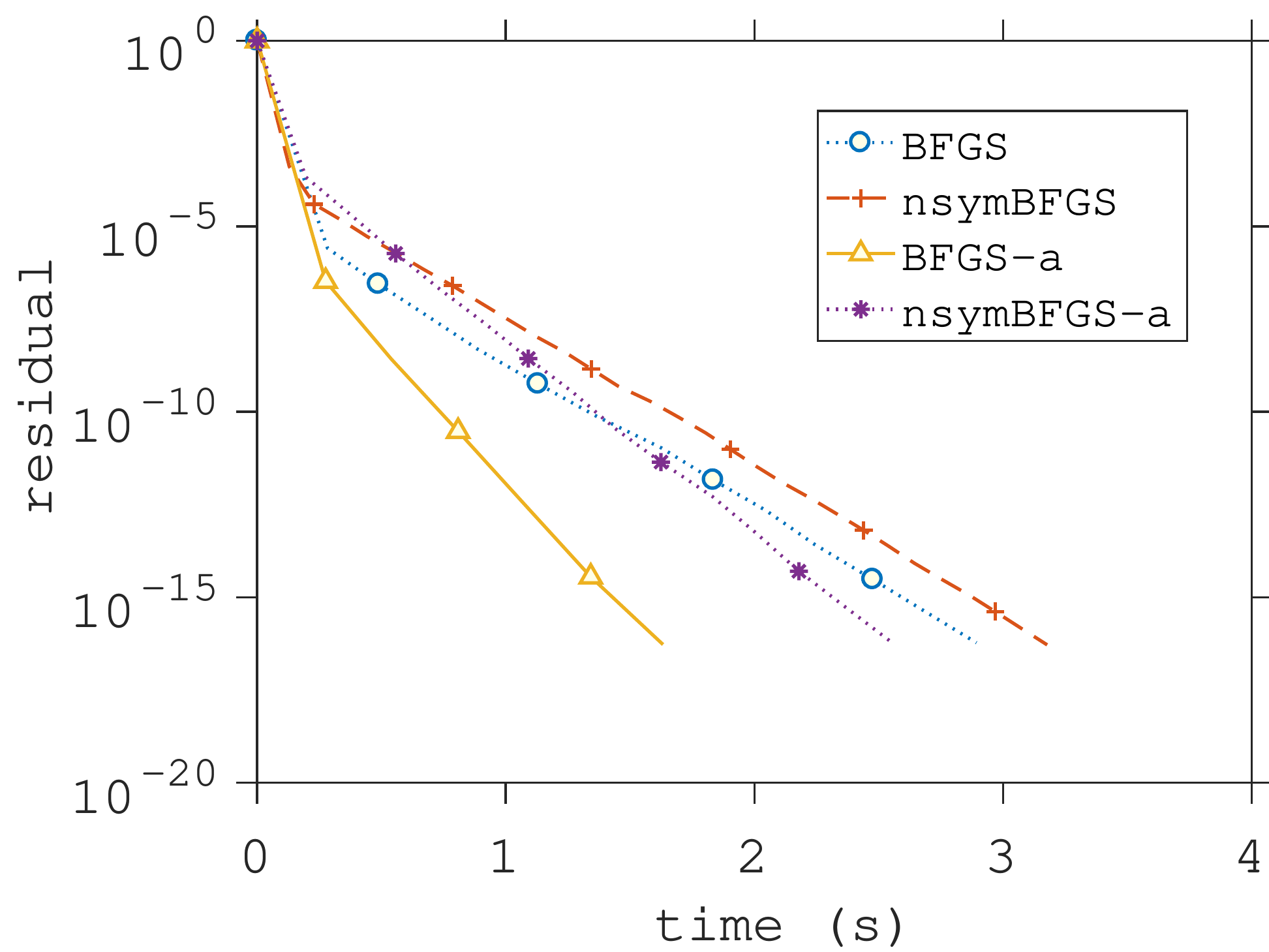}
\end{minipage}%
    \caption{   Eigenvalues set to $1,10,10,\dots 10$. From left to right we have: Coordinate sketch with convenient probabilities, coordinate sketch with uniform probabilities and Gaussian sketch respectively. 
}\label{fig:rand_conv_10}
\end{figure}

\begin{figure}[H]
    \centering
\begin{minipage}{0.30\textwidth}
  \centering
\includegraphics[width =  \textwidth ]{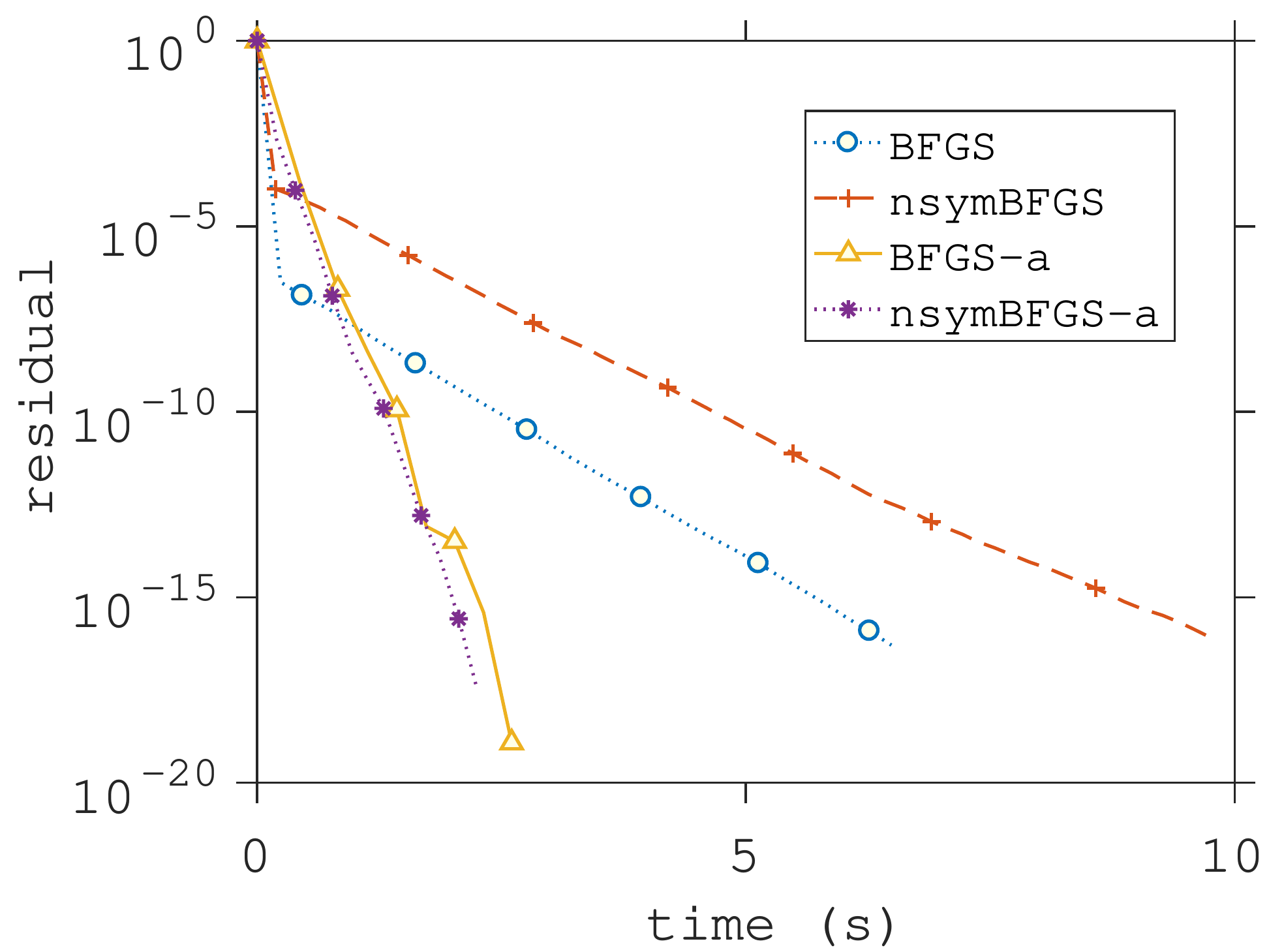}
\end{minipage}%
\begin{minipage}{0.30\textwidth}
  \centering
\includegraphics[width =  \textwidth ]{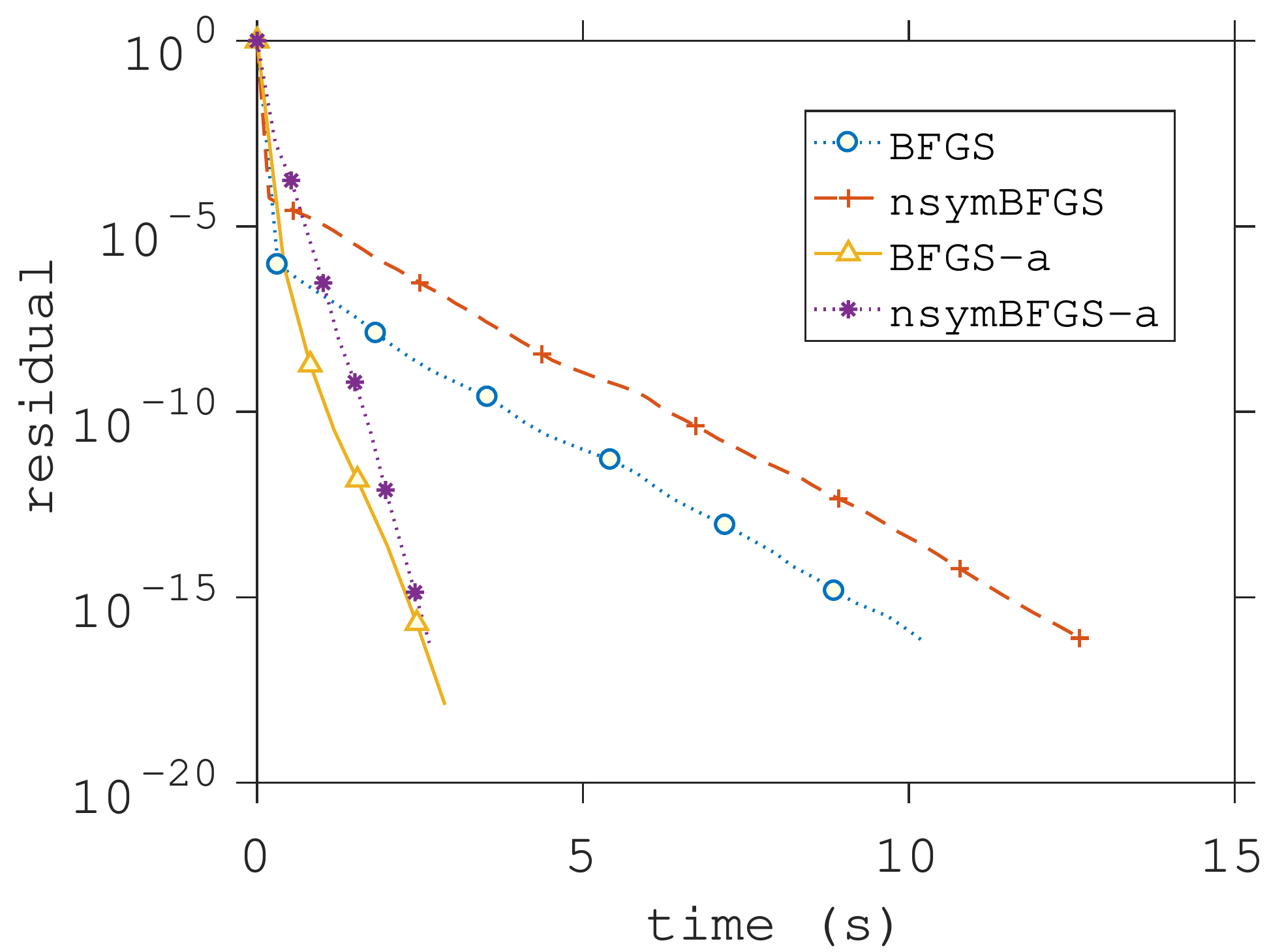}
\end{minipage}%
\begin{minipage}{0.30\textwidth}
  \centering
\includegraphics[width =  \textwidth ]{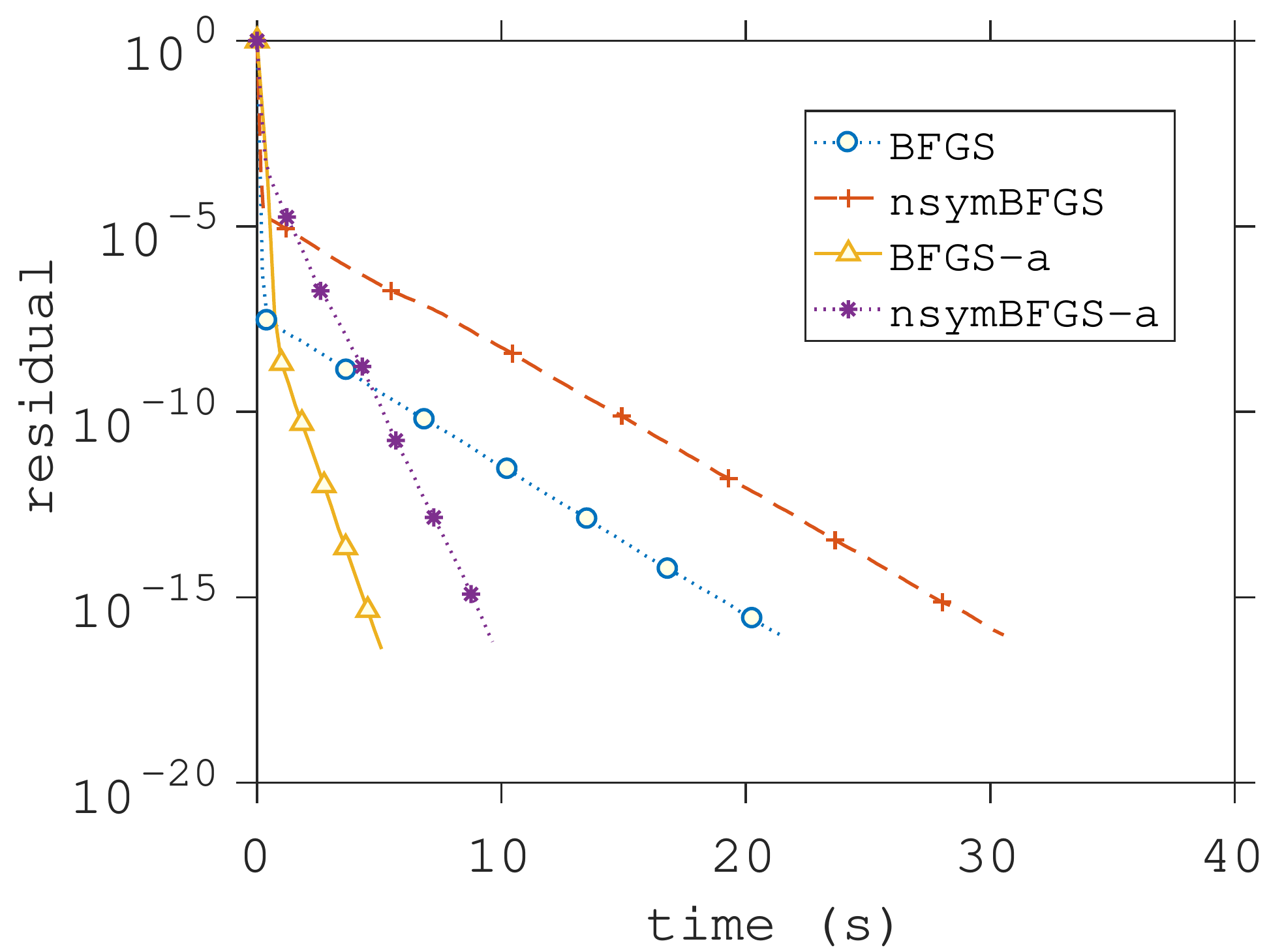}
\end{minipage}%
    \caption{   Eigenvalues set to $1,100,100,\dots 100$. From left to right we have: Coordinate sketch with convenient probabilities, coordinate sketch with uniform probabilities and Gaussian sketch respectively. 
}\label{fig:rand_conv_100}
\end{figure}

\begin{figure}[H]
    \centering
\begin{minipage}{0.30\textwidth}
  \centering
\includegraphics[width =  \textwidth ]{convenientrandomjj=4-time}
\end{minipage}%
\begin{minipage}{0.30\textwidth}
  \centering
\includegraphics[width =  \textwidth ]{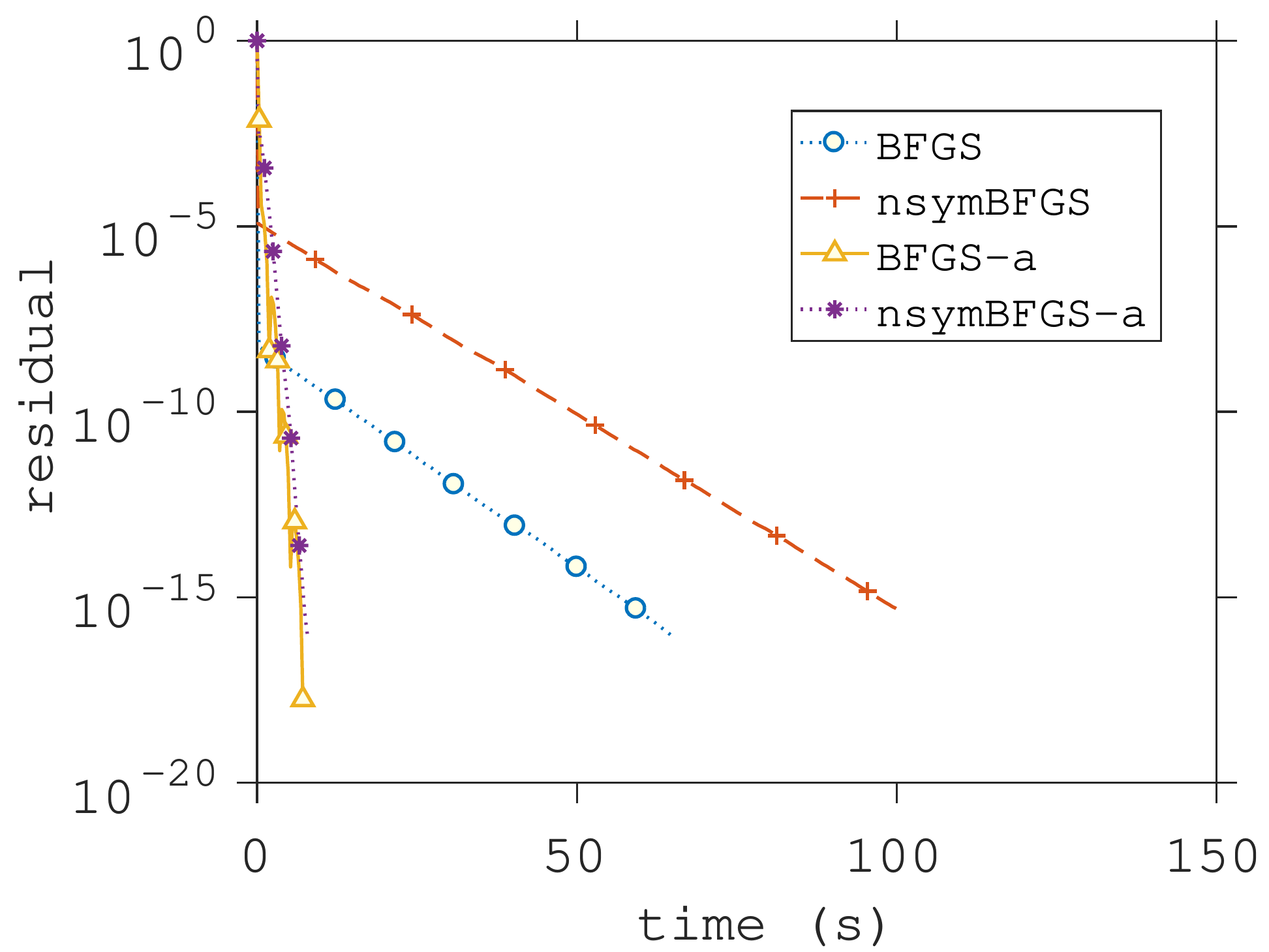}
\end{minipage}%
\begin{minipage}{0.30\textwidth}
  \centering
\includegraphics[width =  \textwidth ]{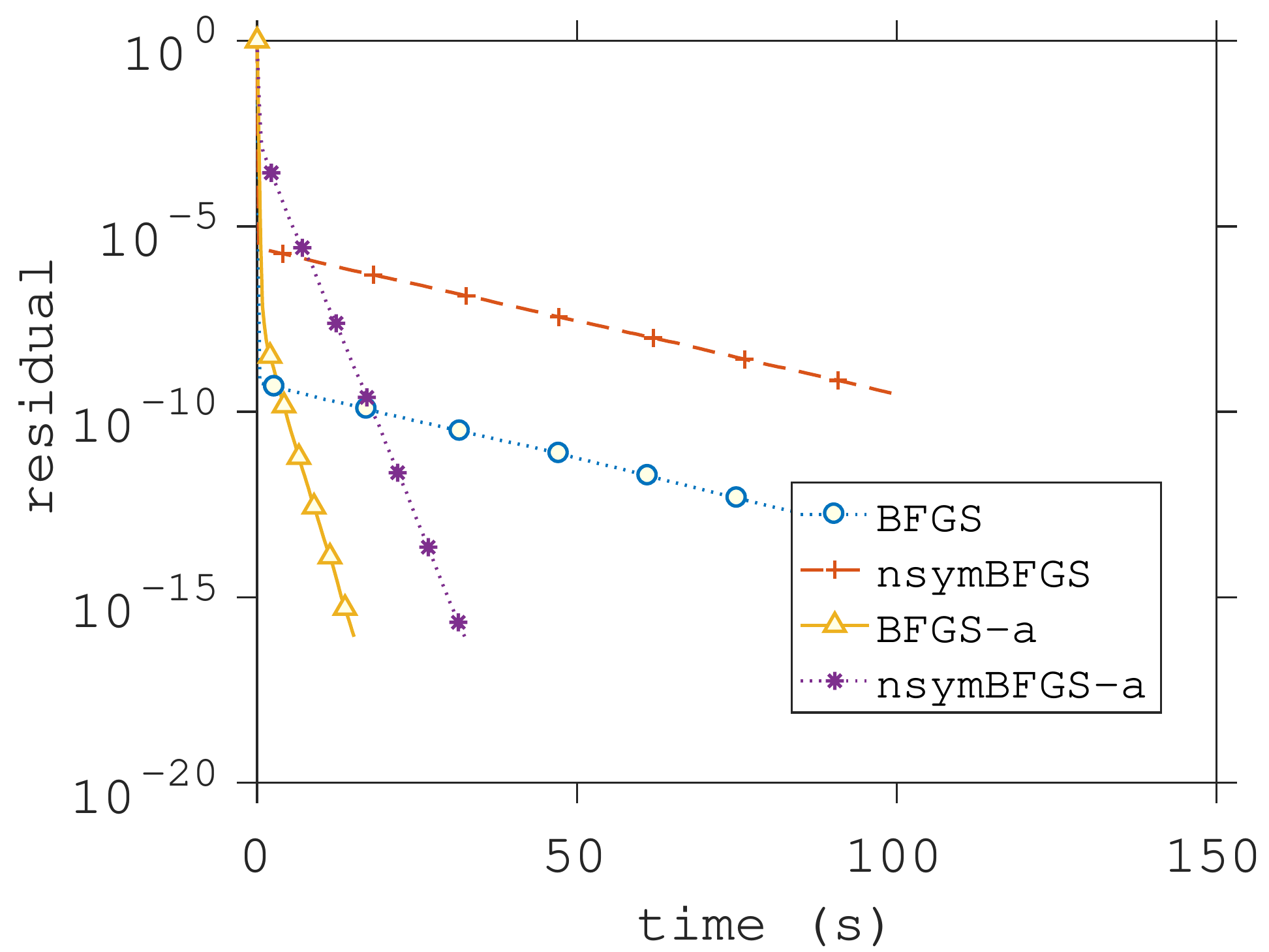}
\end{minipage}%
    \caption{  Eigenvalues set to $1,1000,1000,\dots 1000$. From left to right we have: Coordinate sketch with convenient probabilities, coordinate sketch with uniform probabilities and Gaussian sketch respectively. 
}\label{fig:rand_conv_1000}
\end{figure}

\begin{figure}[H]
    \centering
\begin{minipage}{0.30\textwidth}
  \centering
\includegraphics[width =  \textwidth ]{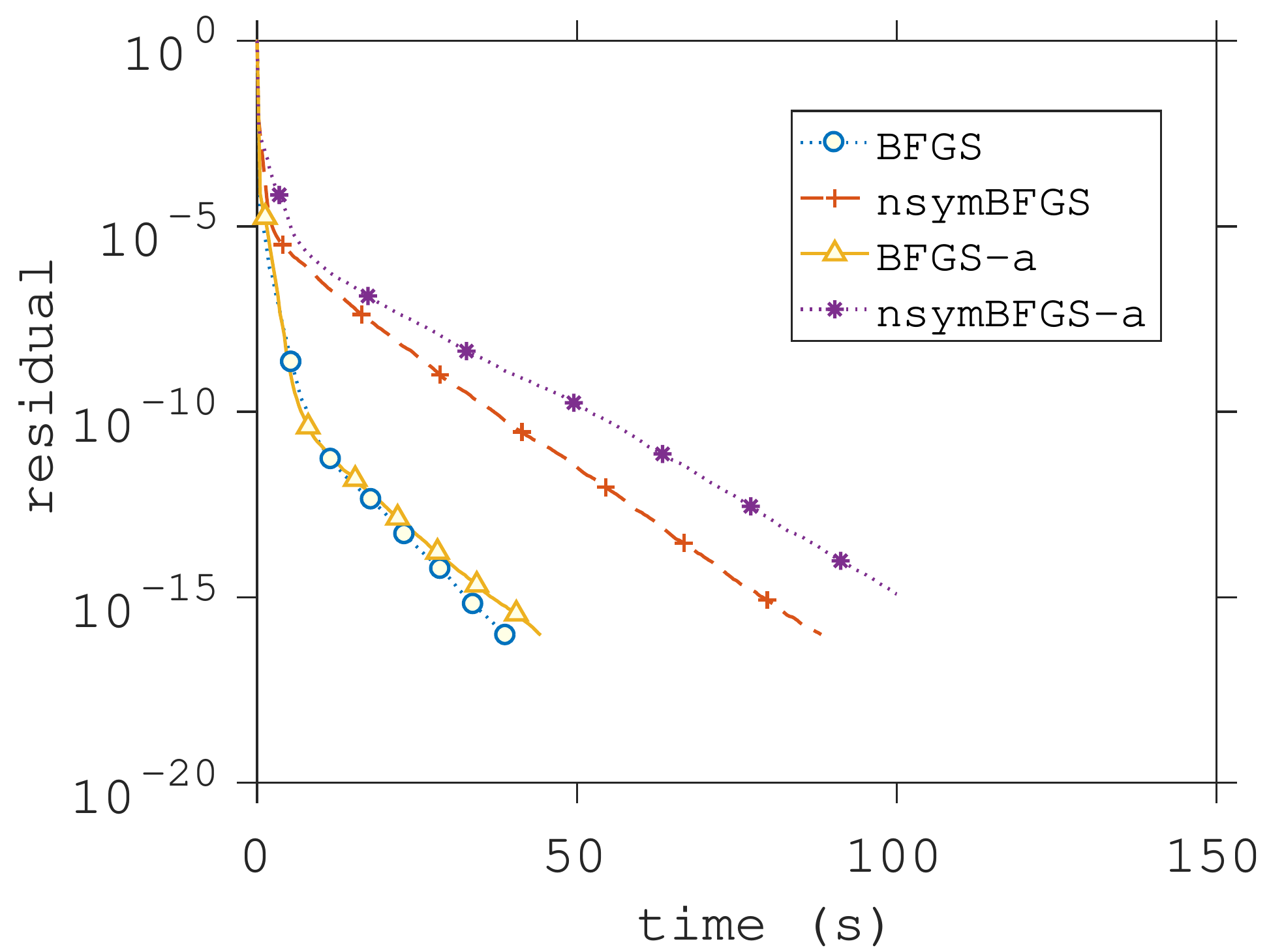}
\end{minipage}%
\begin{minipage}{0.30\textwidth}
  \centering
\includegraphics[width =  \textwidth ]{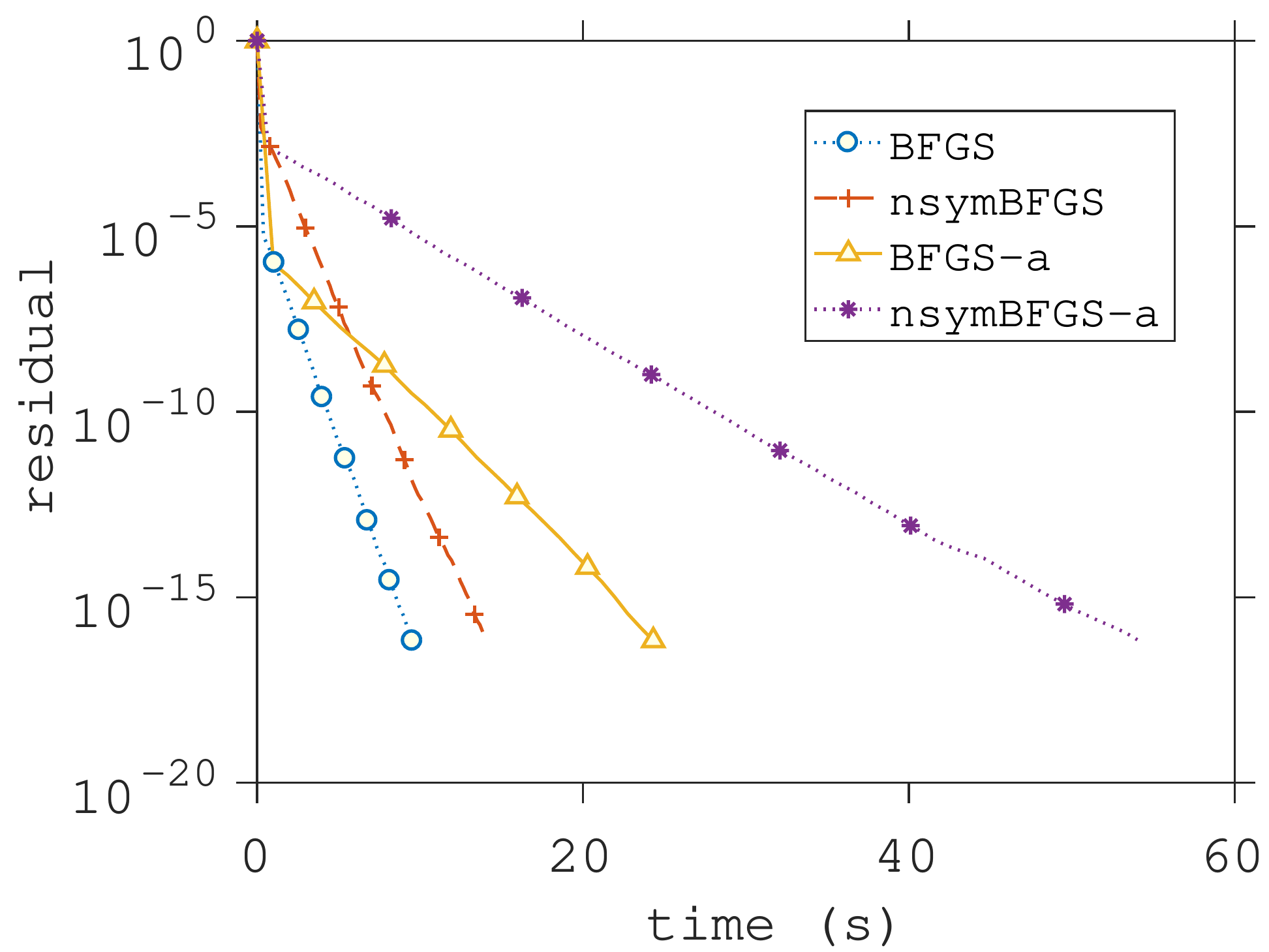}
\end{minipage}%
\begin{minipage}{0.30\textwidth}
  \centering
\includegraphics[width =  \textwidth ]{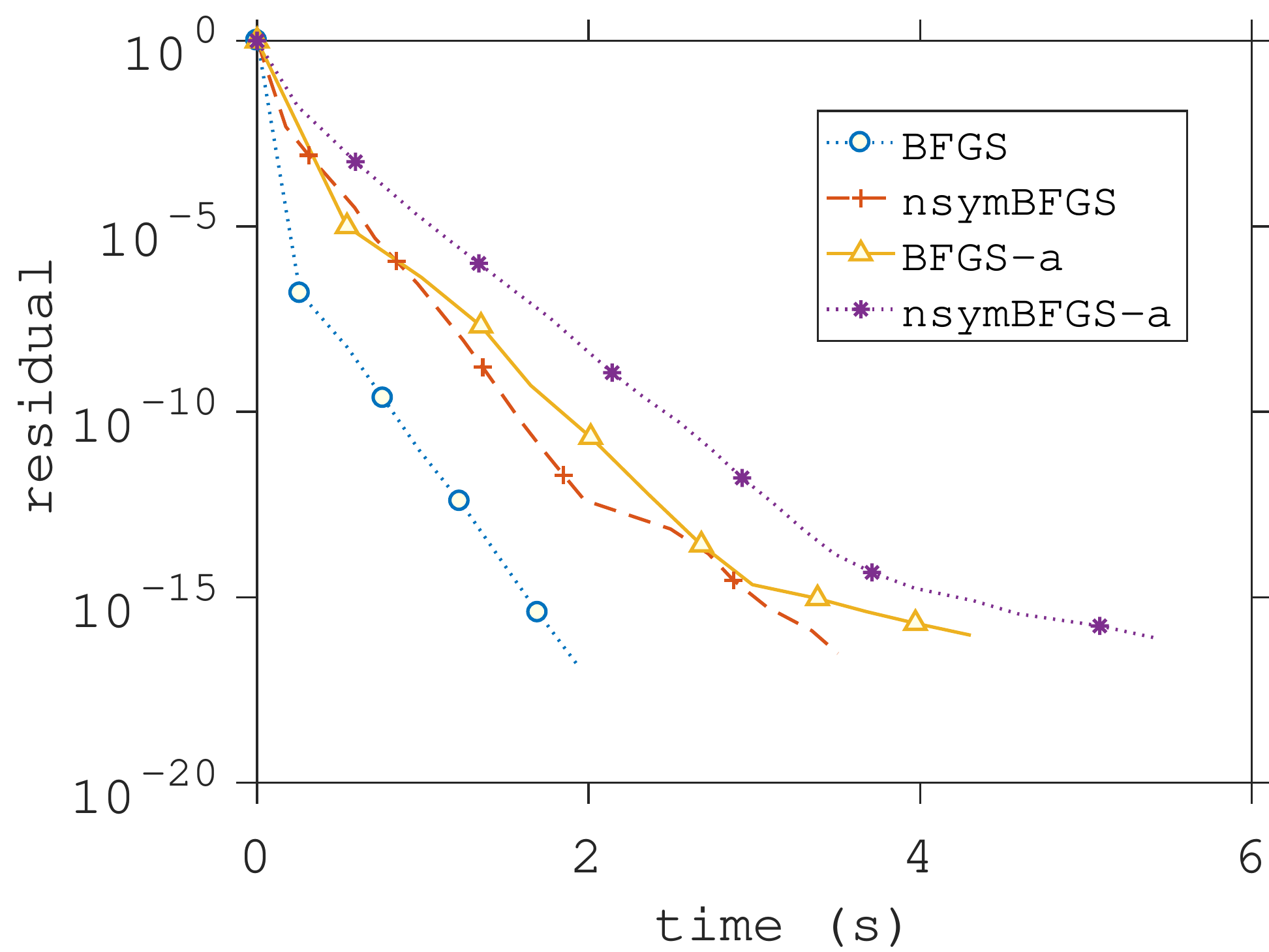}
\end{minipage}%
    \caption{  Eigenvalues set to $10000,1,1,\dots 1$. From left to right we have: Coordinate sketch with convenient probabilities, coordinate sketch with uniform probabilities and Gaussian sketch respectively. 
}\label{fig:rand_conv_110000}
\end{figure}

The numerical experiments in this section indicate that one might choose $\mu ,\nu$ as per Section~\ref{sec:convenient_munu}. In other words, one might pretend to be in the setting when symmetry is not enforced and coordinate sketches with convenient probabilities are used. In fact, the practical speedup coming from the acceleration depends very strongly on the structure of matrix $A$.
 Another message to be delivered is that both preserving symmetry and acceleration yield a better convergence and they combine together well.

We also consider a problem where we pretend to not have access to $\lambda_{\min}(A)$, therefore we cannot choose $\mu=\mu^P$. Instead, we naively choose $\mu=\frac{1}{100\nu}$ and $\mu=\frac{1}{10000\nu}$.

\begin{figure}[H]
    \centering
\begin{minipage}{0.30\textwidth}
  \centering
\includegraphics[width =  \textwidth ]{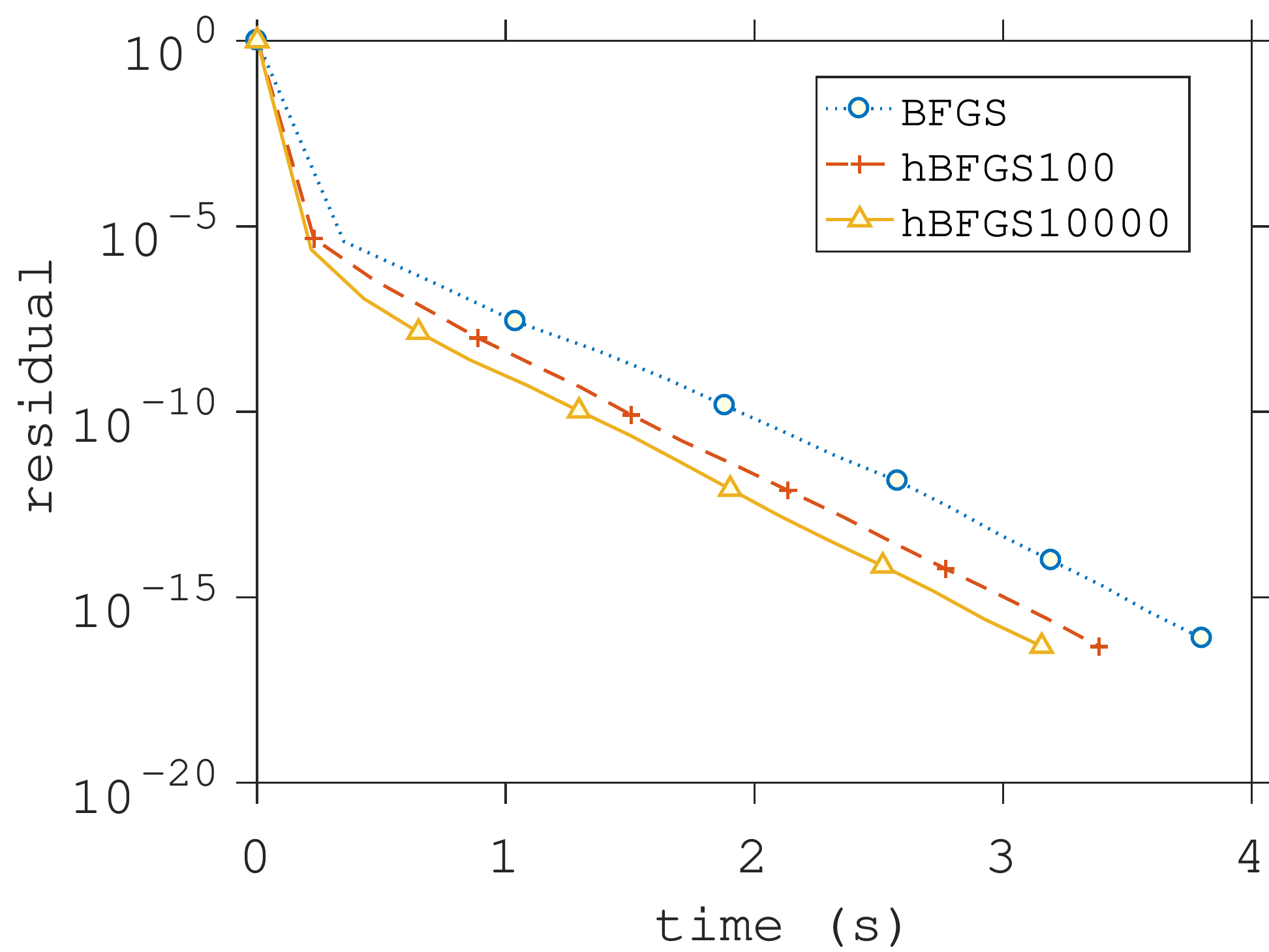}
\end{minipage}%
\begin{minipage}{0.30\textwidth}
  \centering
\includegraphics[width =  \textwidth ]{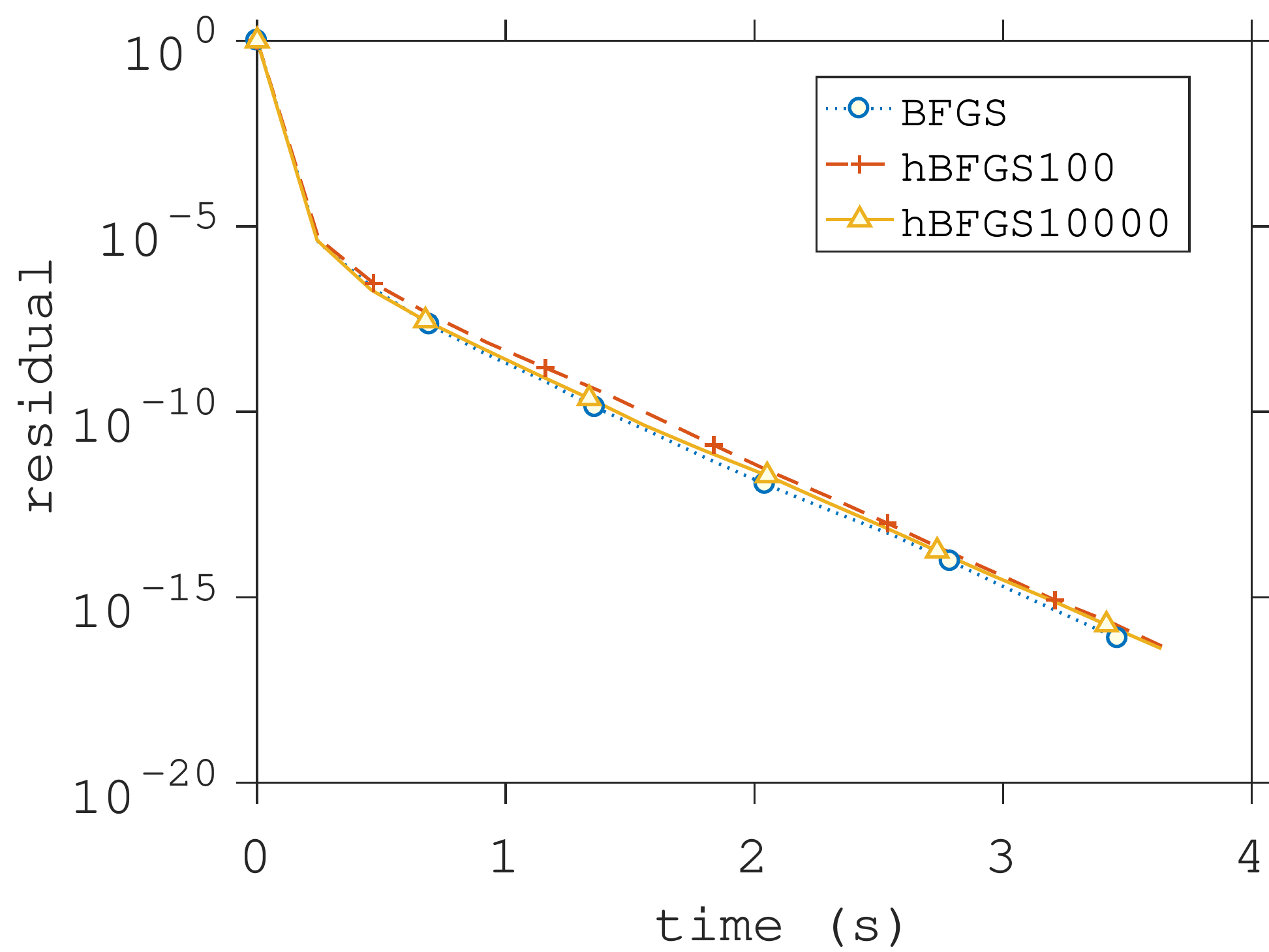}
\end{minipage}%
\begin{minipage}{0.30\textwidth}
  \centering
\includegraphics[width =  \textwidth ]{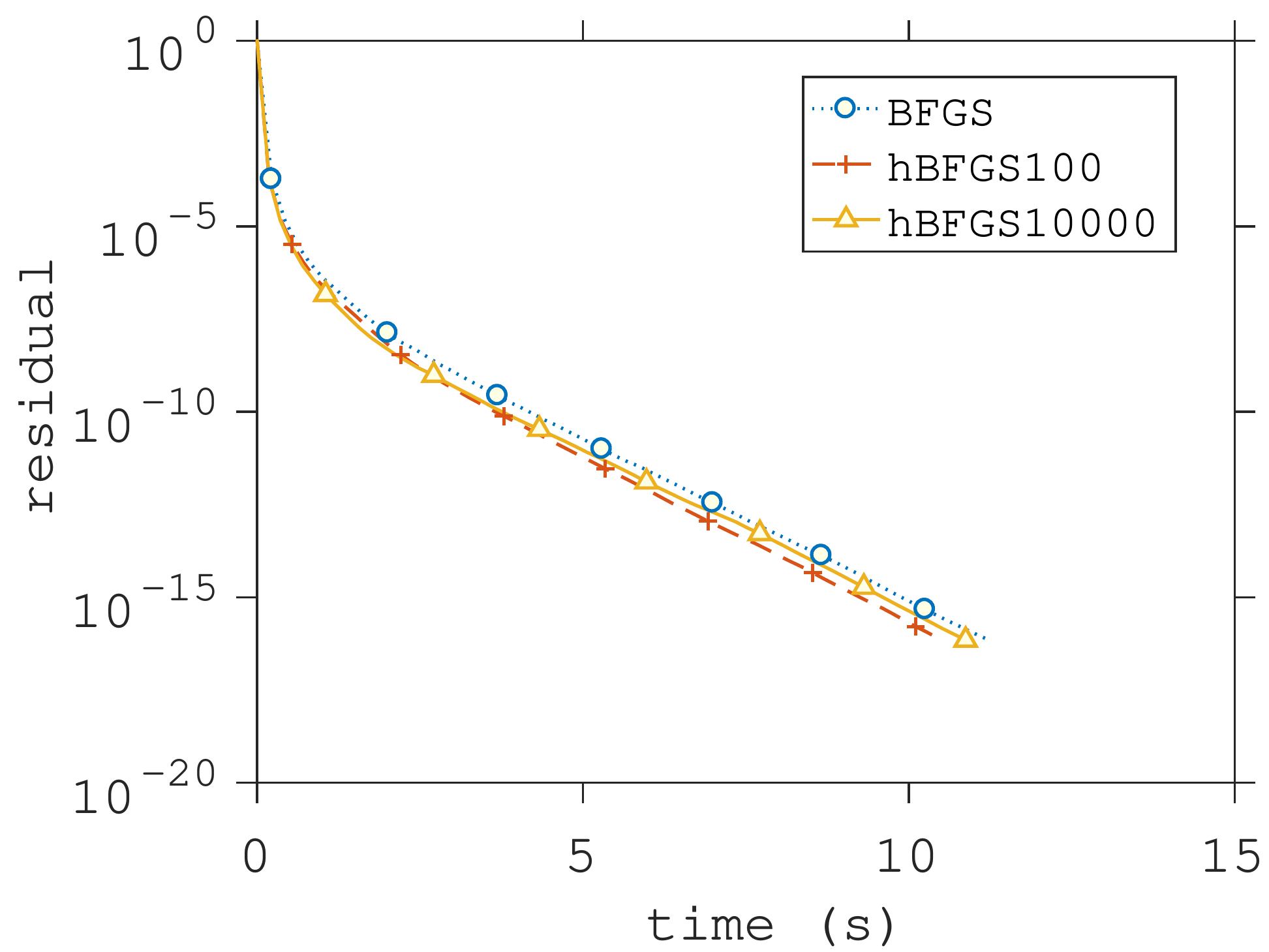}
\end{minipage}%
    \caption{  Eigenvalues set to $1,2,\dots, n$. From left to right we have: Coordinate sketch with convenient probabilities, coordinate sketch with uniform probabilities and Gaussian sketch respectively. 
}\label{fig:rand_conv_110000}
\end{figure}

\begin{figure}[H]
    \centering
\begin{minipage}{0.30\textwidth}
  \centering
\includegraphics[width =  \textwidth ]{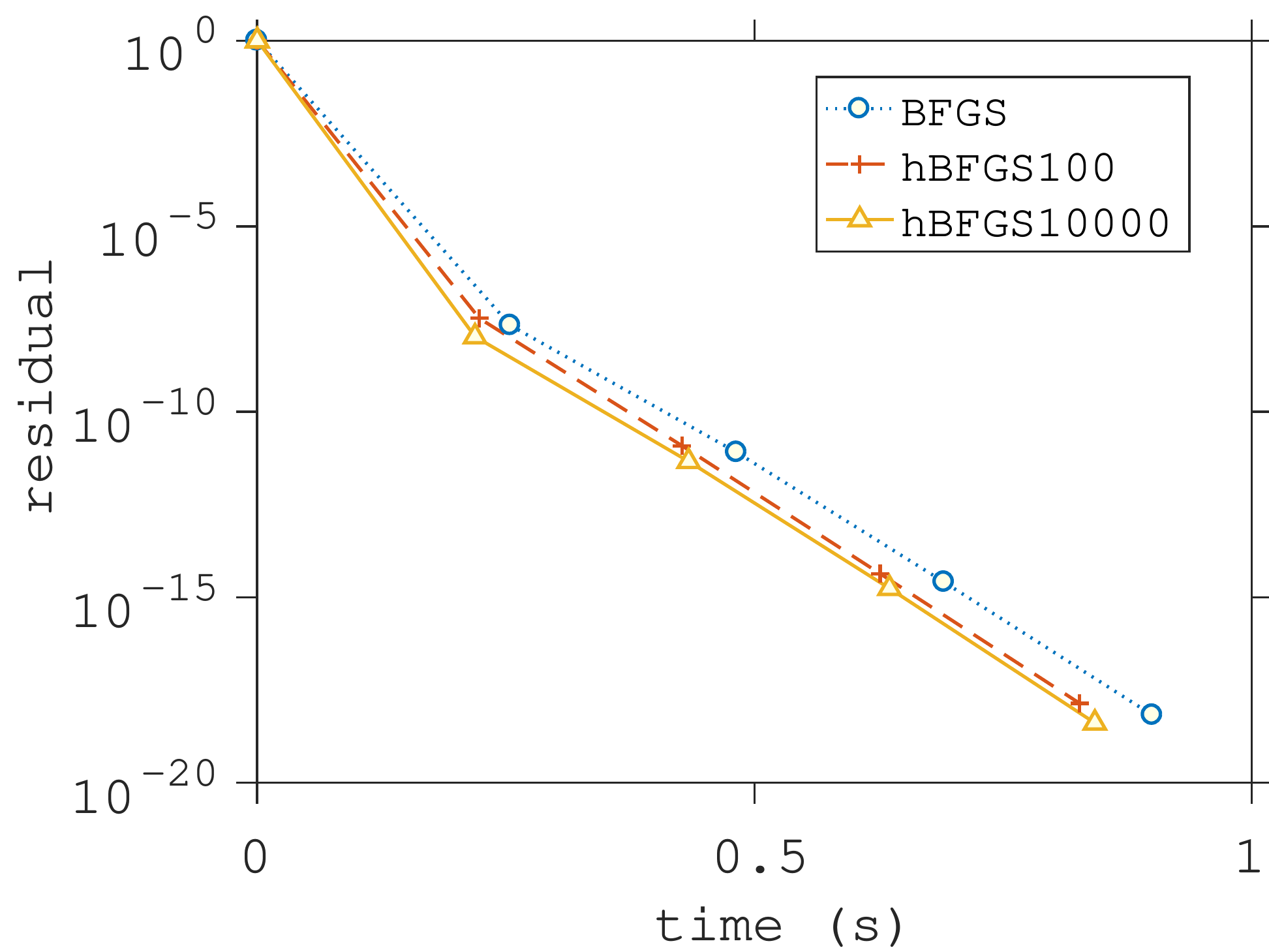}
\end{minipage}%
\begin{minipage}{0.30\textwidth}
  \centering
\includegraphics[width =  \textwidth ]{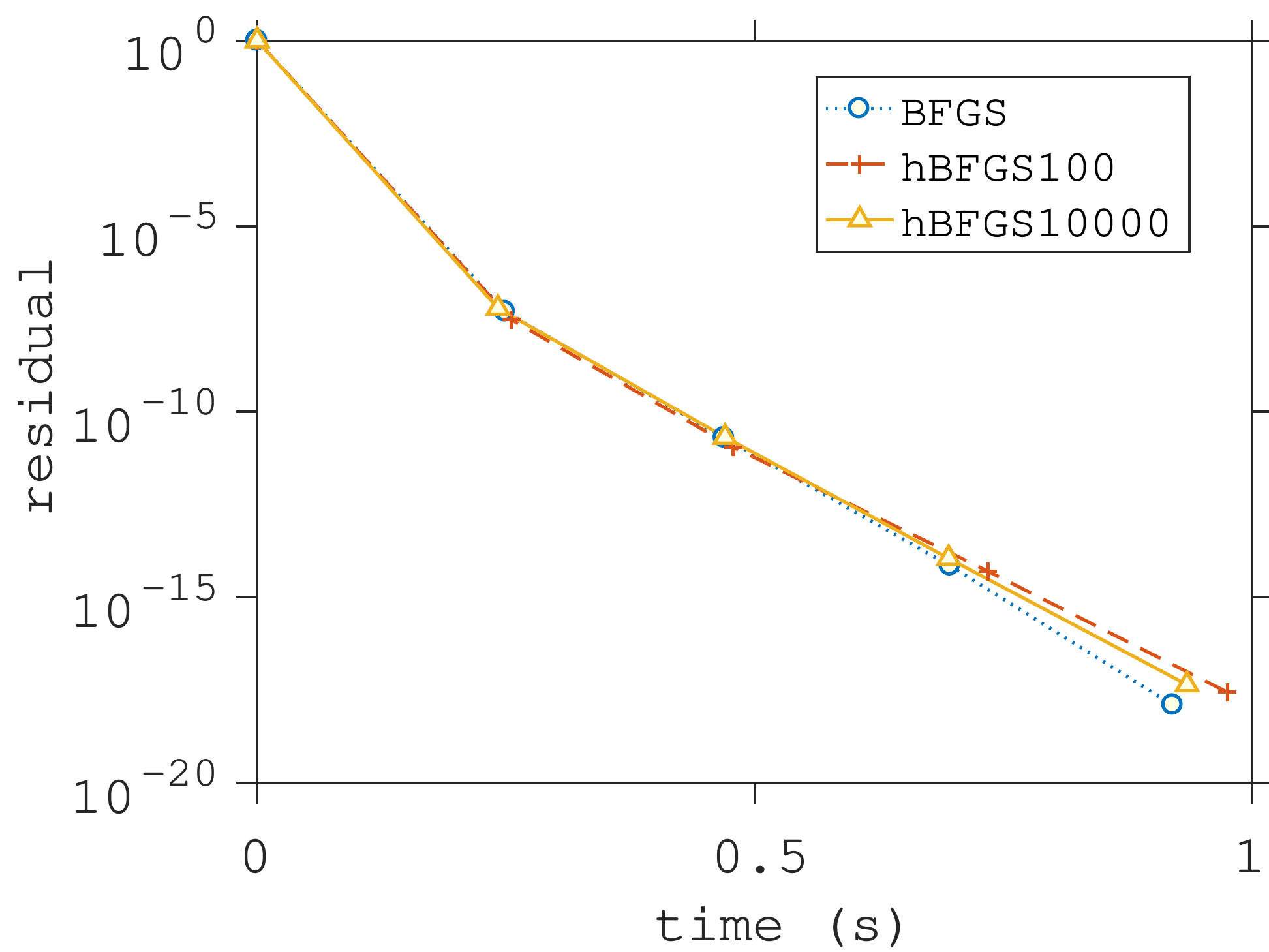}
\end{minipage}%
\begin{minipage}{0.30\textwidth}
  \centering
\includegraphics[width =  \textwidth ]{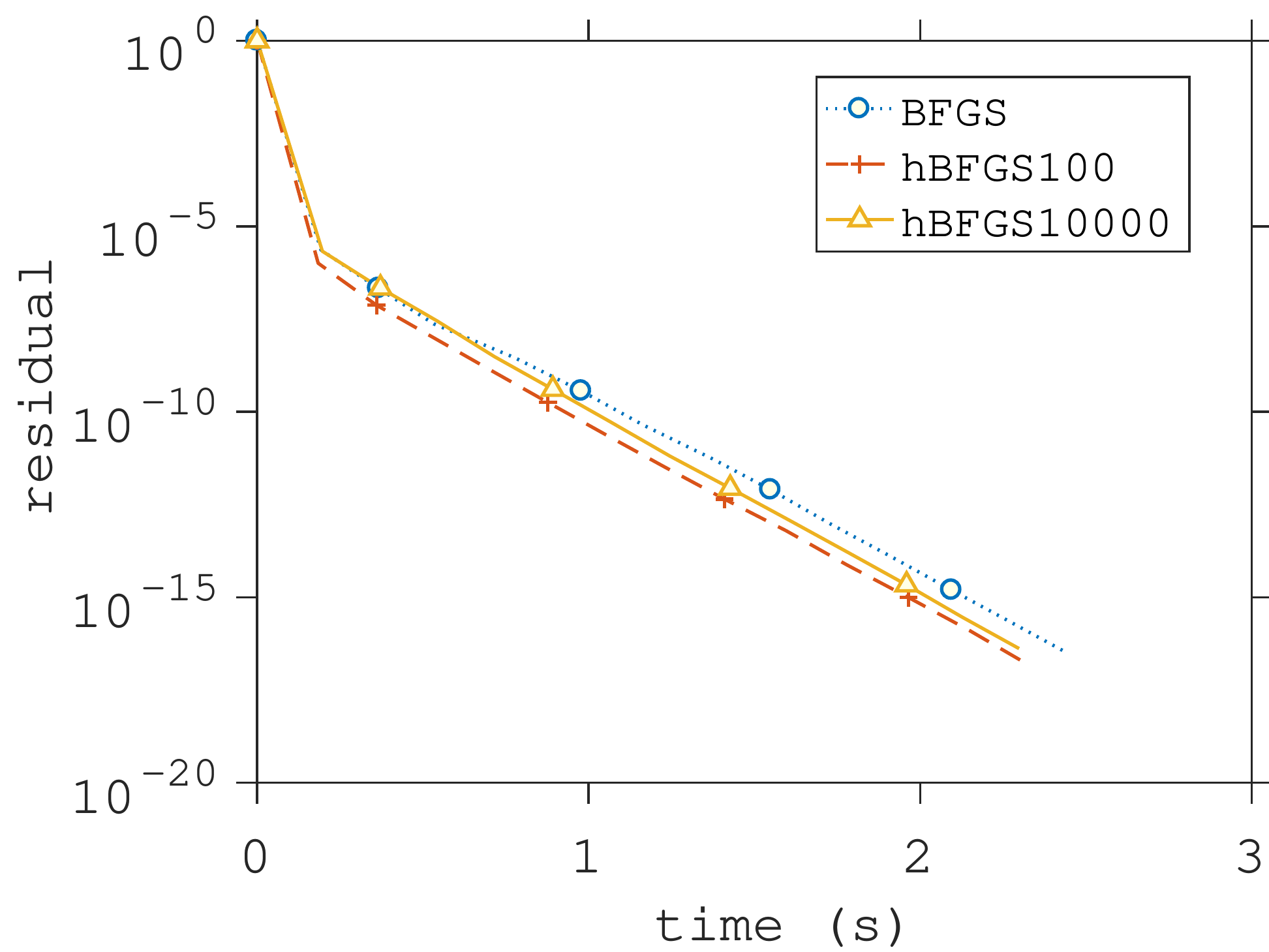}
\end{minipage}%
    \caption{Eigenvalues set to $1,10,10,\dots 10$. Coordinate sketch with convenient probabilities, coordinate sketch with uniform probabilities and Gaussian sketch respectively. 
}\label{fig:rand_conv_110000}
\end{figure}

\begin{figure}[H]
    \centering
\begin{minipage}{0.30\textwidth}
  \centering
\includegraphics[width =  \textwidth ]{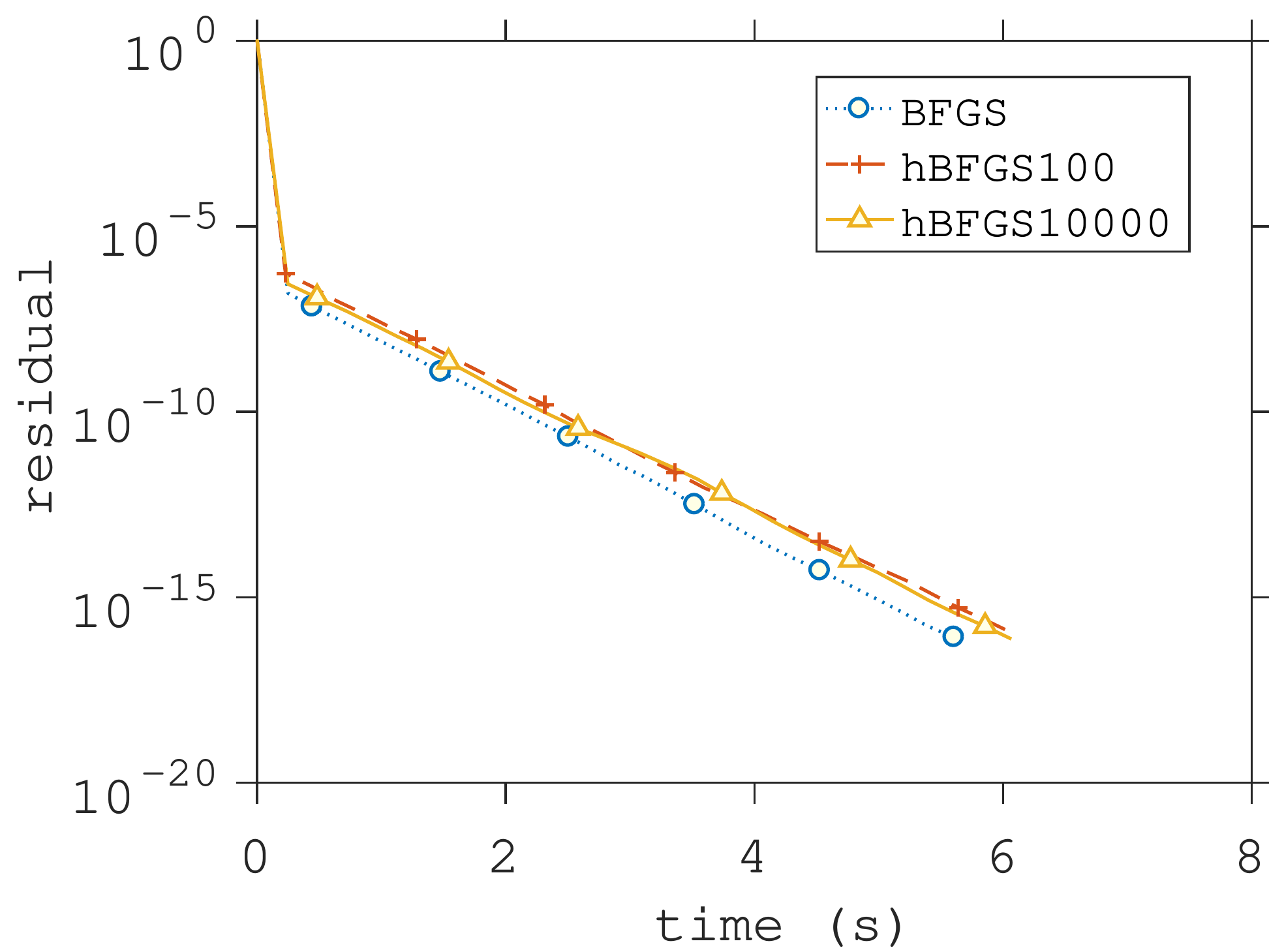}
\end{minipage}%
\begin{minipage}{0.30\textwidth}
  \centering
\includegraphics[width =  \textwidth ]{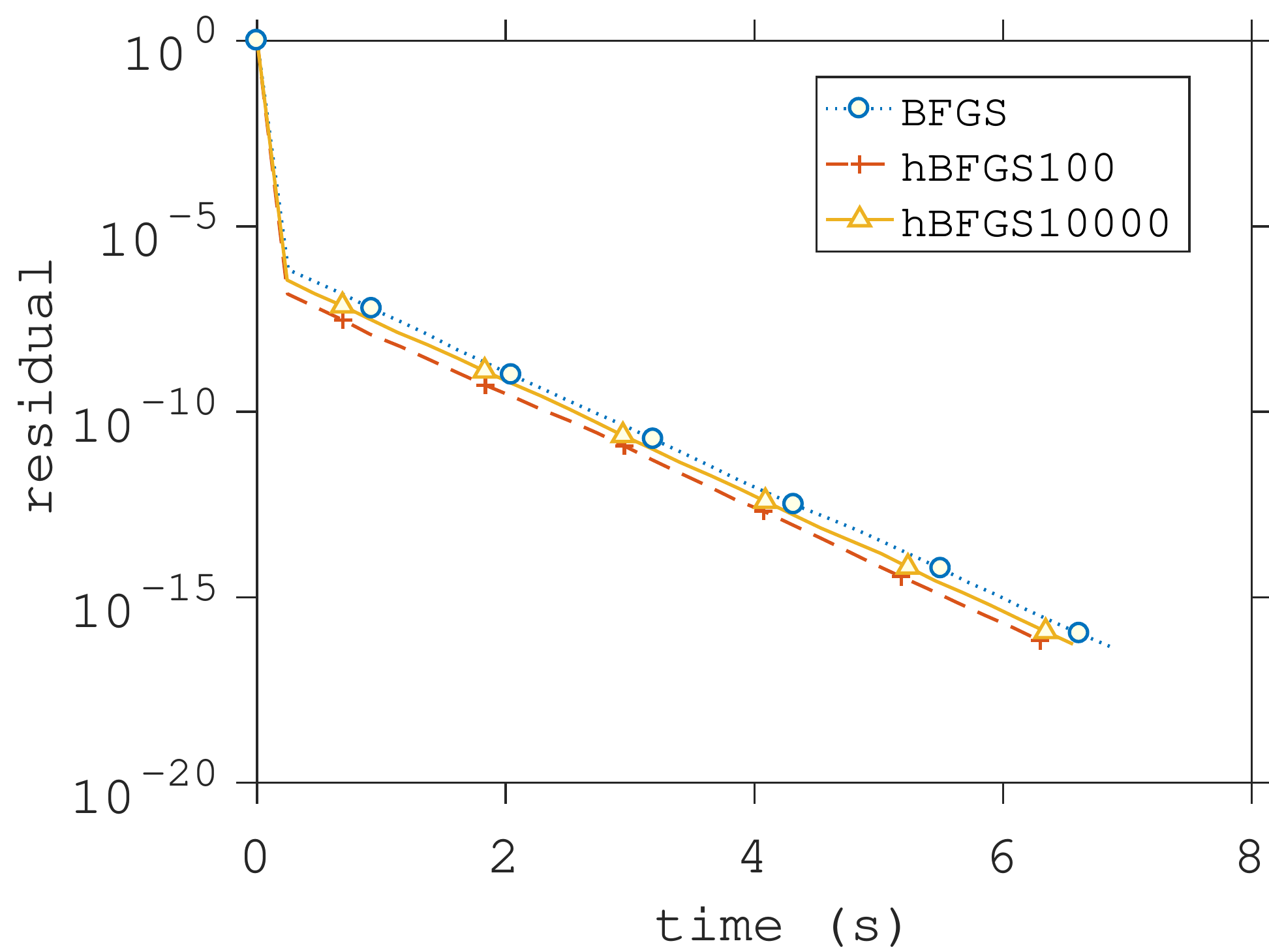}
\end{minipage}%
\begin{minipage}{0.30\textwidth}
  \centering
\includegraphics[width =  \textwidth ]{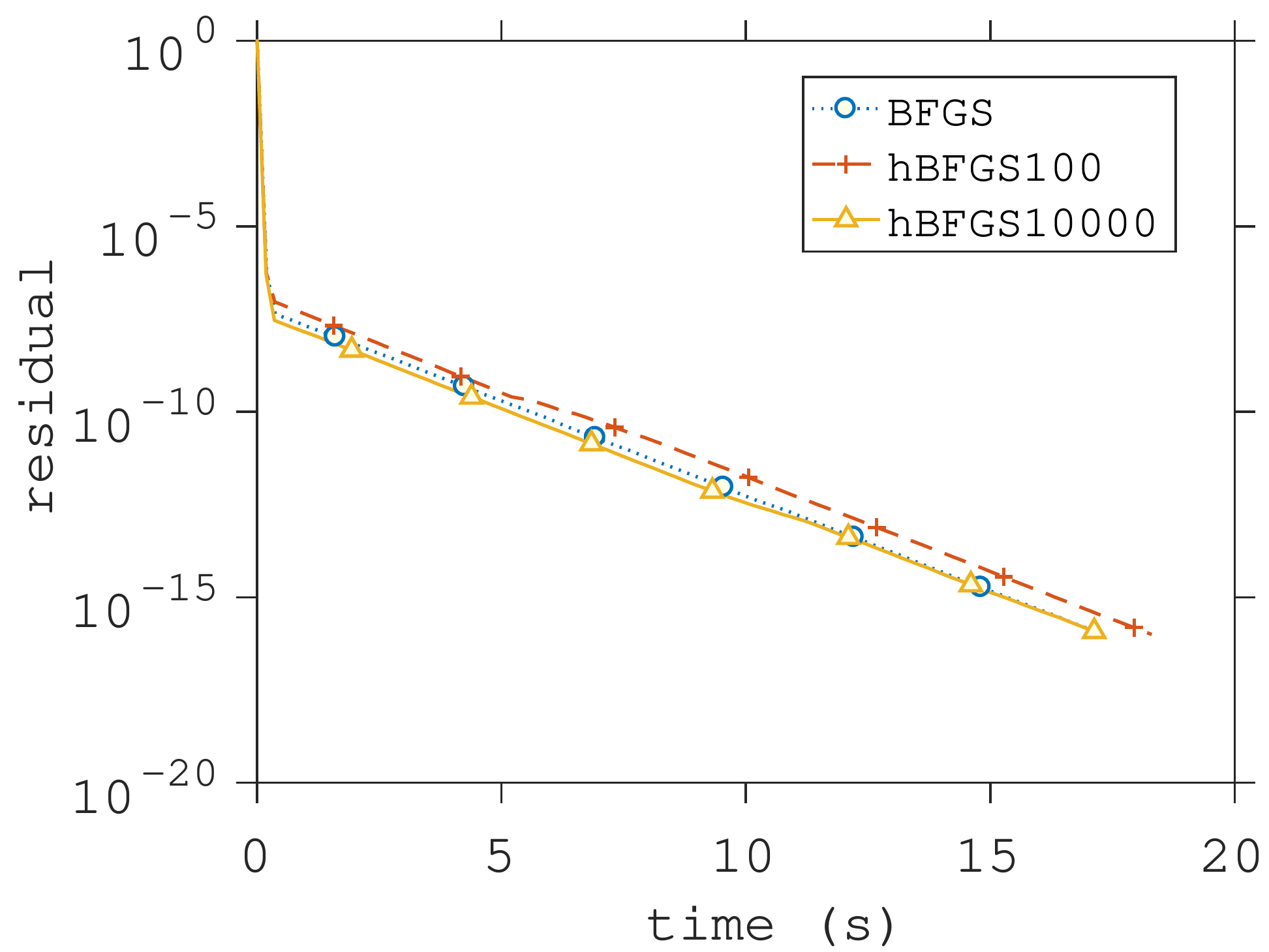}
\end{minipage}%
    \caption{Eigenvalues set to $1,100,100,\dots 100$. From left to right we have: Coordinate sketch with convenient probabilities, coordinate sketch with uniform probabilities and Gaussian sketch respectively. 
}\label{fig:rand_conv_110000}
\end{figure}

\begin{figure}[H]
    \centering
\begin{minipage}{0.30\textwidth}
  \centering
\includegraphics[width =  \textwidth ]{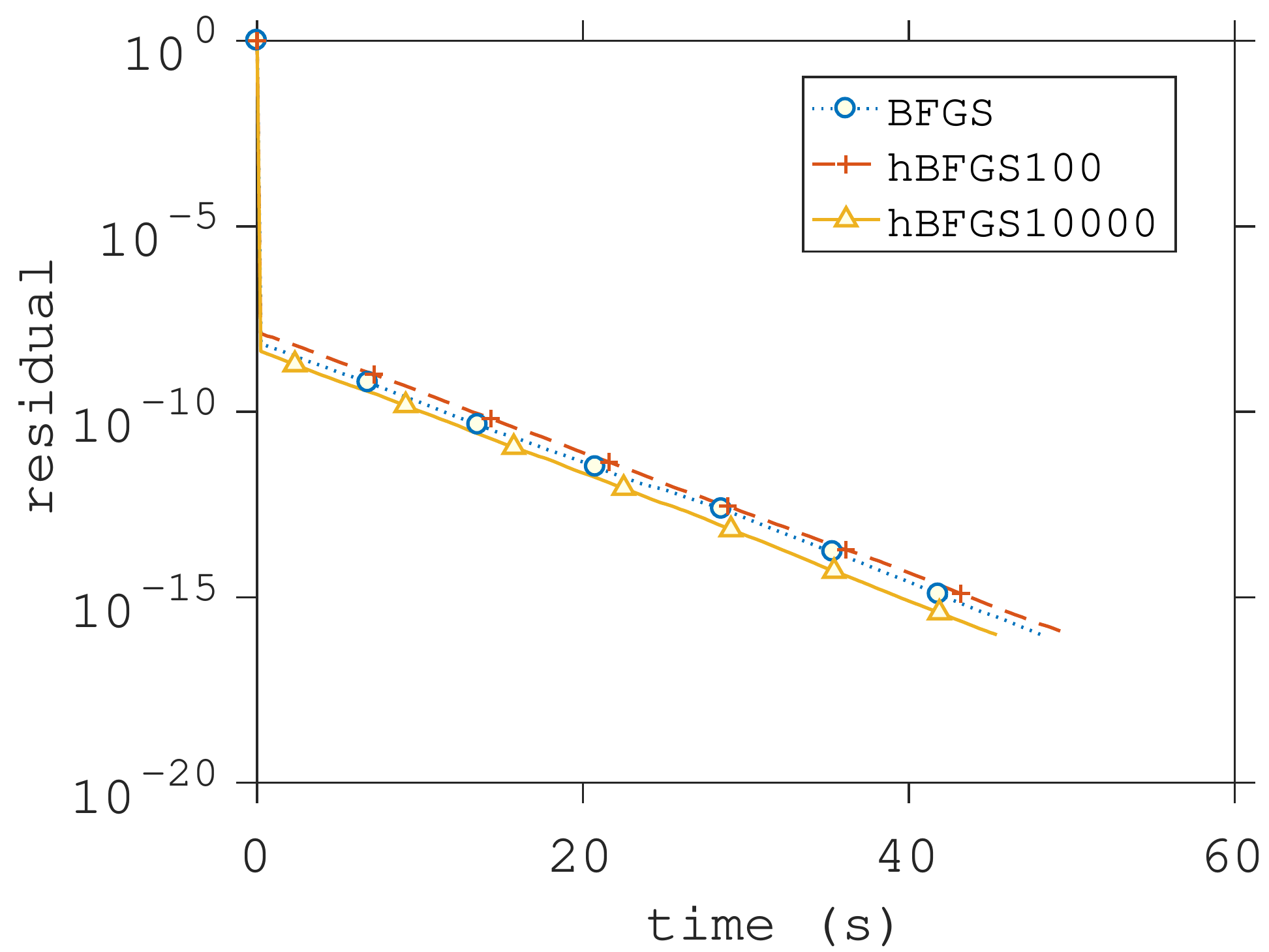}
\end{minipage}%
\begin{minipage}{0.30\textwidth}
  \centering
\includegraphics[width =  \textwidth ]{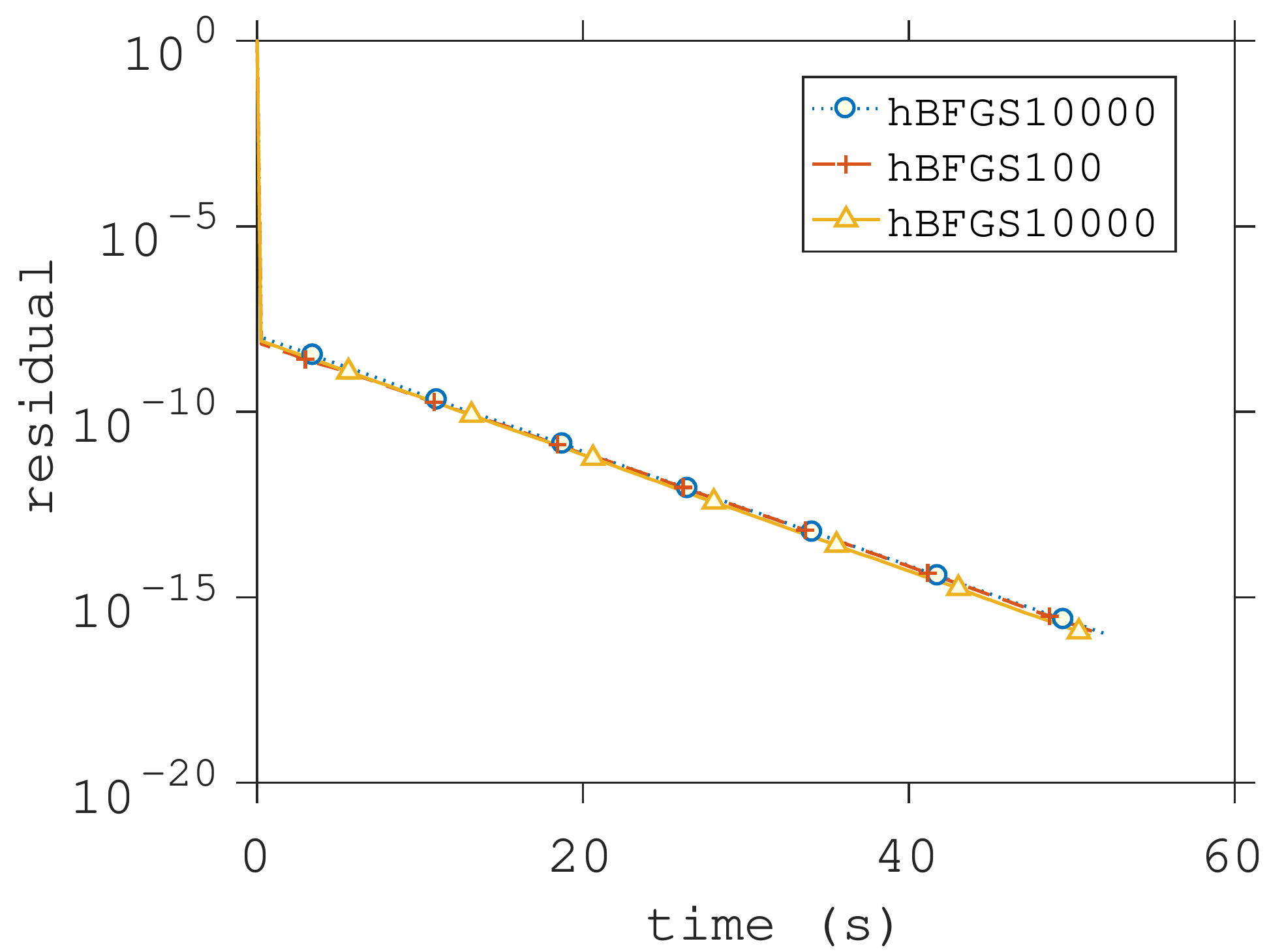}
\end{minipage}%
\begin{minipage}{0.30\textwidth}
  \centering
\includegraphics[width =  \textwidth ]{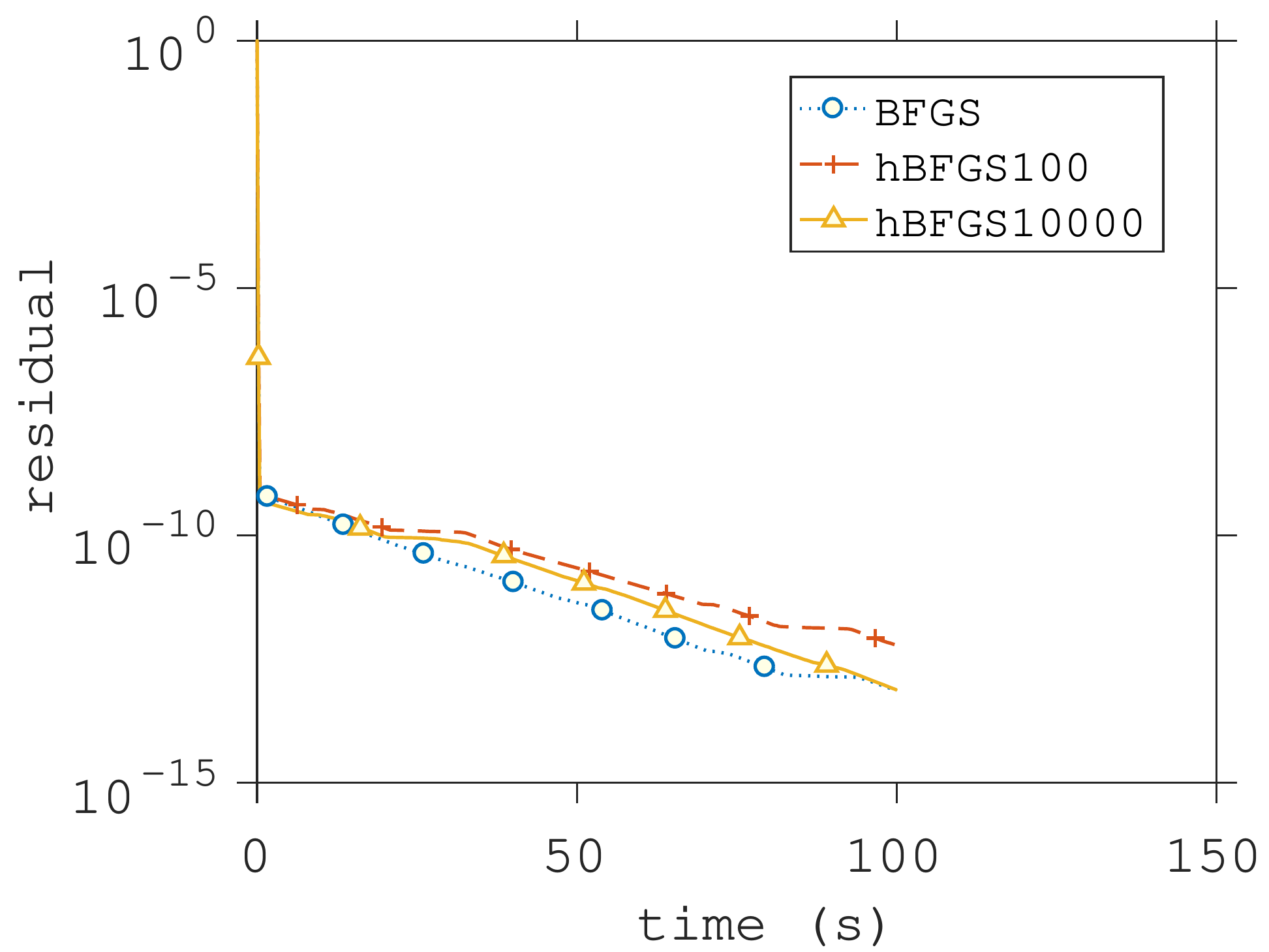}
\end{minipage}%
    \caption{Eigenvalues set to $1,1000,1000,\dots 1000$. From left to right we have: Coordinate sketch with convenient probabilities, coordinate sketch with uniform probabilities and Gaussian sketch respectively. 
}\label{fig:rand_conv_110000}
\end{figure}

\begin{figure}[H]
    \centering
\begin{minipage}{0.30\textwidth}
  \centering
\includegraphics[width =  \textwidth ]{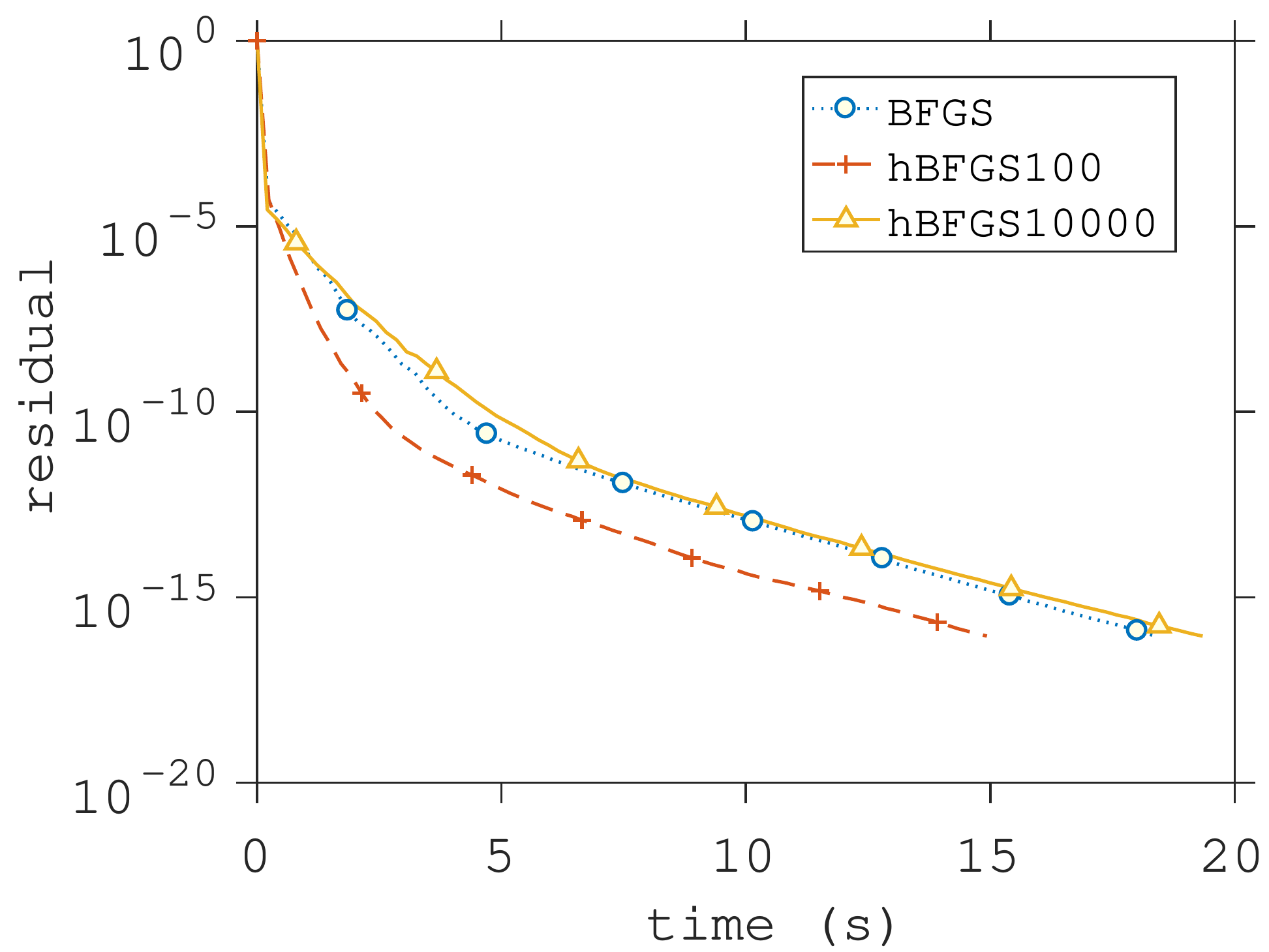}
\end{minipage}%
\begin{minipage}{0.30\textwidth}
  \centering
\includegraphics[width =  \textwidth ]{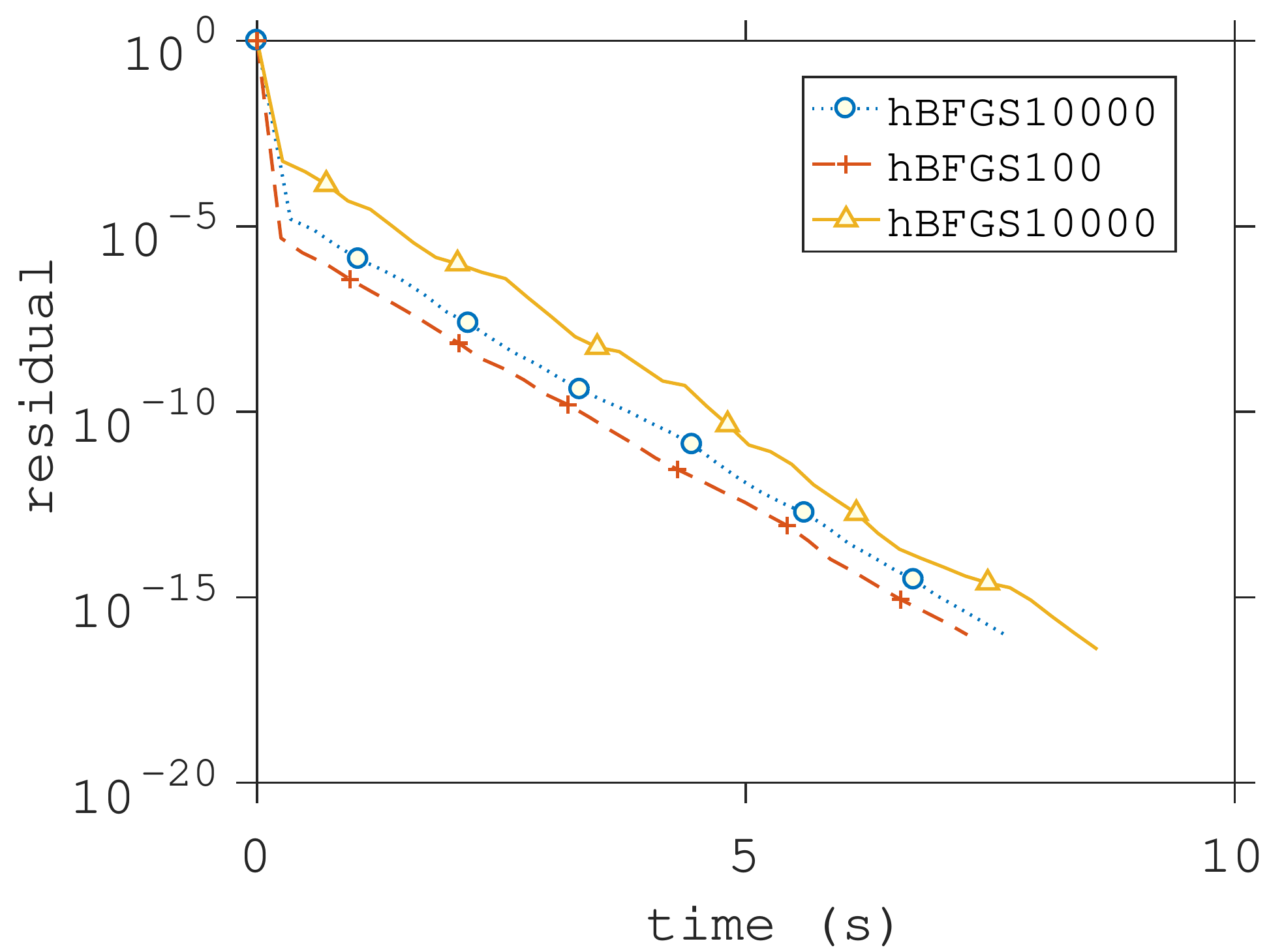}
\end{minipage}%
\begin{minipage}{0.30\textwidth}
  \centering
\includegraphics[width =  \textwidth ]{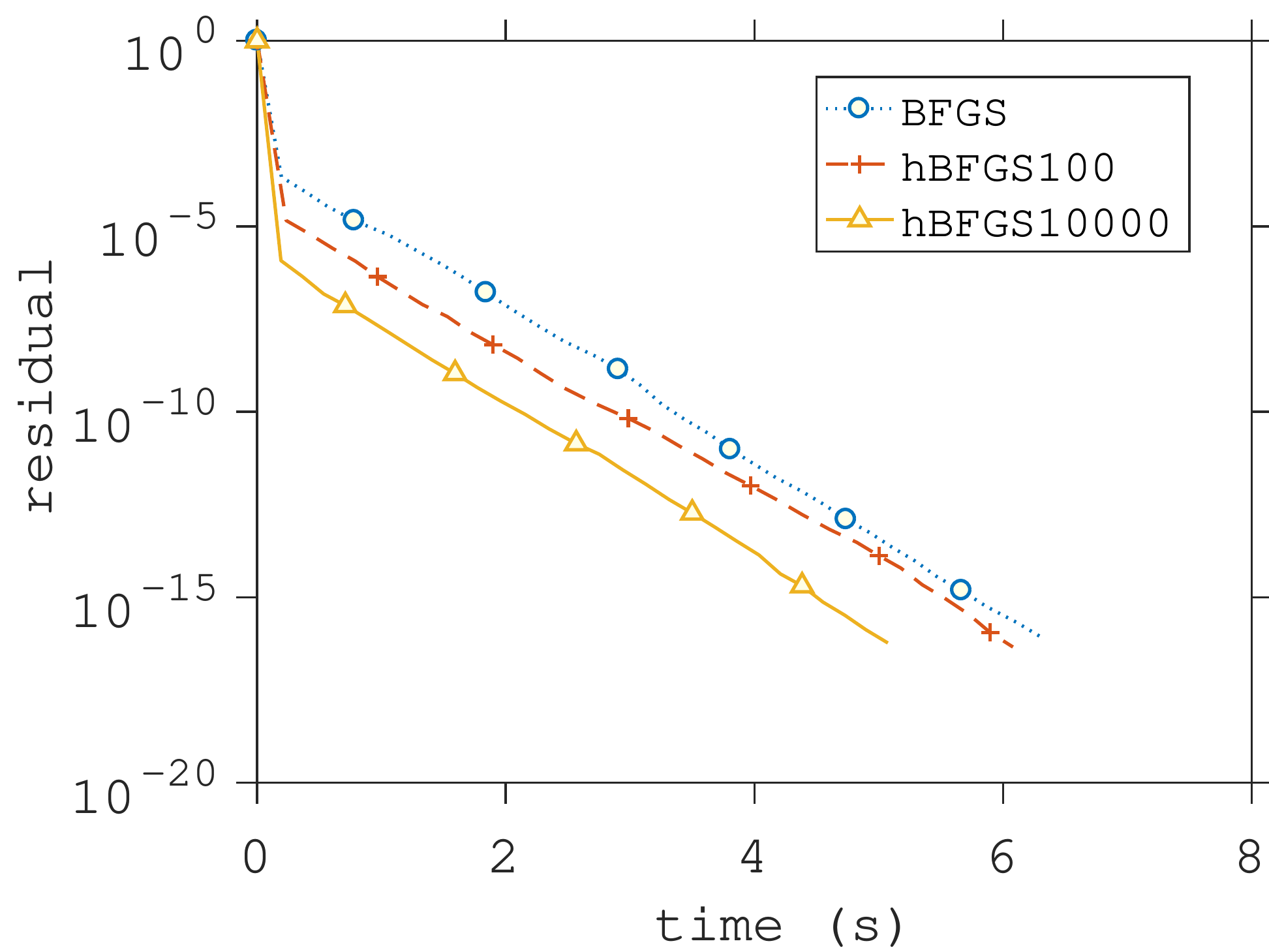}
\end{minipage}%
    \caption{Eigenvalues set to $10000,1,1,\dots 1$. From left to right we have: Coordinate sketch with convenient probabilities, coordinate sketch with uniform probabilities and Gaussian sketch respectively. 
}\label{fig:rand_conv_110000}
\end{figure}

Notice that once the acceleration parameters are not set exactly (but they are still reasonable), we observe that the performance of the accelerated algorithm is essentially the same as the performance of the nonaccelerated algorithm. We have observed the similar behavior when setting $\mu=\mu^P$ for Gaussian sketches. 

\subsubsection{Sensitivity to the acceleration parameters}
Here we investigate the sensitivity of the accelerated BFGS to the parameters $\mu$ and $\nu$. First we compute $\nu^P, \mu^P$ and from this we extract the following exponential grids: $\mu_i =2^{i-4}\mu$ and $\nu_i=5^{i-4}\nu$ for $i=1,2,\dots 7$. To gauge the gain is using acceleration with a particular $(\mu, \nu)$ pair, we run the accelerated algorithm for a fixed time then store the error of the final iterate. We then compute average per iteration decrease and divide it by average per iteration decrease of nonaccelerated algorithm. Thus if the resulting difference is less than one, then the accelerated algorithm was faster to nonaccelerated. 

In the plots below, $n=200$ was chosen. We focused on 2 problems described in the previous section---when the eigenvalues are uniformly distributed and when the the largest eigenvalue have multiplicity $n-1$. 

\begin{figure}[H]
    \centering
\begin{minipage}{0.33\textwidth}
  \centering
\includegraphics[width =  \textwidth]{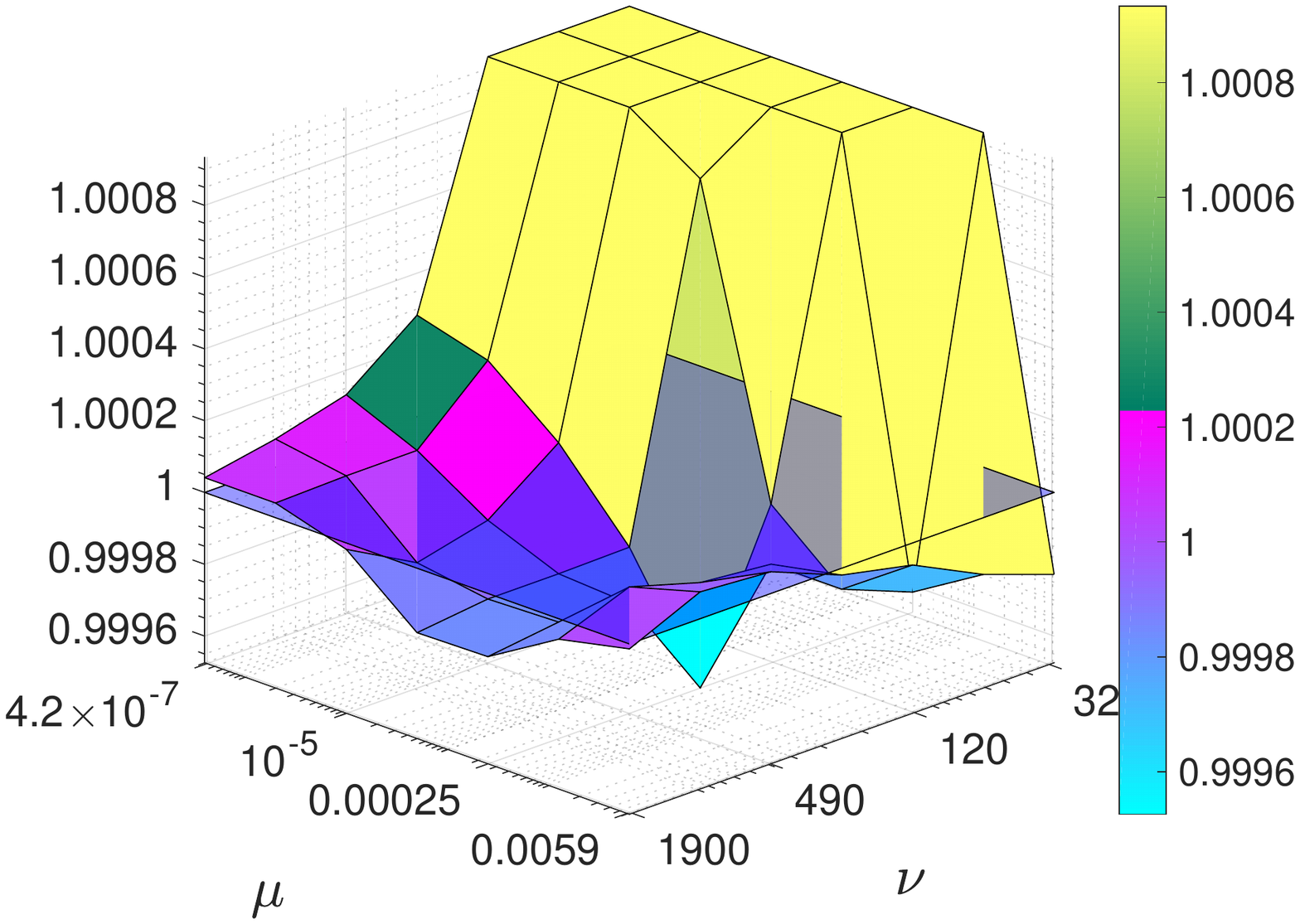}
\end{minipage}%
\begin{minipage}{0.33\textwidth}
  \centering
\includegraphics[width =  \textwidth ]{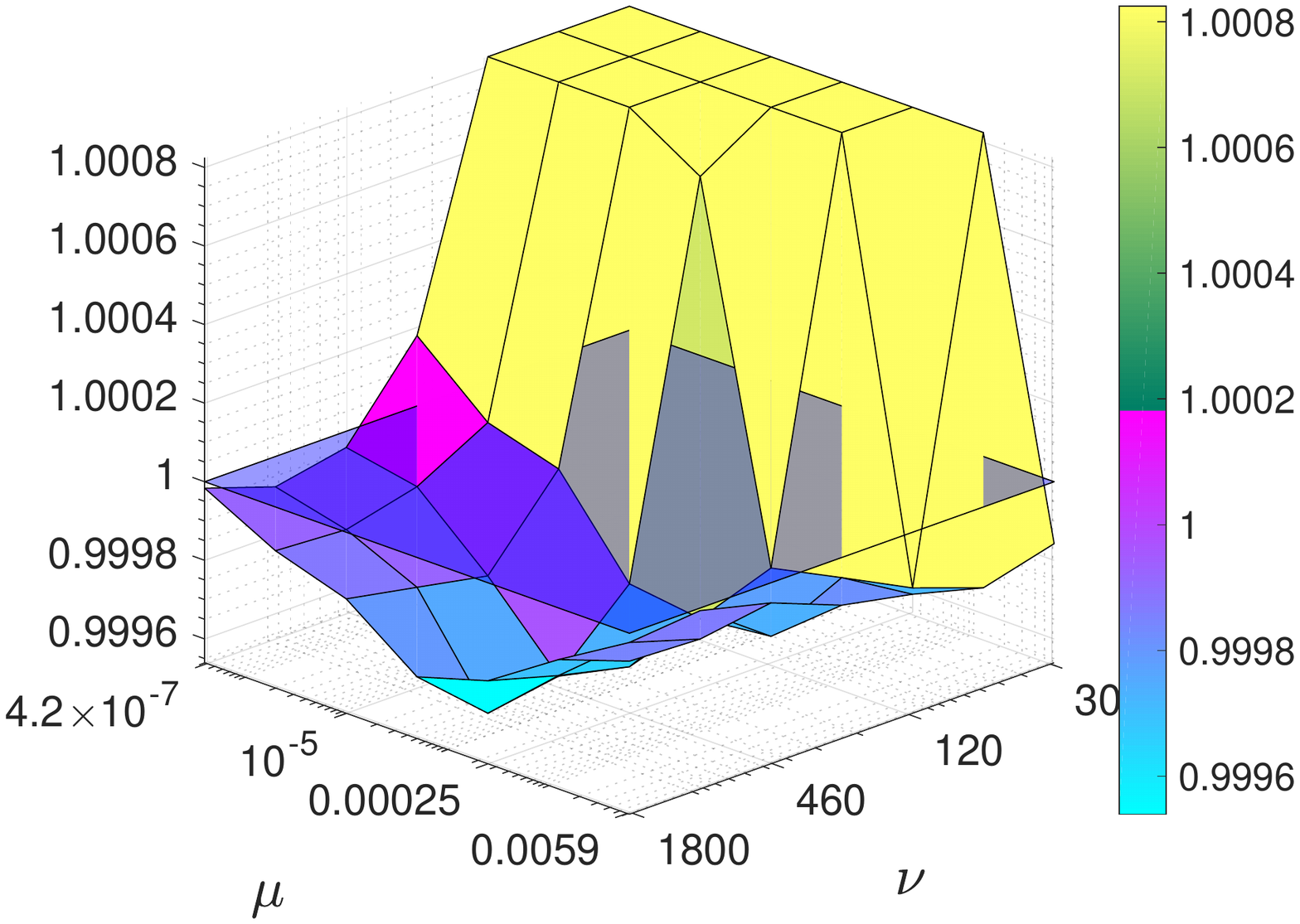}
\end{minipage}%
\begin{minipage}{0.33\textwidth}
  \centering
\includegraphics[width =  \textwidth ]{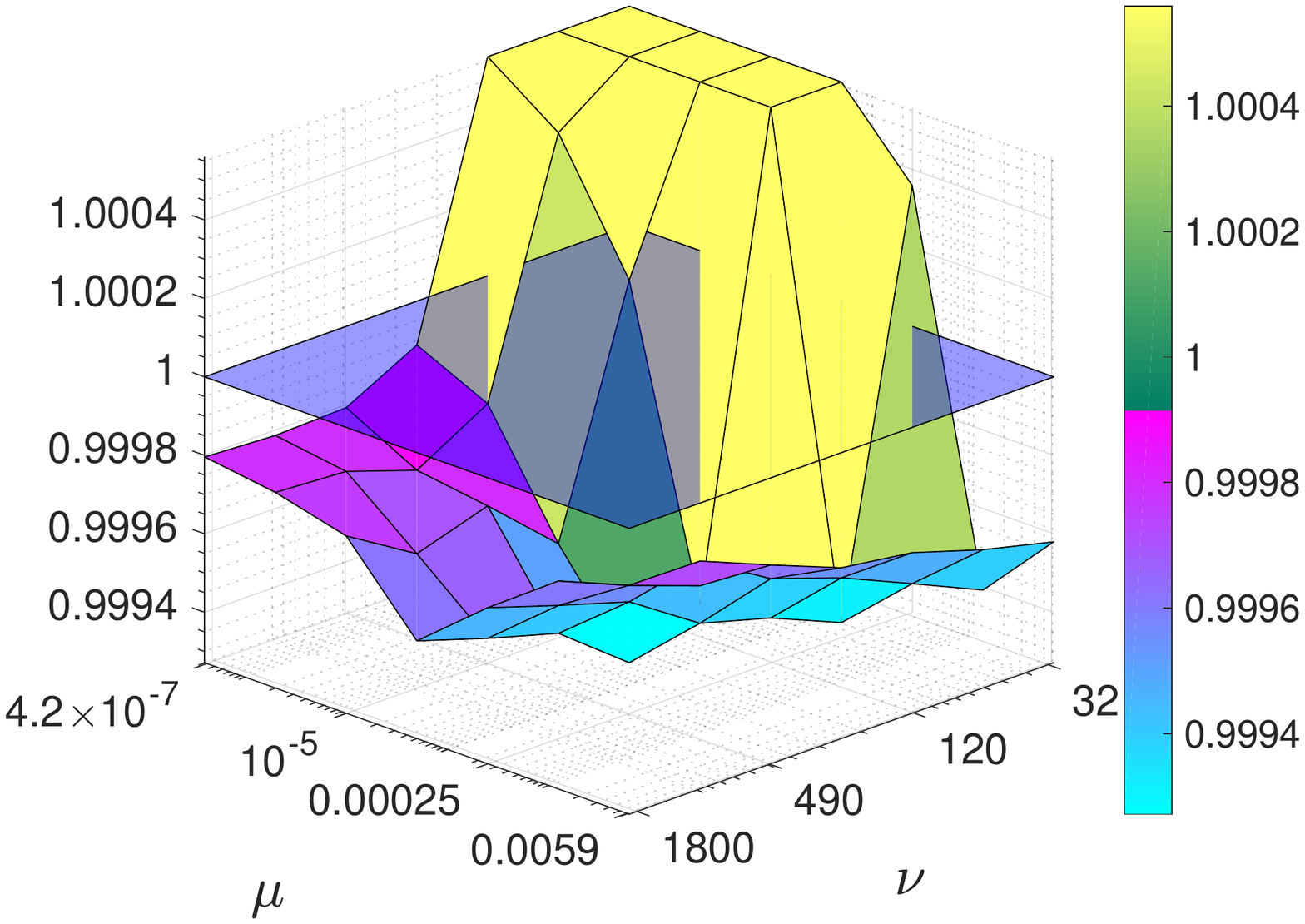}
\end{minipage}%
    \caption{Sensitivity to acceleration parameters. Eigenvalues of $A$ are set to $1,2\dots,n$. From left to right we have: Coordinate sketches with convenient probabilities, coordiante sketches with uniform probabilities and Gaussian sketches. Choice of parameters as per \eqref{eq:munu_conv_paper} in the middle of plots. Each instance was run for 5 seconds. }
\label{fig:heat_gauss_rand}
\end{figure}

\begin{figure}[H]
    \centering
\begin{minipage}{0.33\textwidth}
  \centering
\includegraphics[width =  \textwidth]{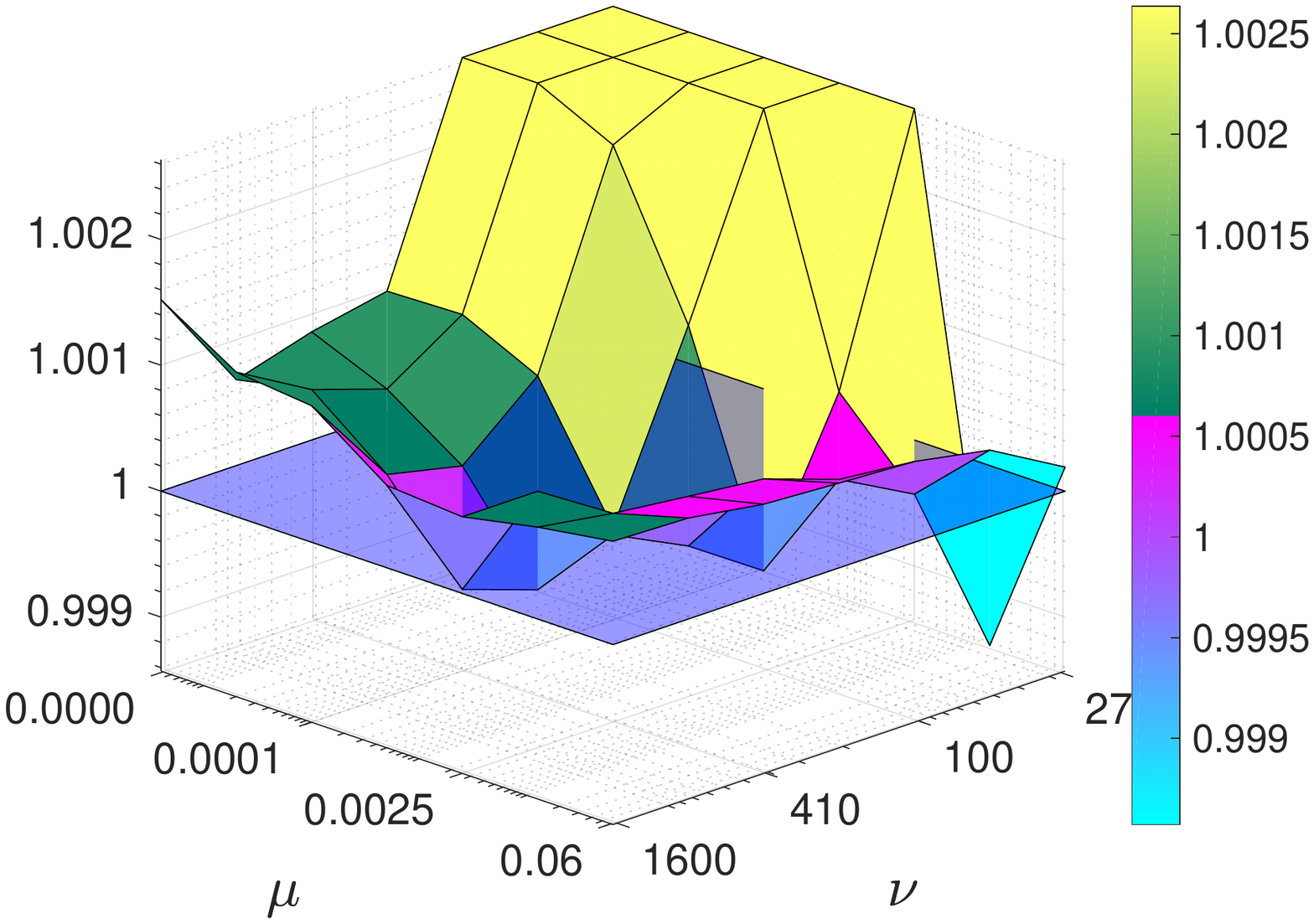}
\end{minipage}%
\begin{minipage}{0.33\textwidth}
  \centering
\includegraphics[width =  \textwidth ]{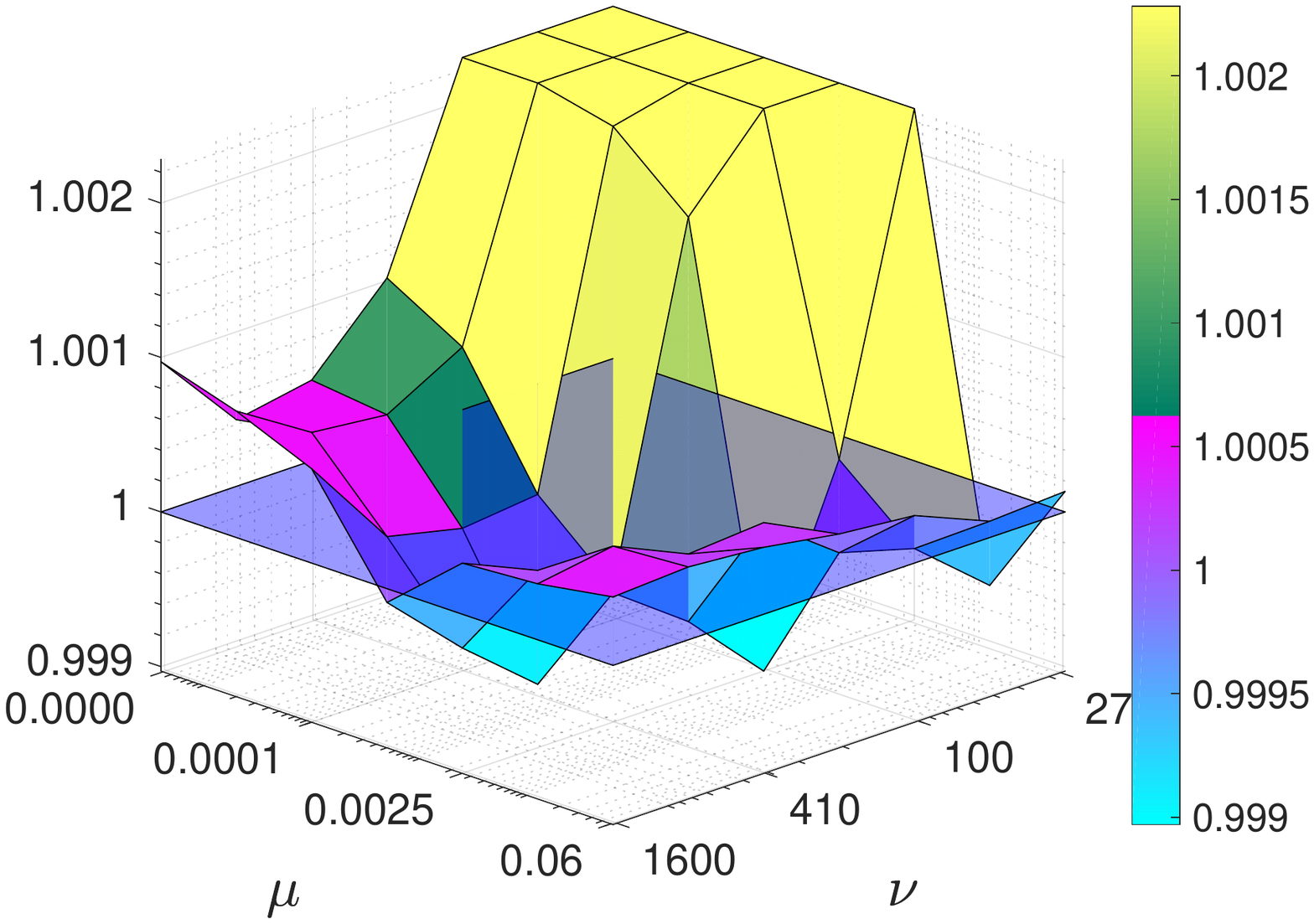}
\end{minipage}%
\begin{minipage}{0.33\textwidth}
  \centering
\includegraphics[width =  \textwidth ]{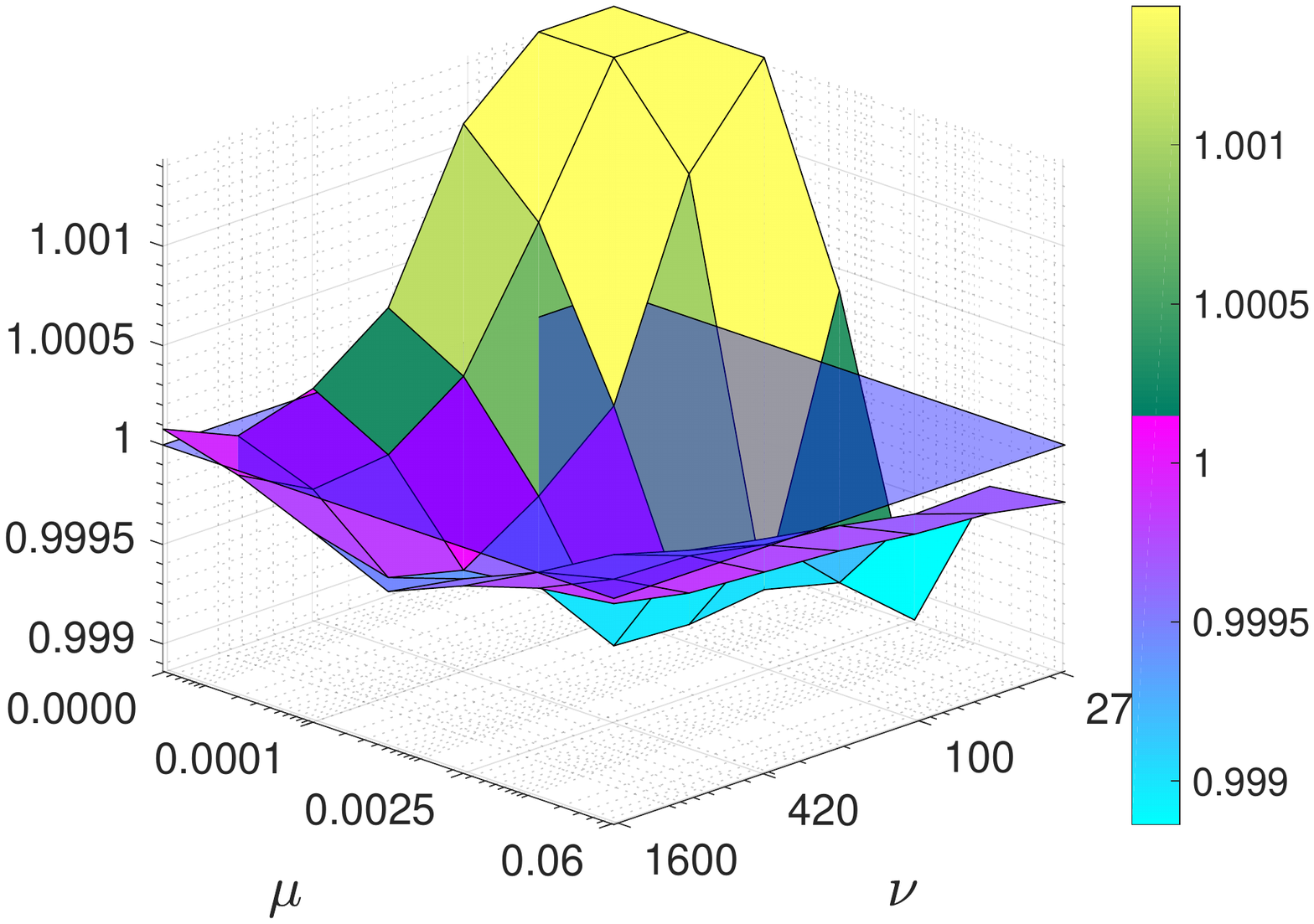}
\end{minipage}%
    \caption{Sensitivity to acceleration parameters. Eigenvalues of $A$ are set to $1,10, 10,\dots,10 $. From left to right we have: Coordinate sketches with convenient probabilities, coordiante sketches with uniform probabilities and Gaussian sketches. Choice of parameters as per \eqref{eq:munu_conv_paper} in the middle of plots. Each instance was run for 2 seconds.}
    \label{fig:heat_gauss_rand}
\end{figure}

\begin{figure}[H]
    \centering
\begin{minipage}{0.33\textwidth}
  \centering
\includegraphics[width =  \textwidth]{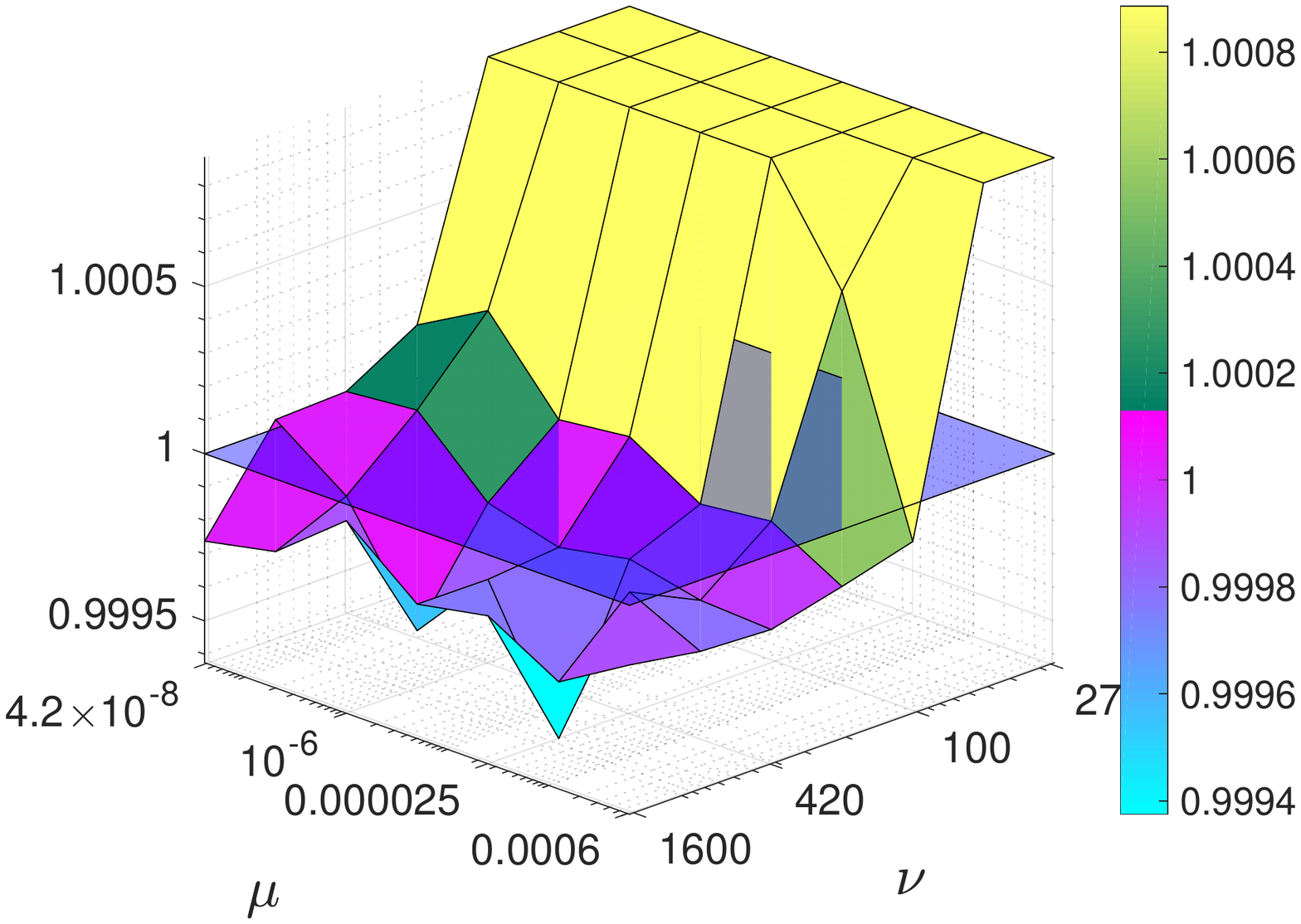}
\end{minipage}%
\begin{minipage}{0.33\textwidth}
  \centering
\includegraphics[width =  \textwidth ]{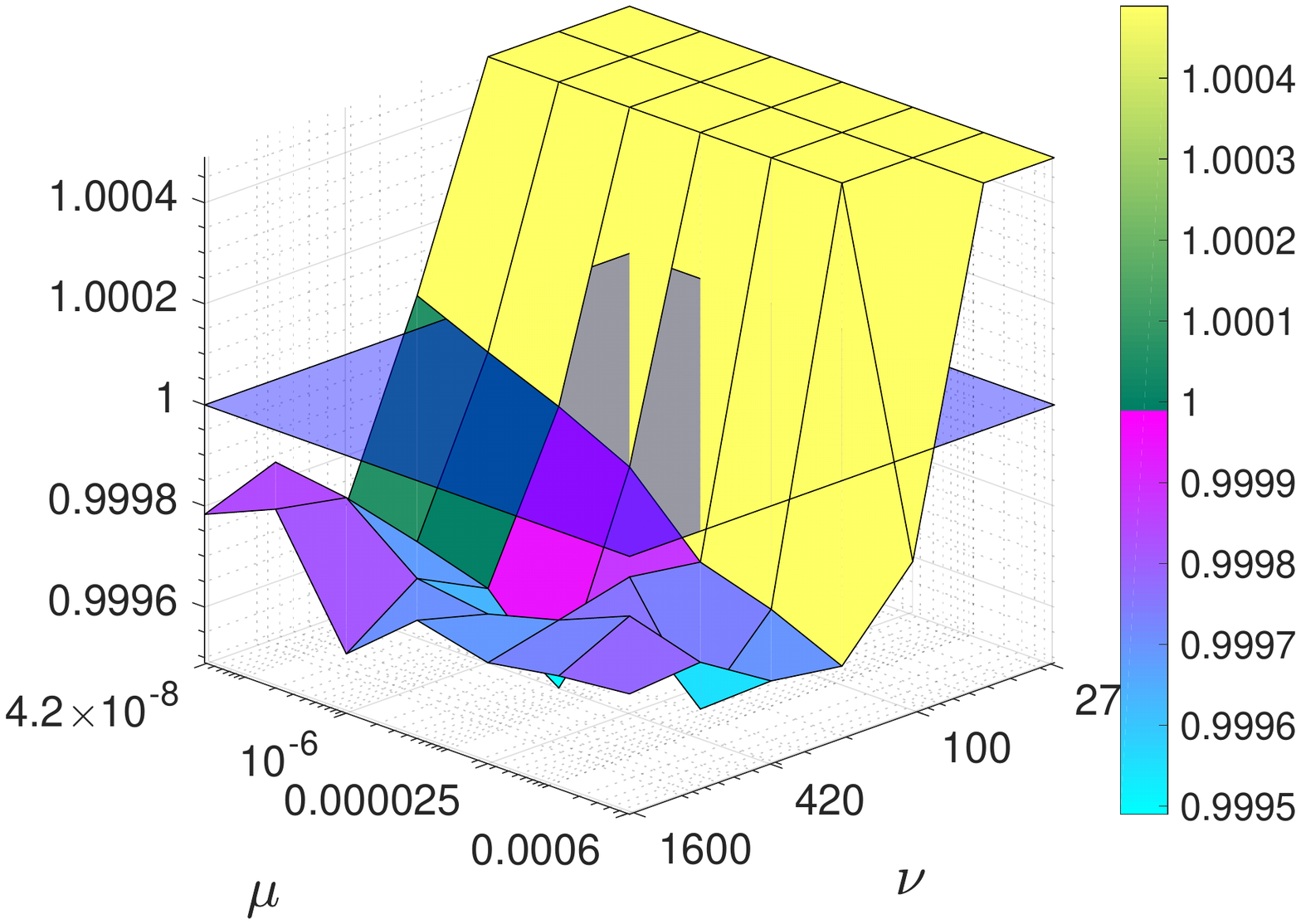}
\end{minipage}%
\begin{minipage}{0.33\textwidth}
  \centering
\includegraphics[width =  \textwidth ]{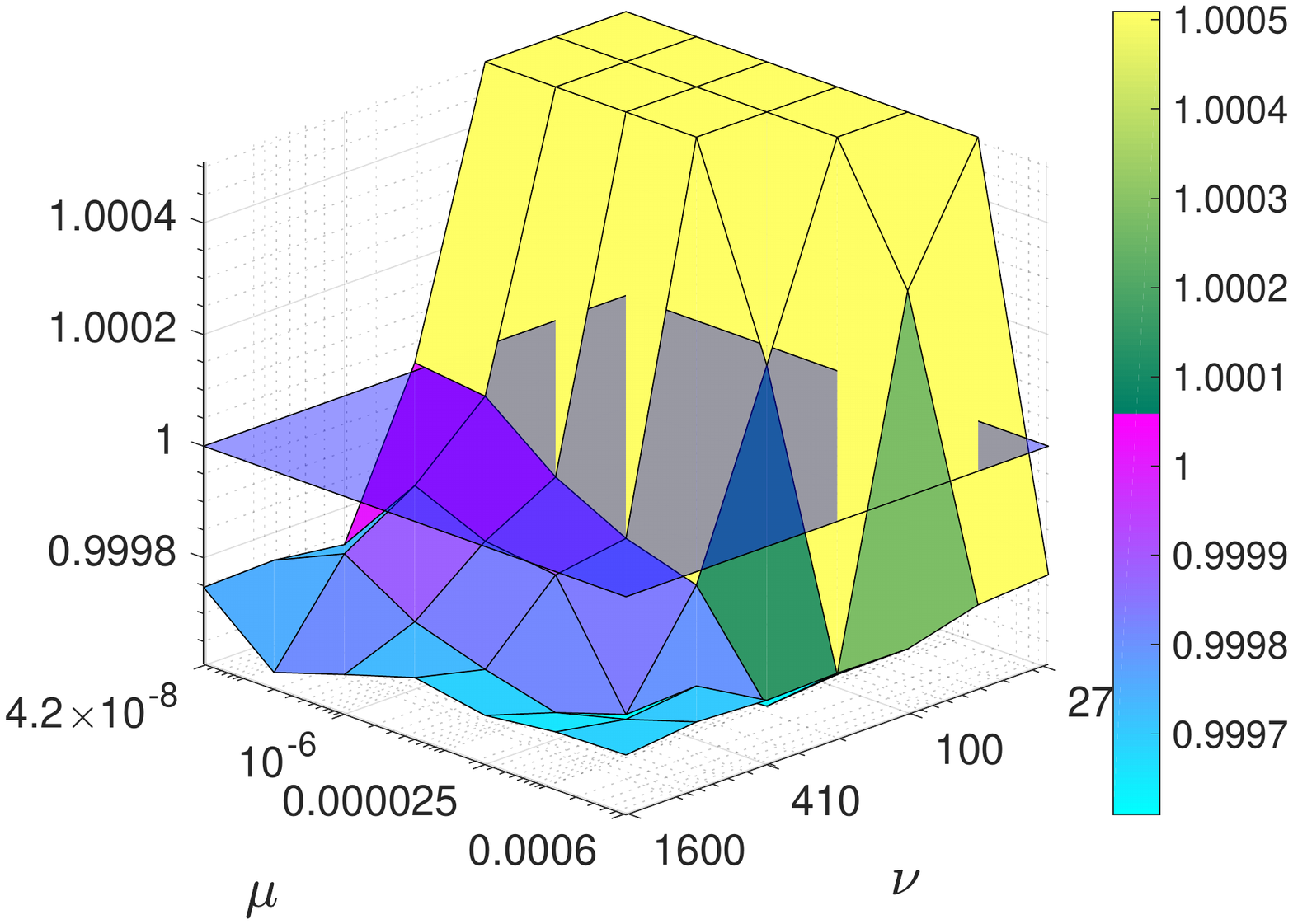}
\end{minipage}%
       \caption{Sensitivity to acceleration parameters. Eigenvalues of $A$ are set to $1,1000, 1000, \dots,1000$. From left to right we have: Coordinate sketches with convenient probabilities, coordiante sketches with uniform probabilities and Gaussian sketches. Choice of parameters as per \eqref{eq:munu_conv_paper} in the middle of plots. Each instance was run for 10 seconds.}
\label{fig:heat_gauss_rand}
\end{figure}

The crucial aspect to make the accelerated algorithm to converge is to set $\nu$ large enough.  In fact, combination of both small $\nu$ and small $\mu$ leads almost always to non-convergent algorithm. On the other hand, it seems that once $\nu$ is chosen correctly, big enough $\mu$ leads to fast convergence. This indicates how to compute $\mu$ in practice (recall that computing $\nu$ is feasible)---one needs just to choose it small enough (definitely smaller than $\frac{1}{\nu}$).

\subsection{Experiments with LIBSVM}

Next we investigate if the accelerated BFGS update improves upon the standard BFGS update when applied to the Hessian $\nabla^2 f(x)$  of ridge regression problems
of the form
\begin{equation}\label{eq:ridgeMatrix}
\min_{x\in \R^n}f(x)\eqdef \frac{1}{2}\norm{Ax-b}_2^2 + \frac{\lambda}{2} \norm{x}_2^2,\quad  \quad\nabla^2 f(x) = A^\top A+\lambda I,
\end{equation}
using data from LIBSVM~\cite{Chang2011}. Datapoints (rows of $A$) were normalized such that $\|A_{i:}\|^2=1$ for all $i$ and the regularization parameter was chosen as $\lambda=\frac{1}{m}$.

First, we run the experiments on smaller problems when parameters $\mu$, $\nu$ are precomputed for coordinate sketches with convenient probabilities \eqref{eq:munu_conv_paper}.

\begin{figure}[H]
    \centering
\begin{minipage}{0.30\textwidth}
  \centering
\includegraphics[width =  \textwidth ]{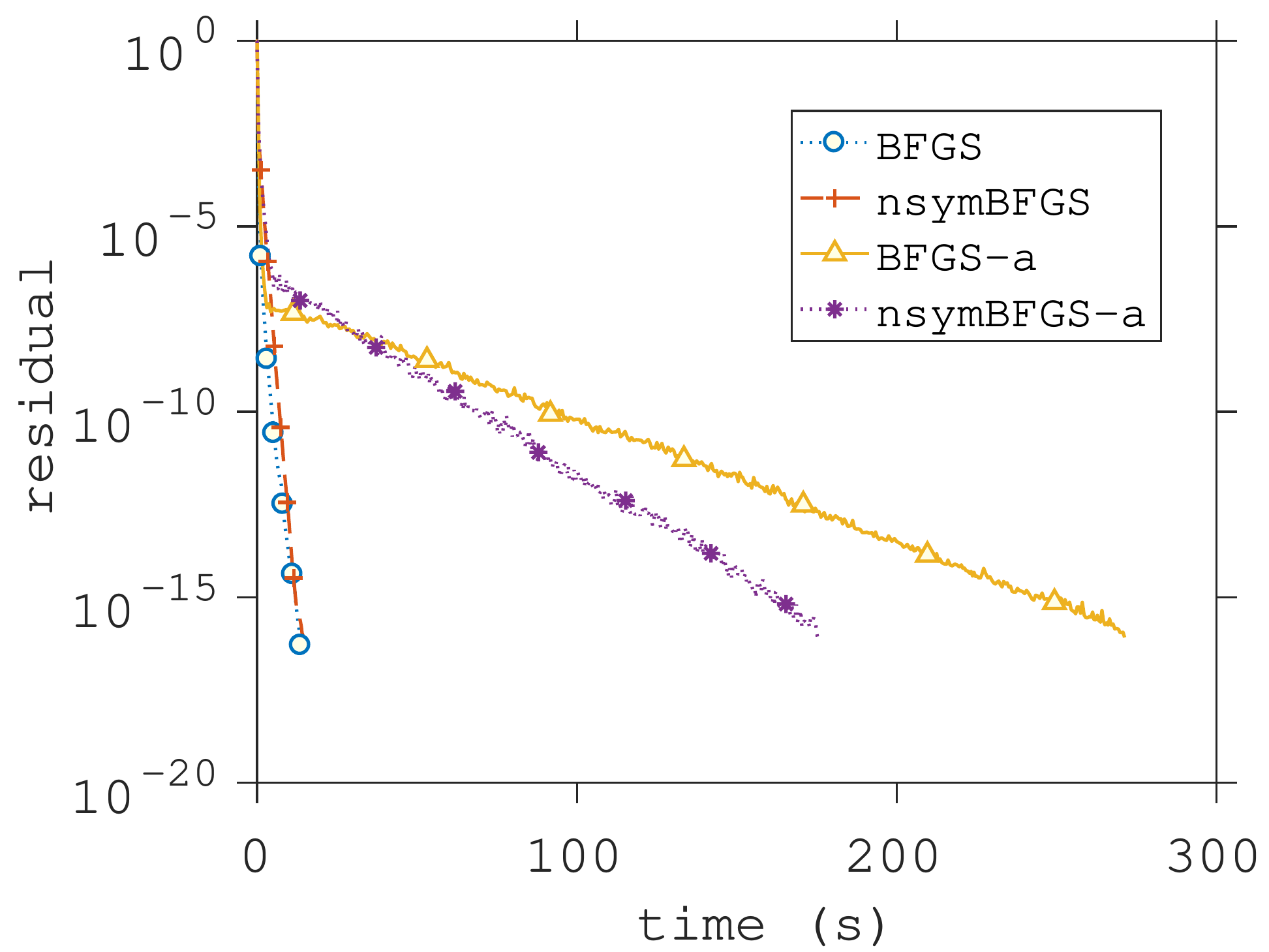}
\end{minipage}%
\begin{minipage}{0.30\textwidth}
  \centering
\includegraphics[width =  \textwidth ]{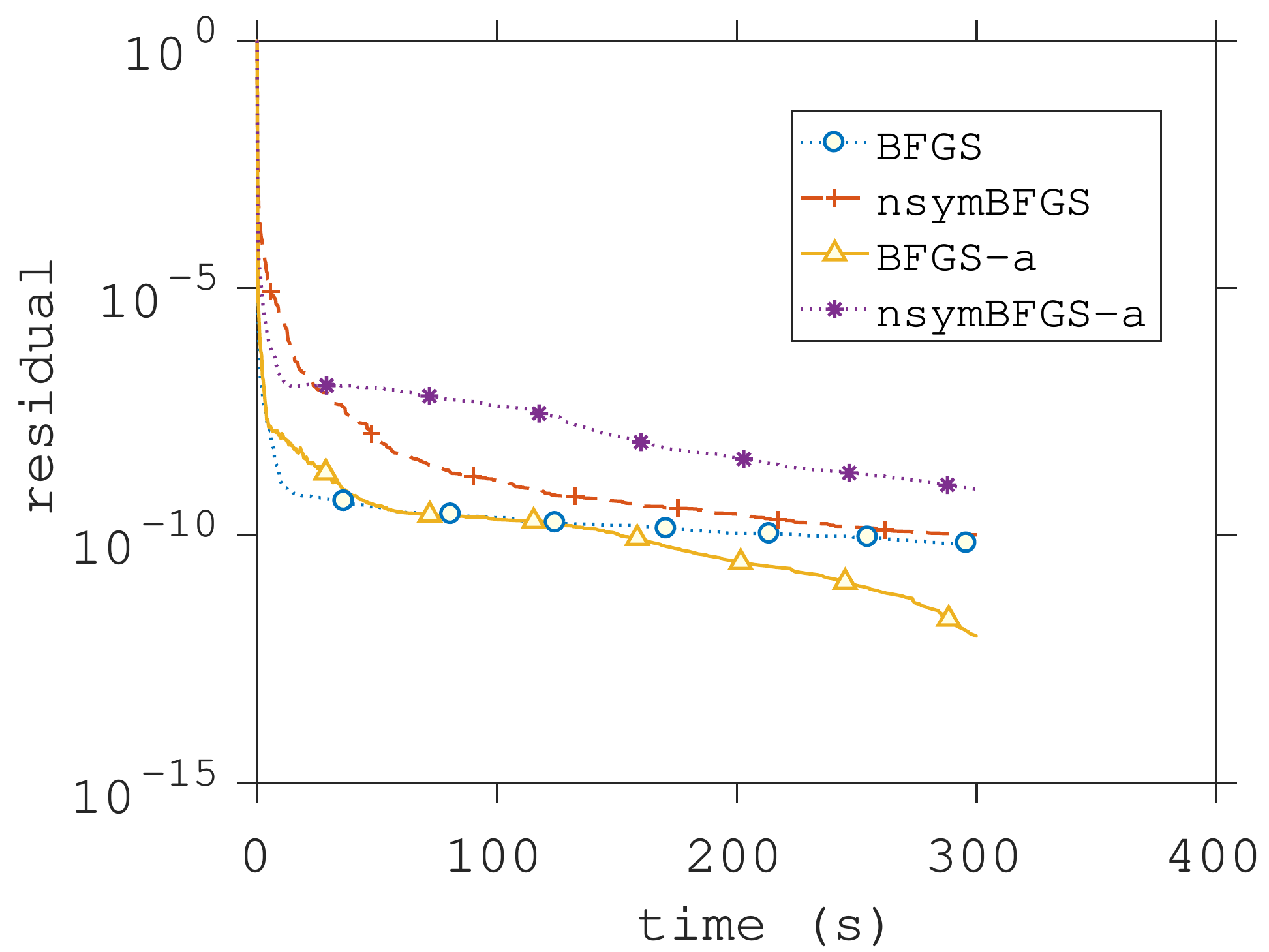}
\end{minipage}%
\begin{minipage}{0.30\textwidth}
  \centering
\includegraphics[width =  \textwidth ]{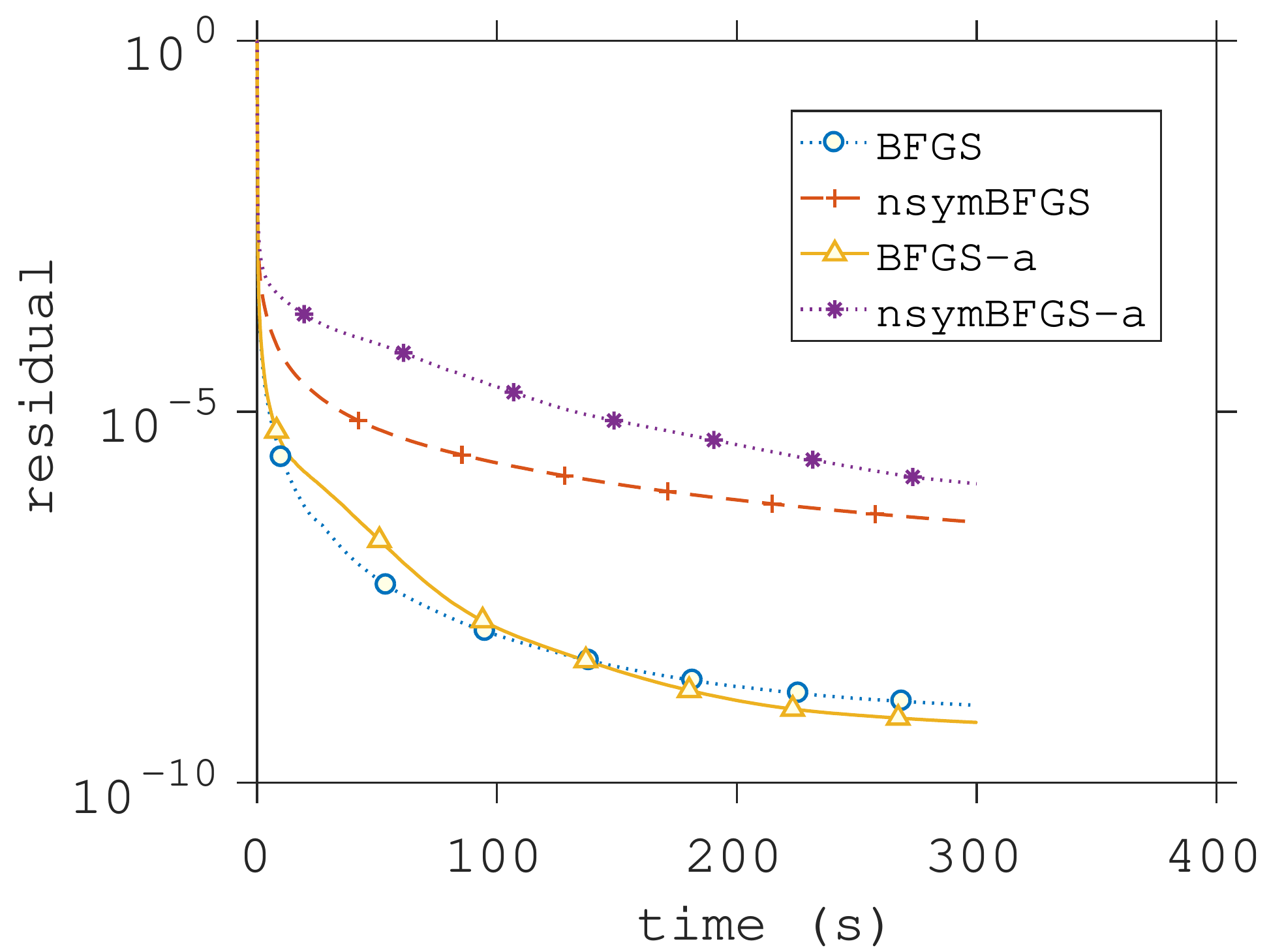}
\end{minipage}%
    \caption{  Dataset aloi: $n=128$. From left to right we have: Coordinate sketch with convenient probabilities, coordinate sketch with uniform probabilities and Gaussian sketch respectively. 
}\label{fig:aloig}
\end{figure}

\begin{figure}[H]
    \centering
\begin{minipage}{0.30\textwidth}
  \centering
\includegraphics[width =  \textwidth ]{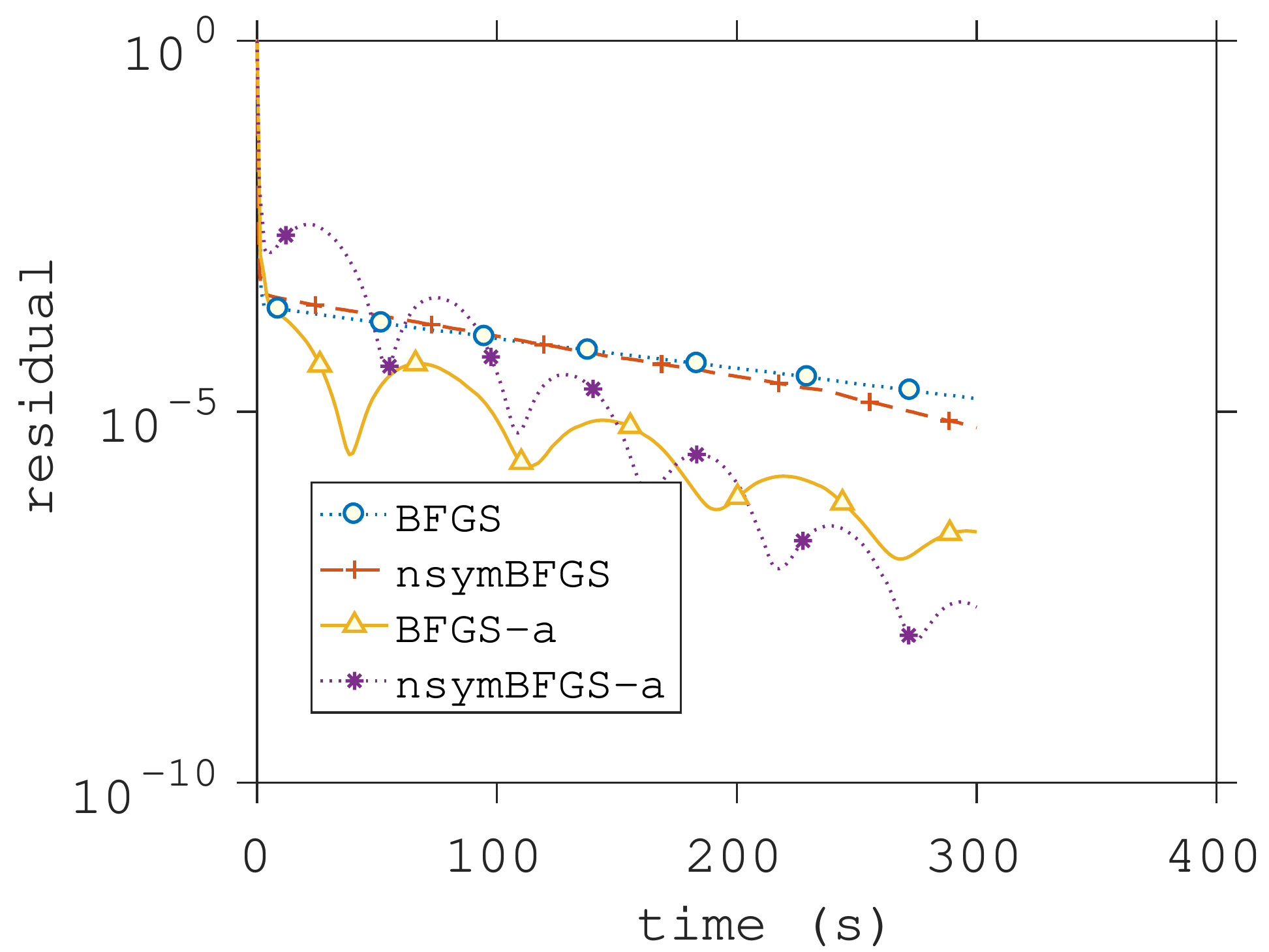}
\end{minipage}%
\begin{minipage}{0.30\textwidth}
  \centering
\includegraphics[width =  \textwidth ]{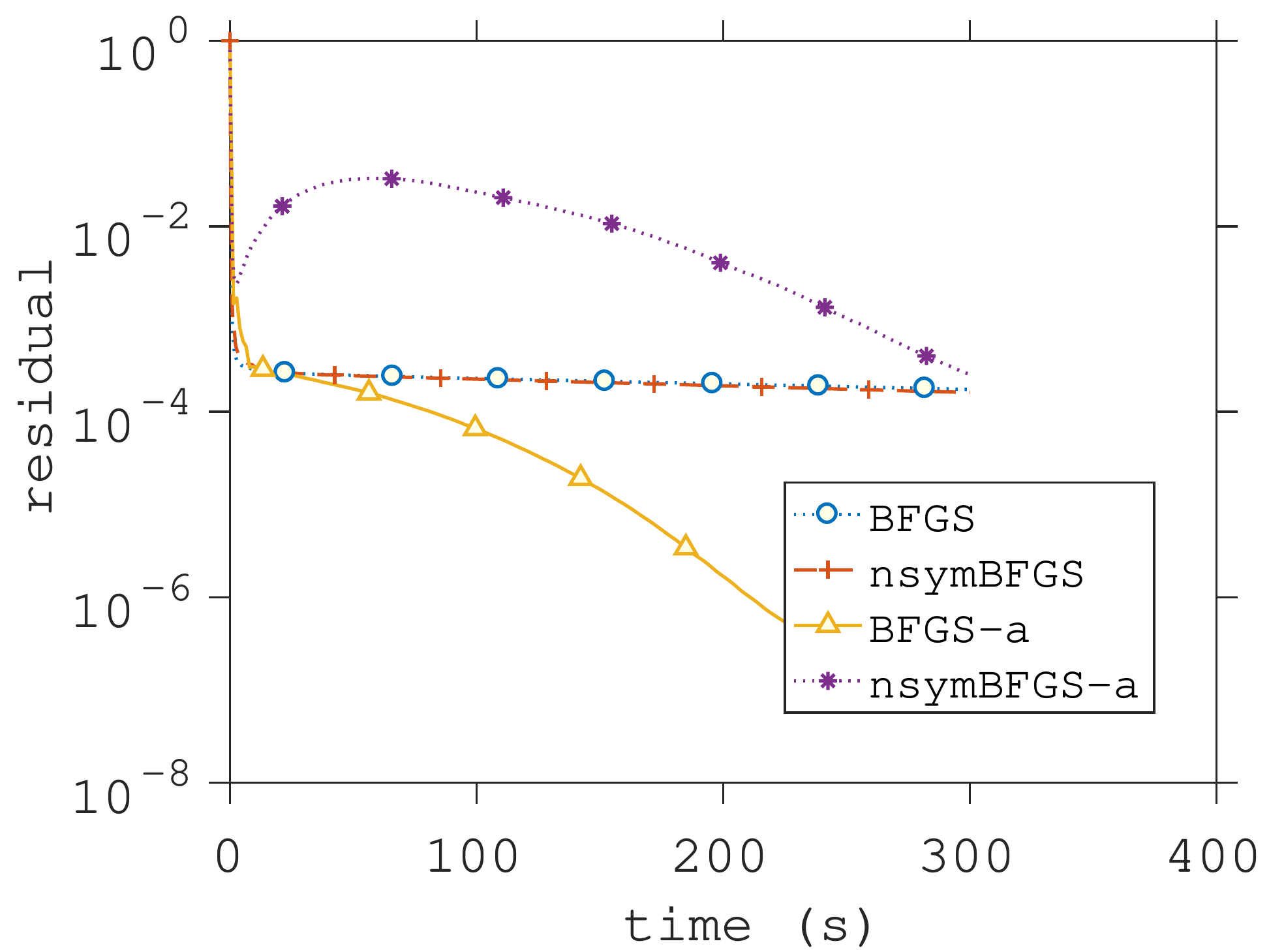}
\end{minipage}%
\begin{minipage}{0.30\textwidth}
  \centering
\includegraphics[width =  \textwidth ]{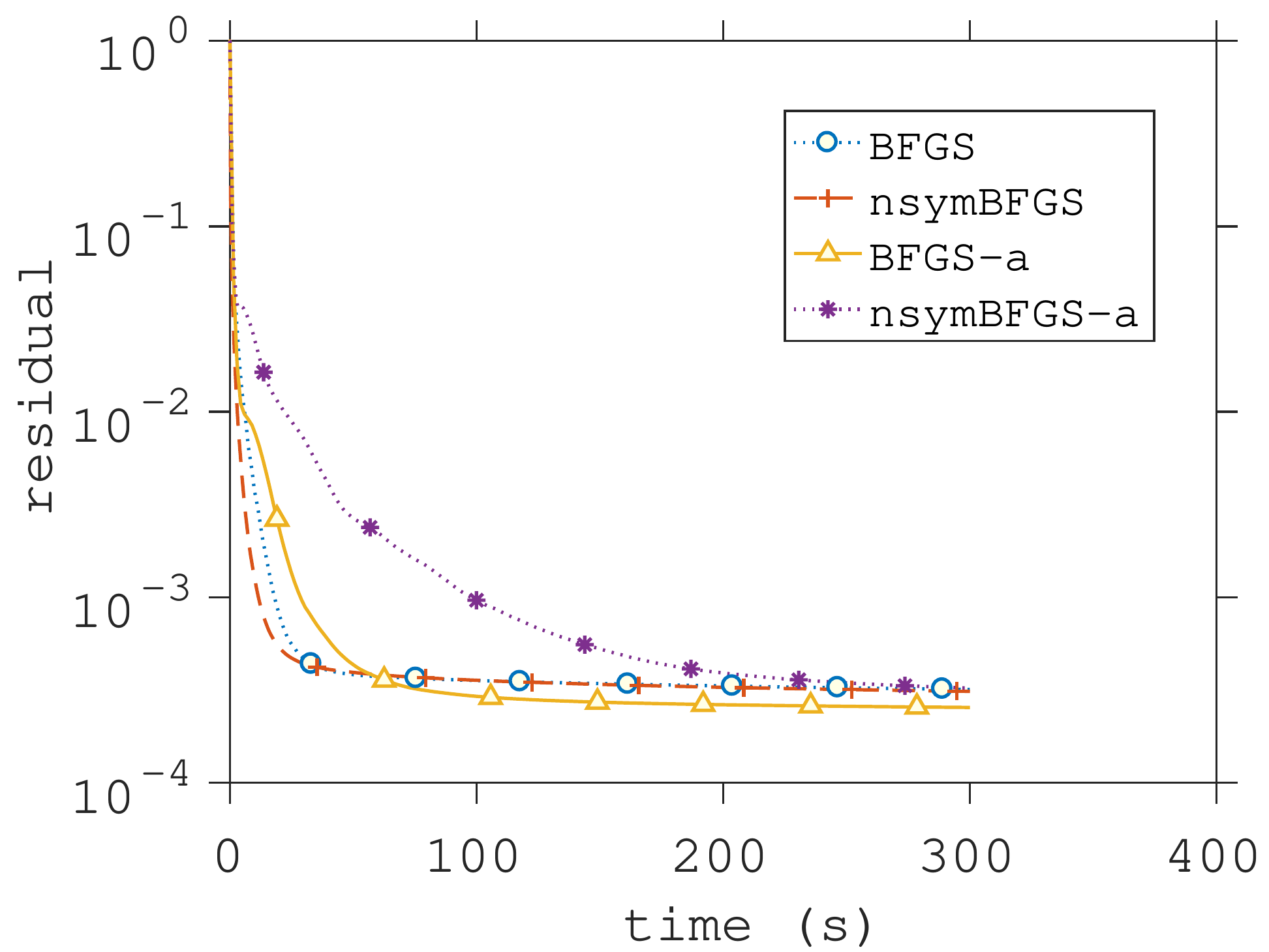}
\end{minipage}%
    \caption{  Dataset w1a: $n=300$. From left to right we have: Coordinate sketch with convenient probabilities, coordinate sketch with uniform probabilities and Gaussian sketch respectively. 
}\label{fig:w1a}
\end{figure}

\begin{figure}[H]
    \centering
\begin{minipage}{0.30\textwidth}
  \centering
\includegraphics[width =  \textwidth ]{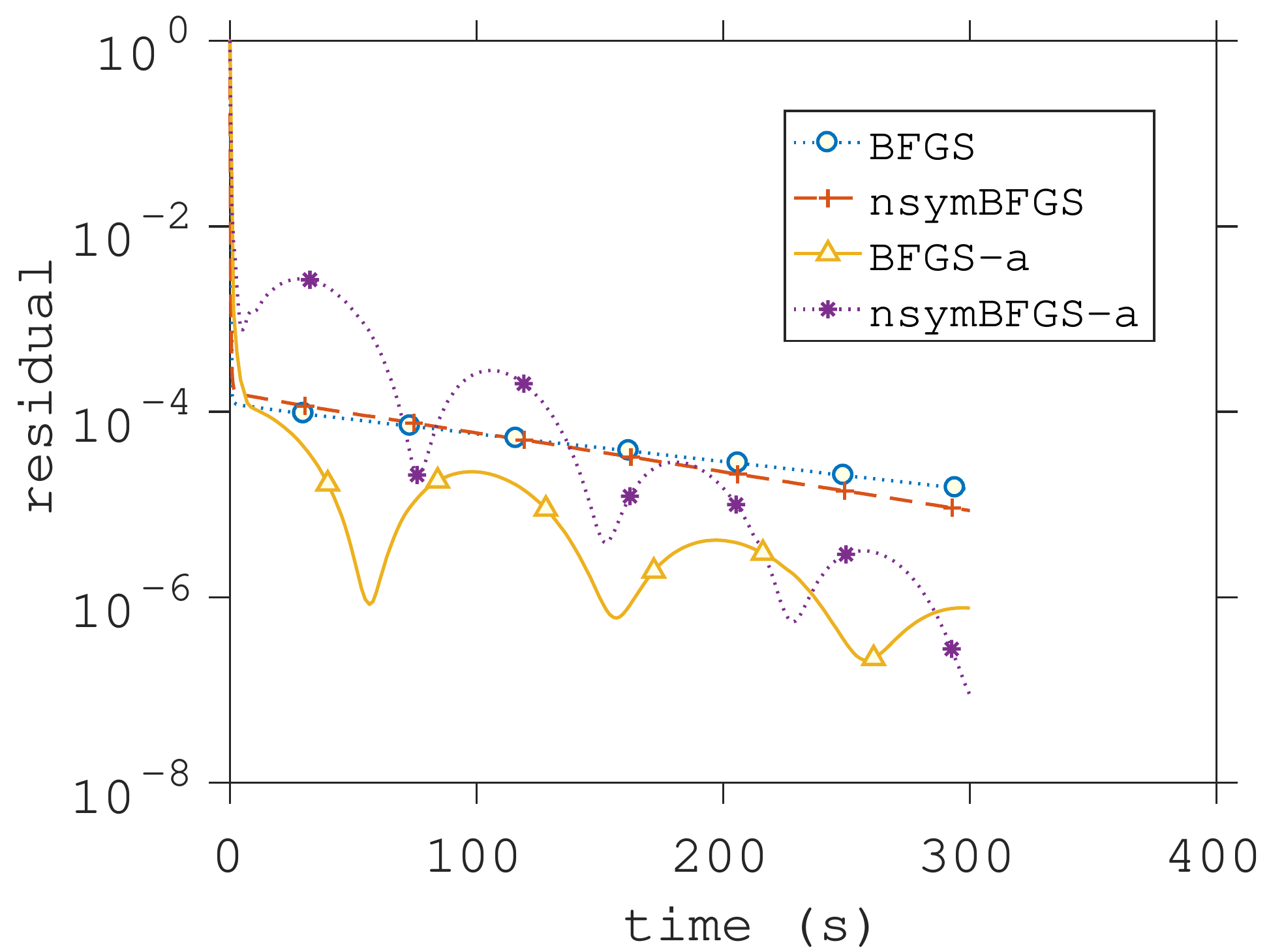}
\end{minipage}%
\begin{minipage}{0.30\textwidth}
  \centering
\includegraphics[width =  \textwidth ]{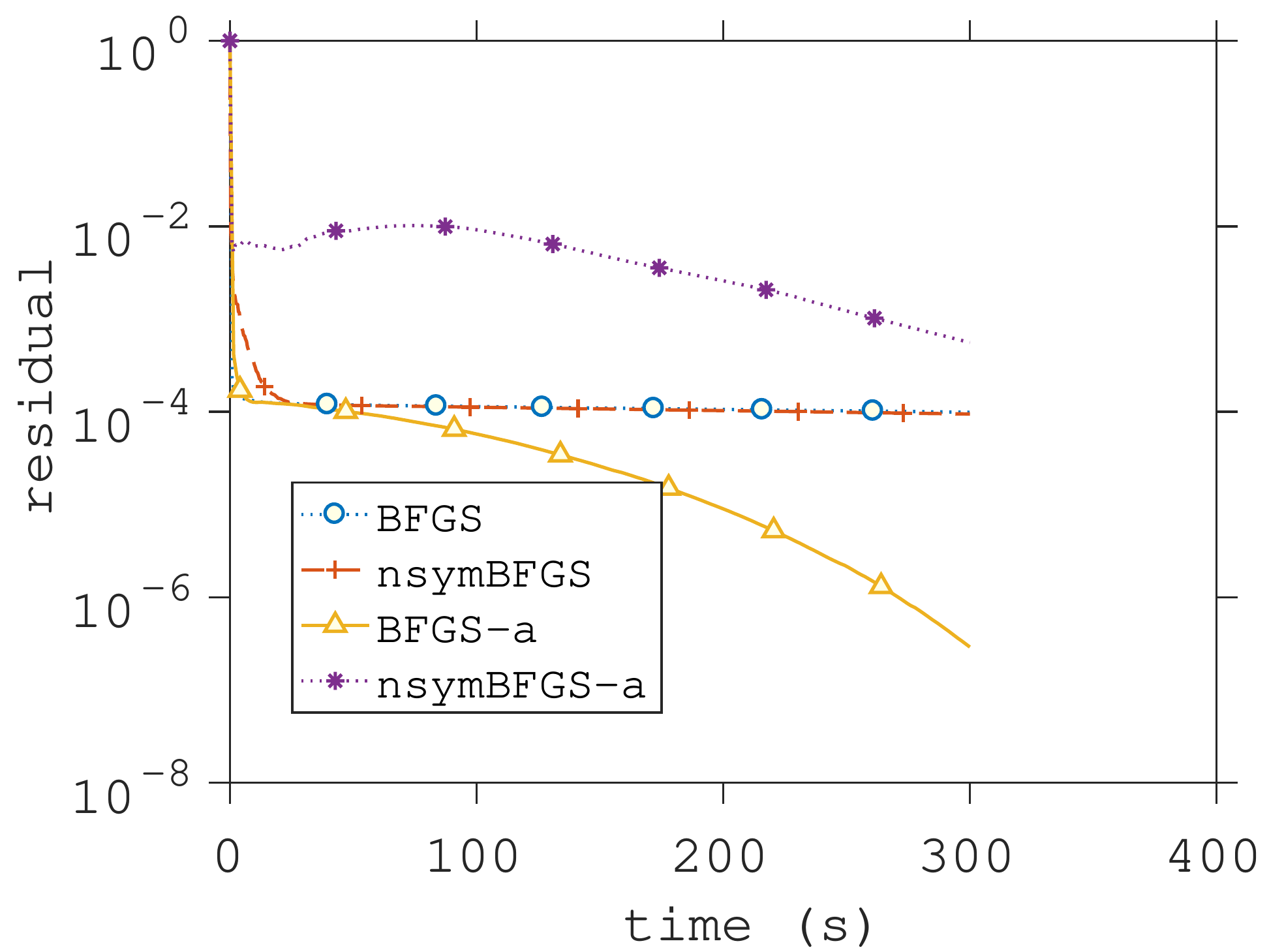}
\end{minipage}%
\begin{minipage}{0.30\textwidth}
  \centering
\includegraphics[width =  \textwidth ]{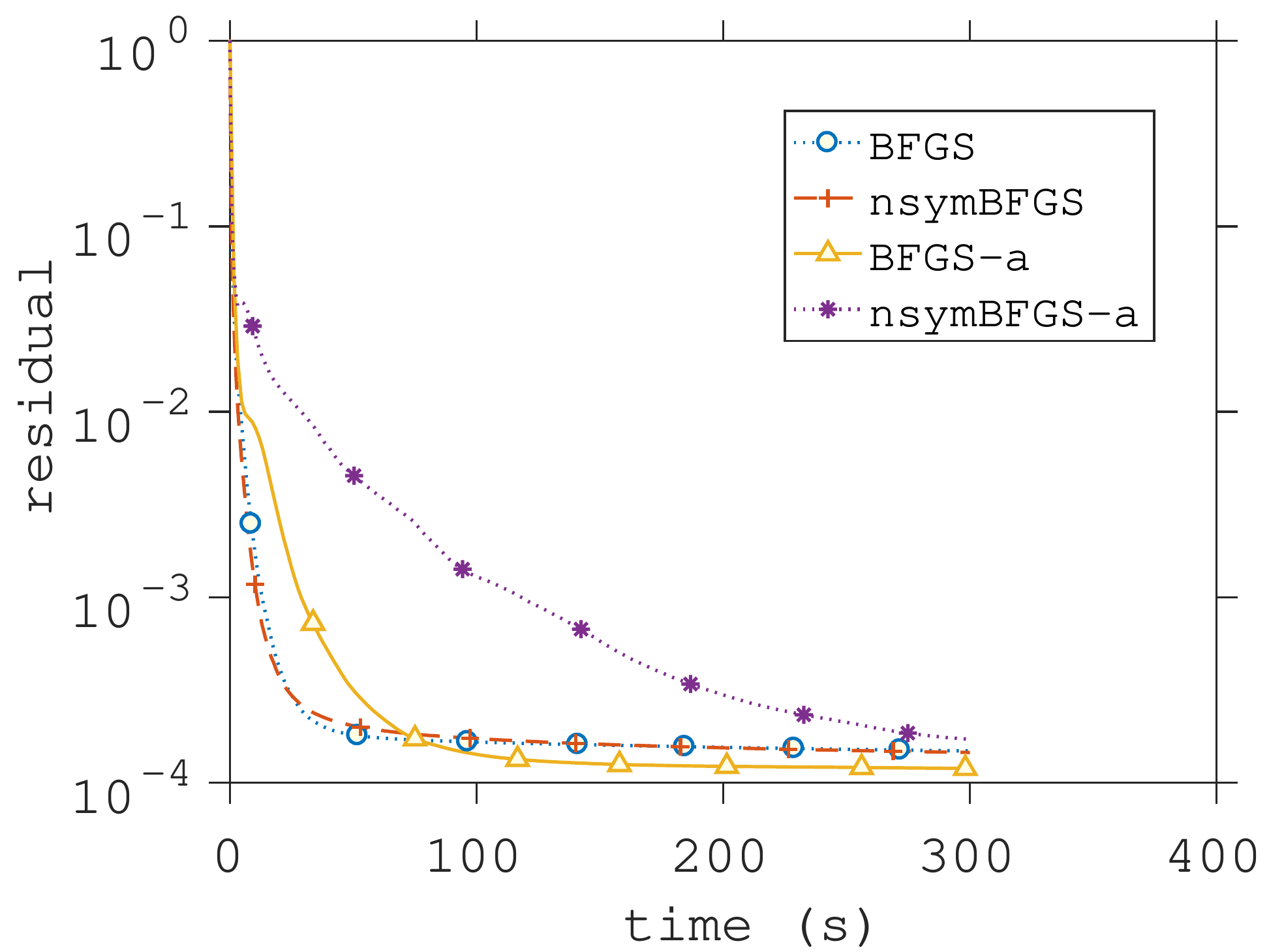}
\end{minipage}%
    \caption{  Dataset w2a: $n=300$. From left to right we have: Coordinate sketch with convenient probabilities, coordinate sketch with uniform probabilities and Gaussian sketch respectively. 
}\label{fig:w2a}
\end{figure}

\begin{figure}[H]
    \centering
\begin{minipage}{0.30\textwidth}
  \centering
\includegraphics[width =  \textwidth ]{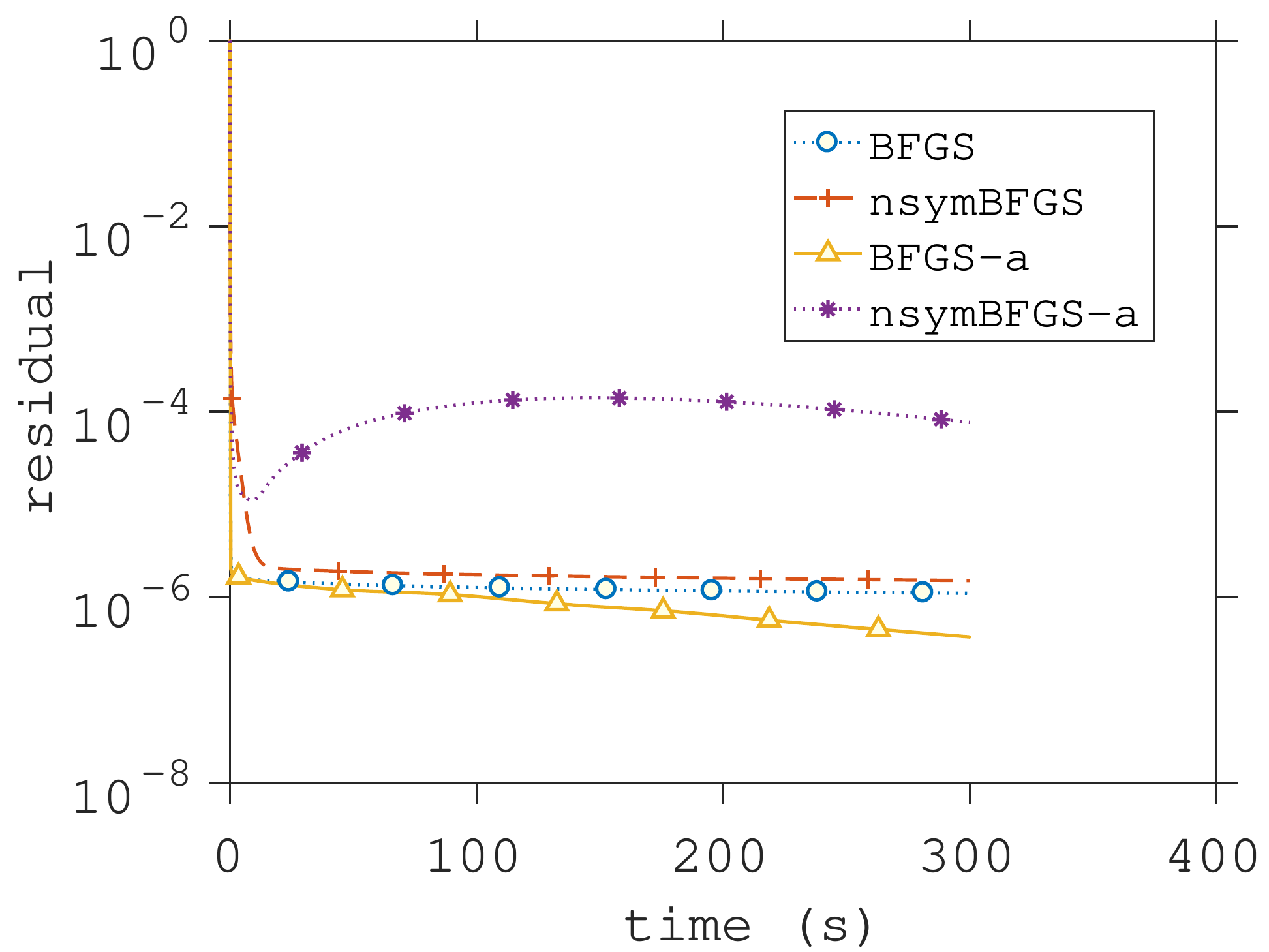}
\end{minipage}%
\begin{minipage}{0.30\textwidth}
  \centering
\includegraphics[width =  \textwidth ]{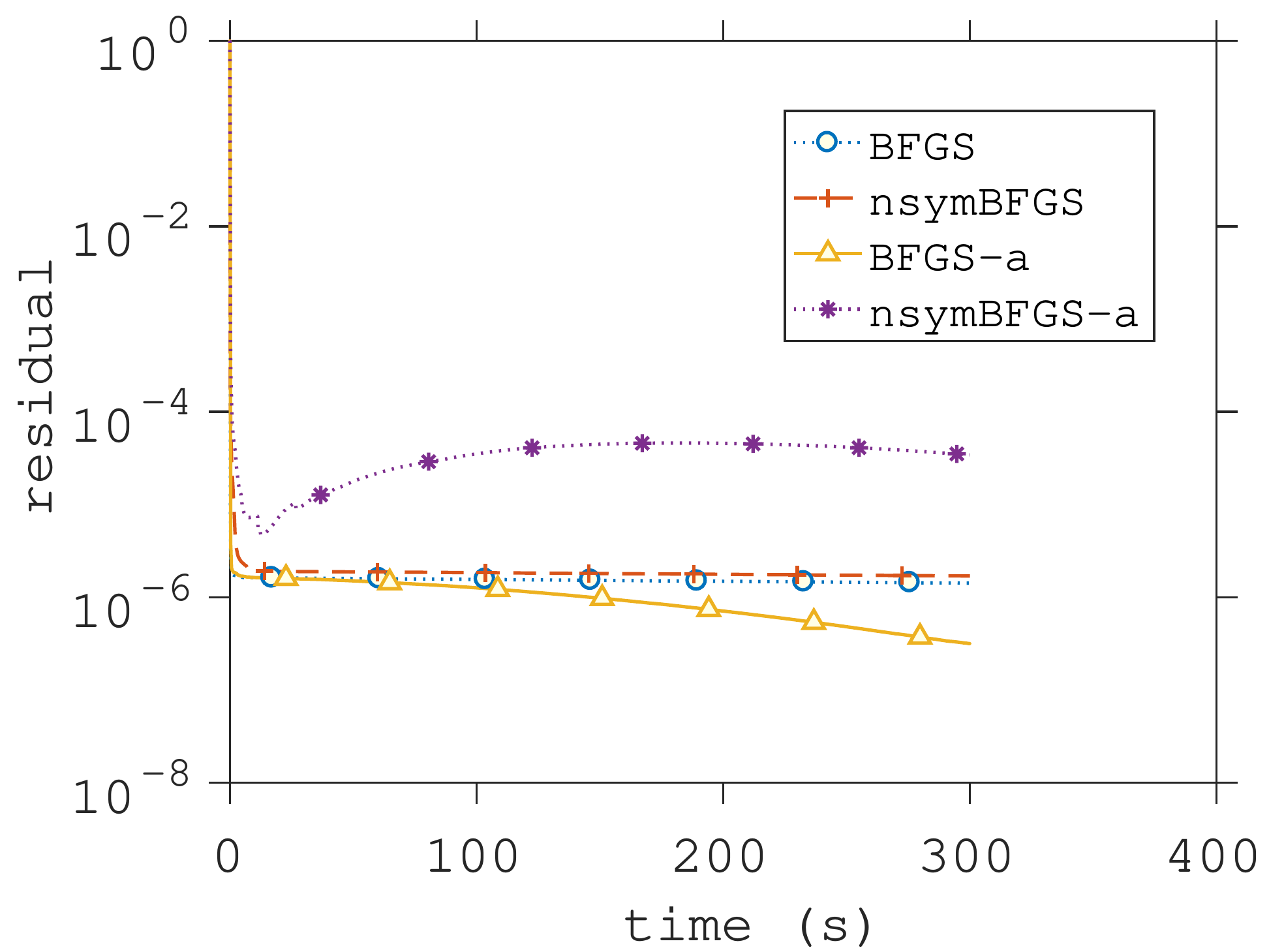}
\end{minipage}%
\begin{minipage}{0.30\textwidth}
  \centering
\includegraphics[width =  \textwidth ]{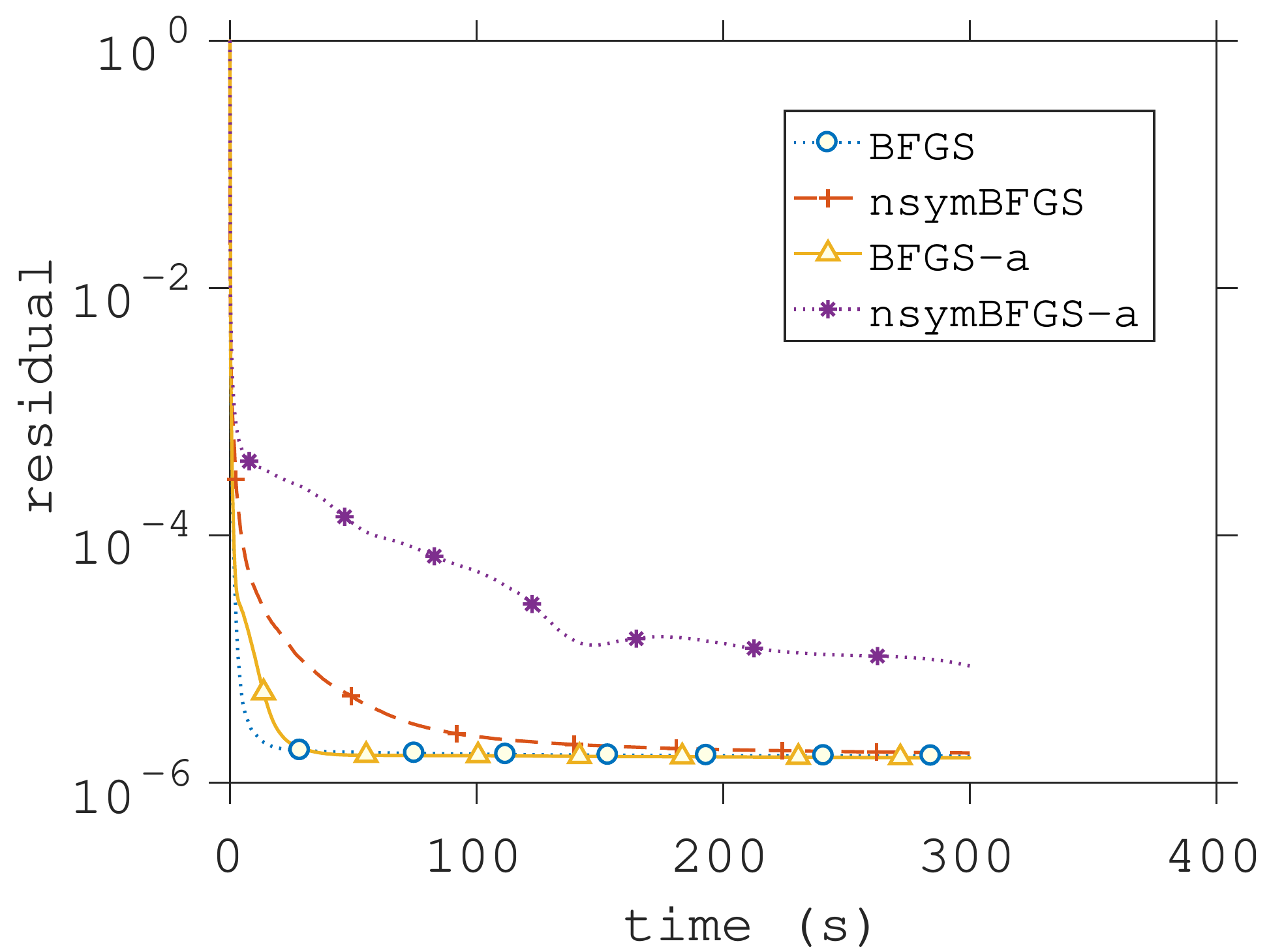}
\end{minipage}%
    \caption{  Dataset mushrooms: $n=112$. From left to right we have: Coordinate sketch with convenient probabilities, coordinate sketch with uniform probabilities and Gaussian sketch respectively. 
}\label{fig:mushroomsxx}
\end{figure}

\begin{figure}[H]
\centering
\begin{minipage}{0.30\textwidth}
  \centering
\includegraphics[width =  \textwidth ]{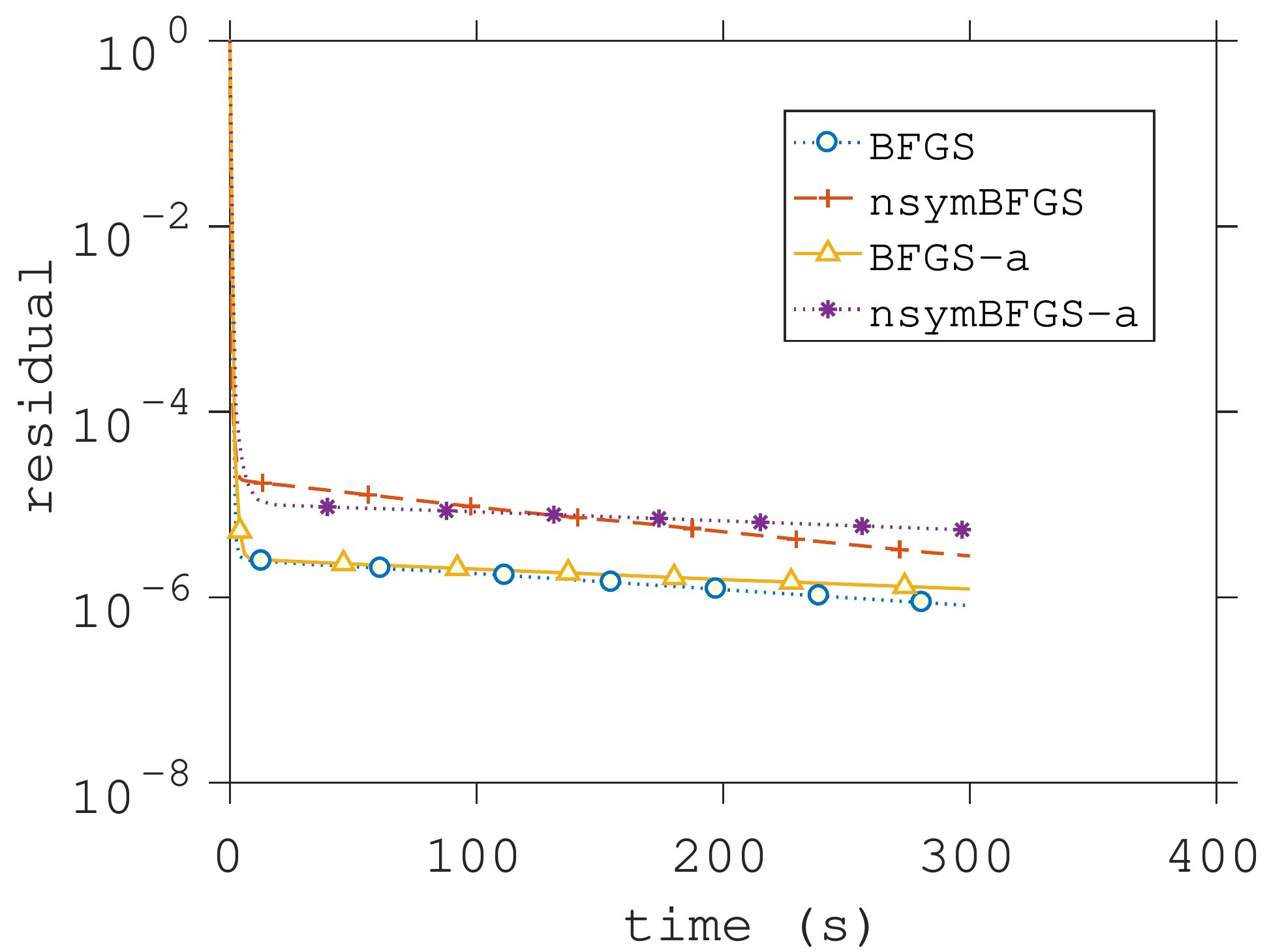}
\end{minipage}%
\begin{minipage}{0.30\textwidth}
  \centering
\includegraphics[width =  \textwidth ]{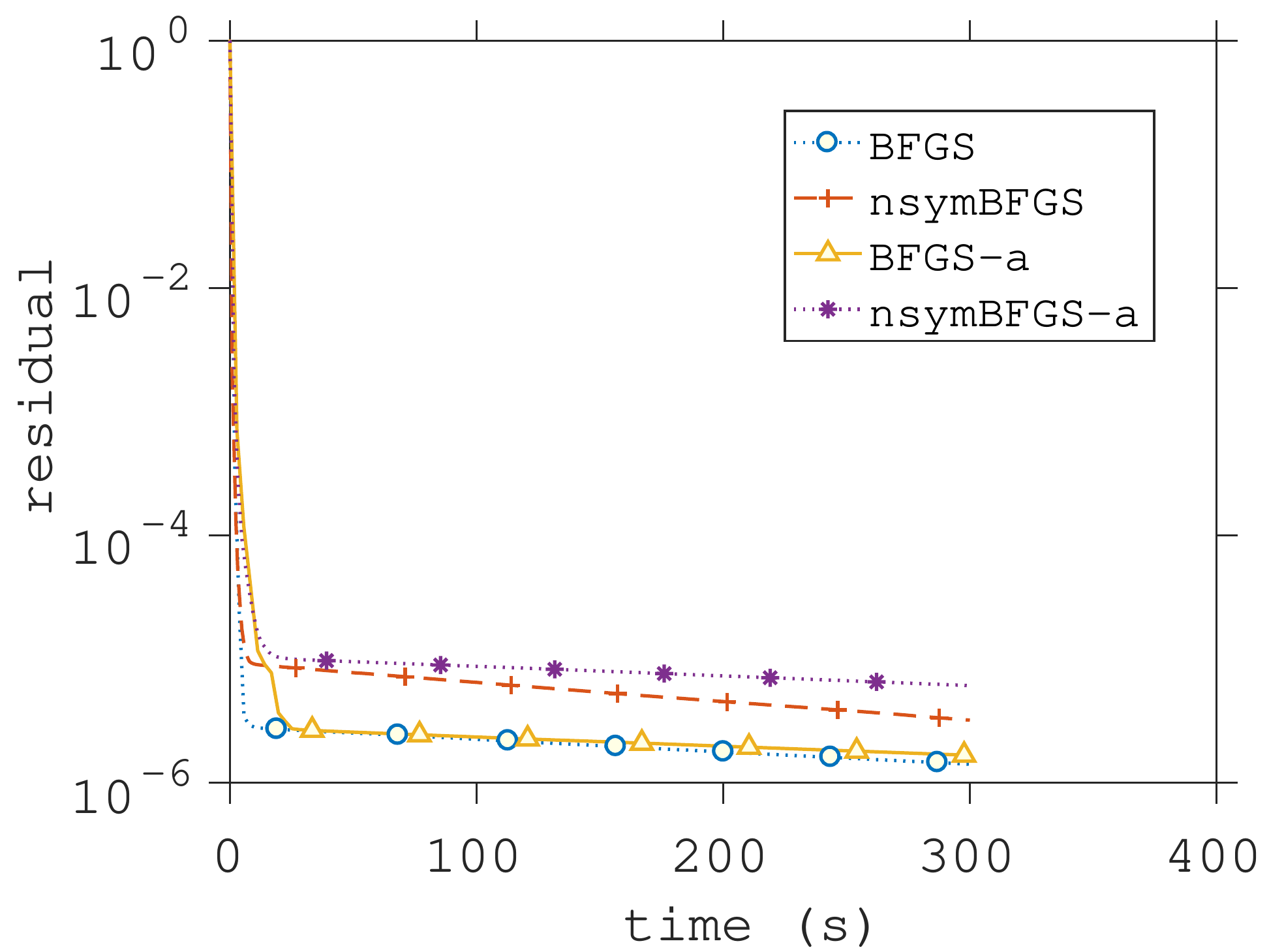}
\end{minipage}%
\begin{minipage}{0.30\textwidth}
  \centering
\includegraphics[width =  \textwidth ]{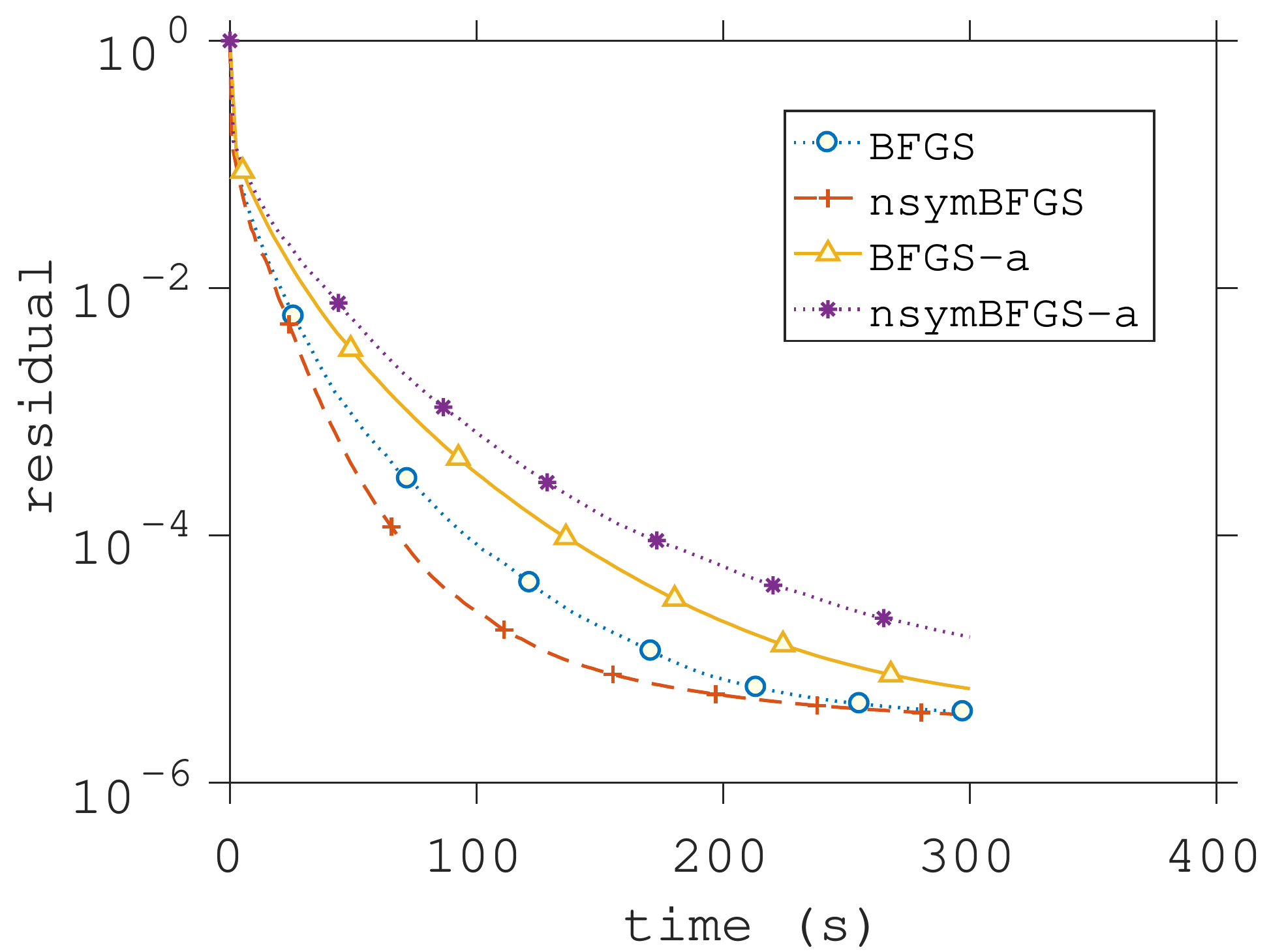}
\end{minipage}%
    \caption{  Dataset protein: $n=357$. From left to right we have: Coordinate sketch with convenient probabilities, coordinate sketch with uniform probabilities and Gaussian sketch respectively. 
}\label{fig:protein}
\end{figure}

\begin{figure}[H]
    \centering
\begin{minipage}{0.30\textwidth}
  \centering
\includegraphics[width =  \textwidth ]{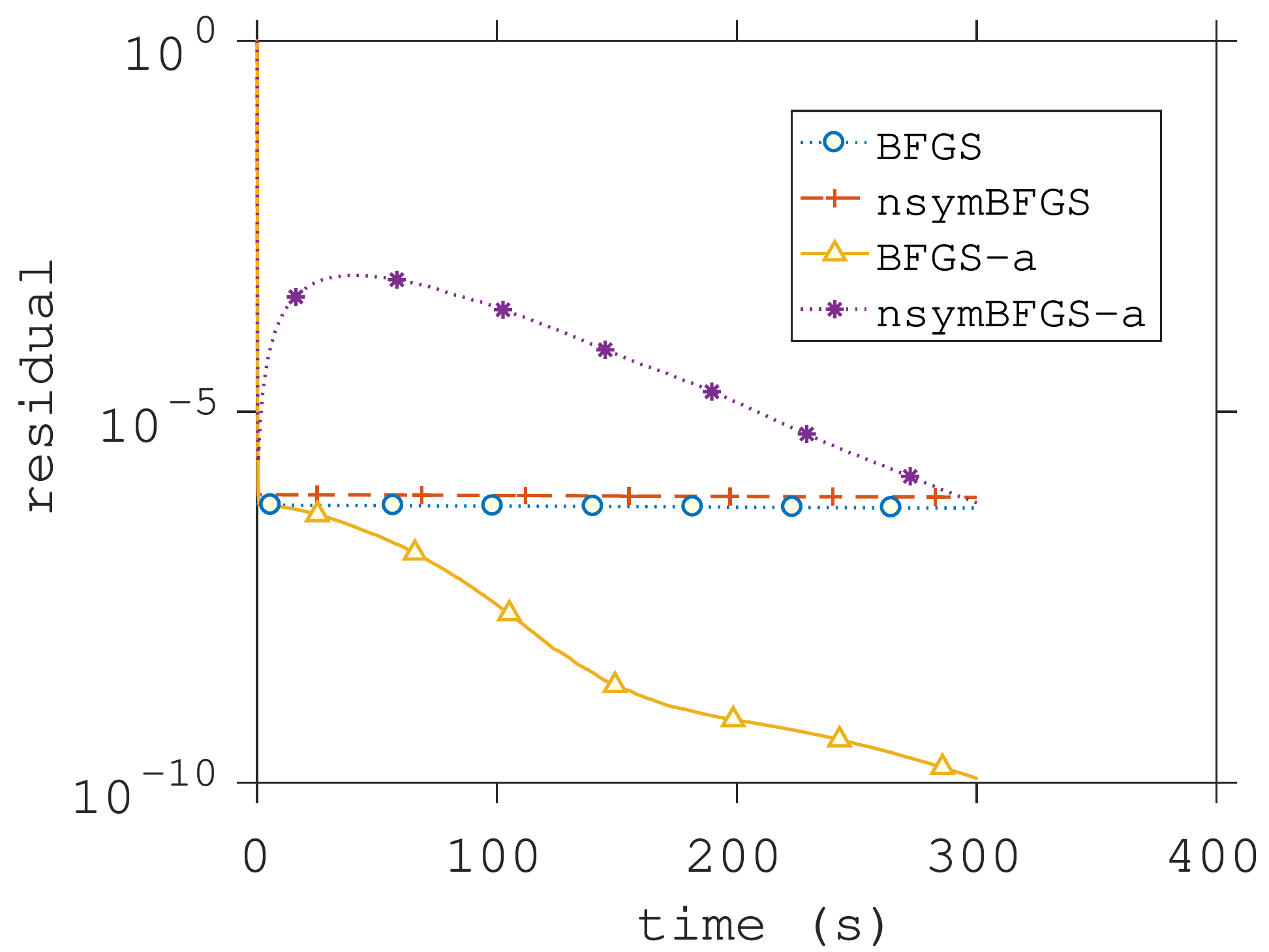}
\end{minipage}%
\begin{minipage}{0.30\textwidth}
  \centering
\includegraphics[width =  \textwidth ]{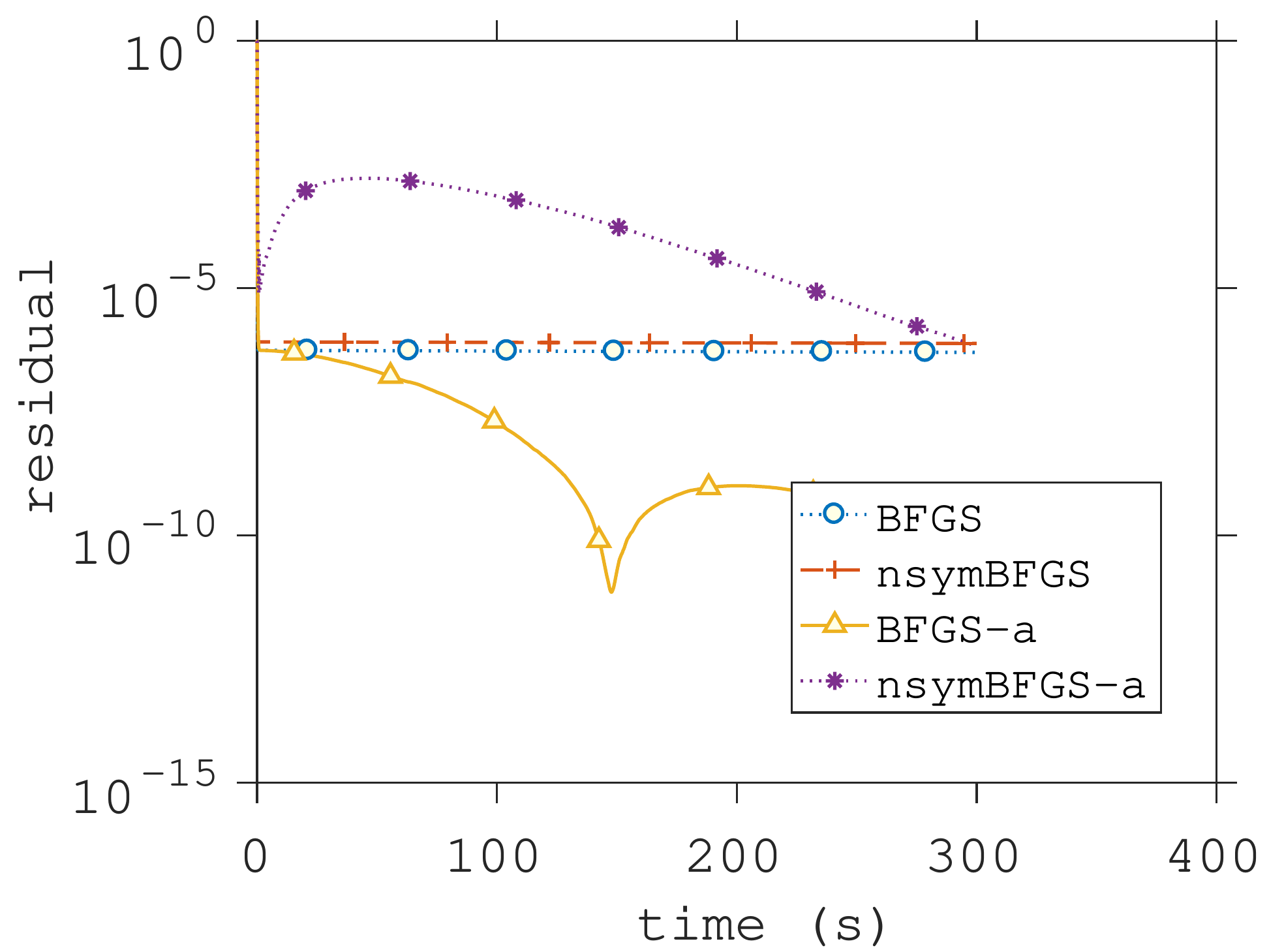}
\end{minipage}%
\begin{minipage}{0.30\textwidth}
  \centering
\includegraphics[width =  \textwidth ]{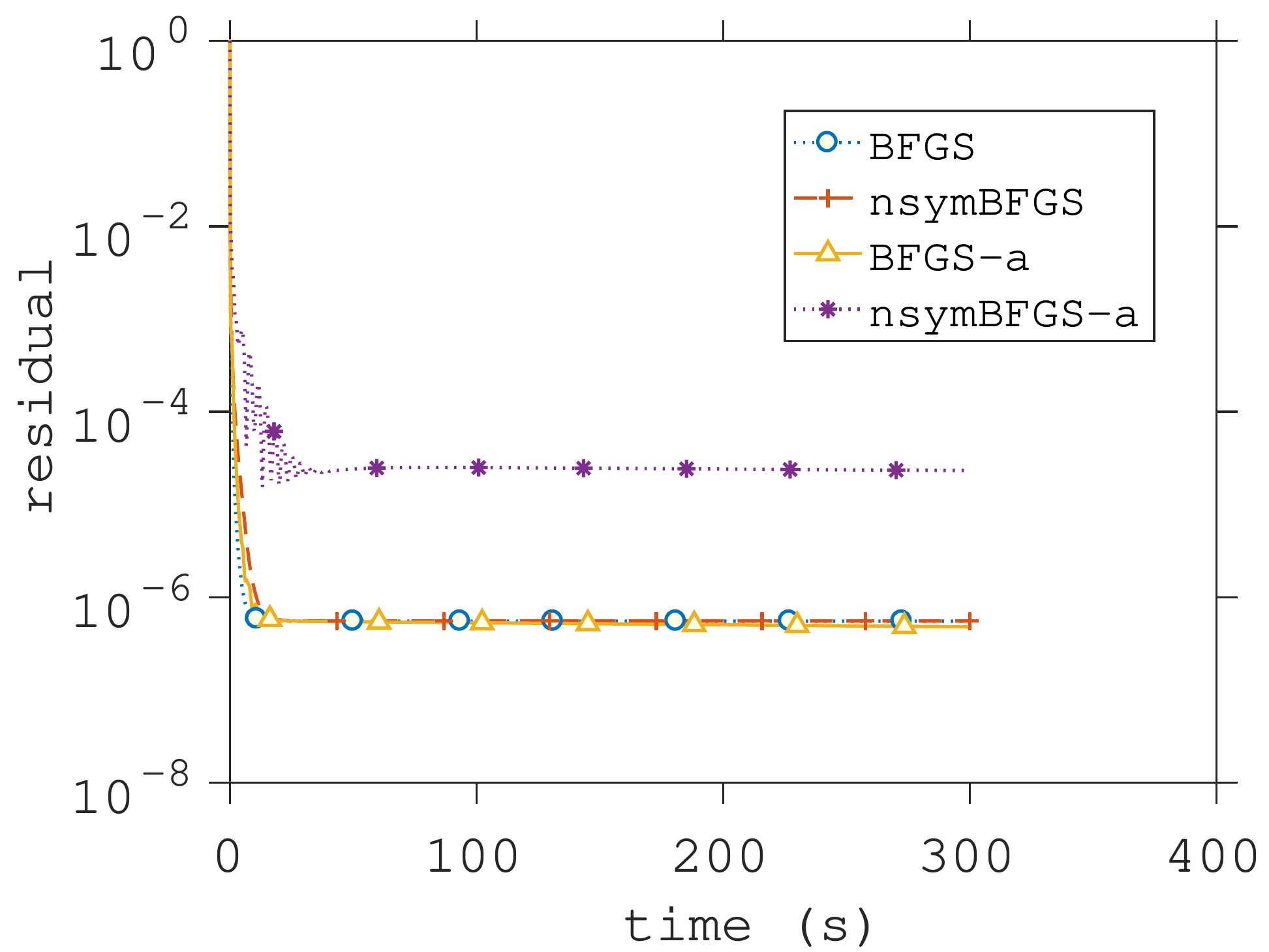}
\end{minipage}%
    \caption{  Dataset phishing: $n=68$. From left to right we have: Coordinate sketch with convenient probabilities, coordinate sketch with uniform probabilities and Gaussian sketch respectively. 
}\label{fig:splice-scale}
\end{figure}

In the vast majority of examples, the accelerated method performed significantly better than the nonaccelerated method for coordinate sketches (with both convenient and uniform probabilities), however the methods were comparable for Gaussian sketches. We believe that this is due to the fact that choice of parameters as per~\eqref{eq:munu_conv_paper} is close to the optimal parameters for coordinate sketches, and further for Gaussian sketches. However, the experiments on coordinate sketches indicates that for some classes of problems, accelerated algorithms with finely tuned parameters bring a great speedup compared to nonaccelerated ones. 

We also consider a problem where we do not compute $\lambda_{\min}(A)$, and therefore we cannot choose $\mu=\mu^P$ in \eqref{eq:munu_conv_paper}. Instead, we choose $\mu=\frac{1}{100\nu}$ and $\mu=\frac{1}{10000\nu}$.

\begin{figure}[H]
    \centering
\begin{minipage}{0.30\textwidth}
  \centering
\includegraphics[width =  \textwidth ]{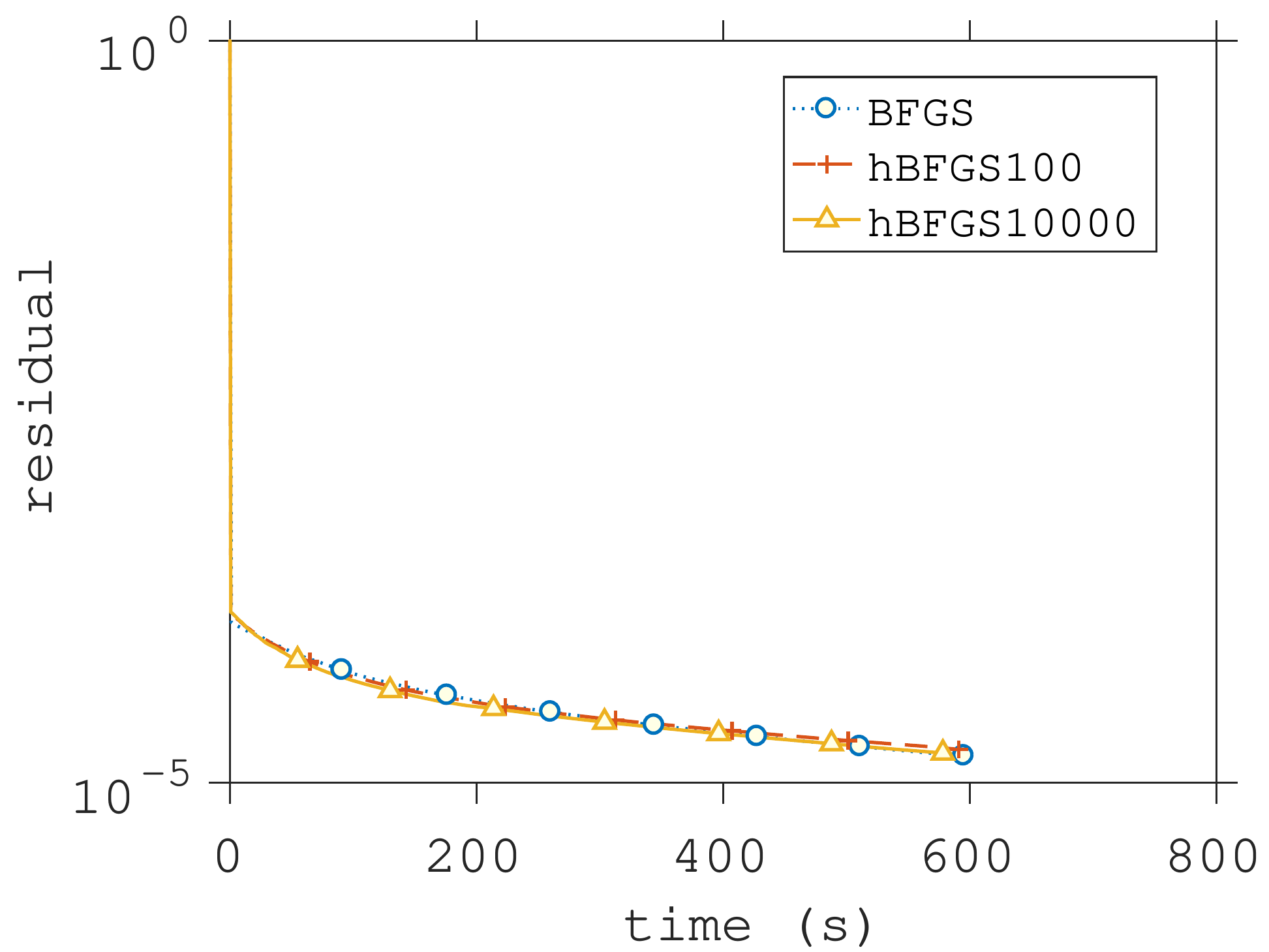}
\end{minipage}%
\begin{minipage}{0.30\textwidth}
  \centering
\includegraphics[width =  \textwidth ]{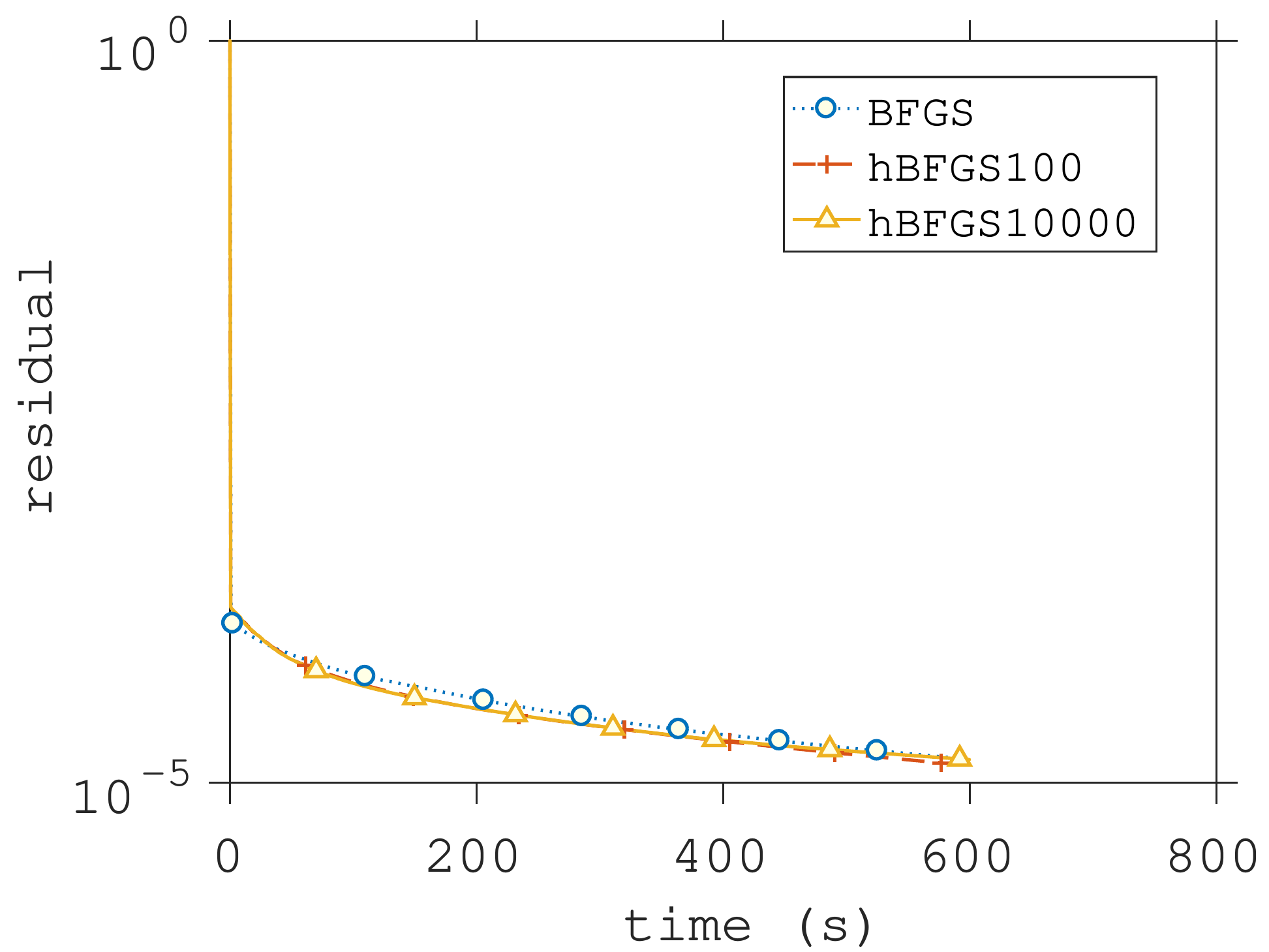}
\end{minipage}%
\begin{minipage}{0.30\textwidth}
  \centering
\includegraphics[width =  \textwidth ]{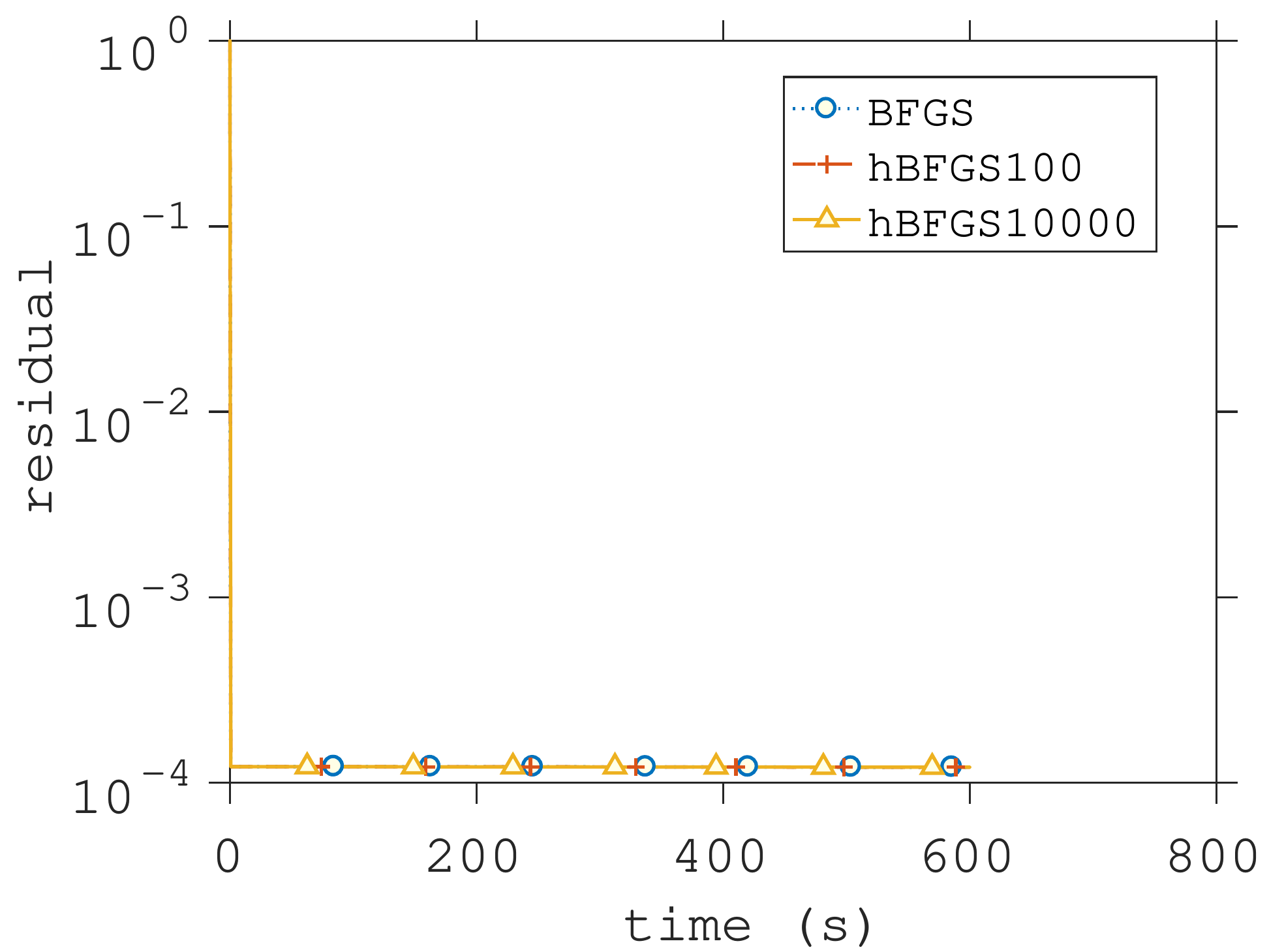}
\end{minipage}%
    \caption{  Dataset madelon: $n=500$. From left to right we have: Coordinate sketch with convenient probabilities, coordinate sketch with uniform probabilities and Gaussian sketch respectively. 
}\label{fig:splice-scale}
\end{figure}

\begin{figure}[H]
    \centering
\begin{minipage}{0.30\textwidth}
  \centering
\includegraphics[width =  \textwidth ]{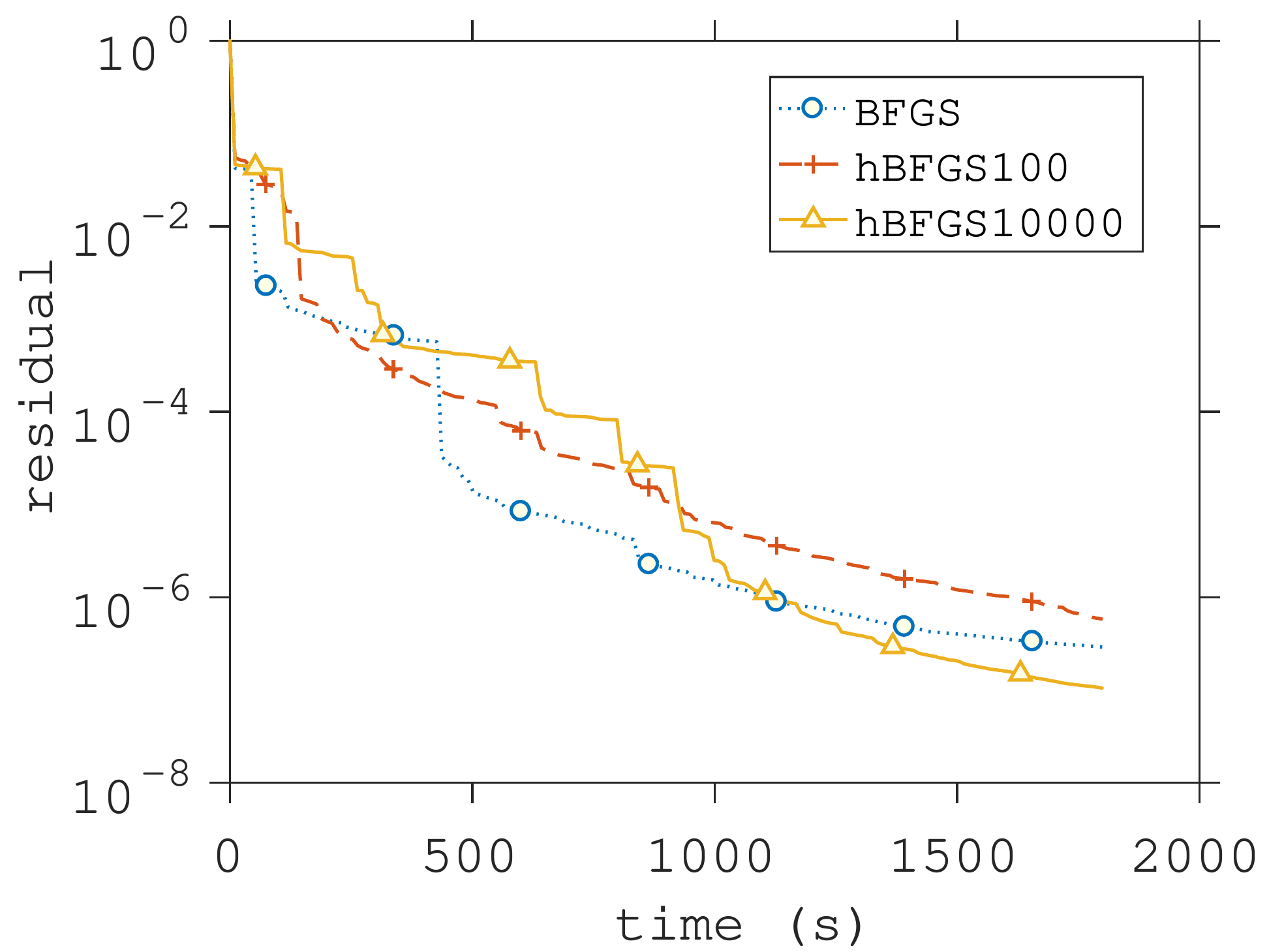}
\end{minipage}%
\begin{minipage}{0.30\textwidth}
  \centering
\includegraphics[width =  \textwidth ]{coordheur2epsilon-normalized-time}
\end{minipage}%
\begin{minipage}{0.30\textwidth}
  \centering
\includegraphics[width =  \textwidth ]{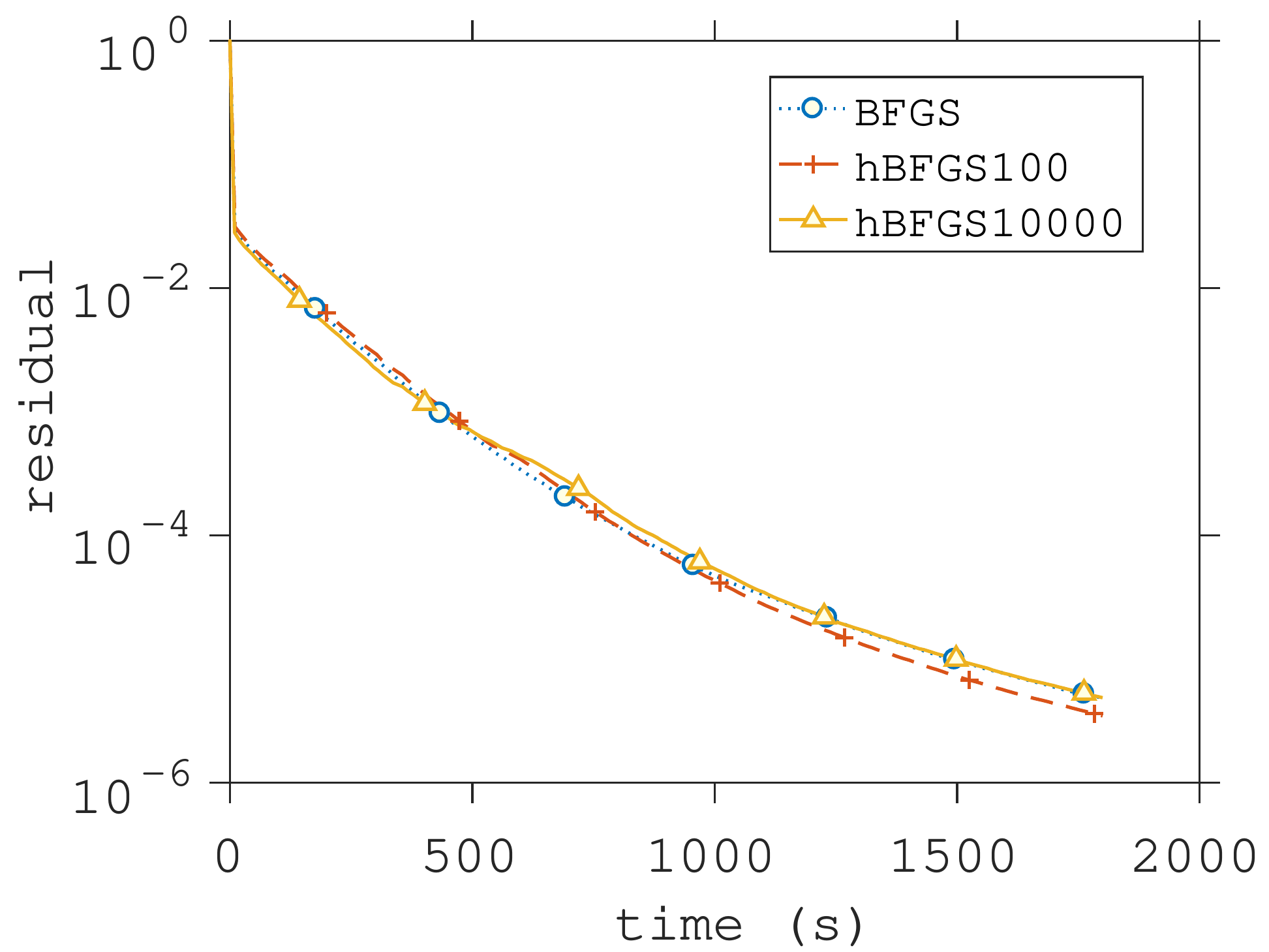}
\end{minipage}%
    \caption{  Dataset epsilon: $n=2000$. From left to right we have: Coordinate sketch with convenient probabilities, coordinate sketch with uniform probabilities and Gaussian sketch respectively. 
}\label{fig:splice-scale}
\end{figure}

\begin{figure}[H]
    \centering
\begin{minipage}{0.30\textwidth}
  \centering
\includegraphics[width =  \textwidth ]{convenientheur2svhn-time}
\end{minipage}%
\begin{minipage}{0.30\textwidth}
  \centering
\includegraphics[width =  \textwidth ]{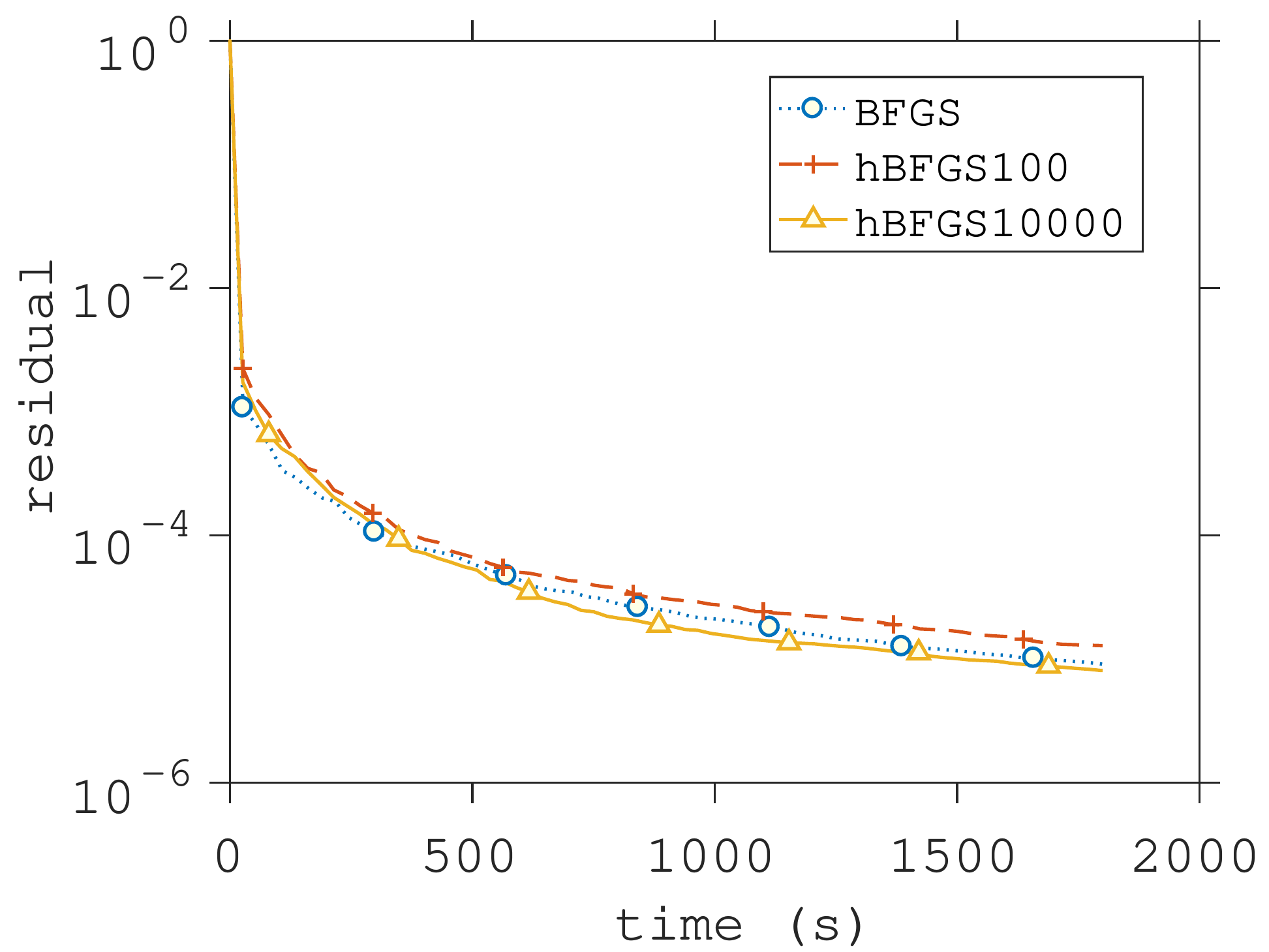}
\end{minipage}%
\begin{minipage}{0.30\textwidth}
  \centering
\includegraphics[width =  \textwidth ]{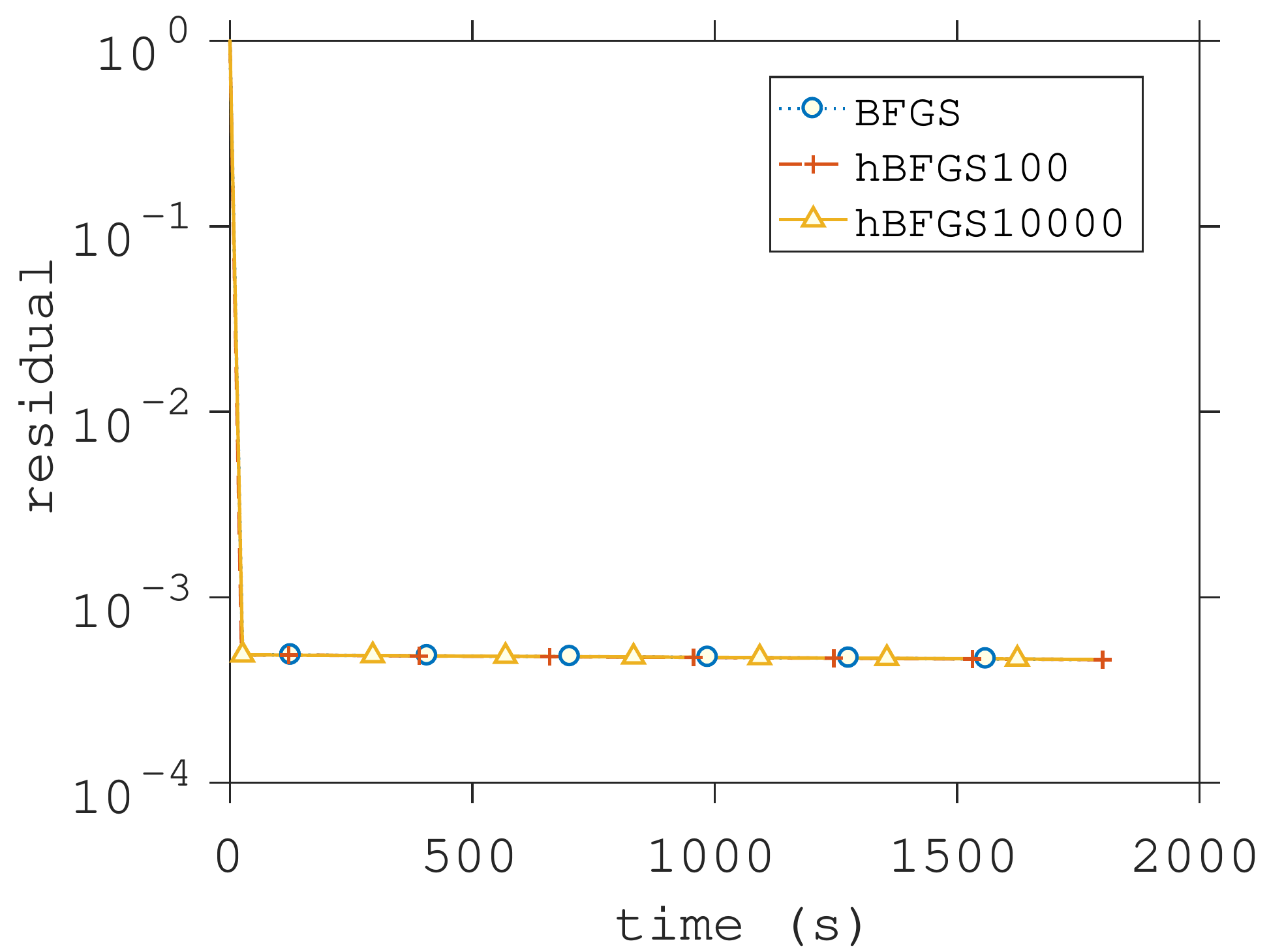}
\end{minipage}%
    \caption{  Dataset svhn: $n=3072$. From left to right we have: Coordinate sketch with convenient probabilities, coordinate sketch with uniform probabilities and Gaussian sketch respectively. 
}\label{fig:splice-scale}
\end{figure}

\begin{figure}[H]
    \centering
\begin{minipage}{0.30\textwidth}
  \centering
\includegraphics[width =  \textwidth ]{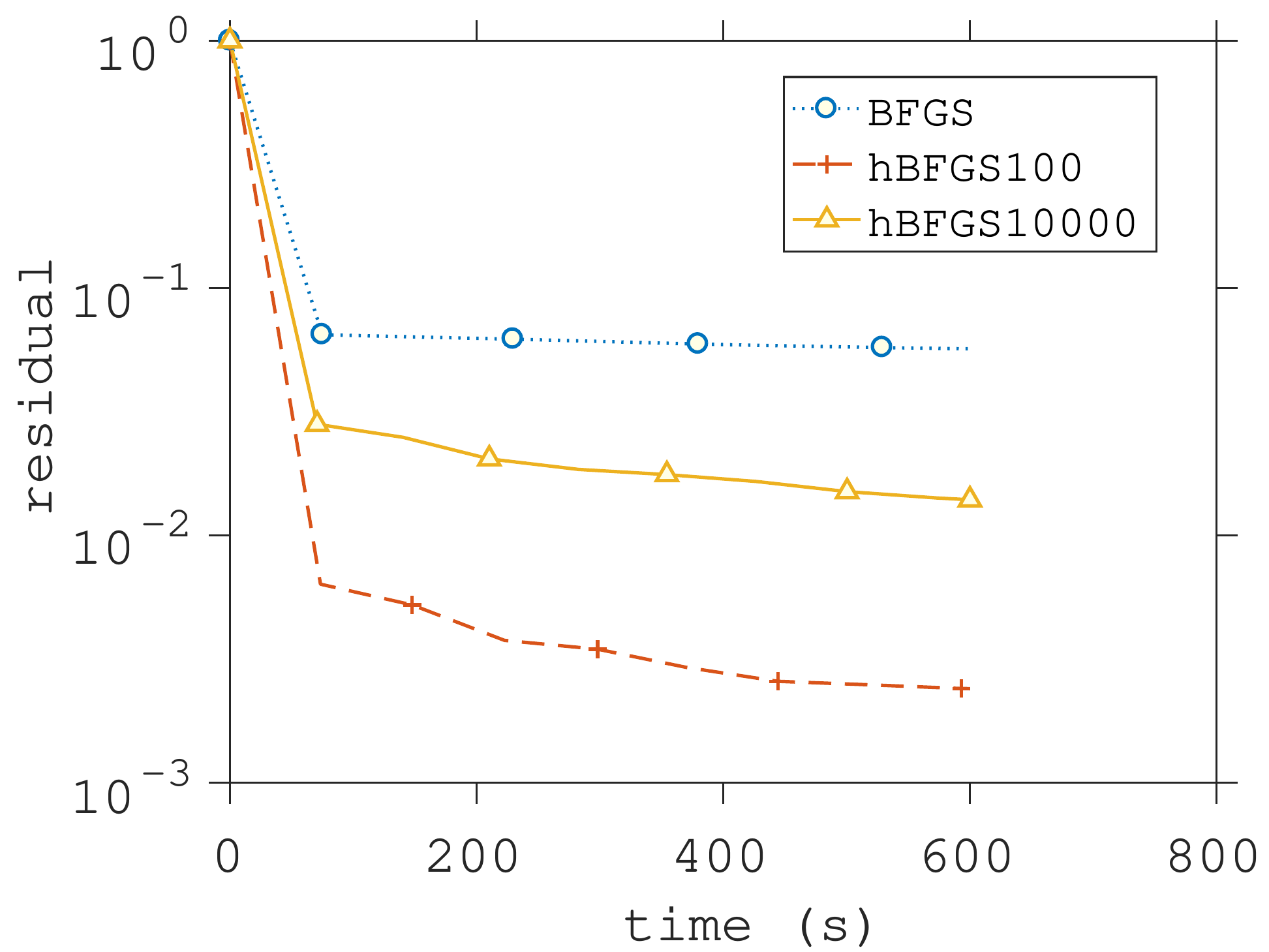}
\end{minipage}%
\begin{minipage}{0.30\textwidth}
  \centering
\includegraphics[width =  \textwidth ]{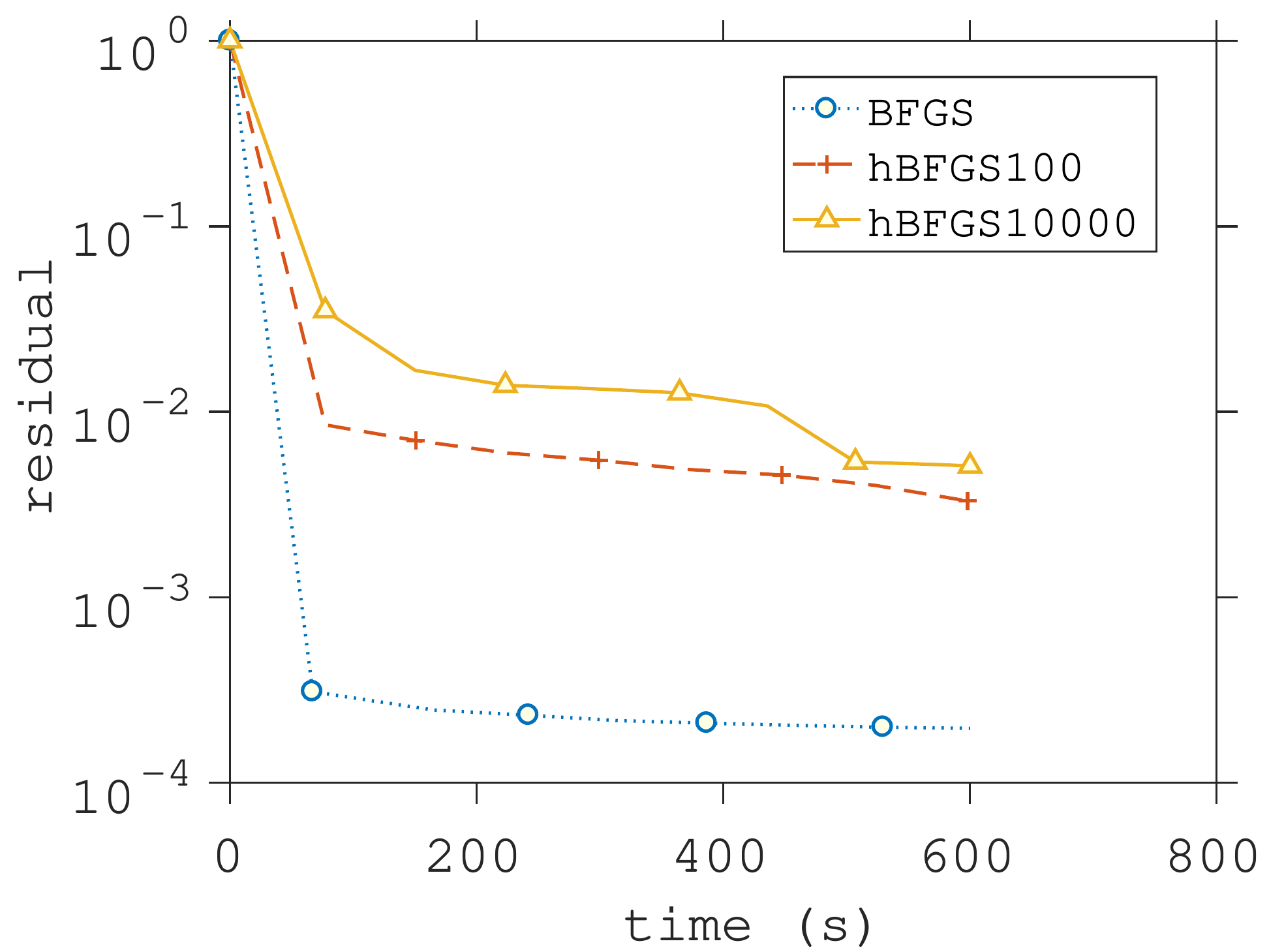}
\end{minipage}%
\begin{minipage}{0.30\textwidth}
  \centering
\includegraphics[width =  \textwidth ]{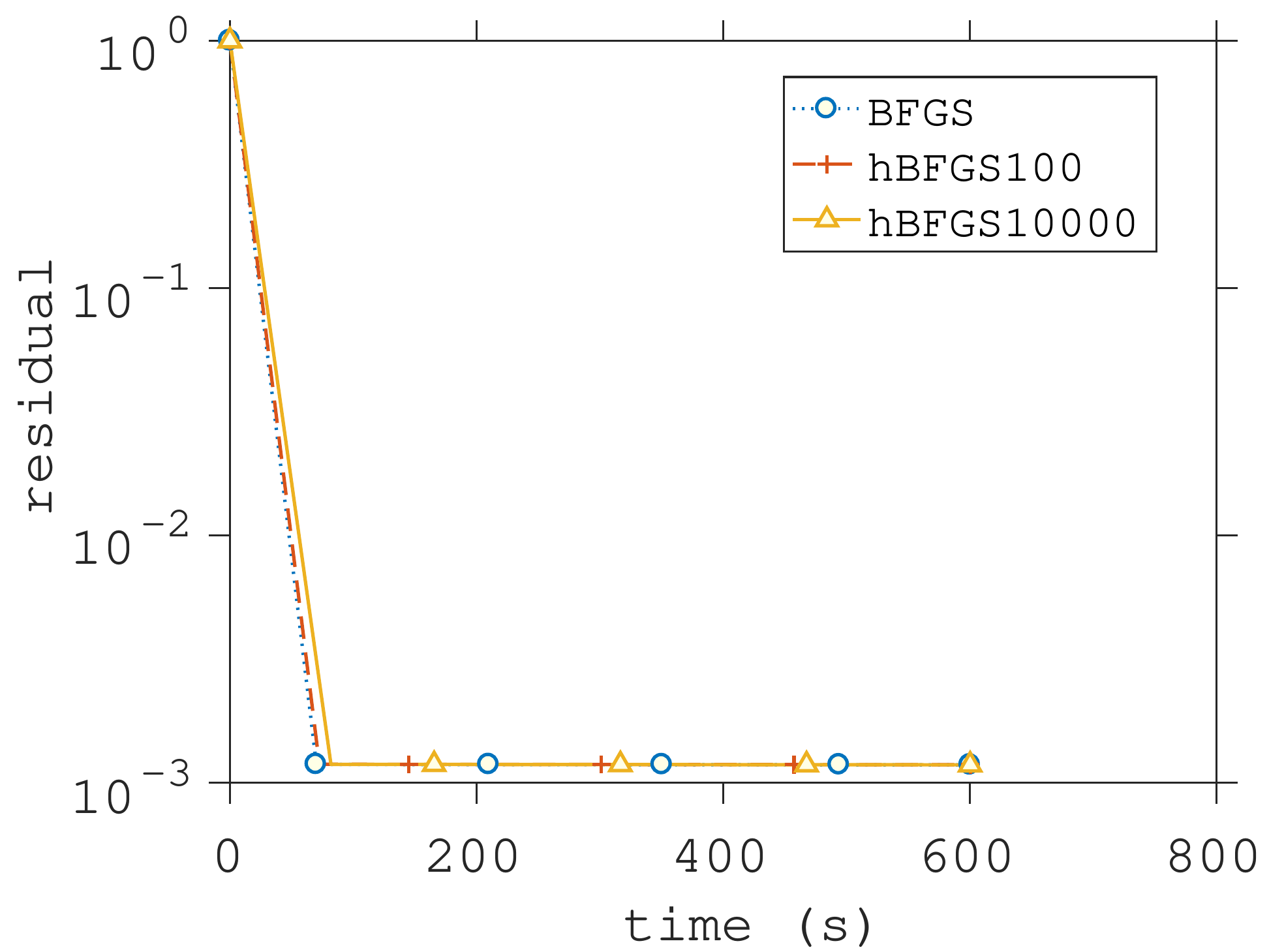}
\end{minipage}%
    \caption{  Dataset gisette: $n=5000$. From left to right we have: Coordinate sketch with convenient probabilities, coordinate sketch with uniform probabilities and Gaussian sketch respectively. 
}\label{fig:splice-scale}
\end{figure}

Notice that once the acceleration parameters are not set exactly (but they are still reasonable), we observe that the performance of the accelerated algorithm is essentially the same as the performance of the nonaccelerated algorithm, which is essentially the same conclusion as for artificially generated examples.  

\subsection{Additional optimization experiments}

In Figure~\ref{fig:australiantime} we solve the same problems with the same setup as in~\ref{fig:australiantime}, but now we plot the time versus the residual (as opposed to iterations versus the residual).  Despite the more costly iterations, the accelerated BFGS method can still converge faster than the classic BFGS method.  
 
\begin{figure}[!h]
    \centering
\begin{minipage}{0.24 \textwidth}
\includegraphics[width =  \textwidth ]{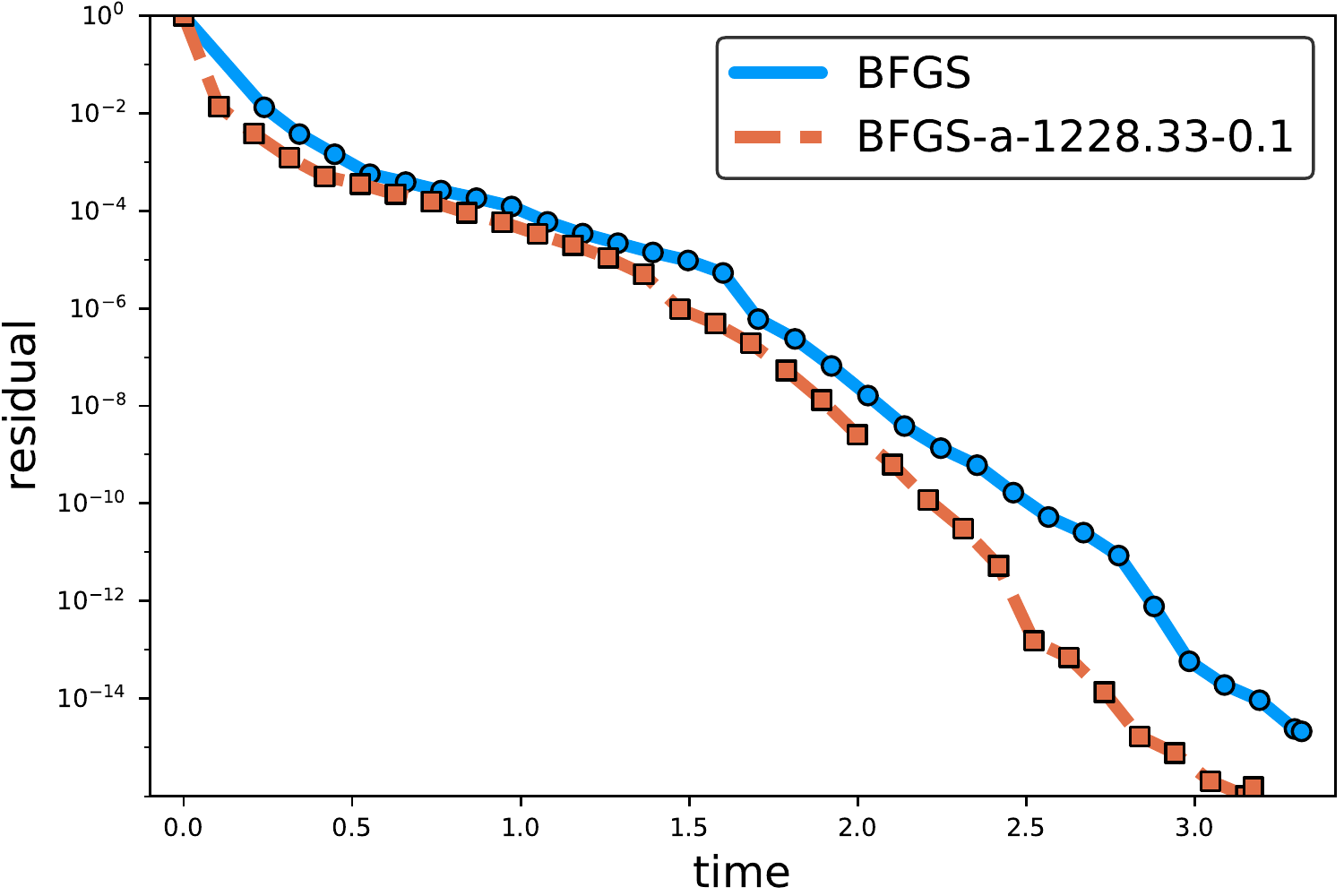}
\end{minipage}%
\begin{minipage}{0.24 \textwidth}
\includegraphics[width =  \textwidth ]{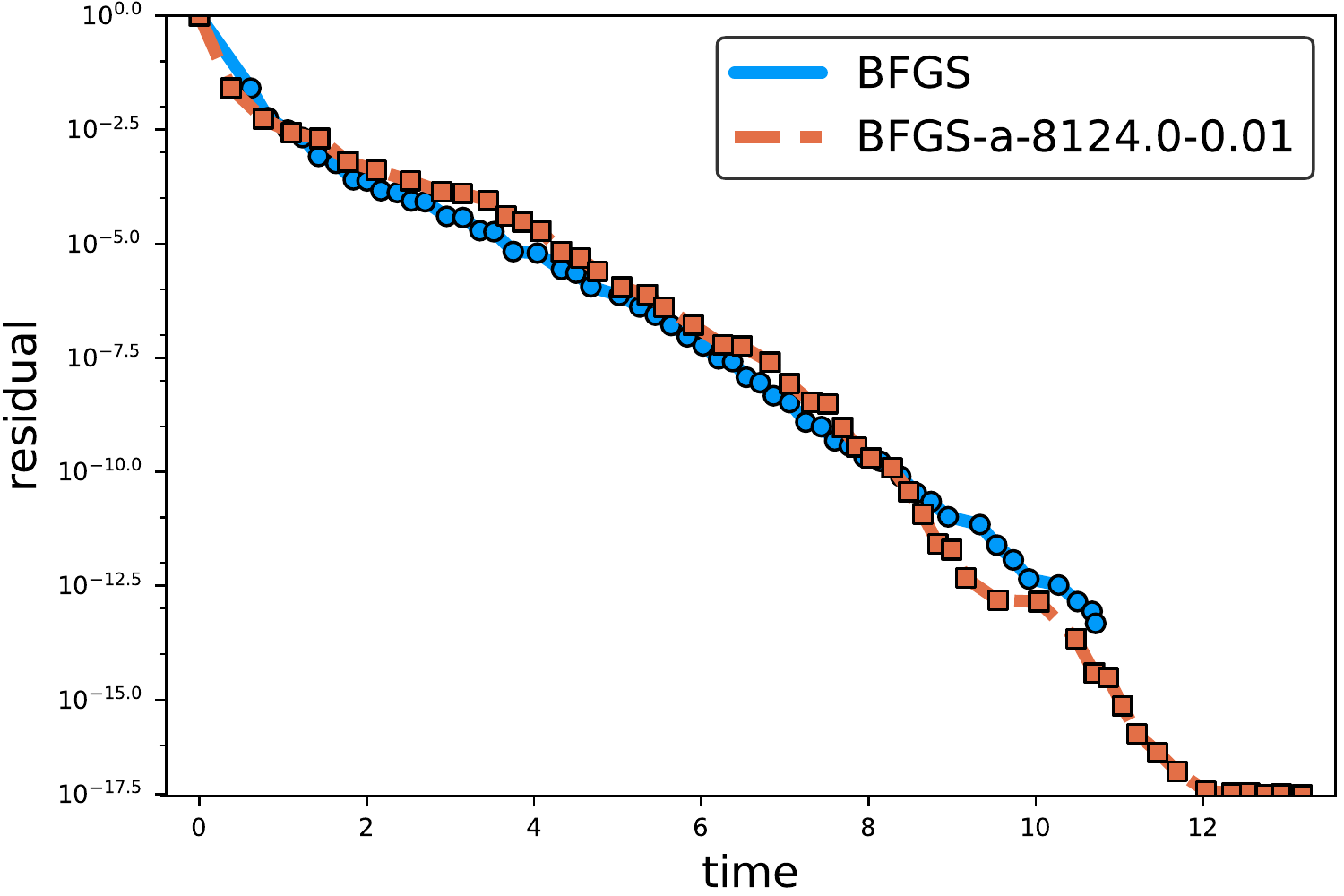}
\end{minipage}
\begin{minipage}{0.24	 \textwidth}
\includegraphics[width =  \textwidth ]{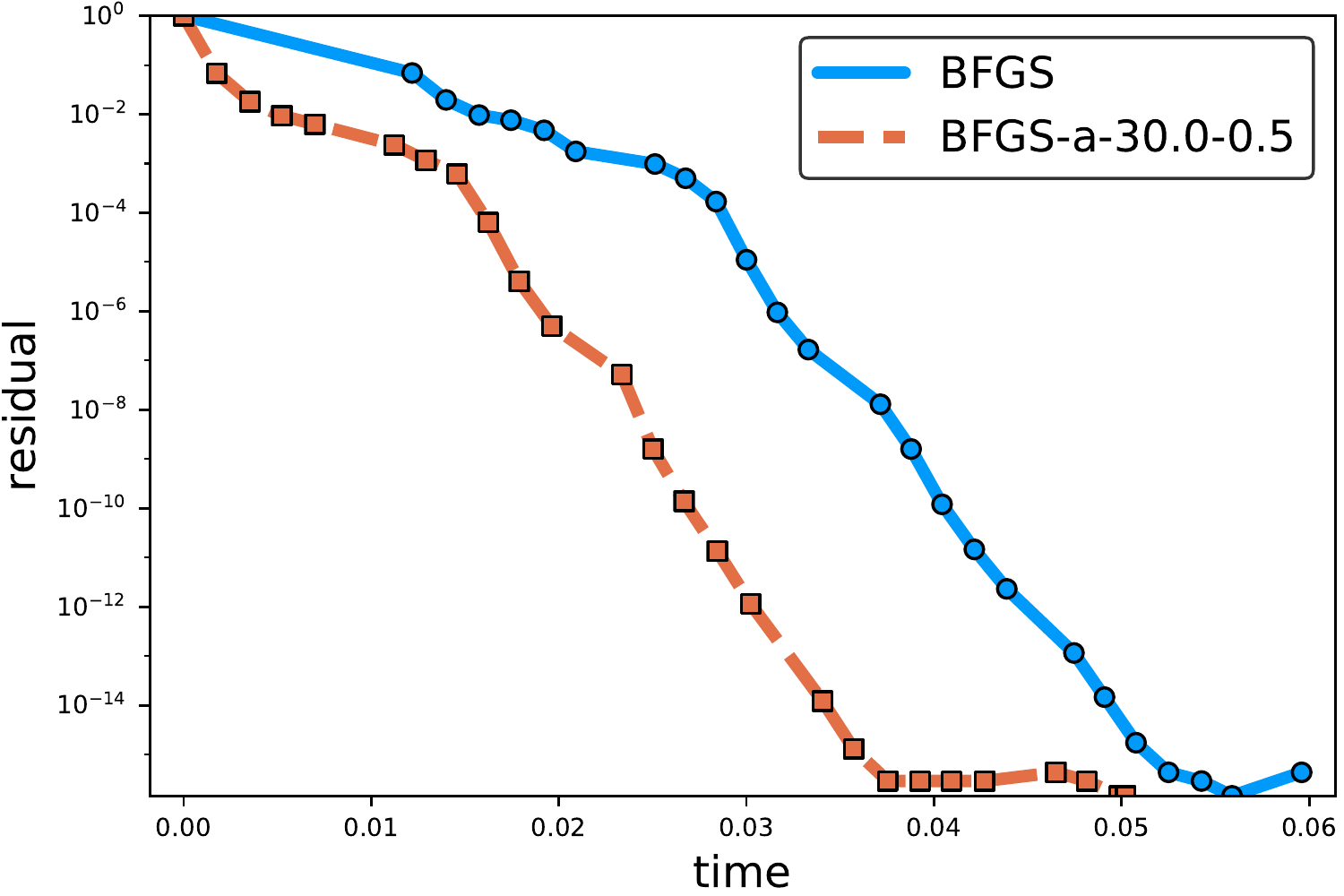}	
\end{minipage}%
\begin{minipage}{0.24 \textwidth}
\includegraphics[width =  \textwidth ]{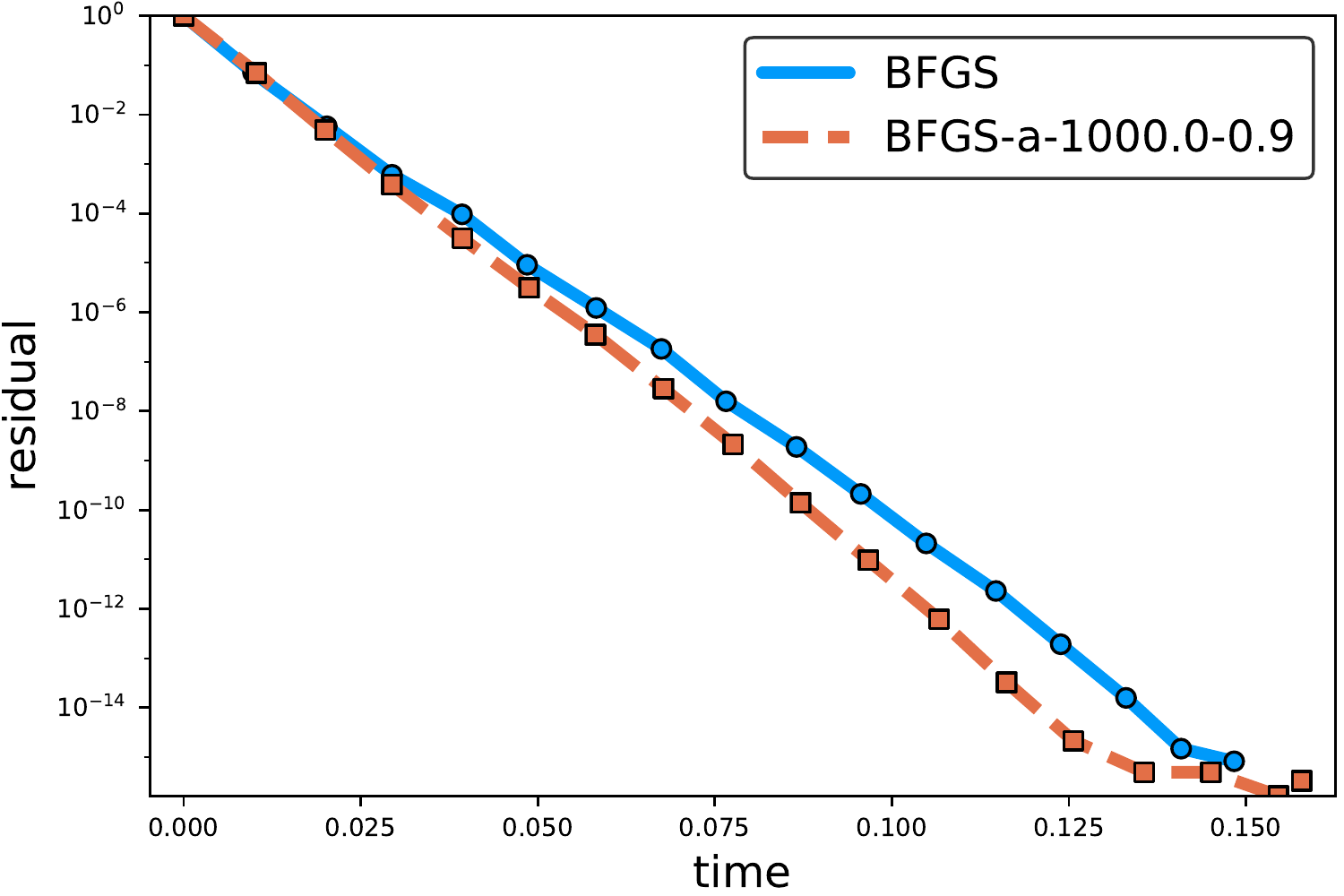}
\end{minipage}%
\caption{Algorithm~\ref{alg:bfgs_opt} (BFGS with accelerated matrix inversion quasi-Newton update) vs standard BFGS. From left to right:  \texttt{phishing}, \texttt{mushrooms}, \texttt{australian} and \texttt{splice} dataset.}
\label{fig:australiantime}
\end{figure}

We also give additional experiments with the same setup to the ones found in Section~\ref{sec:accBFGSexp}. 
Much like the \texttt{phishing} problem in Figure~\ref{fig:australian}, the problems \texttt{madelon}, \texttt{covtype} and \texttt{a9a} in Figures~\ref{fig:bfgs_opt_libsvm4},~\ref{fig:bfgs_opt_libsvm6} and~\ref{fig:bfgs_opt_libsvm7} did not benefit that much from acceleration.
 Indeed, we found in our experiments that even when choosing extreme values of $\mu$ and $\nu$, the generated inverse Hessian would not significantly deviate from the estimate that one would obtain using the standard BFGS update. Thus on these two problems there is apparently little room for improvement by using acceleration.    

\begin{figure}[H]
	\centering
	\begin{minipage}{0.25\textwidth}
		\centering
		\includegraphics[width = \textwidth ]{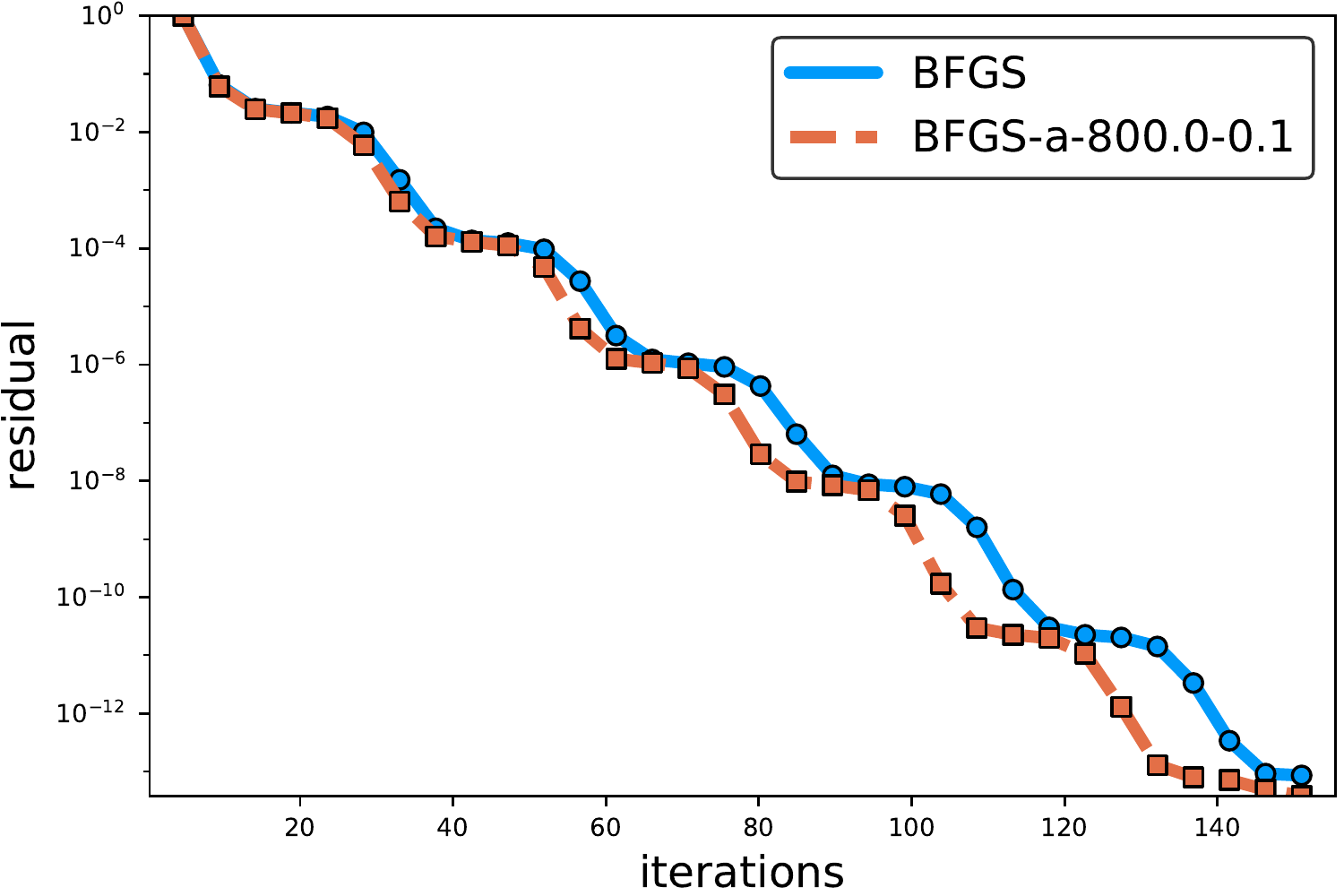}
		\caption{  \texttt{madelon}:}
		\label{fig:bfgs_opt_libsvm4}
	\end{minipage}
	\begin{minipage}{0.25\textwidth}
		\centering
		\includegraphics[width=\textwidth]{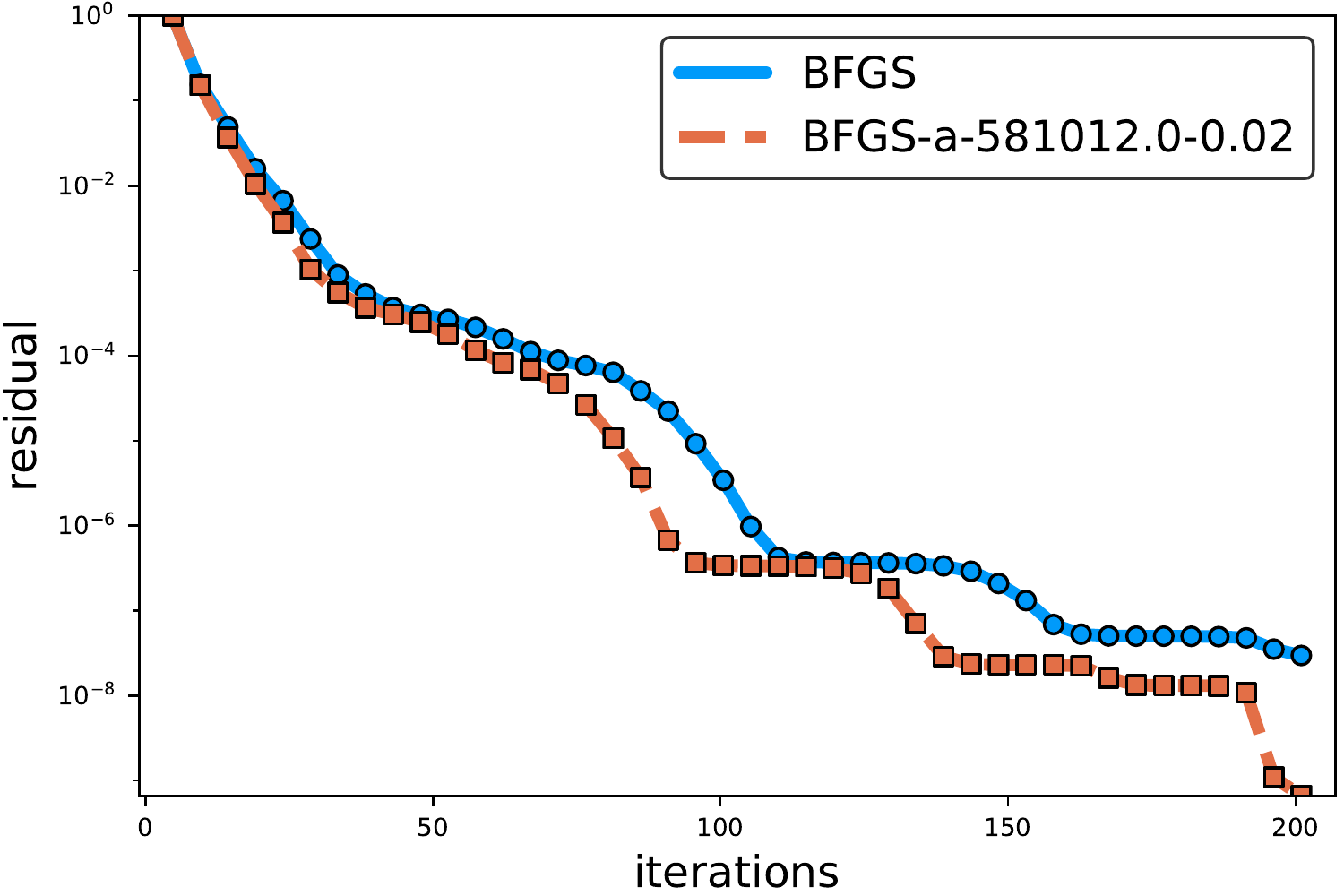}
		\caption{  \texttt{covtype}}
		\label{fig:bfgs_opt_libsvm6}
	\end{minipage}%
		\begin{minipage}{0.25\textwidth}
			\centering
			\includegraphics[width=\textwidth]{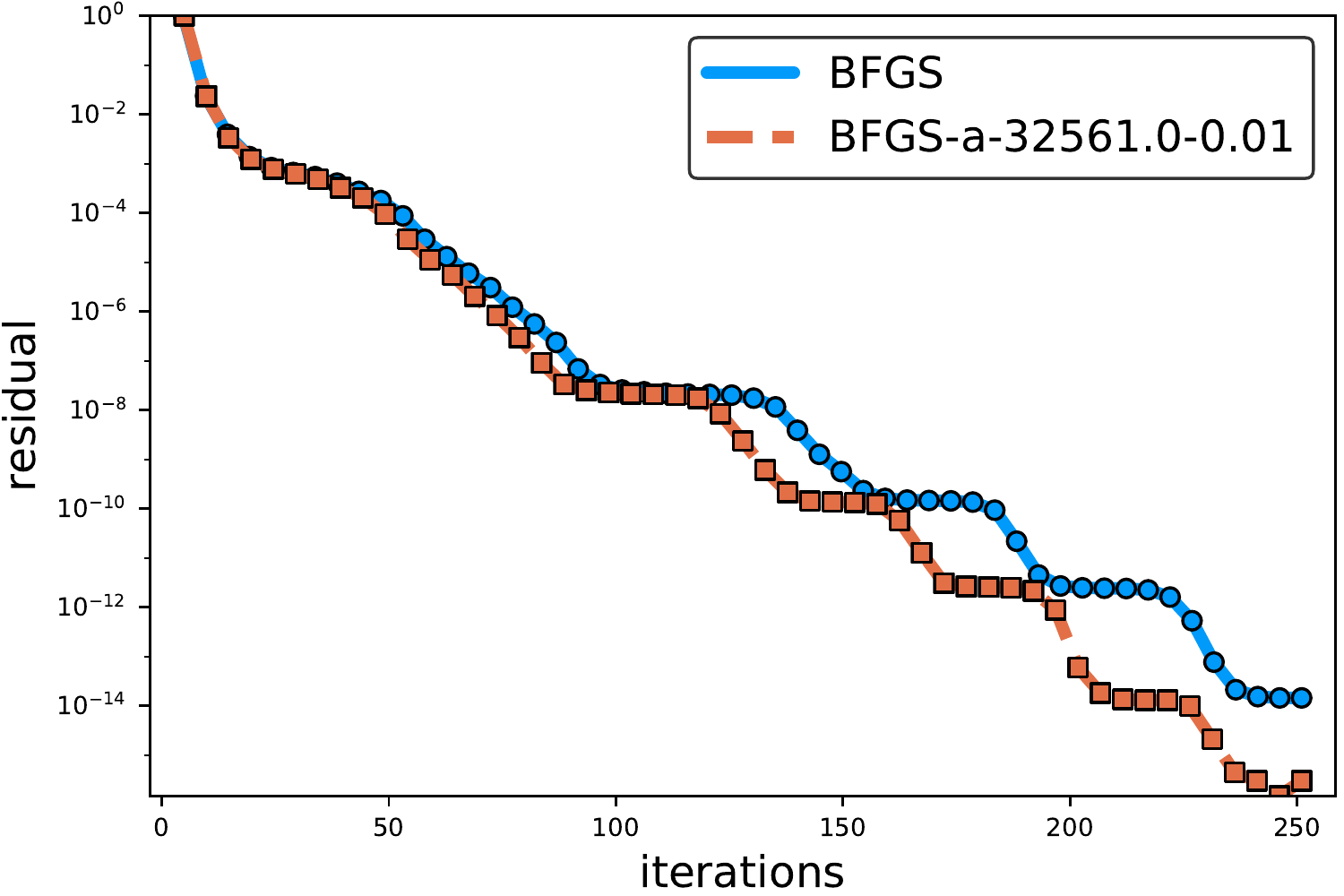}	
			\caption{  \texttt{a9a}}
			\label{fig:bfgs_opt_libsvm7}
		\end{minipage}%
	\label{fig:bfgs_opt_libsvm}
\end{figure}

\section{Conclusions and Extensions}
We developed an accelerated sketch-and-project method for solving linear systems in Euclidean spaces. The method was applied to invert positive definite matrices, while keeping their symmetric structure. Our accelerated matrix inversion algorithm was then incorporated into an optimization framework to develop both accelerated stochastic and deterministic BFGS, which to the best of our knowledge, are {\em the first  accelerated quasi-Newton updates.} 

We show that under a careful choice of the  parameters of the method, and depending on the problem structure and conditioning, acceleration might result into significant speedups both for the matrix inversion problem and for the stochastic BFGS algorithm. We confirm experimentally that our accelerated methods can lead to speed-ups when compared to the classical BFGS algorithm.

As a future line of research, it might be interesting to study the accelerated BFGS algorithm (either deterministic or stochastic)  further, and provide a convergence analysis on a suitable class of functions. Another interesting area of research might be to combine accelerated BFGS with limited memory \cite{liu1989limited} or engineer the method so that it can efficiently compete with first order algorithms for some empirical risk minimization problems, such as, for example \cite{gower2016stochastic}. 

As we show in this work, {\em Nesterov's acceleration can be applied to quasi-Newton updates}. We believe this is a surprising fact, as quasi-Newton updates have not been understood as optimization algorithms, which prevented the idea of applying acceleration in this context.

Since since second-order methods are becoming more and more ubiquitous in machine learning and data science, we hope that our work will motivate further advances at the frontiers of big data optimization.

\bibliographystyle{plain}
\bibliography{accel_mat_inv}

\clearpage
\appendix

\clearpage

\section{Proofs for Section~\ref{sec:ami_paper} \label{sec:proof_Euclidean}}

\subsection{Proof of Lemma~\ref{lem:Z_bounds}}
First note that $Z$ is a self-adjoint positive operator and thus so is $\E{Z}.$  Consequently.
 \begin{eqnarray}
 \mu &\overset{\eqref{eq:mu+nu}}{=}& \inf_{x \in \Range{\A^*}} \frac{\dotprod{\E{Z}x,x}}{\dotprod{x,x}} \nonumber \\
 &\overset{\eqref{eq:exactness}}{=} &
 \inf_{x \in \Range{\E{Z}}} \frac{\dotprod{\E{Z}x,x}}{\dotprod{x,x}} \nonumber\\
 & \overset{\mbox{  Lemma~\ref{lem:pseudo} item~\ref{it:pseudorange} }}{=}& \inf_{x \in \cX} \frac{\dotprod{\E{Z}\E{Z}^\dagger x,\E{Z}^\dagger x}}{\dotprod{\E{Z}^\dagger x,\E{Z}^\dagger x}} \nonumber\\
  & \overset{\mbox{  Lemma~\ref{lem:pseudo} item~\ref{it:pseudoTTdagT} }}{=}& \inf_{x \in \cX} \frac{\dotprod{\E{Z}^\dagger x,x}}{\dotprod{\E{Z}^\dagger x,\E{Z}^\dagger x}} \nonumber\\
   & \overset{\mbox{  Lemma~\ref{lem:squareroot} }}{=}& 
  \inf_{z \in \Range{(\E{Z}^\dagger)^{1/2}}} \frac{\dotprod{ z,z}}{\dotprod{\E{Z}^\dagger z, z}} \qquad (\mbox{set }z = (\E{Z}^\dagger)^{1/2}x)\nonumber  \\
  & \overset{\eqref{eq:RangeGhalf} }{=}& 
  \frac{1}{\norm{\E{Z}^{\dagger}}}.\label{eq:as9d8n923}
 \end{eqnarray} 
 
 For the bounds~\eqref{eq:nubnds} we have that
\begin{eqnarray*}
 \nu &\overset{\eqref{eq:mu+nu}}{=}& \sup_{x \in \Range{\A^*}} \frac{\E{\dotprod{\E{Z}^\dagger Zx,Zx}}}{\dotprod{\E{Z}x,x}}  \\
 &\leq & \sup_{x \in \Range{\A^*}} \frac{\norm{\E{Z}^\dagger} \E{\norm{Zx}_2^2}}{\dotprod{\E{Z}x,x}}\\
 & =& \norm{\E{Z}^\dagger}\\
 & \overset{\eqref{eq:as9d8n923}}{\leq} & \frac{1}{\mu}.
 \end{eqnarray*} 
To bound $\nu$ from below we use that $\E{Z}^\dagger$ is self adjoint together with that the map $X \mapsto \dotprod{X\E{Z}^\dagger Xx,x}$ is convex over the space of self-adjoint operators $X \in L(\cX)$ and for a fixed $x \in \cX$. Consequently by Jensen's inequality
\begin{equation}
\E{\dotprod{Z\E{Z}^\dagger Zx,x}} \geq \dotprod{\E{Z}\E{Z}^\dagger \E{Z}x,x}
\overset{ \mbox{  Lemma~\ref{lem:pseudo} item~\ref{it:pseudoTTdagT} }}{=}\dotprod{\E{Z} x, x}.\label{eq:ZEZZjen}
\end{equation}
Finally
\begin{eqnarray*}
\nu & \overset{\eqref{eq:ZEZZjen}}{\geq} & \sup_{x \in \Range{\A^*}} \frac{\dotprod{\E{Z}x,x}}{\dotprod{\E{Z}x,x}}  = 1.
\end{eqnarray*}

Lastly, to show \eqref{eq:nu_lower} we have
 \begin{eqnarray*}
 \Rank{\A^*} & \overset{\eqref{eq:exactness}}{=}& \Rank{\E{Z}} \\
  & \overset{\mbox{  Lemma~\ref{lem:projrank}+ Lemma~\ref{lem:pseudo} (\emph{\ref{it:pseudoproj}}) }}{=}& \Tr{\E{Z}\E{Z}^\dagger} = \E{\Tr{Z \E{Z}^\dagger}}\\
 &= & \E{\Tr{Z \E{Z}^\dagger Z}} \\
 &\leq & \nu \E{\Tr{Z}} \overset{\mbox{  Lemma~\ref{lem:projrank}}}{=} \nu   \E{\Rank{Z}},
 \end{eqnarray*}
where we used that  $\dotprod{ \E{Z \E{Z}^\dagger Z} u,u}  \leq  \nu \dotprod{\E{Z} u,u}$ for every $u \in \Range{\E{Z}}=\Range{\A^*}=\cX.$ \qed

 {\bf Proof} that $X \mapsto \dotprod{X\E{Z}^{\dagger}Xx,x} = \norm{Xx}_{\E{Z}^{\dagger}}^2$ is convex: Let $G= \E{Z}^{\dagger}$ then
 \begin{eqnarray*}
\norm{(\lambda X +(1-\lambda)Y)x}_G^2 &= &
\lambda^2 \norm{X x}_G^2 +(1-\lambda)^2  \norm{Yx}_G^2
+2\lambda(1-\lambda) \dotprod{x X GY ,x}\\
&= & -\lambda(1-\lambda)\norm{(X-Y)x}_G^2 \\
& &+ \lambda  \norm{X x}_G^2+(1-\lambda)\norm{Yx}_G^2\\
&\leq &\lambda  \norm{X x}_G^2+(1-\lambda)\norm{Y x}_G^2 . \hskip5cm \qed 
\end{eqnarray*}

\subsection{Technical lemmas to prove Theorem~\ref{theo:conv}}

\begin{lemma} \label{lem:invariant}For all $k \geq 0,$ the vectors $y_k-x_*, \, x_k-x_*$ and $v_k-x_*$ belong to $\Range{\A^*}.$
\end{lemma} 
\begin{proof}
 Note that $x_0 =y_0 =x_0$ and in view of~\eqref{eq:xsol} we have $x_* \in x_0 + \Range{\A^*}.$ So $y_0 -x_* \in \Range{\A^*},$  $v_0 -x_* \in \Range{\A^*}$ and  $x_0 -x_* \in \Range{\A^*}.$ Assume by induction that  $y_k -x_* \in \Range{\A^*},$ $v_k -x_* \in \Range{\A^*}$ and  $x_k -x_* \in \Range{\A^*}.$ Since $g_k \in \Range{\A^*}$ and $x_{k+1} = y_k - g_k$ we have
 \[x_{k+1}-x_* = (y_k-x_*) -g_k \in \Range{\A^*}.\]
 Moreover, 
 \[v_{k+1}-x_* = \beta(v_k -x_*) + (1-\beta)(y_k-x*)-\gamma g_k \in \Range{\A^*}.\]
 Finally
  \[y_{k+1}-x_* = \alpha v_{k+1}+(1-\alpha)x_{k+1}-x_* = \alpha (v_{k+1}-x_*)+(1-\alpha)(x_{k+1}-x_*)\in \Range{\A^*}.\]
\end{proof}

\begin{lemma}
\begin{equation}\label{eq:Enormnubnd}
\E{ \norm{Z_k(y_k -x_*)}_{\E{Z}^{\dagger}}^2 \, | \, y_k} \leq \nu \norm{y_k -x_*}_{\E{Z}}^2
\end{equation} 
\end{lemma}
\begin{proof}
Since $y_k -x_* \in \Range{\A^*}$ we have that
\begin{eqnarray*}
\E{ \norm{Z_k(y_k -x_*)}_{\E{Z}^{\dagger}}^2 \, | \, y_k} &=  & 
 \dotprod{\E{Z_k\E{Z}^{\dagger}Z_k}(y_k -x_*),(y_k -x_*)} \\
 & \overset{\eqref{eq:mu+nu}}{\leq} &  \nu \dotprod{\E{Z}(y_k -x_*),(y_k -x_*)} \\
 &= & \nu \norm{y_k -x_*}_{\E{Z}}^2. 
\end{eqnarray*}
\end{proof}

\begin{lemma}
\label{lemma:threepoint}
\begin{equation}\label{eq:ykxkident}
 \norm{y_k-x_*}_{\E{Z}}^2 = \norm{y_k-x_*}^2 -\E{\norm{x_{k+1}-x_*}^2 \, | \, y_k}\end{equation}
\end{lemma}
\begin{proof}
\begin{eqnarray*}
\E{\norm{x_{k+1}-x_*}^2 \, | \, y_k} &= & \E{\norm{(I-Z_k)( y_k -x_*)}^2 \, | \, y_k} \\
&= &\dotprod{(I-\E{Z})( y_k -x_*),y_k -x_*}\\
&= &\norm{y_k -x_*}^2 - \norm{y_k -x_*}^2_{\E{Z}}. 
\end{eqnarray*} 
\end{proof}
\subsection{Proof of Theorem \ref{theo:conv}}
\label{subsec:proof_main_thm}
Let $r_k \eqdef \norm{v_{k}-x_*}^2_{\E{Z}^\dagger}$. It follows that
\begin{eqnarray}
r_{k+1}^2 & =& \norm{v_{k+1}-x_*}^2_{\E{Z}^\dagger} \nonumber\\
& =& \norm{\beta v_k +(1-\beta)y_k -x_*- \gamma Z_k(y_k -x_*)}^2_{\E{Z}^\dagger}\nonumber\\
&=&\underbrace{\norm{\beta v_k +(1-\beta)y_k -x_*}^2_{\E{Z}^\dagger}}_{I} +\gamma^2\underbrace{\norm{Z_k(y_k -x_*)}^2_{\E{Z}^\dagger}}_{II} \nonumber \\
& &-2\gamma \underbrace{\dotprod{\beta (v_k-x_*) +(1-\beta)(y_k -x_*), \E{Z}^\dagger Z_k(y_k -x_*)}}_{III} \nonumber \\
& =& I + \gamma^2 II -2\gamma III.\label{eq:proofstep1}
\end{eqnarray}

The first term can be upper bounded as follows
\begin{eqnarray}
I &=& \norm{\beta (v_k-x_*) +(1-\beta)(y_k -x_*)}^2_{\E{Z}^\dagger} \nonumber \\
& =&\beta^2 \norm{v_k-x_*}^2_{\E{Z}^\dagger}  + (1-\beta)^2 \norm{y_k-x_*}^2_{\E{Z}^\dagger} +2\beta(1-\beta) \dotprod{v_k-x_*,y_k -x_*}_{\E{Z}^\dagger}\nonumber \\
&\overset{\eqref{eq:paral1}}{=} & \beta\norm{v_k-x_*}^2_{\E{Z}^\dagger}  + (1-\beta) \norm{y_k-x_*}^2_{\E{Z}^\dagger} -\beta(1-\beta) \norm{v_k-y_k}^2_{\E{Z}^\dagger} \nonumber \\
& \leq & \beta r_k^2 + (1-\beta) \norm{y_k-x_*}^2_{\E{Z}^\dagger},\label{eq:Ibnded}
\end{eqnarray} 
where in the third equality we used a form of the parallelogram identity
\begin{equation}\label{eq:paral1}
2\dotprod{u,v} = \norm{u}^2 + \norm{v}^2 - \norm{u-v}^2,
\end{equation}
with $u = v_k-x_*$ and $v = y_k-x_*.$

Taking expectation with to $\cS_k$ in the third term in~\eqref{eq:proofstep1} gives
\begin{eqnarray}
\E{III \, | \, y_k, v_k, x_k} &= & \dotprod{\beta v_k +(1-\beta)y_k -x_*, \E{Z}^\dagger \E{Z}(y_k -x_*)} \nonumber \\
& =& \dotprod{\beta v_k +(1-\beta)y_k -x_*,y_k -x_*} \label{eq:explainsteplat} \\
& =& \dotprod{\beta \left[\frac{1}{\alpha}y_k - \frac{1-\alpha}{\alpha} x_k\right] +(1-\beta)y_k -x_*,y_k -x_*}\nonumber \\
&= & \dotprod{y_k - x_* +\beta \frac{1-\alpha}{\alpha}(y_k- x_k) ,y_k -x_*} \nonumber \\
& =& \norm{y_k - x_*}^2 + \beta \frac{1-\alpha}{\alpha}\dotprod{y_k- x_k,y_k-x_*}\nonumber \\
&=&\norm{y_k - x_*}^2  - \beta \frac{1-\alpha}{2\alpha}\left(\norm{x_k-x_*}^2-\norm{y_k- x_k}^2-\norm{y_k- x_*}^2 \right)\label{eq:IIIbnded}
\end{eqnarray}
where in the second equality~\eqref{eq:explainsteplat} we used that $y_k -x_* \in \Range{\A^*} \overset{\eqref{eq:exactness}}{ =} \Range{\E{Z}}$ together with a defining property of pseudoinverse operators $\E{Z}^\dagger\E{Z} w = w$ for all $w  \in \Range{\E{Z}}.$ In the last equality~\eqref{eq:IIIbnded} we used yet again the identity~\eqref{eq:paral1} with $u = y_k- x_k$ and $v = y_k-x_*.$

Plugging~\eqref{eq:Ibnded} and~\eqref{eq:IIIbnded} into~\eqref{eq:proofstep1}
and taking conditional expectation gives
\begin{eqnarray}
\E{r_{k+1}^2 \, | \, y_k, v_k, x_k} & =& I + \gamma^2 \E{II \, | \, y_k}  -2 \gamma \E{III \, | \, y_k,v_k,x_k} \nonumber\\
& \overset{\eqref{eq:Ibnded}+\eqref{eq:IIIbnded}+\eqref{eq:Enormnubnd}}{ =}& 
\beta r_k^2 + (1-\beta) \norm{y_k-x_*}^2_{\E{Z}^\dagger} + \gamma^2  \nu \norm{y_k-x_*}^2_{\E{Z}} \nonumber \\
& & +2\gamma \left(-\norm{y_k - x_*}^2  + \beta \frac{1-\alpha}{2\alpha}\left(\norm{x_k-x_*}^2-\norm{y_k- x_k}^2-\norm{y_k- x_*}^2 \right)\right)\nonumber\\
& \overset{\eqref{eq:ykxkident}+\eqref{eq:nubnds}}{\leq } & 
\beta r_k^2 + \frac{1-\beta}{\mu} \norm{y_k-x_*}^2 + \gamma^2  \nu\left( \norm{y_k-x_*}^2 -\E{\norm{x_{k+1}-x_*}^2 \, | \, y_k}\right) \nonumber \\
& & +2\gamma \left(-\norm{y_k - x_*}^2  + \beta \frac{1-\alpha}{2\alpha}\left(\norm{x_k-x_*}^2-\norm{y_k- x_*}^2\right)\right). \label{eq:continue_omega_proof}
\end{eqnarray}
Therefore we have that
\begin{eqnarray}
\E{r_{k+1}^2  +\gamma^2  \nu\norm{x_{k+1}-x_*}^2\, | \, y_k, v_k, x_k }
& \leq & 
\beta \left(r_k^2 + \underbrace{\gamma\frac{1-\alpha}{\alpha}}_{P_1}\norm{x_k-x_*}^2\right) \nonumber \\
& &+\left( \underbrace{\frac{1-\beta}{\mu} -2\gamma+\gamma^2  \nu - \beta \gamma\frac{1-\alpha}{\alpha}}_{P_2}\right)\norm{y_k-x_*}^2.    \nonumber 
\end{eqnarray}
To establish a recurrence, we need to choose the free parameters $\gamma, \alpha$ and  $\beta$ so that $P_1 =\gamma^2  \nu$ and $P_2 =0.$ Furthermore we should try to set $\beta$ as small as possible so as to have a fast rate of convergence. 
Choosing $\beta  =1 - \sqrt{\frac{\mu}{\nu}},$ $\gamma = \sqrt{\frac{1}{\mu \nu}},$ $\alpha = \frac{1}{1+\gamma\nu}$ gives $P_2 =0$, $\gamma^2 \nu = 1/\mu$ and 
 \begin{eqnarray}
\E{r_{k+1}^2  +\tfrac{1}{\mu}\norm{x_{k+1}-x_*}^2\, | \, y_k, v_k, x_k }
& \leq & \left(1 - \sqrt{\frac{\mu}{\nu}} \right) \left(r_k^2 + \tfrac{1}{\mu}\norm{x_k-x_*}^2\right).
\end{eqnarray}
Taking expectation and using the tower rules gives the result.\qed
%

\subsection{Changing norm \label{sec:change_norm}}
Given an invertible positive self-adjoint $B \in L(\cX),$ suppose we want to find the least norm solution of~\eqref{eq:primal} under the norm defined by
$\norm{x}_B  \eqdef \sqrt{\dotprod{Bx,x}}$ as the metric in $\cX$. That is,
we want to solve
\begin{equation} \label{eq:primalB_change_norm}
x^* \eqdef \arg\min_{x \in \cX} \tfrac{1}{2}\norm{x-x_0}_B^2, \quad \mbox{subject to} \quad \A x = b.
\end{equation}
By changing variables $x = B^{-1/2}z$ we have that the above is equivalent to solving
\begin{equation} \label{eq:primalz}
z^* \eqdef \arg\min_{z \in \cX} \tfrac{1}{2}\norm{z-z_0}^2, \quad \mbox{subject to} \quad \A B^{-1/2} z = b,
\end{equation}
with $x^* = B^{-1/2}z^*$, and $B^{1/2}$ is the unique symmetric square root of $B$ (see Lemma~\ref{lem:squareroot}). We can now apply Algorithm~\ref{alg:SketchJac} to solve~\eqref{eq:primalz} where $\A B^{-1/2}$ is  the system matrix.  Let $x_k$ and $v_k$ be the resulting iterates of applying Algorithm~\ref{alg:SketchJac}.
To make explicit this change in the system matrix we define the matrix
\[
Z_B \eqdef B^{-1/2}\A^*\cS_k^*(\cS_k\A B^{-1}\A^*\cS_k^*)^{\dagger}\cS_k \A B^{-1/2},
\]
and the constants
\begin{equation} \label{eq:muB}
\mu_B  \eqdef   \inf_{x \in \Range{B^{-1/2}\A^*}} \frac{\dotprod{\E{Z_B}x,x}}{\dotprod{x,x}}
\end{equation}
and
\begin{equation} \label{eq:nuB}
\nu_B  \eqdef   \sup_{x \in \Range{B^{-1/2}\A^*}} \frac{\dotprod{\E{Z_B\E{Z_B}^\dagger Z_B}x,x}}{\dotprod{\E{Z_B}x,x}}.
\end{equation}

 Theorem~\ref{theo:conv} then guarantees that
\[
\E{\norm{v_{k+1} -z_*}_{\E{Z_B }^\dagger}^2 +\frac{1}{\mu_B}\norm{x_{k+1} -z_*}^2 } \leq \left(1 - \sqrt{\frac{\mu_B}{\nu_B}} \right) \E{\norm{v_k -z_*}_{\E{Z_B }^\dagger}^2 + \frac{1}{\mu_B}\norm{x_k-z_*}^2}.
\]
 Reversing our change of variables $\bar{x}_k = B^{-1/2}x_k$ and $\bar{v}_k = B^{-1/2} v_k$ in the above displayed equation gives
  \begin{eqnarray}
&& \E{\norm{\bar{v}_{k+1} -x_*}_{B^{1/2}\E{Z_B}^\dagger B^{1/2}}^2 +\frac{1}{\mu_B}\norm{\bar{x}_{k+1} -x_*}_B^2 } \notag \\
&& \qquad \leq  \left(1 - \sqrt{\frac{\mu_B}{\nu_B}} \right) \E{\norm{\bar{v}_k -x_*}_{B^{1/2}\E{Z_B}^\dagger B^{1/2}}^2 + \frac{1}{\mu_B}\norm{\bar{x}_k-x_*}_B^2}.\label{eq:convzvbar2}
\end{eqnarray} 
Thus we recover the same exact from the main theorem in~\cite{MartinRichtarikAccell}, but in a much more general setting.

%

\section{Proof of Corollary~\ref{cor:sss}}

Clearly,  $Z=\frac{1}{A_{i,i}}A^{\frac12}SS^\top A^{\frac12}$, and hence
$ \E{Z}=\frac{A}{\Tr{A}}$ and $ \mu^P=\frac{\lambda_{\min}(A)}{\Tr{A}}.$
After  simple algebraic manipulations we get
\[
\E{\E{Z}^{-\frac12 }Z \E{Z}^{-1} Z\E{Z}^{-\frac12 }}=\Tr{A}^2\E{ \tfrac{1}{A_{i,i}^2}SS^\top SS^\top }  =\Tr{A}\diag{A_{i,i}^{-1}},
\]
and therefore
$
\nu^P= \lambda_{\max} \E{\E{Z}^{-\frac12 }Z \E{Z}^{-1} Z\E{Z}^{-\frac12 }} = \tfrac{\Tr{A}}{\min_i A_{i,i}}.
$

 \section{Adding  a stepsize $\omega$ \label{sec:omega}}
 
 In this section we enrich Algorithm~\ref{alg:SketchJac} with several {\em additional} parameters and study their effect on convergence of the resulting method.
 
 First, we consider an extension of Algorithm~\ref{alg:SketchJac} to a variant which uses a {\em stepsize parameter}  $0<\omega<2$. That is, instead of performing the update  
 \begin{equation} \label{eq:98gf98909} x_{k+1}=y_k - g_k,\end{equation}
 we perform the update
\begin{equation} \label{eq:089d8h09ff} x_{k+1}=y_k-\omega g_k.\end{equation}  Parameters $\alpha, \beta, \gamma$ are adjusted accordingly. The resulting method enjoys the rate
${\cal O}\left( \left(1-\sqrt{\frac{\nu}{\mu}\omega(2-\omega)}\right)^k \right),$
recovering the rate from Theorem~\ref{theo:conv} as a special case for $\omega=1$.  The formal statement follows.

\begin{theorem}
\label{thm:omega} 
Let $0 < \omega < 2$ be an arbitrary stepsize and define
\begin{align}
 \eta \eqdef 2\omega - \omega^2 \geq 0\,.
\end{align}
Consider a modification of Algorithm~\ref{alg:SketchJac} where instead of \eqref{eq:98gf98909} we perform the update \eqref{eq:089d8h09ff}. If we use the  parameters
\begin{align}
 \alpha &= \tfrac{1}{1 + \gamma \nu} &
 \beta  &= 1 -\sqrt{\tfrac{\mu \eta}{\nu}} & 
 \gamma &= \sqrt{\tfrac{\eta}{\mu \nu}},
\end{align}
then the iterates $\{v_k,x_k\}_{k \geq 0}$ of Algorithm~\ref{alg:SketchJac} satisfy
\begin{align*}
 &\E{\norm{v_{k} -x_*}_{\E{Z}^\dagger}^2 +\tfrac{1}{\mu}\norm{x_{k} -x_*}^2 } 
 \leq \left(1 - \sqrt{\tfrac{\mu \eta}{\nu}} \right)^k \E{\norm{v_0 -x_*}_{\E{Z}^\dagger}^2 + \tfrac{1}{\mu}\norm{x_0-x_*}^2}.
\end{align*}
\end{theorem}
\begin{proof} See Appendix~\ref{sec:89g8f9009jJJJ}.
\end{proof}
 
 \section{Allowing for  different $\alpha$}\label{sec:alpha}
 
 In this section we study how the choice of  the key parameter $\alpha$ affects the convergence  rate.

This parameter determines how much the sequence $y_k = \alpha v_k + (1-\alpha) x_k$ resembles the sequence given by $x_k$ or by $v_k$. For instance, when $\alpha = 0$, $y_k \equiv x_k$, i.e., we recover the steps of the non-accelerated method, and thus one would expect to obtain the same convergence rate as the non-accelerated method. Similar considerations hold in the other extreme, when $\alpha \to 1$. We investigate this hypothesis, and especially discuss how $\beta$ and $\gamma$ must be chosen as a function of $\alpha$ to ensure convergence. 
 
 The following statement is a generalization of Theorem~\ref{theo:conv}. For simplicity, we assume that the optional stepsize that was introduced in Theorem~\ref{thm:omega} is set to one again, $\omega \equiv 1$.


\begin{theorem}
\label{thm:family}
Let $0 < \alpha < 1$ be fixed. Then the iterates $\{v_k,x_k\}_{k \geq 0}$ of Algorithm~\ref{alg:SketchJac} 
with parameters
    \begin{align}
    \beta(s)  &= \frac{1+s - s \sqrt{\frac{\nu + 4\mu s - 2\nu s + \nu s^2}{\nu s^2}}}{2s}\,, &
    \gamma(s) &= \frac{1}{(1-s \beta(s))\nu}\,. \label{eq:def_bg}
    \end{align}
where $\tau \eqdef \frac{1-\alpha}{\alpha}$ and $s \eqdef \frac{\tau}{\beta \gamma}$, 
satisfy
\begin{align*}
 &\E{\norm{v_{k} -x_*}_{\E{Z}^\dagger}^2 + \gamma \tau \norm{x_{k} -x_*}^2 } 
 \leq \rho^k \E{\norm{v_0 -x_*}_{\E{Z}^\dagger}^2 + \gamma \tau\norm{x_0-x_*}^2}.
\end{align*}
(or put differently):
\begin{align*}
 &\E{\norm{v_{k} -x_*}_{\E{Z}^\dagger}^2 + (1-\alpha)\gamma \norm{x_{k} -x_*}^2 } 
 \leq \rho^k \E{\norm{v_0 -x_*}_{\E{Z}^\dagger}^2 + (1-\alpha)\gamma \norm{x_0-x_*}^2}.
\end{align*}
where $\rho = \max\{\beta(s),s\beta(s)\}\leq 1$.
\end{theorem}

We can now exemplify a few special parameter settings.

\begin{example}
For $\alpha = 1$, i.e., if $s \to 0$, we get the rate $\rho = 1-\frac{\mu}{\nu}$ with $\beta = 1-\frac{\mu}{\nu}$, $\gamma = \frac{1}{\nu}$.
\end{example}

\begin{example}
For $\alpha \to 0$, i.e., in the limit $s \to \infty$, we get the rate  $\rho = 1-\frac{\mu}{\nu}$. 
\end{example}

\begin{example}
The rate $\rho$ is minimized for $s=1$, i.e.,  $\beta=1- \sqrt{\frac{\nu}{\mu}}$ and $\gamma = \sqrt{\frac{1}{\mu \nu}}$; recovering Theorem~\ref{theo:conv}.
\end{example}

The best case, in terms of convergence rate for both non-unit stepsize and a variable parameter choice happened to be the default parameter setup. The non-optimal parameter choice was studied in order to have theoretical guarantees for a wider class of parameters, as in practice one might be forced to rely on sub-optimal / inexact parameter choices.

\section{Proof of Theorem~\ref{thm:omega}} \label{sec:89g8f9009jJJJ}

The proof follows by slight modifications of the proof of Theorem~\ref{theo:conv}.

First we adapt Lemma~\ref{lemma:threepoint}. As we have $x_{k+1} - x_* = (1-\omega Z_k)(y_k - x_*)$ the following statement follows by the same arguments as in the proof of Lemma~\ref{lemma:threepoint}.
\begin{lemma}[Lemma~\ref{lemma:threepoint}']
\label{lemma:mod_lemma}
\begin{equation}
 \eta \norm{y_k-x_*}_{\E{Z}}^2 = \norm{y_k-x_*}^2 -\E{\norm{x_{k+1}-x_*}^2 \, | \, y_k}
 \end{equation}
\end{lemma}
\begin{proof}
\begin{eqnarray*}
\E{\norm{x_{k+1}-x_*}^2 \, | \, y_k} &= & \E{\norm{(I-Z_k)( y_k -x_*)}^2 \, | \, y_k} \\
&= & \E{\dotprod{(I-\omega Z_k)( y_k -x_*),(I-\omega Z_k)y_k -x_*}}\\
&= &\norm{y_k -x_*}^2 - \eta \norm{y_k -x_*}^2_{\E{Z}}. 
\end{eqnarray*} 
\end{proof}

We now follow the same steps as in proof of Theorem~\ref{theo:conv} in Section~\ref{subsec:proof_main_thm}.
We observe, that the first time Lemma~\ref{lemma:threepoint} is applied is in equation~\eqref{eq:continue_omega_proof}. Using Lemma~\ref{lemma:mod_lemma} instead, gives
\begin{eqnarray}
\E{r_{k+1}^2 \, | \, y_k, v_k, x_k} & \leq &  
\beta r_k^2 + \frac{1-\beta}{\mu} \norm{y_k-x_*}^2 + \frac{\gamma^2  \nu}{\eta} \left( \norm{y_k-x_*}^2 -\E{\norm{x_{k+1}-x_*}^2 \, | \, y_k}\right) \nonumber \\
& & +2\gamma \left(-\norm{y_k - x_*}^2  + \beta \frac{1-\alpha}{2\alpha}\left(\norm{x_k-x_*}^2-\norm{y_k- x_*}^2\right)\right).
\end{eqnarray}
Therefore we have that
\begin{eqnarray}
\E{r_{k+1}^2  +\gamma^2  \nu\norm{x_{k+1}-x_*}^2\, | \, y_k, v_k, x_k }
& \leq & 
\beta \left(r_k^2 + \underbrace{\gamma\frac{1-\alpha}{\alpha}}_{P_1'}\norm{x_k-x_*}^2\right) \nonumber \\
& &+\left( \underbrace{\frac{1-\beta}{\mu} -2\gamma+ \frac{\gamma^2  \nu}{\eta} - \beta \gamma\frac{1-\alpha}{\alpha}}_{P_2'}\right)\norm{y_k-x_*}^2.    \nonumber 
\end{eqnarray}
Noting that $\frac{1-\alpha}{\alpha}= \gamma \nu$ and $\frac{\gamma^2 \nu}{\eta} = \frac{\gamma(1-\alpha)}{\eta \alpha} = \frac{1}{\mu}$, we observe $P_2' = 0$ and deduce the statement of Theorem~\ref{thm:omega}.

\section{Proof of Theorem~\ref{thm:family}} \label{sec:j98**JJ*(}

It suffices to study equation~\eqref{eq:continue_omega_proof}. We observe that for convergence the big bracket, $P_2$, should be negative,
\begin{align}
(1-\beta) \frac{1}{\mu} + \gamma^2 \nu - 2\gamma - \gamma \beta \frac{1-\alpha}{\alpha} \leq 0 \label{eq:cond1}
\end{align} 
The convergence rate is then
\begin{align}
\rho \eqdef \max \left\{\beta, \frac{(1-\alpha)\beta}{\alpha \gamma \nu}\right\}\,. \label{eq:rate}
\end{align}
or in the notation of Theorem~\ref{thm:family}, $\rho = \max\{\beta,s\beta\}$.

This means, that in order to obtain the best convergence rate, we should therefore choose parameters $\beta$ and $\gamma$ such that $\beta$ is as small as possible. This observation is true regardless of the value of $s$ (which itself depends on $\gamma$).

With the notation $\tau=s \gamma \beta$, we reformulate~\eqref{eq:cond1} to obtain
\begin{align}
\frac{1}{\mu} + \gamma^2 \nu - 2 \gamma \leq \beta \left( \frac{1}{\mu}  + s \gamma^2 \nu \right)
\end{align}
Thus we see, that $\beta$ cannot be chosen smaller than
\begin{align}
\beta^\star(s,\gamma) = \frac{1 + \mu \gamma^2 \nu - 2 \mu \gamma}{1 + s \mu \gamma^2 \nu}
\end{align}
Minimizing this expression in $\gamma$ gives
\begin{align}
\beta^\star(s) = \frac{1+s - s \sqrt{\frac{\nu + 4\mu s - 2\nu s + \nu s^2}{\nu s^2}}}{2s} \label{eq:beta_star}
\end{align}
with $\gamma^\star(s) = \frac{1}{(1-s \beta^\star(s))\nu}$.

We further observe that this parameter setting indeed guarantees convergence, i.e. $\rho \leq 1$. From~\eqref{eq:beta_star} we observe ($\nu > 0$, $s \geq 0$, $\mu \geq 0$):
\begin{align}
\beta^\star(s) \leq  \frac{1+s - \sqrt{\frac{\nu - 2\nu s + \nu s^2}{\nu}}}{2s} = \frac{1+s - (s-1)}{2s} = \frac{1}{s}
\end{align}
Hence $ s \beta^\star(s)  \leq 1$. On the other hand, $(1-s) \leq \sqrt{(1-s)^2 + \frac{4\mu s}{\nu}}$ and hence $(1+s) -\sqrt{(1-s)^2 + \frac{4\mu s}{\nu}} \leq 2s$, which shows $\beta^\star (s) \leq 1$.

\section{Proofs and Further Comments on Section~\ref{sec:asqn_pap}}

\subsection{Proof of Theorem~\ref{theo:qn}}

We perform a change of coordinates since it is easier to work with the standard Frobenius norm as opposed to the weighted Frobenius norm. Let $\hat{X} = A^{1/2}X A^{1/2}$  
so that~\eqref{eq:primalqN} and~\eqref{eq:qunac} become
\begin{equation} \label{eq:primalqNH}
\hat{X}_* \eqdef I=   \arg\min \norm{\hat{X}}_{F}^2 \quad \mbox{subject to} \quad  \hat{X} = I, \quad \hat{X} = \hat{X}^\top,
\end{equation}
and
\begin{equation}\label{eq:qunacI}
\hat{X}_{k+1}  =P+ \left(I-P\right) \hat{X}_{k}\left(I -P \right),
\end{equation}
respectively, where $P = A^{1/2}S(S^\top AS)^{-1}S^\top A^{1/2}.$  The linear operator that encodes the constaint in~\eqref{eq:primalqNX} is given by
$\hat{\A}(X) = \left( X, \, X - X^\top\right)$ the adjoint of which is given by
$\hat{\A}^*(Y_1,Y_2) = Y_1 + Y_2-Y_2^\top.$ Since
$\hat{\A}^*$ is clearly surjective, it follows that $\Range{\hat{\A}^*} = \R^{n \times n}$. 

Subtracting the identity matrix from both sides of~\eqref{eq:qunacI} and using that $P$ is a projection matrix, we have that
\begin{equation}\label{eq:qunacIres}
\hat{X}_{k+1} -I =\left(I-P\right) (\hat{X}_{k}-I)\left(I -P \right).
\end{equation}
To determine the $Z$ operator~\eqref{eq:Z}, from~\eqref{eq:IZprojres} and~\eqref{eq:qunacIres} we know that
\[ \left(I-P\right) (\hat{X}_{k}-I)\left(I -P \right)= (I - Z) (\hat{X}_k -I).\]
Thus for every matrix $X \in \R^{n \times n}$ we have that
\begin{equation} \label{eq:Zqn}
 Z(X) = X - \left(I-P\right) X\left(I -P \right) = XP + PX(I-P).\end{equation}
Denote column-wise vectorization of $X$ as $x$: $x\eqdef \Vect{X}$. To calculate a useful lower bound on~$\mu$, note that
\begin{eqnarray} \label{eq:XZXlower}
 \Tr{X^\top Z(X)} &= &\Tr{X^\top XP} + \Tr{X^\top PX(I-P)}\nonumber 
 \\
 &= &
 x^\top \Vect{XP}+x^\top \Vect{ PX(I-P)}\nonumber 
 \\
  &= &
 x^\top (P\otimes I)x +x^\top((I-P)\otimes P)x \nonumber 
 \\
& \overset{\eqref{eq:bigz}}{= } &
   x^\top \bigZ x ,
   \label{eq:XZXeq}
\end{eqnarray} 
where we used that $\Tr{A^\top B}=\Vect{A}^\top \Vect{B}$ and $\Vect{ AXB}=(B^\top \otimes A) \Vect{x}$ holds for any $A,B,X$.

Consequently, $\mu$ is equal to
\[
\mu \overset{\eqref{eq:mu}}{=}  \inf_{X \in \R^{n\times n}} \frac{\dotprod{\E{Z}X,X}_F}{\norm{X}_F^2}
 \overset{\eqref{eq:XZXeq}}{= }  \inf_{x \in \R^{n^2\times n^2}} \frac{   x^\top \E{\bigZ} x  }{x^\top x} 
 = \lambda_{\min}(\E{\bigZ}).
\]
Notice that we have $2\lambda_{\min}(\E{P}) \geq\lambda_{\min}(\E{\bigZ})\geq \lambda_{\min}(\E{P})$ since $(P\otimes I)+(I\otimes P)\geq \bigZ\geq (P\otimes I)$.

In light of Algorithm~\ref{alg:SketchJac}, the iterates of the accelerated version of~\eqref{eq:qunacI} are given by
\begin{eqnarray}
 \hat{Y}_k &=& \alpha \hat{V}_k + (1-\alpha) \hat{X}_k \nonumber\\
 \hat{G}_k & =& Z_k(\hat{Y}_k - I ) \nonumber\\
\hat{X}_{k+1} &=& \hat{Y}_k - \hat{G}_k \nonumber\\
\hat{V}_{k+1} &=& \beta \hat{V}_k +(1-\beta)\hat{Y}_k - \gamma \hat{G}_k \label{eq:qunacacc}\
\end{eqnarray}
where $\hat{Y}_k,\hat{V}_{k},\hat{G} \in \R^{n\times n}.$ From Theorem~\ref{theo:conv} we have that $\hat{V}_k$ and $\hat{X}_k$ converge to the identity matrix according to
\begin{equation} \label{eq:aas89hah}
 \E{\norm{\hat{V}_{k+1} -I}_{\E{Z}^\dagger}^2 +\frac{1}{\mu}\norm{\hat{X}_{k+1} -I}^2_{F} } \leq \left(1 - \sqrt{\frac{\mu}{\nu}} \right) \E{\norm{\hat{V}_k -I}_{\E{Z}^\dagger}^2 + \frac{1}{\mu}\norm{\hat{X}_k-I}^2_{F}},
 \end{equation}
where $\norm{X}_{\E{Z}^\dagger}^2 = \dotprod{\E{Z}^\dagger X, X}_F.$
Changing coordinates back to $ \hat{X}_k = A^{1/2}X_k A^{1/2}$ and defining
$Y_k \eqdef A^{-1/2} \hat{Y}_k A^{-1/2}$, $V_k \eqdef A^{-1/2} \hat{V}_k A^{-1/2}$ and $ G_k \eqdef A^{-1/2} \hat{G}_k A^{-1/2}$, we have that~\eqref{eq:aas89hah} gives~\eqref{eq:qnaccconv}. Furthermore, using the same coordinate change applied to the iterates~\eqref{eq:qunacacc} gives Algorithm~\ref{alg:qn}.

\subsection{Matrix inversion as linear system~\label{sec:alternate}}
Denote $x=\Vect{\mX}$, i.e. $x$ is $n^2$ dimensional vector such that $X_{(n(i-1)+1):ni}=\mX_{:,i}$. Similarly, denote $e=\Vect{\mI}$. System \eqref{eq:system} can be thus rewritten as 
\begin{equation}
 (I\otimes A)x=e.
\label{eq:system_vector}
\end{equation}

Notice that all linear sketches of the original system $\mA \mX=\mI$ can be written as 
\begin{equation}\label{eq:sketch_general}
\gS^\top  (I\otimes A)x=\gS^\top  e
\end{equation}
for a suitable $n^2\times n^2$ matrix $\gS$, therefore the setting is fairly general.

\subsubsection{Alternative proof of Theorem \ref{theo:qn}}
Let us now, for a purpose of this proof, consider sketch matrix $\gS$ to capture only sketching the original matrix system $AX=I$ by left multiplying by $S$, i.e. $\gS=(I\otimes S)$, as those are the considered sketches in the setting of Section~\ref{sec:asqn_pap}. 

As we have 
\[
\Tr{BX^\top BX}=\Vect{BX B}^\top x=x^\top (B\otimes B) x, 
\]
weighted Frobenius norm of matrices is equivalent to a special weighted euclidean norm of vectors. Define also $C$ to be a matrix such that $Cx=0$ if and only if $X=X^\top$. Therefore, \eqref{eq:primalqNX} is equivalent to 

\begin{equation} \label{eq:primalqNXvec}
x_{k+1}=   \arg\min \norm{x - x_k}_{A\otimes A}^2 \quad \mbox{subject to} \quad ( I\otimes S^\top) (I\otimes A) x = ( I\otimes S^\top) e, \quad Cx = 0,
\end{equation}
which is a sketch-and-project method applied on the linear system, with update as per \eqref{eq:qunac}:
\[
x^{k+1}=x^k-(H\otimes I)((I\otimes A) x -e ) -(I\otimes H)((I\otimes A) x -e )  + (HA\otimes H)  ((I\otimes A) x -e )
\]
for 
$H\eqdef S\left(S^\top AS\right)^{-1}S^\top.$
Using substitution $\hat{x}=(A^{\frac12}\otimes A^\frac12)x; \hat{S}=A^{\frac12}S$ and comparing to \eqref{eq:IZprojres}, we get 
\[
Z=I\otimes I-(I-P)\otimes(I-P)
\]
for $P$ as defined inside the statement of Theorem~\ref{theo:qn}. 
Therefore, we have all necessary information to apply the results from \cite{MartinRichtarikAccell}, recovering Theorem \ref{theo:qn}.

\section{Linear Operators in Euclidean Spaces}
Here we provide some technical lemmas and results for linear operators in Euclidean space, that we used in the main body of the paper. Most of these results can be found in standard textbooks of analysis, such as~\cite{pedersen1996}. We give them here for completion.
  
Let $\cX, \cY, \cZ$ be Euclidean spaces, equipped with inner products. Formally, we should use a notation that distinguishes the inner product in each space. But instead we use $\dotprod{\cdot,\,\cdot}$ to denote the inner product on all spaces, as it will be easy to determine from which space the elements are in. That is, for $x_1,x_2 \in \cX$, we denote by $\dotprod{x_1,x_2}$ the inner product between $x_1$ and $x_2$ in $\cX.$

Let
\[\norm{T} \eqdef \sup_{\norm{x}\leq 1} \norm{T x},\]
denote the operator norm of $T$.  Let $0 \in L(\cX,\cY)$ denote the zero operator and $I \in L(\cX,\cY)$ the identity map.

\paragraph{The adjoint.}
Let $T^* \in L( \cY , \cX)$ denote the unique operator that satisfies
\[\dotprod{Tx,y} = \dotprod{x,T^*y},\]
for all $x \in \cX$ and $y \in \cY.$ We say that $T^*$ is the \emph{adjoint} of $T$. We say $T$ is \emph{self-adjoint} if $T =T^*.$
Since for all $x\in \cX$ and $s\in \cS$,
\[ \langle x , (ST)^*s \rangle = \langle ST  x , s \rangle_\cS = \langle T  x , S^* s \rangle_\cY =  \langle  x , T^* S^* s \rangle,\]
we have \[(ST)^* = T^* S^*.\]

\begin{lemma} \label{lem:decomp}For $T \in L(\cX,\cY)$ we have that $\Range{T^*}^{\perp} = \Null{T}.$ Thus
	\begin{eqnarray}
	\cX &= & \Range{T^*} \oplus \Null{T}\\
	\cY &= & \Range{T} \oplus \Null{T^*}
	\end{eqnarray}
\end{lemma}
\begin{proof}
	See 3.2.6 in~\cite{pedersen1996}.
\end{proof}

\subsection{Positive Operators}
We say that $G\in L(\cX)$ is positive if it is self-adjoint and if
$\dotprod{x,Gx} \geq 0$ for all $ x \in \cX$.
Let $(e_j)_{j=1}^{\infty} \in \cX$ be an orthonormal basis. The trace of $G$ is defined as
\begin{equation}\label{eq:tracedef}
\Tr{G} \eqdef \sum_{j=1}^{\infty} \dotprod{G e_j, e_j}.
\end{equation}
The definition of trace is independent of the choice of basis due to the following lemma.
\begin{lemma} If $U$ is unitary and $G\geq 0$ then
$\Tr{UGU^*} = \Tr{G}.$
\end{lemma}
\begin{proof}
	See 3.4.3 and 3.4.4 in~\cite{pedersen1996}.
\end{proof}
\begin{lemma} \label{lem:projrank} If $P \in L(\cX)$ is a projection matrix then
	$\Tr{P} = \dim(\Range{P}) = \Rank{P}.$
\end{lemma}
\begin{proof}  Let $d = \dim(\Range{P})$ which is possibly infinite.
	Given that $P$ is a projection we have that $\Range{P}$ is a closed subspace and thus there exists orthonormal basis  $(e_j)_{j=1}^{d} $ of $\Range{P}$. Consequently,
	$\Tr{P} \overset{\eqref{eq:tracedef}}{=} \sum_{j=1}^{d} 1 = d= \dim(\Range{P}).$
\end{proof}

A \emph{square root} of an operator $G \in L(\cX)$ is an operator $R\in L(\cX) $ such that $R^2 =G.$
\begin{lemma} \label{lem:squareroot}
	If $G: \cX \rightarrow \cX$ is positive, then there exists a unique positive square root of $G$ which we denote by $G^{1/2}.$ 
\end{lemma}
\begin{proof}
	See 3.2.11 in~\cite{pedersen1996}. 
\end{proof}

\begin{lemma} \label{lem:decompTGT} For any $T \in L(\cX,\cY)$ and any $G \in L(\cY,\cY)$ that is positive and injective,
	\begin{equation} \label{eq:NullTGT}
	\Null{T} =  \Null{T^*GT},
	\end{equation}
	and
	\begin{equation}\label{eq:RangeTGT}
	\overline{\Range{T^*}} =  \overline{\Range{T^*GT}}.
	\end{equation}
\end{lemma}
\begin{proof}
	The inclusion $\Null{T} \subset \Null{T^*GT}$ is immediate. For the opposite inclusion, let $x \in \Null{T^*GT}.$ Since $G$ is positive we have by Lemma~\ref{lem:squareroot} that there exists a square root with $G^{1/2}G^{1/2} =G.$ 
Therefore,
	$\dotprod{x,T^*GT x} = \dotprod{G^{1/2}Tx,G^{1/2}Tx} =0,$
	which implies that $G^{1/2}Tx =0$. Since $G$ is injective, it follows that $G^{1/2}$ is injective and thus $x \in \Null{T}$. Finally~\eqref{eq:RangeTGT} follows by taking the orthogonal complements of~\eqref{eq:NullTGT} and observing Lemma~\ref{lem:decomp}. 
\end{proof}

As an immediate consequence of~\eqref{eq:NullTGT} and~\eqref{eq:RangeTGT} we have the following lemma.
\begin{corollary} For $G: \cX \rightarrow \cX$ positive we have that
	\begin{eqnarray}
	\Null{G^{1/2}} & =& \Null{G} \label{eq:NullGhalf}\\
	\overline{\Range{G^{1/2}}} & =& \overline{\Range{G}}  \label{eq:RangeGhalf}
	\end{eqnarray}
\end{corollary}

\subsection{Pseudoinverse}
 For a bounded linear operator $T$ define the pseudoinverse of $T$ as follows.
\begin{definition} \label{def:pseudoinverse} Let $T \in L(\cX,\cY)$ such that $\Range{T}$ is closed. $T^{\dagger}: \cY \rightarrow \cX$ is said to be the pseudoinverse if
\begin{enumerate}[i)]
\item $T^{\dagger} Tx = x $ for all $x \in \Range{T^* }.$
\item $T^{\dagger} x = 0$ for all $x \in \Null{T^* }.$
\item If $x \in \Null{T}$ and $y \in \Range{T^* }$ then $T^{\dagger}(x+y)=T^{\dagger}x +T^{\dagger}y.$
\end{enumerate}
\end{definition}
It follows directly from the definition (see~\cite{Desoer1963} for details) that $T^{\dagger}$ is a unique bounded linear operator. The following properties of pseudoinverse will be important.

\begin{lemma}[Properties of pseudoinverse] \label{lem:pseudo} Let $T \in L(\cX,\cY)$ such that $\Range{T}$ is closed. It follows that 
\begin{enumerate}[i)]
\item $ T T^\dagger T = T$ \label{it:pseudoTTdagT}
\item $\Range{T^{\dagger}} = \Range{T^*}$ and $\Null{T^{\dagger}} = \Null{T^*}$\label{it:pseudorange}
\item $(T^*)^{\dagger} = (T^{\dagger})^*$ \label{it:pseudoadjoint}
\item If $T$ is self-adjoint and positive then $T^{\dagger}$ is self-adjoint and positive.  \label{it:pseudoadjpos}
\item   $T^{\dagger}T T^* = T^*$, that is, $T^{\dagger}T$ projects orthogonally onto $ \Range{T^* }$ and along $\Null{T}.$  \label{it:pseudoproj}
\item Consider the linear system $Tx=d$ where $d \in \Range{T}$. It follows that
\begin{equation}
\textstyle
\label{eq:pseudoleastnorm}T^{\dagger}d = \arg \min_{x \in \cX} \tfrac{1}{2}\norm{x}^2 \quad \mbox{subject to} \quad Tx=d.
\end{equation}     \label{it:pseudoleastnorm}
\item $T^{\dagger} = T^* (TT^*)^{\dagger}$ \label{it:pseudoTadjTTadj}
\end{enumerate}
\end{lemma}
\begin{proof}
The proof of items~{\it \ref{it:pseudoTTdagT}, \ref{it:pseudorange}, \ref{it:pseudoadjoint}, \ref{it:pseudoadjpos}, \ref{it:pseudoproj}} can be found in~\cite{Desoer1963}. The proof of item~{\it\ref{it:pseudoleastnorm}} is alternative characterization of the pseudoinverse and it can be established by using that $d \in \Range{T}$ together with item~{\it\ref{it:pseudoTTdagT}} thus $TT^{\dagger}d =d$. The proof then follows by using the orthogonal decomposition $\Range{T^*} \oplus \Null{T}$ to show that $T^{\dagger}d $ is indeed the minimum of~\eqref{eq:pseudoleastnorm}. Finally item~\eqref{it:pseudoTadjTTadj} is a direct consequence of the previous items.
\end{proof}

\end{document}

%% file: stuff.tex
\usepackage{mdframed} 
\usepackage{thmtools}

\definecolor{shadecolor}{gray}{0.95}
\declaretheoremstyle[
headfont=\normalfont\bfseries,
notefont=\mdseries, notebraces={(}{)},
bodyfont=\normalfont,
postheadspace=0.5em,
spaceabove=1pt,
mdframed={
  skipabove=8pt,
  skipbelow=8pt,
  hidealllines=true,
  backgroundcolor={shadecolor},
  innerleftmargin=4pt,
  innerrightmargin=4pt}
]{shaded}

\usepackage{enumerate}
\usepackage{fullpage}
\usepackage{microtype}
\usepackage{graphicx}
 \usepackage{subfigure}
 \usepackage{booktabs} 

\usepackage{amsfonts,amssymb,amsmath}
\usepackage{amsthm, verbatim}

\usepackage[colorinlistoftodos,bordercolor=orange,backgroundcolor=orange!20,linecolor=orange,textsize=scriptsize]{todonotes}

\usepackage{tikz}
\usetikzlibrary{decorations.pathreplacing,calc,positioning,arrows,shapes}

\usepackage{graphicx}
\usepackage{xcolor}

\usepackage[normalem]{ulem}
\usepackage[makeroom]{cancel}
\usepackage{listings}

\graphicspath{{./figures/}}  
\newcommand{\R}{\mathbb{R}}

\newcommand{\N}{\mathbb{N}}

\newcommand{\A}{\mathcal{A}}

  \newcommand{\argmin}{\text{\rm argmin}}
 \newcommand{\eqdef}{\overset{\text{def}}{=}} 
\providecommand{\Null}[1]{\mathbf{Null}\left( #1\right)}
\providecommand{\Rank}[1]{\mathbf{Rank}\left( #1\right)}

\providecommand{\Range}[1]{\mathbf{Range}\left( #1\right)}

\providecommand{\Tr}[1]{\mathbf{Tr}\left( #1\right)}

\providecommand{\diag}[1]{\mathbf{Diag}\left( #1\right)}
\providecommand{\E}[1]{\mathbf{E}\left[ #1\right]}

  \providecommand{\dotprod}[1]{\langle #1\rangle} 

\providecommand{\norm}[1]{\lVert#1\rVert}

\newtheorem{lemma}{Lemma}
\newtheorem{corollary}[lemma]{Corollary}

\newtheorem{definition}[lemma]{Definition}

\newtheorem{theorem}[lemma]{Theorem}
\newtheorem{example}[lemma]{Example}

\newcommand{\cD}{{\mathcal D}}

\newcommand{\cN}{{\mathcal N}}

\newcommand{\cS}{{\mathcal S}}

\newcommand{\cX}{{\mathcal X}}
\newcommand{\cY}{{\mathcal Y}}

\newcommand{\cZ}{{\mathcal Z}}

\newcommand{\mA}{{ A}}

\newcommand{\mI}{{I}}

\newcommand{\mS}{{ S}}

\newcommand{\mX}{{X}}

\newcommand{\bigZ}{{ \mathbf{Z}}}

\newcommand{\gS}{ {\mS_0}}

\providecommand{\Vect}[1]{\mathbf{Vec}\left( #1\right)}

\usepackage[colorlinks=true,linkcolor=blue]{hyperref}

%% file: Arxiv3.bbl
\begin{thebibliography}{10}

\bibitem{agarwal2017second}
Naman Agarwal, Brian Bullins, and Elad Hazan.
\newblock Second-order stochastic optimization for machine learning in linear
  time.
\newblock {\em The Journal of Machine Learning Research}, 18(1):4148--4187,
  2017.

\bibitem{BerahasBN17}
Albert~S. Berahas, Raghu Bollapragada, and Jorge Nocedal.
\newblock An investigation of {N}ewton-sketch and subsampled {N}ewton methods.
\newblock {\em CoRR}, abs/1705.06211, 2017.

\bibitem{broyden1967quasi}
Charles~G Broyden.
\newblock Quasi-{N}ewton methods and their application to function
  minimisation.
\newblock {\em Mathematics of Computation}, 21(99):368--381, 1967.

\bibitem{Byrd2015}
Richard~H. Byrd, S.~L. Hansen, Jorge Nocedal, and Yoram Singer.
\newblock A stochastic quasi-newton method for large-scale optimization.
\newblock {\em {SIAM} Journal on Optimization}, 26(2):1008--1031, 2016.

\bibitem{Chang2011}
Chih-Chung Chang and Chih-Jen Lin.
\newblock Libsvm: A library for support vector machines.
\newblock {\em ACM Trans. Intell. Syst. Technol.}, 2(3):27:1--27:27, May 2011.

\bibitem{Desoer1963}
C.~A. Desoer and B.~H. Whalen.
\newblock {A} note on pseudoinverses.
\newblock {\em Journal of the Society of Industrial and Applied Mathematics},
  11(2):442--447, 1963.

\bibitem{fletcher1970new}
Roger Fletcher.
\newblock A new approach to variable metric algorithms.
\newblock {\em The computer journal}, 13(3):317--322, 1970.

\bibitem{goldfarb1970family}
Donald Goldfarb.
\newblock A family of variable-metric methods derived by variational means.
\newblock {\em Mathematics of computation}, 24(109):23--26, 1970.

\bibitem{gower2016stochastic}
Robert Gower, Donald Goldfarb, and Peter Richt{\'a}rik.
\newblock Stochastic block {B}{F}{G}{S}: Squeezing more curvature out of data.
\newblock In {\em International Conference on Machine Learning}, pages
  1869--1878, 2016.

\bibitem{Gower2015c}
Robert~M. Gower and Peter Richt{\'{a}}rik.
\newblock Stochastic dual ascent for solving linear systems.
\newblock {\em arXiv:1512.06890}, 2015.

\bibitem{Gower:2017}
Robert~M. Gower and Peter Richt{\'a}rik.
\newblock Randomized quasi-{N}ewton updates are linearly convergent matrix
  inversion algorithms.
\newblock {\em SIAM Journal on Matrix Analysis and Applications},
  38(4):1380--1409, 2017.

\bibitem{Gower2015}
Robert~Mansel Gower and Peter Richt{\'{a}}rik.
\newblock Randomized iterative methods for linear systems.
\newblock {\em SIAM Journal on Matrix Analysis and Applications},
  36(4):1660--1690, 2015.

\bibitem{K-1937}
S.~Kaczmarz.
\newblock {{A}ngen\"{a}herte {A}ufl\"{o}sung von {S}ystemen linearer
  {G}leichungen}.
\newblock {\em Bulletin International de l'Acad\'{e}mie Polonaise des Sciences
  et des Lettres}, 35:355--357, 1937.

\bibitem{liu1989limited}
Dong~C Liu and Jorge Nocedal.
\newblock On the limited memory {B}{F}{G}{S} method for large scale
  optimization.
\newblock {\em Mathematical programming}, 45(1-3):503--528, 1989.

\bibitem{Liu:2016}
Ji~Liu and Stephen~J. Wright.
\newblock An accelerated randomized {K}aczmarz algorithm.
\newblock {\em Math. Comput.}, 85(297):153--178, 2016.

\bibitem{loizou2017momentum}
Nicolas Loizou and Peter Richt{\'a}rik.
\newblock Momentum and stochastic momentum for stochastic gradient, {N}ewton,
  proximal point and subspace descent methods.
\newblock {\em arXiv preprint arXiv:1712.09677}, 2017.

\bibitem{Mokhtari2014}
Aryan Mokhtari and Alejandro Ribeiro.
\newblock Global convergence of online limited memory {BFGS}.
\newblock {\em The Journal of Machine Learning Research}, 16:3151--3181, 2015.

\bibitem{moritz2016linearly}
Philipp Moritz, Robert Nishihara, and Michael Jordan.
\newblock A linearly-convergent stochastic {L}-{B}{F}{G}{S} algorithm.
\newblock In {\em Artificial Intelligence and Statistics}, pages 249--258,
  2016.

\bibitem{Nesterov12}
Yu. Nesterov.
\newblock Efficiency of coordinate descent methods on huge-scale optimization
  problems.
\newblock {\em SIAM Journal on Optimization}, 22(2):341--362, 2012.

\bibitem{nesterov1983method}
Yurii Nesterov.
\newblock A method of solving a convex programming problem with convergence
  rate ${O}(1/k^2)$.
\newblock {\em Soviet Mathematics Doklady}, 27(2):372--376, 1983.

\bibitem{Nesterov:2017}
Yurii Nesterov and Sebastian~U. Stich.
\newblock Efficiency of the accelerated coordinate descent method on structured
  optimization problems.
\newblock {\em SIAM Journal on Optimization}, 27(1):110--123, 2017.

\bibitem{pedersen1996}
G.K. Pedersen.
\newblock {\em Analysis Now}.
\newblock Graduate Texts in Mathematics. Springer New York, 1996.

\bibitem{PilanciW17}
Mert Pilanci and Martin~J. Wainwright.
\newblock Newton sketch: {A} near linear-time optimization algorithm with
  linear-quadratic convergence.
\newblock {\em {SIAM} Journal on Optimization}, 27(1):205--245, 2017.

\bibitem{MartinRichtarikAccell}
Peter Richt{\'a}rik and Martin Tak\'{a}\v{c}.
\newblock Stochastic reformulations of linear systems: accelerated method.
\newblock {\em Manuscript, October 2017}, 2017.

\bibitem{RT2017_stoch_reformulations}
Peter Richt{\'a}rik and Martin Tak\'{a}\v{c}.
\newblock Stochastic reformulations of linear systems: algorithms and
  convergence theory.
\newblock {\em arXiv:1706.01108}, 2017.

\bibitem{Schraudolph2007}
Nicol~N Schraudolph and G~Simon.
\newblock A stochastic quasi-{Newton} method for online convex optimization.
\newblock {\em In Proceedings of 11th International Conference on Artificial
  Intelligence and Statistics}, 2007.

\bibitem{shanno1970conditioning}
David~F Shanno.
\newblock Conditioning of quasi-{N}ewton methods for function minimization.
\newblock {\em Mathematics of computation}, 24(111):647--656, 1970.

\bibitem{Stich2016}
S.~U. Stich, C.~L. M{\"u}ller, and B.~G{\"a}rtner.
\newblock Variable metric random pursuit.
\newblock {\em Mathematical Programming}, 156(1):549--579, Mar 2016.

\bibitem{Stich14}
Sebastian~U. Stich.
\newblock {\em Convex Optimization with Random Pursuit}.
\newblock PhD thesis, ETH Zurich, 2014.
\newblock Diss., Eidgenössische Technische Hochschule ETH Zürich, Nr. 22111.

\bibitem{strohmer2009randomized}
Thomas Strohmer and Roman Vershynin.
\newblock A randomized {K}aczmarz algorithm with exponential convergence.
\newblock {\em Journal of Fourier Analysis and Applications}, 15(2):262, 2009.

\bibitem{TuVWGJR17}
Stephen Tu, Shivaram Venkataraman, Ashia~C. Wilson, Alex Gittens, Michael~I.
  Jordan, and Benjamin Recht.
\newblock Breaking locality accelerates block {G}auss-{S}eidel.
\newblock In {\em Proceedings of the 34th International Conference on Machine
  Learning, {ICML} 2017, Sydney, NSW, Australia, 6-11 August 2017}, pages
  3482--3491, 2017.

\bibitem{Wang2014}
Xiao Wang, Shiqian Ma, Donald Goldfarb, and Wei Liu.
\newblock Stochastic quasi-newton methods for nonconvex stochastic
  optimization.
\newblock {\em {SIAM} Journal on Optimization}, 27(2):927--956, 2017.

\bibitem{wang2017stochastic}
Xiao Wang, Shiqian Ma, Donald Goldfarb, and Wei Liu.
\newblock Stochastic quasi-{N}ewton methods for nonconvex stochastic
  optimization.
\newblock {\em SIAM Journal on Optimization}, 27(2):927--956, 2017.

\bibitem{Wright:2015}
Stephen~J. Wright.
\newblock Coordinate descent algorithms.
\newblock {\em Math. Program.}, 151(1):3--34, June 2015.

\bibitem{Xu2016}
Peng Xu, Jiyan Yang, Farbod Roosta-Khorasani, Christopher R{\'e}, and Michael~W
  Mahoney.
\newblock Sub-sampled newton methods with non-uniform sampling.
\newblock In {\em Advances in Neural Information Processing Systems}, pages
  3000--3008, 2016.

\end{thebibliography}
